\documentclass[11pt]{article}
\usepackage{amssymb}
\usepackage{amsmath}
\usepackage{amsthm}
\usepackage{mathrsfs}
\usepackage{mathtools}
\usepackage{color}
\usepackage[dvipdfmx]{hyperref}
\usepackage{xcolor}
\usepackage{yhmath}
\usepackage{cite}
\usepackage{cases}
\usepackage[normalem]{ulem}
\usepackage[scale=0.80]{geometry}

\allowdisplaybreaks[1]

\theoremstyle{plain}
\newtheorem{thm}{Theorem}[section]
\newtheorem{lem}[thm]{Lemma}
\newtheorem{cor}[thm]{Corollary}
\newtheorem{pro}[thm]{Proposition}

\theoremstyle{definition}
\newtheorem{rem}[thm]{Remark}
\newtheorem{dfi}[thm]{Definition}
\newtheorem{exm}[thm]{Example}

\newcommand{\ZZ}{\mathbb{Z}}
\newcommand{\E}{\mathcal{E}}
\newcommand{\NN}{{\mathbb N}}
\newcommand{\QQ}{\mathbb{Q}}
\newcommand{\CC}{\mathbb{C}}
\newcommand{\rD}{\mathrm{D}}
\newcommand{\D}{\mathcal{D}}
\newcommand{\DD}{\mathbb{D}}
\newcommand{\ep}{\varepsilon}
\newcommand{\bw}{\overline{w}}

\newcommand{\w}{\overline{w}^1}
\newcommand{\ww}{\overline{w}^2}
\newcommand{\wN}{\overline{w^N}^1}
\newcommand{\wwN}{\overline{w^N}^2}

\newcommand{\RR}{{\mathbb R}}

\newcommand{\tr}{{\rm tr}}
\newcommand{\la}{\lambda}

\newcommand{\FC}{\mathcal{FC}^{\infty}_b}

\newcommand{\g}{\mathfrak{g}}
\newcommand{\ad}{{\rm ad}}
\newcommand{\Pe}{P_e(G)}
\newcommand{\PeH}{P_e(G)^H}
\newcommand{\Peomega}{P_e(G)^{\Omega}}
\newcommand{\Pea}{P_{e,a}(G)}
\newcommand{\PeaH}{P_{e,a}(G)^H}
\newcommand{\Peaomega}{P_{e,a}(G)^{\Omega}}

\newcommand{\Ent}{\mathrm{Ent}}

\newcommand{\bX}{\mathbf{X}}
\newcommand{\X}{\overline{X}^1}
\newcommand{\XX}{\overline{X}^2}
\newcommand{\tY}{\tilde{Y}}
\newcommand{\bY}{\mathbf{Y}}
\newcommand{\Y}{\overline{Y}^1}
\newcommand{\YY}{\overline{Y}^2}

\newcommand{\tNk}{\tau^N_k}
\newcommand{\tNkN}{\tau^N_{k-1}}
\newcommand{\tNiN}{\tau^N_{i-1}}
\newcommand{\tNi}{\tau^N_i}
\newcommand{\oTkSa}{\overline{T_kS_a}^W}

\newcommand{\hyp}{\mathchar`-}
\newcommand{\su}{\mathfrak{su}}
\newcommand{\LL}{\mathcal{L}}
\newcommand{\DS}{D_{S_a}}
\newcommand{\LS}{L_{\la, S_a}}
\newcommand{\tLS}{\tilde{L}_{\la, S_a}}
\newcommand{\supp}{\mathrm{supp}}

\newcommand{\tw}{\tilde{w}}
\newcommand{\teta}{\tilde{\eta}}
\newcommand{\nn}{\nonumber}

\newcommand{\bce}{\mathbf{\check{e}}}
\newcommand{\be}{\mathbf{e}}
\newcommand{\YD}{Y^{-1}(\D)}
\newcommand{\rIm}{\mathrm{Im}}
\numberwithin{equation}{section}

\makeatletter
\@namedef{subjclassname@2020}{%
  \textup{2020} Mathematics Subject Classification}
\makeatother

\begin{document}
\newcounter{aaa}
\newcounter{bbb}
\newcounter{ccc}
\newcounter{ddd}
\newcounter{eee}
\newcounter{xxx}
\newcounter{xvi}
\newcounter{x}
\setcounter{aaa}{1}
\setcounter{bbb}{2}
\setcounter{ccc}{3}
\setcounter{ddd}{4}
\setcounter{eee}{32}
\setcounter{xxx}{10}
\setcounter{xvi}{16}
\setcounter{x}{38}
\title{
Asymptotics of lowlying Dirichlet eigenvalues of\\
Witten Laplacians on domains in pinned path groups
}
\author{Shigeki Aida
\thanks{Dedicated to Professor Leonard Gross on the occasion of his 95th birthday.}
\\
Graduate School of Mathematical Sciences\\
The University of Tokyo\\
3-8-1 Komaba Meguro-ku Tokyo, 153-8914, JAPAN\\
e-mail: aida@ms.u-tokyo.ac.jp}
\date{}
\maketitle

\begin{abstract}
 Let $G$ be a compact connected Lie group
and $P_{e,a}(G)=C([0,1]\to G~|~\gamma(0)=e, \gamma(1)=a)$ be
the pinned path space with a pinned Brownian motion measure $\nu_{\lambda,a}$
defined by the heat kernel $p(\lambda^{-1}t,x,y)$, 
where $\lambda$ is a positive parameter.
We consider a Witten Laplacian $-L_{\lambda,\mathcal{D}}$ 
acting on functions with the Dirichlet boundary condition
on a certain domain $\mathcal{D}\subset P_{e,a}(G)$ 
which includes finitely many geodesics $\{l_1,\ldots,l_N\}$
between $e$ and
$a$. $\nu_{\lambda,a}$ has the formal path integral expression 
$\nu_{\lambda,a}(d\gamma)=Z_{\lambda}^{-1}\exp \left(-\lambda E(\gamma)\right)d\gamma$,
where $E(\gamma)=\frac{1}{2}\int_0^1|\dot{\gamma}(t)|^2dt$ and
$E$ is a Morse function when $a$ is not a point of the cut-locus
of $e$.
Hence, by the analogy of 
finite dimensional cases, one may expect that the lowlying spectrum of 
$-\lambda^{-1}L_{\lambda,\mathcal{D}}$
can be approximated by the spectral sets of Ornstein-Uhlenbeck type
operators which approximate
$-\lambda^{-1}L_{\lambda,\mathcal{D}}$ in each small neighborhood of
critical points $\{l_i\}$ when $\lambda\to\infty$.
However, in contrast to the finite dimensional case, the spectral sets of
the approximate Ornstein-Uhlenbeck type operators contain essential spectrum.
It may be difficult to analyze the behavior of the spectrum of
$-\lambda^{-1}L_{\lambda,\mathcal{D}}$ near the set of the essential spectrum.
In this paper, we study the asymptotic behavior of the lowlying 
discrete spectrum of
$-\lambda^{-1}L_{\lambda,\mathcal{D}}$ in the complement of the neighborhood of
the set of essential spectrum
of the approximate Ornstein-Uhlenbeck type operators at $\{l_i\}$
as $\lambda\to\infty$.
\end{abstract}

\tableofcontents

\section{Introduction}

Let $E$ be a smooth function on $\RR^d$.
Let $\la>0$ and for all sufficiently large $\la$, assume
$e^{-\la E}\in L^1(\RR^d,dx)$ and
let $\nu_{\la}(dx)=Z_{\la}^{-1}e^{-\la E(x)}dx$ be the normalized probability measure
on $\RR^d$.
Consider a Dirichlet form $\E_{\la}(f,f)=\int_{\RR^d}|Df(x)|^2d\nu_{\la}(x)$
in $L^2(\RR^d,\nu_{\la})$ and denote the nonnegative generator by $-L_{\la}$.
Under assumptions on $E$ such that
\begin{itemize}
\item[{\rm (i)}] $E$ has finitely many critical points $\{c_1,\ldots,c_N\}$,
\item[{\rm (ii)}] the Hessians $D^2E(c_i)$ are nondegenerate $(1\le i\le N)$
\end{itemize}
and with technical assumptions at infinity, 
there have been many studies on the asymptotics
of the spectrum of $-\la^{-1}L_{\la}$ from several view points
(\cite{begk}, \cite{helffer-nier}) as $\la\to\infty$.
The function $E$ which satisfies (i) and (ii) is called a Morse function.
To understand the relation between the spectrum of $-L_{\la}$ of order
$O(\la)$ and the set of critical points $\{c_i\}$ of $E$,
consider a unitary transformation $M_{\la} : L^2(\RR^d,\nu_{\la})\to L^2(\RR^d,dx)$
defined by $M_{\la}f=\rho_{\la}^{1/2}f$, where $\rho_{\la}=Z_{\la}^{-1}e^{-\la E}$.
Then, we have $M_{\la} (-L_{\la}) M_{\la}^{-1}f=-H_{\la}f$, where
$-H_{\la}=-\Delta+\frac{\la^2}{4}|DE|^2-\frac{\la}{2}\Delta E$.
By the assumption on $E$, $|DE(x)|^2$ is strictly positive except on $\{c_i\}$.
This roughly implies that if $f\in L^2(\RR^d)$ does not vanish outside a
neighborhood of $\{c_i\}$ in an appropriate sense, 
then $(-H_{\la}f,f)$ is very large comparing
the order of $O(\la)$.
Hence, it suffices to consider $f\in L^2(\RR^d,dx)$ whose support is included in 
a neighborhood of $\{c_i\}$ to study the asymptotics of order $O(\la)$ of
the spectrum of $-L_{\la}$.
Furthermore, near $c_i$, we can approximate $-H_{\la}$ by
a Schr\"odinger operator with a quadratic potential function
(quantum harmonic oscillator) and the spectrum is completely known and
the spectrum determines the asymptotic behavior
of the spectrum of $-L_{\la}$ of order $O(\la)$.
We refer the readers for the
above rigorous arguments in the case of Schr\"odinger operators 
to \cite{simon2}.
Moreover, if there are local minimum points of $E$ other than global minimum, 
we can expect the small eigenvalues to be of order $e^{-c \la}$ $(c>0)$
(\cite{begk}, \cite{helffer-nier}).

We are interested in these problems in infinite dimensional cases.
In infinite dimensional space $X$ also,
there exist probability measures $\nu_{\la}$
which have formal path integral representations 
$\nu_{\la}(d\gamma)=Z_{\la}^{-1}\exp(-\la E(\gamma))d\gamma$
and in some cases, we can define Dirichlet forms $(\E_{\la},\rD(\E_{\la}))$ in
$L^2(X,\nu_{\la})$.
However, we cannot use the unitary transformation $M_{\la}$ and
it is not trivial to see the relation between the set of critical points
of $E$ and the spectrum of 
$-L_{\la}$ of the order of $O(\la)$.
For this problem, we used
a semiboundedness theorem for Schr\"odinger operators and
large deviation estimate for the measure $\nu_{\la}$
to study asymptotics of the lowest eigenvalues of
$P(\phi)_2$ type Hamiltonians and the 
asymptotics of the spectral gap of the generators of Dirichlet forms
on pinned path spaces over Riemannian manifolds
in \cite{a2003-2},\cite{a2007}, \cite{a2009}, \cite{a2012, a2015, a2016}.
The semiboundedness theorem is called, {\it e.g.},  a NGS bound or
the Federbush semi-boundedness theorem (\cite{simon}, \cite{g3})
and can be derived under the validity of logarithmic Sobolev inequality
or the hyperboundedness of the semigroup.
Also, we mention that 
there are studies on exponentially small eigenvalues on
domains in pinned path spaces (\cite{eberle1}, \cite{eberle2})
and the sharp exponential estimate 
of the spectral gap of stochastic one-dimensional
Allen-Cahn equation in finite volume (\cite{b-digesu}).

In this paper, we consider a pinned path space over a 
compact connected
Lie group $G$,
$\Pea=C([0,1]\to G~|~\gamma(0)=e, \gamma(1)=a)$
with the pinned Brownian motion measure $\nu_{\la,a}$
which is defined by the heat kernel $p(\la^{-1}t,x,y)$ on
$G$. The measure $\nu_{\la,a}$ has a formal representation
$d\nu_{\la,a}(\gamma)=Z_{\la}^{-1}\exp(-\la E(\gamma))d\gamma$,
where $E(\gamma)=\frac{1}{2}\int_0^1|\dot{\gamma}(t)|^2dt$.
We refer the readers for more precise representation formula and a rigorous
explanation to \cite{ad}.
If $a$ is not a conjugate point of $e$ along any geodesics between
$e$ and $a$, then the set of the geodesics
$\{l_i\}_{i=1}^{\infty}$
between $e$ and $a$ are the set of critical points of $E$ and
the Hessian of $E$ at $l_i$ are nondegenerate (\cite{milnor}),
that is a Morse function.
Also, we can define a Dirichlet form in $L^2(\Pea,\nu_{\la,a})$
by using the $H$-derivative $\nabla$ on $\Pea$.
We call the nonnegative generator $-L_{\la,\Pea}$ a Witten Laplacian.
This is a second order differential operator acting on functions.
Note that original Witten Laplacians are defined as operators acting 
on differential forms on Riemannian manifolds (\cite{witten}).
We consider operators acting on functions only in this paper.
We will make small comments on 
Witten Laplacians acting on differential forms on $\Pea$ in
Remark~\ref{remark on main theorem} (4).
Although $\Pea$ is a curved space, by considering 
an appropriate local coordinate chart in a neighborhood 
of the geodesic $l_i$,
we can construct an Ornstein-Uhlenbeck type operator
$L_{\la,T_{l_i},W_0}$
on a (linear) Wiener space $W_0$, which 
is an approximation to $-L_{\la,\Pea}$ 
in this neighborhood.
Note that the spectral set $-\la^{-1}L_{\la,T_i,W_0}$ is independent of
$\la$ and
$-L_{\la,T_{l_i},W_0}$ is unitarily equivalent operator to
the infinite dimensional 
counterpart Schr\"odinger operator with a quadratic potential function
which we explained in the finite dimensional cases in the above.
However, unlike the finite dimensional case, 
there exists nonempty essential spectrum in
the spectral set of
$-L_{\la,T_{l_i},W_0}$.
It may be a difficult problem to study the asymptotic behavior of
the lowlying spectrum of $-\la^{-1}L_{\la,\Pea}$ near the 
essential spectrum of $-\la^{-1}L_{\la,T_{l_i},W_0}$ $(i\ge 1)$.
Also, at the moment, it is not clear for the author whether the appropriate
semi-boundedness theorem of $-L_{\la,\Pea}$ holds or not.
Taking these into account, in this paper, we consider Witten Laplacian 
$L_{\la,\D}$
with the Dirichlet boundary condition on a certain domain $\D\subset\Pea$.
The domains $\D$ which we consider in this paper 
contain finitely many geodesics
$\{l_i\}_{i=1}^N$ only and there are no geodesics in the boundary
$\partial \D$.
Moreover we have a refined version of Gross's log-Sobolev inequality in
\cite{a2008} with a potential
function on $\Pea$ which gives a suitable lower bound of $-L_{\la,\D}$.
We may consider a similar problem in the case of general Riemannian 
manifold.
We already have several log-Sobolev inequalities with potential functions
on pinned path space over compact Riemmanian manifold
(\cite{a1996},\cite{gong-ma}).
However, it is not clear that the inequalities are good enough for our problem.
Also, we mention that log-Sobolev inequality without potential function 
was proved for heat kernel measure on $P_{e,e}(G)$ by \cite{driver3}.
Let us explain our main theorem more precisely.
Let $R>0$.
We consider the set $\Sigma_{R,r}$ which is the intersection of
$[0,R]$ and the complement of $r$-neighborhood of the union of
the sets of essential spectrum of $-L_{\la,T_i,W_0}$ $(1\le i\le N)$.
We can prove that $-\la^{-1}L_{\la,\D}$ has discrete spectrum only
in the set $\Sigma_{R,r}$ for large $\la$ and the asymptotic behavior of them
can be determined by the discrete spectrum of
$-\la^{-1}L_{\la,T_i,W_0}$ $(1\le i\le N)$ similarly to finite dimensional cases.

The organization of this paper is as follows.
In Section 2, we prepare necessary ingredients of our study
and state our main theorem.
In Section 2.1, first, we prepare necessary notations and 
we recall the definition of the pinned Brownian motion measure 
and the $H$-derivative on
$\Pea$.
We next introduce domains $\D\subset \Pea$ and
define Witten Laplacian $L_{\la,\D}$ with the Dirichlet boundary condition on
$\D$.
In Section 2.2, we recall the explicit form of the Hessian of
the energy function $E(\gamma)=\frac{1}{2}\int_0^1|\dot{\gamma}(t)|^2dt$
and the eigenvalues of the Hessian (\cite{a2003}).
We next state our main theorem (Theorem~\ref{main theorem}). 
In our main theorem, we are concerned with the set of discrete spectrum
of approximate Ornstein-Uhlenbeck type operators.
We need to give examples for which 
there are many discrete spectrum of $-L_{\la,T_i,W_0}$.
We give such an example in the case where $G=SU(2)$ 
in Proposition~\ref{example of discrete spectrum}.
The proof will be given in Appendix.
In Section 2.3, we prepare necessary tools from rough path analysis and
Malliavin calculus.
A certain subset
$\Peaomega$ of $\Pea$ is isomorphic to
a submanifold $S_a=\{w\in \Omega~|~Y(1,e,w)=a\}$, where
$Y$ is a solution of rough differential equation,
$dY_t=Y_td \bar{w}_t$ and $\bar{w}$ is a geometric rough path lift of
$w\in \Omega\subset W$.
Here, $W$ is a classical Wiener space and $\Omega$ is the subset for which
the smooth rough path lift $\overline{w^N}$ of the dyadic polygonal approximation
$w^N$ converges in rough path topology.
The isomorphism is given by the mapping 
$Y : S_a\ni w\mapsto Y(\cdot,e,w)\in \Peaomega(\subset\Pea)$.
By using the isomorphism, we can reduce several problems in $\Pea$ to 
the corresponding problem in $S_a$.
The finite measure on $S_a$ associated with the positive generalized Wiener functional
$\delta_a(Y(1))$ in the sense of Watanabe plays important role.
In Section 2.4, we define the operator $L_{\la,T_i,W_0}$ in general settings.
In Section 2.5, we recall the refined version of Gross's log-Sobolev inequality
with a potential function on $\Pea$ and prove necessary semi-boundedness theorem
for $\E_{\la,\D}(f,f)=\int_{\D}|\nabla f|^2d\nu_{\la,\D}$
by using the inequality.

In Section 3, we construct local coordinate system on $S_a$ and prove a
change of variable formula by using 
estimates in rough path analysis. This also provides a 
local coordinate system on
$\Peaomega$ via a isomorphism which is given in Section 2.3.

In Section 4, we introduce cut-off functions in the topology of rough path
and approximate eigenfunctions of $-\la^{-1}L_{\la,\D}$ by using the local 
coordinate given in Section 3.

In Section 2.5, we give a lower bound of $\E_{\la,\D}(f,f)$.
The inequality contains integral of certain exponential functional.
We need to show the integral remains finite under the limit $\la\to\infty$.
In Section 5, we prove such an estimate.

In Section 6, we prove our main theorem.

\section{Preliminary and statement of main results}

\subsection{Witten Laplacians acting on functions on pinned path groups}

Let $G$ be a $d$-dimensional compact connected Lie group.
Denote the unit element by $e$.
We denote the Lie algebra of $G$ by
$\g$ which is identified with 
$T_{e}G$.
Actually, there exists a positive integer
$n$ such that $G$ is isomorphic to a Lie subgroup of
$n$-dimensional unitary group $SU(n)$ and
the Lie algebra is isomorphic to the Lie subalgebra of
${\mathfrak su}(n)$.
In this case, $e$ corresponds to the identity matrix $I$.
By this result, we may assume that $G$ is a matrix group.
That is $G\subset SU(n)\subset M(n,\CC)\cong \CC^{n^2}$
and $\g\subset \mathfrak{su}(n)\subset M(n,\CC)\cong \CC^{n^2}$.
$M(n,\CC)$ denotes the set of all $n\times n$-matrices whose elements
are complex numbers.
 Let $x, y\in G$. 
We may write $xy=L_xy=R_yx$, that is $L_x$ and $R_y$ are 
the left multiplication and the right multiplication respectively.
By identifying $T_yG$ with the set $\{vy~|~v\in \g\}$,
we may write $(R_y)_{\ast}v=vy$.
Also, as usual, we write $Ad(x)y=xyx^{-1}$.
The derivative of the mapping $Ad(x) : G\to G$ at $e$
is also written as $Ad$ which is a linear map on $\mathfrak{g}=T_eG$.
Since we identify $G$  and $\g$ with a subgroup and Lie subalgebra of $M(n,\CC)$,
$Ad(x)v$ $(v\in \g)$ is identified with the matrix which is obtained by the
matrix product $xvx^{-1}$ $(v\in \g)$.
We write $[u,v]=\ad(u)v$.
Note that $[u,v]=uv-vu$ if we view $u,v$ as the elements of $\g$,
where $uv, vu$ is usual matrix product.

For $A,B\in M(n,\CC)$, define an inner product
$(A,B)=\tr AB^{\ast}$.
This defines an inner product on the real vector space $\mathfrak{su}(n)$
which is invariant by the $Ad$ action
and hence, a bi-invariant Riemannian metric on
$G$.
We fix an orthonormal basis of $\g$, say,  $\{\ep_i\}_{i=1}^d$ and we identify
$\mathfrak{g}$ with $\RR^d$ by the mapping $\ep_i\to e_i$ $(1\le i\le d)$,
where $\{e_i\}_{i=1}^d$ denotes the standard basis of $\RR^d$.
Note that $\ep_i\in \su(n)\subset M(n,\CC)$ holds and
$\ep_i$ is a matrix.
We denote by $dx$ the Riemannian volume measure (Haar measure) on $G$
which is bi-invariant by the bi-invariance of the metric.
For $A\in \g$, $e^A, \exp A$ denote the matrix exponential element.
This exponential map coincides with the exponential map 
in Riemannian geometry sense starting from $e$.

We now introduce the Brownian motion and pinned Brownian motion on $G$.
Let $P_{e}(G)$ be a set of continuous paths
$\gamma(t)$~$(0\le t\le 1)$ with values in $G$
with $\gamma(0)=e$.
Let $\la>0$ which is a large parameter and the inverse corresponds to
small semi-classical parameter in quantum physics.
There exists a probability measure which is called the Brownian motion
measure $\nu_{\la}$ on $P_{e}(G)$ such that
\begin{align*}
&\nu_{\la}\left(\gamma_{t_1}\in A_1, \ldots, 
\gamma_{t_N}\in A_N
\right)\\
& =\int_{A_1}dx_1\cdots\int_{A_N}dx_N
\prod_{i=1}^{N}p\left(\la^{-1}(t_{i}-t_{i-1}),x_{i-1},x_i\right),
\end{align*}
where $0=t_0<\cdots <t_N=1,\, x_0=e$ and $N\in \NN$.
Here $p(t,x,y)$ denotes the heat kernel of
$e^{t\Delta/2}$, where $\Delta$ is the Laplace-Bertlami operator defined by
the Riemannian metric. 
Let $\Pea$ be a subset of $P_{e}(G)$
such that $\gamma(1)=a$.
There exists a probability measure $\nu_{\la,a}$
which is called the pinned Brownian motion
measure on $\Pea$ such that
\begin{align*}
&\nu_{\la,a}\left(\gamma_{t_1}\in A_1, \ldots, \gamma_{t_N}\in A_N
\right)\\
&=p(\la^{-1},e, a)^{-1}\int_{A_1}dx_1\cdots\int_{A_N}dx_N
\prod_{i=1}^{N+1}p\left(\la^{-1}(t_{i}-t_{i-1}),x_{i-1},x_i\right),
\end{align*}
where $0=t_0<\cdots<t_{N+1}=1, x_0=e, x_{N+1}=a$.
Let 
\begin{align*}
H&=\left\{
h : [0,1]\to \RR^d~\Big |~
h(0)=0, \|h\|_{H}^2:=\int_0^1|\dot{h}(t)|^2dt<\infty
\right\},\\
H_{0,0}&=
\left\{h\in H~\Big |~ h(1)=0\right\}.
\end{align*}
The notation $H_{0,0}$ may be unusual notation but
we will use the notation $H_0$ to denote more general subspace of $H$
later (see Definition~\ref{W_0}).
Below, we use the notation $X$ to denote 
$\Pea, P_{e}(G)$.

Let $\FC(X)$ be the set of
smooth cylindrical functions on $X$.
That is,
$\FC(X)$ consists of functions $F: X\to\RR$ such that there exists
$f\in C^{\infty}(G^N)$ $(N\in \NN)$ and a time sequence
$\{0<t_1<\cdots<t_N\le 1\}$ satisfying
$F(\gamma)=f(\gamma(t_1),\ldots,\gamma(t_N))$.
We define the $H$-derivative $(\nabla F)(\gamma)$
of $F(\gamma)\in \FC(X)$
by the unique element of 
$H$ (if $X=P_e(G)$)~(or $H_{0,0}$ if $X=\Pea$) satisfying that
\begin{align*}
&(\nabla F)(\gamma)[h]=\left((\nabla F)(\gamma), h\right)_H:=
\lim_{\kappa\to 0}\frac{F\left(e^{\kappa h(\cdot)}\gamma(\cdot)\right)
-F(\gamma)}{\kappa},\\
&\qquad  \mbox{for all $h\in H$ (if $X=P_e(G)$) 
(or $h\in H_{0,0}$ if $X=\Pea$)}.
\end{align*}
Note that we identify $\RR^d$ with $\g$ by using the standard basis 
$\{e_i\}\subset \RR^d$
and the orthonormal system $\{\ep_i\}\subset \g\subset \su(n)$.
That is, $e^{\kappa h(t)}\gamma(t)$
can be defined as the elements in matrix.
Typical examples of derivatives are as follows:
\begin{align}
 \nabla_h \gamma_t=h_t\gamma_t, \quad 
\nabla_h\gamma^{-1}_t=-\gamma_t^{-1}\nabla_h\gamma_t\gamma_t^{-1}
=-\gamma^{-1}_th_t.\label{derivative of gamma}
\end{align}
Explicitly, we have
\begin{align*}
 (\nabla F)(\gamma)_t&=
\begin{cases}
\sum_{i=1}^N(R_{\gamma(t_i)})_{\ast}^{-1}
(\partial f)_i(\gamma)t_i\wedge t\in H,&
\text{if $X=P_e(G)$},\\
 \sum_{i=1}^N(R_{\gamma(t_i)})_{\ast}^{-1}
(\partial f)_i(\gamma)(t_i\wedge t-t_it)\in H_{0,0},& 
\text{if $X=\Pea$},
\end{cases}
\end{align*}
where $(\partial f)_i(\gamma)$ is the derivative with respect to 
$i$-th element $\gamma(t_i)$ and
$(R_a)_{\ast} : T_xG\to T_{xa}G$ is the derivative of 
the right multiplication mapping,
$R_a(x)=xa$.
In other words, we consider the Bismut tangent space 
along $\gamma$ by using
the right invariant connection on $G$.
Note that the tangent space depends on the choice of the connection on $G$.
For example, the derivative which is defined by the Levi-Civita connection and
left invariant connection
are totally different from the one in the above (\cite{driver2}).
We now define a Witten Laplacian $-L_{\la,\Pea}$ acting on functions on $\Pea$.
First note that $\FC(\Pea)$ is a dense subset of $L^2(\Pea,\nu_{\la,a})$.
We consider the following densely defined
symmetric form $\E_{\la}$ on $L^2(\Pea,\nu_{\la,a})$:
\begin{align*}
 \E_{\la}
(F,G)&=\int_{\Pea}\left((\nabla F)(\gamma),(\nabla G)(\gamma)\right)_{H_{0,0}}
d\nu_{\la,a}(\gamma), 
\end{align*}
where $F, G\in \FC(\Pea)$.
We have the integration by parts formula
\begin{align*}
 \int_{\Pea}\nabla_hF(\gamma)G(\gamma)d\nu_{\la,a}(\gamma)=
\int_{\Pea}F(\gamma)\left\{
-\nabla_hG(\gamma)+\la\,(b,h)G(\gamma)\right\}
d\nu_{\la,a}(\gamma),
\end{align*}
where $b(t)$ is the stochastic development of $\gamma_t$ by the right
invariant connection, that is,
$
 b(t)=\int_0^t\circ d\gamma_s\gamma_s^{-1},
$
and $(b,h):=\int_0^1(\dot{h}_s,db_s)$.
Here $\circ d\gamma_s\gamma_s^{-1}$ denotes
the Stratonovich integral and the product is a matrix product.

We refer the readers for $b(t)$ and $(h,b)$ to
Remark~\ref{development of path}.
By using this integration by parts formula and the denseness of
the set 
$\{\sum_{i=1}^NF_i(\gamma)h_i~|~F_i\in \FC(\Pea),
h_i\in H_{0,0},\,
N\in \mathbb{N}\}$
in $L^2(\Pea\to H_{0,0},\nu_{\la,a})$,
we see that if $F_1=F_2\in \FC(\Pea)$ $\nu_{\la,a}$-a.s., 
then $\nabla F_1=\nabla F_2$ $\nu_{\la,a}$ a.s. holds.
Hence, $(\nabla F(\gamma),\nabla G(\gamma))$ can be defined for
the equivalence class of $\nu_{\la,a}$-a.s.equal $F, G\in \FC(\Pea)$.
Consequently,
we can check that $\E_{\la}$ is a symmetric form for $\nu_{\la,a}$
almost surely equal equivalence class of $\FC(\Pea)$.
Also, by using the integration by parts formula, it is easy to see that
it is closable.
We consider the smallest closed extension Dirichlet form
$\E_{\la}$ and we denote the domain of $\E_{\la}$ by 
$\rD(\E_{\la})$ and the nonnegative generator and the domain 
by $-L_{\la,\Pea}$ and
$\rD(L_{\la,\Pea})$ respectively.
Note that for $F\in\rD(\E_{\la})$, 
$\nabla F\in L^2(\Pea\to H_{0,0},\nu_{\la,a})$ can be defined.
We now consider a Dirichlet form defined on a subset
of $\Pea$ with the Dirichlet boundary condition.
To this end, we prepare a lemma.

\begin{lem}\label{Dirichlet laplacian denseness}
 Let $\varphi\in \rD(\E_{\la})$ and we fix a positive number $R$.
Let $\D=\{\gamma\in \Pea~|~\varphi(\gamma)<R\}$ and suppose that
$\nu_{\la,a}(\D)>0$.
Then the subset of functions on $\Pea$
\begin{align*}
V_{\varphi,R}=\{F(\gamma)\chi(\varphi(\gamma))~|~F\in \FC(\Pea),\, 
\text{$\chi$ is a $C^1_b$ function on $\RR$ satisfying
$\supp\,\chi\subset (-\infty,R)$}
\}
\end{align*}
is dense in $L^2(\D,\nu_{\la,a})$ and for any
 $\tilde{F}\in V_{\varphi,R}$, 
$\tilde{F}(\gamma)=0$ 
for $\nu_{\la,a}$-almost all 
$\gamma\in \D^\complement$
and $\tilde{F}\in \rD(\E_{\la})$.
\end{lem}

\begin{proof}
 The denseness follows from the denseness 
$\FC(\Pea)\subset L^2(\Pea,\nu_{\la,a})$
with an easy calculation.
The remaining part follows from the definition of $\tilde{F}$.
\end{proof}

\begin{rem}\label{development of path}
(1) Note that the pinned Brownian motion $\gamma_t$ under $\nu_{\la, a}$ 
is a semimartingale as well as non-pinned process $\gamma_t$ under $\nu_{\la}$.
Also, $b(t)$ and the stochastic development process 
$w_t=\int_0^t\gamma_s^{-1}\circ d\gamma_s$
by the left invariant connection
are standard Brownian motions 
on $\RR^d(\simeq \g)$ satisfying
$E[(w_t,e_i)(w_s,e_j)]=E[(b_t,e_i)(b_s,e_j)]=\la^{-1}\delta_{i,j}t\wedge s$
respectively under the probability measure $\nu_{\la}$.
We refer the readers to \cite{g2}, \cite{hsu2}, \cite{driver1}, \cite{driver1-2},
\cite{a2000} 
for these results
as well as the integration by parts formula.
In this paper, we fix a versions of them in Section~\ref{preliminary} 
with the help of rough path analysis.

\noindent
(2) Of course, we can define a Dirichlet form in 
$L^2(\Pe,\nu_{\la})$ similarly to the above.
\end{rem}

By Lemma~\ref{Dirichlet laplacian denseness},
the following definition is well-defined.

\begin{dfi}\label{Dirichlet form with DBC}
Suppose that $\varphi$ satisfies the assumption 
in Lemma~$\ref{Dirichlet laplacian denseness}$
and we consider the set $\D$ in the lemma.
Let $\nu_{\la,\D}(d\gamma)=\nu_{\la,a}(d\gamma)/\nu_{\la,a}(\D)$
be the normalized probability measure on $\D$.
We consider the Dirichlet form on $L^2(\D,\nu_{\la,\D})$,
\begin{align*}
& \E_{\la,\D}(F,G)=\int_{\D}\left(\nabla F(\gamma),\nabla G(\gamma)\right)
d\nu_{\la,\D}(\gamma),
\end{align*}
with the domain
$
 \rD(\E_{\la,\D})=\left\{
F\in\rD(\E_{\la})~|~
\text{$F=0$ on $\D^{\complement}$
$\nu_{\la,a}$-a.s.$\gamma$}
\right\}.
$
We denote the non-negative generator of this Dirichlet form by
$-L_{\la,\D}$.
\end{dfi}

\begin{rem}\label{remark on the generator}
Let $F\in \rD(L_{\la,\Pea})$ and assume $F=0$ $\nu_{\la,a}$-a.s. 
$\gamma\in \D^\complement$.
Then, by the definition of the Dirichlet form and the generator,
we see that $F\in \rD(L_{\la,\D})$ and
$(L_{\la,\D}F)(\gamma)=(L_{\la,\Pea}F)(\gamma)$ for a.s. $\gamma$ on $\D$.
\end{rem}

We are interested in the asymptotic behavior of the spectrum of
$-L_{\la,\D}$ under $\la\to\infty$.
Actually, we consider more concrete domain $\D$.
We introduce exhaustion domain of $P_{e,a}(G)$.
Since $\gamma_t$ is a path on $G$, it may be natural to consider
the Riemmanian distance $d(\gamma_t,\gamma_s)$, but,
we adopt the norm $|\gamma_t-\gamma_s|$ in the ambient space
$M(n,\CC)$ to evaluate the distance for simplicity.
Note that the two distances are equivalent since $G$ is compact.

\begin{dfi}\label{definition of domain}
(1) For $K,m\in \NN$, $0<\theta<1$, $0<\beta<1/2$, define
\begin{align*}
 \varphi_{\infty,K}(\gamma)&=\max_{i=0,1,\ldots,K-1}\sup_{s,t\in [i/K,(i+1)/K]}
|\gamma_t-\gamma_s|,\\
\varphi_{B, m,\theta}(\gamma)&=\left(\iint_{0<s<t<1}
\frac{|\gamma_t-\gamma_s|^{4m}}{|t-s|^{2+2m\theta}}dsdt\right)^{1/(4m)},\qquad
\varphi_{H,\beta}(\gamma)=\sup_{0\le s<t\le 1}
\frac{|\gamma_t-\gamma_s|}{|t-s|^{\beta}}.
\end{align*}

\noindent
(2)
Let $\delta, M>0$ in addition to $K, m, \theta, \beta$.
Using the above functional, we define subset of $\Pea$ as follows:
\begin{align*}
& D_{K,\delta}
=
\left\{
\gamma~\Big |~\varphi_{\infty,K}(\gamma)<\delta\right\},\quad
B_{M,m,\theta}=
\left\{\gamma~\Big |~
\varphi_{B,m,\theta}(\gamma)<M
\right\},\\
& C_{M,\beta}=\left\{\gamma~\Big |~
\varphi_{H,\beta}(\gamma)<M
\right\},\quad
 \mathscr{D}_M=\left\{\gamma~\Big |~
|b(1,\gamma)|<M\right\}.
\end{align*}

\noindent
(3) Let $\D$ be one of the above domains.
We call the sets
\begin{align*}
 \{\varphi_{\infty,K}(\gamma)=\delta\},\quad
\{\varphi_{B,m,\theta}(\gamma)=M\},\quad
\{\varphi_{H,\beta}(\gamma)=M\},\quad
\{|b(1,\gamma)|=M\}
\end{align*}
boundary of each of $\D$
and we denote them by $\partial \D$.
Also we say that the boundary measure is $0$ if $\partial\D$ is null set with respect to
$\nu_{\la,a}$.
\end{dfi}

\begin{rem}
(1) $\varphi_{B,m,\theta}(\gamma)$ is the Besov norm of $\gamma$.
The choice of $4m$ is just for the convenience of the calculation below.
Later, we define a similar functional 
$\bar{\Gamma}_{m,\theta}(w)$ for $\bar{w}$ which is a 
geometric rough path lift of
the standard Brownian motion $w_t=\int_0^t\gamma_s^{-1}\circ d\gamma_s$
in Lemma~\ref{cut-off function}.
See also Remark~\ref{remark on Besov norm}.
Recall that, we define
$
 b(t)=\int_0^t\circ d\gamma_s\gamma_s^{-1}
$
as an almost surely defined random variable.
In this paper, we fix a version of $b(t)$ in Lemma~\ref{pathwise gamma} and
prove $b(t)\in \rD(\E_\la)$ in Lemma~\ref{derivative of b1}.

\noindent
(2) 
The boundary measure of $\D=\{\gamma~|~\varphi(\gamma)<R\}$ is 0 is equivalent to
that $R$ is not an atom of the distribution of $\varphi$ 
under $\nu_{\la,a}$.
It is plausible that all boundaries are measure zero according to
the absolute continuity property of
the distribution of several maximum processes
which can be found in \cite{nualart}.
\end{rem}

\begin{lem}
Suppose $\gamma$ is a continuous path.
 \begin{enumerate}
  \item[$(1)$] It holds that $\varphi_{H,\theta/2}(\gamma)\le 
C\varphi_{B,m,\theta}(\gamma)$
and $\varphi_{B,m,\theta}(\gamma)\le C\varphi_{H,\beta}(\gamma)$,
where $\beta>\frac{\theta}{2}+\frac{1}{4m}$.
\item[$(2)$] $C_{K^{\beta}\delta,\beta}\subset D_{K,\delta}$.
 \end{enumerate}
\end{lem}

\begin{proof}
We refer the proof of (1) to Remark~\ref{remark on Besov norm}.
The proof of (2) is an elementary calculation.
\end{proof}

\begin{lem}
It holds that
$\varphi_{\infty,K}, \varphi_{B,m,\theta}, \varphi_{H,\beta}, |b(1)|\in 
\rD(\E_{\la})$.
\end{lem}

\begin{proof}
Let $f_m(\gamma)=|\gamma_t-\gamma_s|^{2m}$.
Then $f_m(\gamma)=
\left(\tr(\gamma_t-\gamma_s)(\gamma_t^{\ast}-\gamma_s^{\ast}))\right)^{m}$.
Thus, for $h\in H_{0,0}$ and $\kappa\in \RR$
\begin{align*}
 f_m(e^{\kappa h}\gamma)=
\left[\tr\left\{
\left((\gamma_t-\gamma_s)+\kappa(h_t\gamma_t-h_s\gamma_s)\right)
\left((\gamma_t^{\ast}-\gamma_s^{\ast})-
\kappa(\gamma_t^{\ast}h_t-\gamma_s^{\ast}h_s)\right)\right\}
+O(\kappa^2)\right]^m
\end{align*}
This implies
\begin{align*}
 \frac{\partial}{\partial\kappa}f_m(e^{\kappa h}\gamma)\Big|_{\kappa=0}&=
m|\gamma_t-\gamma_s|^{2(m-1)}
\tr\left((\gamma_t\gamma_s^{\ast}-\gamma_s\gamma_t^{\ast})(h_t-h_s)^{\ast}\right)\\
&=m|\gamma_t-\gamma_s|^{2(m-1)}
\tr\left\{(\Gamma_{s,t}-\frac{\tr\, \Gamma_{s,t}}{n} I)(h_t-h_s)^{\ast}\right\},
\end{align*}
where
$\Gamma_{s,t}=\gamma_t\gamma_s^{\ast}-\gamma_s\gamma_t^{\ast}$.
Note that $\Gamma_{s,t}^{\ast}=-\Gamma_{s,t}$ and 
$\tr\, \Gamma_{s,t}^{\ast}=-\tr\, \Gamma_{s,t}$.
The second result implies that $\tr\, \Gamma_{s,t}\in \sqrt{-1}\RR$.
Thus, $\Gamma_{s,t}-\frac{\tr\, \Gamma_{s,t}}{n}I\in \su(n)$.
This shows
\[
 (\nabla f_m)(\gamma)_u=m|\gamma_t-\gamma_s|^{2(m-1)}
\left(\Gamma_{s,t}-\frac{\tr\, \Gamma_{s,t}}{n}I\right)
\left\{\chi_{t}(u)-\chi_{s}(u)\right\},
\]
where $\chi_{t}(u)=u\wedge t-ut$.
Noting
 $\Gamma_{s,t}=\gamma_t\gamma_s^{-1}-\gamma_s\gamma_t^{-1}
=(\gamma_t-\gamma_s)\gamma_t^{-1}+(\gamma_t-\gamma_s)\gamma_s^{-1}$,
we get $|\Gamma_{s,t}|\le 2|\gamma_t-\gamma_s|$.
By using this and a limiting argument, we obtain
$\varphi_{B,m,\theta}^{4m}\in \rD(\E_{\la})$ and
\begin{align*}
 \nabla_h\varphi_{B,m,\theta}^{4m}(\gamma)=
\iint_{0<s<t<1}
\frac{2m|\gamma_t-\gamma_s|^{2(2m-1)}
\left(\Gamma_{s,t}-
\frac{\tr \Gamma_{s,t}}{n}I,h_t-h_s\right)}{|t-s|^{2+2m\theta}}dsdt.
\end{align*}
Using H\"older's inequality, the estimate on $\Gamma_{s,t}$ and 
$\varphi_{B,m,\theta}(h)\le C\|h\|_{H_{0,0}}$, we obtain
$|\nabla_h\varphi_{B,m,\theta}(\gamma)|\le C_m\|h\|_{H_{0,0}}$.
This also implies $|\nabla\varphi_{B,m,\theta}|_{H_{0,0}}$ is bounded, 
that is, $\varphi_{B,m,\theta}$ is a Lipschitz function.

We next prove $\varphi_{H,\beta}\in \rD(\E_{\la})$.
The proof is similar to that of Lemma~2.2 in \cite{a2000} 
as follows.
Let $g(\gamma)=|\gamma_t-\gamma_s|$ and
$g_{\ep}(\gamma)=\sqrt{f_1(\gamma)+\ep}$.
Then $\lim_{\ep\to 0}g_{\ep}(\gamma)=g(\gamma)$ for all $\gamma$.
Noting that $|\Gamma_{s,t}-\frac{\tr \Gamma_{s,t}}{n}I|\le 
C|\gamma_t-\gamma_s|$,
\begin{align*}
 (\nabla g_{\ep})(\gamma)_u&=\frac{\nabla f_1(\gamma)}{2g_{\ep}(\gamma)}
=\frac{(\Gamma_{s,t}-\frac{\tr \Gamma_{s,t}}{n}I)(\chi_t-\chi_s)}
{\sqrt{|\gamma_t-\gamma_s|^2+\ep}}
\to 
\begin{cases}
0 & \text{if $\gamma_t=\gamma_s$}\\
\frac{(\Gamma_{s,t}-\frac{\tr \Gamma_{s,t}}{n}I)(\chi_t-\chi_s)}
{|\gamma_t-\gamma_s|} & \text{if $\gamma_t\ne \gamma_s$}
\end{cases}
\end{align*}
Consequently, $g\in \rD(\E_{\la})$ and
\begin{align}
 (\nabla g)(\gamma)&
=\frac{(\Gamma_{s,t}-\frac{\tr \Gamma_{s,t}}{n}I)(\chi_t-\chi_s)}
{|\gamma_t-\gamma_s|}1_{(0,\infty)}(|\gamma_t-\gamma_s|),\label{nablag}
\end{align}
where $1_{(0,\infty)}$ is the indicator function of $(0,\infty)$.
Moreover, $\|(\nabla g)(\gamma)\|_{H_{0,0}}\le C (t-s)^{1/2}$,
where $C$ is a constant independent of $\gamma$.
Let $\{t_i\}_{i=1}^{\infty}$ be a dense subset of $[0,1]$ such that
$t_i\ne t_j$ if $i\ne j$.
Then $\varphi_{\beta,N}(\gamma)=\max_{1\le i\ne j\le N}
\frac{|\gamma_{t_i}-\gamma_{t_j}|}{|t_i-t_j|^{\beta}}$
converges to $\varphi_{H,\beta}(\gamma)$ for all $\gamma\in \Pea$
and in $L^2(\nu_{\la,a})$ as $N\to\infty$.
By approximating the maximum function by $\ell^{n}$ norm as $n\to\infty$
and using (\ref{nablag}), we get
$\varphi_{\beta,N}\in \rD(\E_\la)$ and 
$\|(\nabla\varphi_{\beta,N})(\gamma)\|_{H_{0,0}}\le C\max_{1\le i\ne j\le N}
|t_i-t_j|^{(1/2)-\beta}\,\, a.s.$ 
Hence letting $N\to\infty$, we obtain 
$\varphi_{H,\beta}\in\rD(\E_\la)$ and
$\|\nabla\varphi_{H,\beta}(\gamma)\|_{H_{0,0}}\le C$ $a.s.$.
The proof for $\varphi_{\infty,K}$ is similar to that of
$\varphi_{H,\beta}$.
We postpone the proof of $|b(1)|\in\rD(\E_{\la})$ 
to Lemma~\ref{derivative of b1}.
\end{proof}

\subsection{Hessian of the energy function of 
\texorpdfstring{$H^1$}-paths and main theorem}
\label{statement of main theorem}

Let $\PeH$ and $\PeaH$ denote the subset of all paths $\gamma\in \Pe$
and $\gamma\in \Pea$ respectively,
whose energies $E(\gamma)
:=\frac{1}{2}\int_0^1|\dot{\gamma}(t)|^2dt$
are finite.
By using the right invariant connection on $G$, we define
the right invariant Riemmanian metric $\langle \cdot,\cdot\rangle$ 
on $\PeaH$ as follows.
\begin{enumerate}
 \item[(1)] 
Let $h :[0,1]\to TM$ be the vector field along $\gamma$, that is,
$h(t)\in T_{\gamma(t)}G$ $(0\le t\le 1)$.
\item[(2)]
We say that the vector field $h$ along $\gamma$
belongs to the $H^1$ tangent space of $\gamma\in \PeaH$
if and only if $\{(R_{\gamma(t)})_{\ast}^{-1}h(t)\}_{0\le t\le 1}\in H_{0,0}$
holds
and the Riemannian metric is defined by
$\langle h,h\rangle:=\|(R_{\gamma(t)})_{\ast}^{-1}h(\cdot)\|_{H_{0,0}}^2$.
\end{enumerate}

We consider the differential equation
\begin{align}
 \dot{Y}(t,e,h)=Y(t,e,h)\dot{h}(t),\qquad Y(0,e,h)=e.\label{ODE}
\end{align}
Then the inverse map $Y\mapsto h$ defines the Cartan development map
defined by the left-invariant connection.
Let $S_{a}^H=\{h\in H~|~Y(1,e,h)=a\}$.
Then it is easy to see that $S_{a}^H$ is a smooth submanifold in $H$ and 
the mapping $Y : h\in S_{a}^H\mapsto Y(\cdot,e,h)\in\PeaH$ is a diffeomorphism mapping.
Note that we can consider the tangent space $T_hS_a^H$ $(h\in S_a^H)$ and
the induced Riemmanian metric on $S_a^H$ by $H$.
Then we have
\begin{pro}\label{derivative smooth case}
The derivative $DY : T_{h}S_a^H\to T_{Y(h)}\PeaH$ 
is a Riemmanian isometry mapping.
More precisely, we have the following identity:
for any $k\in H$,
\begin{align*}
 (R_{Y(t,e,h)})_{\ast}^{-1}DY(t,e,h)[k]=\int_0^tAd(Y(s,e,h))\dot{k}(s)ds.
\end{align*}
\end{pro}
Although there do not exist the pinned Brownian motion measure and
the induced surface measure on
$\PeaH$ and $S_{a}^H$, 
the corresponding rigorous measure-theoretical
objects $\Pea$ and $S_a=\{w~|~Y(1,e,w)=a\}$ 
can be defined by using a solution of
a Stratonovich stochastic differential equation driven by a Brownian motion $w$ 
\begin{align}
 dY(t,e,w)=Y(t,e,w)\circ dw(t),\qquad Y(0,e,w)=e.\label{sde}
\end{align}
We give the precise definition of $S_a$ in Section~\ref{preliminary}
by the solution of a rough differential equation.
Roughly speaking, the mapping $Y : w\in S_a\mapsto \Pea$
is an isomorphism (see Proposition~\ref{isomorphism})
and using this we will construct a local coordinate system around each geodesic
and we reduce the local problem on the curved space $\Pea$ to the problem on
the small neighborhood in an Wiener space in Section~\ref{local coordinate0}.

The following is proved in \cite{a2003} except the final statement in
(4) (ii).
Note that (1) is well-known result.
For example, we refer the readers to \cite{milnor}.

\begin{lem}\label{hessian of E}
 Let $a\in G$ and consider the energy functional $E$ on $\PeaH$.
\begin{enumerate}
 \item[$(1)$] The sets of critical points of $E$ $(\text{the set of geodesics})$
is given by 
$\{l_{\xi}(t)=e^{t\xi}~|~\xi\in V\}$, where
$V$ is the set of all elements $\xi\in \g$ satisfying $e^{\xi}=a$.
When $a$ is not a conjugate point of $e$ along any geodesics,
$V$ is a countable set and
$\lim_{i\to\infty}|\xi_i|_{\g}=\infty$ $(V=\{\xi_i\}_{i=1}^{\infty})$ holds.
\item[$(2)$] The Hessian $\nabla^2E$ at $l_\xi$ is equal to
$I_{H_{0,0}}+T_\xi$, where $T_\xi$ is a self-adjoint 
Hilbert-Schmidt operator on
$H_{0,0}$ given by
\[
 (T_\xi h)(t)=\int_0^t[h(s),\xi]ds-t\int_0^1[h(s),\xi]ds.
\]
\item[$(3)$] There exists an
orthonormal basis 
$\{u_i\}_{i=1}^l\cup\{\tilde{u}_i\}_{i=1}^l
\cup\{w_j\}_{j=1}^{d-2l}$
of $\g$ and positive number $\zeta_i$ such that
\begin{align*}
 \ad(\xi)u_i=2\pi\zeta_i \tilde{u}_i,\quad 
\ad(\xi)\tilde{u}_i=-2\pi\zeta_i u_i,\quad
\ad(\xi)w_j=0,\quad 1\le i\le l,\quad 1\le j\le d-2l.
\end{align*}
\item[$(4)$] 
Let $\Xi(\xi)$ denote the set of all eigenvalues of 
$I_{H_{0,0}}+T_\xi$ counting multiplicities.
\begin{enumerate}
 \item[{\rm (i)}] We have
\begin{align*}
 \Xi(\xi)=\left\{
1\pm \frac{\zeta_i}{m},\,\, 
1\pm \frac{\zeta_i}{m}\,\Big |\,\, m\in \NN,\,\, 1\le i\le l
 \right\}
\cup \Xi_0,
\end{align*}
where $\Xi_0=\{1,1,\ldots\}$.
\item[{\rm (ii)}] Let $G=SU(n)$ and consider the case
$a=D[e^{2\pi\theta_1\sqrt{-1}},\ldots,e^{2\pi\theta_n\sqrt{-1}}]\in G$, where
$\theta_i$ are real numbers satisfying
$\sum_{i=1}^n\theta_i=0$ and
$\theta_i-\theta_j\notin \ZZ$  if $i\ne j$.
Also $D[a_1,\ldots,a_n]$ denotes the diagonal matrix whose $(i,i)$ element
is $a_i$.
Then 
\begin{align*}
 V&=\left\{2\pi\sqrt{-1}D[\theta_1+k_1,\ldots, \theta_n+k_n]~\Big |~
\text{$k_i\in \ZZ$ $(1\le i\le n)$\, and\, $\sum_{i=1}^nk_i=0$}\right\}
\end{align*}
and for $\xi=2\pi\sqrt{-1}D[\theta_1+k_1,\ldots, \theta_n+k_n]$,
$\ad(\xi)^2$ has zero eigenvalues with multiplicities $n-1$,
the positive eigenvalues are
$
\left\{(2\pi)^2\{(\theta_i-\theta_j)+(k_i-k_j)\}^2~\Big 
|~1\le i\ne j\le n\right\}
$
and $\Xi(\xi)$ is given by
\begin{align}
&\left\{1\pm\frac{|(\theta_i+k_i)-(\theta_j+k_j)|}{m},\,\,
1\pm\frac{|(\theta_i+k_i)-(\theta_j+k_j)|}{m}~\Big |~m\in \NN,
\,\, 1\le i<j\le n
\right\}\cup \Xi_0, \label{eigenvalues}
\end{align}
where $\Xi_0=\{1,1,\ldots\}$.
In particular, $a$ is not a conjugate point of $e$ along any geodesics
under the above assumption on
$a$. Moreover, the minimum geodesic is unique.
Hence, $a$ is not a point of the cut-locus of $e$.
Moreover, except the unique global minimal geodesics, 
indices of all geodesics are 
positive even number.
That is, there are no local minimum geodesics between $e$ and $a$ other than
the shortest one.
\end{enumerate}
\end{enumerate}
\end{lem}

\begin{proof}
It suffices to
prove that there are no local minimum geodesics other than the unique minimum one
in the setting (4)(ii).
Let $\xi=2\pi\sqrt{-1}D[\theta_1,\ldots,\theta_n]\in V$ 
be the element which attains the
smallest value of the norm $\|\xi\|_{\g}$ in $V$.
This $(\theta_i)$ should satisfy $|\theta_i-\theta_j|\le 1$ for all
$i,j$.
This follows from the fact:
If $\theta_i-\theta_j>1$, then the element $\tilde{\xi}$
which is obtained by
replacing $\theta_i$ and $\theta_j$ by $\theta_i-1$ and $\theta_j+1$
respectively has smaller norm than $\xi$.
This argument shows the uniqueness of the minimizer $\xi$ also.
Actually 
the assumption $\theta_i-\theta_j\notin \ZZ$ $(i\ne j)$
implies that $|\theta_i-\theta_j|<1$ for all $i, j$.
Clearly we see that all eigenvalues of $I_{H_{0,0}}+T_\xi$ are positive.
Taking (\ref{eigenvalues}) into account,
we see that
\begin{itemize}
 \item[(i)] $a$ is not a conjugate point along any geodesics between
$e$ and $a$,
\item[(ii)] minimal geodesic between $e$ and $a$ is unique
\end{itemize}
which is equivalent to that
$a$ does not belong to the cut-locus of $e$
(see \cite{ghl}).
Let $\xi(k)=v+2\pi\sqrt{-1}(k_1,\ldots,k_n)(\ne \xi)\in V$.
Then, there exist distinct $i$ and $j$ such that
$k_i\ge 1$ and $k_j\le -1$.
For this $i,j$, we obtain
$|(k_i+\theta_i)-(k_j+\theta_j)|\ge |k_i-k_j|-|\theta_i-\theta_j|>2-1>1.$
Hence $1-|(k_i+\theta_i)-(k_j+\theta_j)|<0$.
This implies that 
$I_{H_{0,0}}+T_{\xi(k)}$
 has negative eigenvalues with even multiplicities
which implies the desired results.
\end{proof}

\begin{rem}\label{remark hessian of E}
(1) For simplicity, we write $\Xi(\xi)=\{\zeta_i(\xi)\}_{i=1}^{\infty}$ 
counting multiplicities.

\noindent
(2) In the case of $SU(n)$ and the choice of $a$ in (4), local minimum geodesic
is just a global minimum geodesic only.
When we consider general compact Lie group or if we choose different kind of end point
$a$, there may exist local minimum geodesics other than global one.
In that case, exponentially small eigenvalue may appear.
However, we do not study exponential precise asymptotic behavior in this paper
and we leave the related problem to future study.
\end{rem}

Let us recall the results for Witten Laplacians acting on functions
on $\RR^d$.
Let $E$ be a Morse function which has finitely many critical points 
$\{c_i\}_{i=1}^N$ and
the Hessians are nondegenerate at them.
Assume $\nu_{\la}(dx)=\rho_{\la}(x)dx$ is a probability measure on $\RR^d$, 
where $dx$ is the Lebesgue measure,  $\rho_{\la}(x)=e^{-\la E(x)}/Z_{\la}$ and
$Z_{\la}$ is the normalized constant.
Let $-L_{\la}$ be the nonnegative generator of the
Dirichlet form $\mathcal{E}_{\la}(f,f)=\int_{\RR^d}|Df(x)|^2d\nu_{\la}(x)$
on $L^2(\RR^d,\nu_{\la})$.
By using the unitary transformation 
$
M_\la : L^2(\RR^d,\nu_{\la})
\to L^2(\RR^d,dx)
$ defined by $M_\la f=\rho_{\la}^{1/2}f$, we see that
the generator $-L_{\la}$ is unitarily equivalent to
the Schr\"odinger operator
$-H_{\la}=-\Delta+\frac{\la^2}{4}|DE(x)|^2-\frac{\la}{2}\Delta E(x)$
on $L^2(\RR^d,dx)$.
Under appropriate assumptions on $E$, when $\la\to\infty$, the
lowlying spectrum of $-\la^{-1}H_{\la}$ can be approximated by the spectral set of
the quantum harmonic oscillators 
$-H_{i,\la}=-\Delta+\frac{\la^2}{4}D^2E(c_i)((x-c_i)^{\otimes 2})
-\frac{\la}{2}(\Delta E)(c_i)$
which can be obtained by replacing the potential function by the quadratic 
approximate functions of the potential function at $\{c_i\}_{i=1}^N$.
Also the spectral set of $-H_{i,\la}$ in $L^2(\RR^d,dx)$ consists of discrete spectrum
only and explicitly we have
\begin{align}
 \sigma\left(-\la^{-1}H_{i,\la}\right)&=\Biggl\{
\sum_{k=1}^d|\zeta^k_i|n_k+
\sum_{\{k\,|\,\zeta^k_i<0\}}|\zeta^k_i|~\Big |~
n_k\in \ZZ_{\ge 0}\Biggr\},\label{finite dimension}
\end{align}
where
$\{\zeta^k_i\}_{k=1}^d$ is the eigenvalues of $(D^2E)(c_i)$.
This implies that the lowlying spectrum of $-\la^{-1}L_{\la}$
can be approximated by the set of numbers of (\ref{finite dimension})
when $\lambda\to\infty$.
This argument is valid for Witten Laplacians with Dirichlet boundary conditions
under suitable assumptions on the domain and $E$.
By this analogy, in the setting of Lemma~\ref{hessian of E},
for $\xi\in V$, using the set $\Xi(\xi)=\{\zeta_k(\xi)\}_{k=1}^{\infty}$ in the lemma,
the following set of numbers $\Lambda(\xi)$, counting multiplicities,
is a candidate of the approximate lowlying
eigenvalues of approximate operator of $-\la^{-1}L_{\la,\D}$ around $l_{\xi}$
and this set is actually eigenvalues of
an infinite dimensional
Ornstein-Uhlenbeck type operator
$-\la^{-1}L_{\la,T_{\xi},W_0}$ 
which corresponds to $H_{i,\la}$ in finite dimensional cases.
We introduce $-L_{\la,T_{\xi},W_0}$ 
in Section~\ref{cons in wiener space} in a general setting.

\begin{dfi}
\begin{align}
\Lambda(\xi)
&=\left\{
\sum_{k=1}^{\infty}n_k|\zeta_k(\xi)|+E_0(\xi)
~\Big|~n_k\in \ZZ_{\ge 0}\,\,
\sum_kn_k<\infty\right\}\label{approximate ev},
\end{align}
where $E_0(\xi)=\sum_{\{k\,|\,\zeta_k(\xi)<0\}}|\zeta_k(\xi)|$.
\end{dfi}

We give an example of $\Lambda(\xi)$.

\begin{exm}\label{example su2}
We consider the case $G=SU(2)$ and the situation in
Lemma~\ref{hessian of E} (4) (ii).
There are freedom of choice of $\theta_1, \theta_2$.
Here, we choose so that $(\theta_1, \theta_2)$ corresponds to
the shortest geodesic.
Write $|\theta_1|=|\theta_2|=\theta>0$.
By checking the calculation in the proof of Lemma~\ref{hessian of E},
we see that $2\theta<1$.
Also note that $2|\theta+k|>1$ if $k\in \ZZ\setminus \{0\}$.
We can identify $V$ with the set consists of
$\theta(k)$ $(k\in \ZZ)$ which are defined by
\begin{align*}
 \theta(k)=
\begin{cases}
 (\theta+k,-\theta-k) & \text{if $\theta_1=\theta$,}\\
(-\theta-k,\theta+k) &\text{if $\theta_2=\theta$.}
\end{cases}
\end{align*}
Then the spectrum of 
$I_{H_{0,0}}+T_{\theta(k)}$, counting multiplicities, is
\begin{align*}
&\left\{1\pm\frac{2|\theta+k|}{m},\,\,
1\pm\frac{2|\theta+k|}{m}~\Big |~m\in \NN
\right\} \cup\{1,1,\ldots\}.
\end{align*}
Let $k\ge 0$.
Then all negative eigenvalues of $I_{H_{0,0}}+T_{\theta(k)}$ 
are given by
\begin{align*}
 \left\{1-\frac{2(\theta+k)}{m},\quad 1\le m\le 2k\right\}.
\end{align*}
Also all positive eigenvalues, except 1, are
\begin{align*}
 \left\{1+\frac{2(\theta+k)}{m},\quad 1\le m\le 2k\right\},\quad \quad
\left\{1\pm \frac{2(\theta+k)}{m}, \quad m\ge 2k+1\right\}.
\end{align*}
Note that the multiplicity of them are two.
Also
\begin{align}
E_0(\theta(k))=2\sum_{m=1}^{2k}\left(\frac{2(\theta+k)}{m}-1\right)=
4\left(k\sum_{m=2}^{2k}\frac{1}{m}+\theta\sum_{m=1}^{2k}\frac{1}{m}\right).
\label{bottom nonnegative}
\end{align}
We consider the case $k\le -1$.
Then all negative eigenvalues of 
$I_{H_{0,0}}+T_{\theta(k)}$ 
are given by
\begin{align*}
 \left\{1-\frac{2(|k|-\theta)}{m},\quad 1\le m\le 2|k|-1\right\}
\end{align*}
Also all positive eigenvalues, except 1,  are
\begin{align*}
 \left\{1+\frac{2(|k|-\theta)}{m},\quad 1\le m\le 2|k|-1\right\},\quad \quad
\left\{1\pm \frac{2(|k|-\theta)}{m}, \quad m\ge 2|k|\right\}.
\end{align*}
Again the multiplicity of them are two and
\begin{align}
E_0(\theta(k))=2\sum_{m=1}^{2|k|-1}\left(\frac{2(|k|-\theta)}{m}-1\right)=
2\left(2|k|\sum_{m=2}^{2|k|-1}\frac{1}{m}+1-2\theta\sum_{m=1}^{2|k|-1}
\frac{1}{m}\right).
\label{bottom negative}
\end{align}
Consequently, 
for any $k\in \ZZ$, we have
\begin{align}
& \Lambda(\theta(k))
=\Biggl\{
E_0(\theta(k))+\sum_{m=1}^{(|2k|\vee |2k+1|)-1}n_m\left(\frac{2|\theta+k|}{m}-1\right)
+\sum_{m=1}^{(|2k|\vee |2k+1|)-1}n_{m}'\left(\frac{2|\theta+k|}{m}+1\right)\nonumber\\
&\quad\qquad\qquad\qquad
+\sum_{m\ge |2k|\vee |2k+1|}N'_m\left(1+\frac{2|\theta+k|}{m}\right)
+\sum_{m\ge |2k|\vee |2k+1|}N_m\left(1-\frac{2|\theta+k|}{m}\right)+n\nonumber\\
&\quad\qquad\qquad\qquad \Bigg |~
\text{$n_m, n'_m, N'_m, N_m, n\in \ZZ_{\ge 0}$
and $\sum_{m}(n_m+n'_m+N'_m+N_m)<\infty$}
\Biggr\},
\label{identity of eigenvalue1}
\end{align}
where note that
\[
|2k|\vee |2k+1|
=
 \begin{cases}
 2k+1 & \text{$k\ge 0$,}\\
2|k| & \text{$k<0$.}
 \end{cases}
\]
\end{exm}

To state our main theorem,
we need to consider the set of essential spectrum of 
Ornstein-Uhlenbeck type operator
$-L_{\la,T_{\xi},W_0}$ at the geodesic $l_\xi$.

\begin{lem}\label{accumulation point}
 Let $\Lambda(\xi)^a$ be the set of accumulation points or infinite multiplicity
points of $\Lambda(\xi)$.
\begin{enumerate}
 \item[$(1)$] $\Lambda(\xi)^a=\Lambda(\xi)+\NN:=\left\{x+n~\Big |~ x\in 
\Lambda(\xi), 
n\in \mathbb{N}\right\}$ holds and $\Lambda(\xi)^a$ is a closed set.
Here in the definition of $\Lambda(\xi)+\NN$, 
we do not take the multiplicity into account.
Furthermore, $\Lambda(\xi)^a\subset \Lambda(\xi)$.
\item[$(2)$] $\Lambda(\xi)^a$ coincides with the set of essential spectrum 
of 
$-\la^{-1}L_{\la,T_{\xi},W_0}$ at the geodesic $l_\xi$.
We refer the readers to Theorem~$\ref{cons}$ for the
definition of $-L_{\la,T,W_0}$.

\item[$(3)$] Let $R>0$, $0<r<1$ and define 
\begin{align}
\Sigma_{R,r}=\left\{\left(\cup_{i=1}^N\Lambda(\xi_i)\right)
\cap [0,R]\right\}\setminus
\left(\cup_{i=1}^N(\Lambda(\xi_i)^a+(-r,r))\right),\label{e_1e_L}
\end{align}
counting multiplicities.
Then $\Sigma_{R,r}$ is a finite set.
\end{enumerate}
\end{lem}

\begin{rem}
 The results (1) and (2) above imply that
all real numbers which belong to
essential spectrum of $-\la^{-1}L_{\la,T_{\xi},W_0}$ 
are eigenvalues of $-\la^{-1}L_{\la,T_{\xi},W_0}$.
\end{rem}

\begin{proof}[Proof of Lemma~$\ref{accumulation point}$]
After $-L_{\la,T_\xi,W_0}$ will be defined, 
 the proof is easy and we omit the proof.
\end{proof}

Note that the function $N_R(r)=|\Sigma_{R,r}|$ $(r\in (0,1))$
is a decreasing 
function, where $|\cdot|$ denotes the cardinality of the set.

We also recall the definition of
the injectivity radius.
Let $i(G)$ be the injectivity radius of $G$.
That is,
\begin{align}
i(G)&=\sup\left\{r~|~\mbox{$\exp : \{\xi\in \g~|~|g|<r\}\to G$ gives a
local chart at $e$}
\right\}.\label{injectivity radius}
\end{align}

The following is our main theorem.

\begin{thm}\label{main theorem} 
Suppose that
$a$ does not belong to the cut-locus of $e$.
Let $\D$ be one of $D_{K,\delta}$, $B_{M,m,\theta}$, $C_{M,\beta}$, 
$\mathscr{D}_M$
which includes the minimal geodesic.
We assume the boundary measure is 0 in the sense of 
Definition~$\ref{definition of domain}$.
When $\D=D_{K,\delta}$, we assume $\delta$ is sufficiently small.
Let $\{e^{t\xi_i}\}_{i=1}^N$ be all geodesics in
$\D$ and suppose that $\partial \D$ does not
contain any geodesics.
Let $R$ be a positive number such that $R\notin \cup_{i=1}^N\Lambda(\xi_i)$.
Let us choose $r\in (0,1)$ such that $r$ is a continuous point of the function $N_R$.
Let $L=N_R(r)$
and
write $\Sigma_{R,r}=\{e_1,\ldots,e_L\}$
counting multiplicities in ascending order.
Then there exists $\la_0>0$ such that for all $\la\ge\la_0$, the following hold
for the spectral set $\sigma(-\la^{-1}L_{\la,\D})$.
\begin{enumerate}
 \item[$(1)$] 
$
 \sigma(-\la^{-1}L_{\la,\D})\cap 
[0,R]\cap
\left(\cup_{i=1}^N(\Lambda(\xi_i)^a+(-r,r))\right)^{\complement}
$
consists of $L$ eigenvalues
$
\left\{e_i(\la)\right\}_{i=1}^L
$
counting multiplicity in ascending order.
\item[$(2)$] $\lim_{\la\to\infty}e_i(\la)=e_i$\,\, $(1\le i\le L)$ hold.
\end{enumerate}
\end{thm}

\begin{rem}\label{remark on main theorem}
(1) 
Consider the case where $G=SU(n)$ and the situation in
Lemma~\ref{hessian of E} (4) (ii).
If $l_{\xi}$ is the minimum geodesic, then
$E_0(\xi)=0$.
Also
we see that $\lim_{|\xi|\to\infty}E_0(\xi)=\infty$.
Hence $0\notin \cup_{\xi\in V}\Lambda(\xi)^a$ and $e_1=0$.
For any $\ep>0$, by taking $\theta$ to be sufficiently small, 
we see that $\inf (E_0(\xi)\setminus \{0\})\ge 2-\ep$ holds.
Hence there exists $M$ such that
$\{1\pm \frac{|\theta_i-\theta_j|}{m} ; 1\le i<j\le n,\, m\ge M\}
\cap \cup_{\xi\in V}\Lambda(\xi)^a=\emptyset$.
This implies that $\lim_{r\to 0}N_R(r)=\infty$.

\noindent
(2)
We consider the case $G=SU(2)$ and the setting in
Example~\ref{example su2}.
In this case, $\min_{k\ne 0} E_0(\theta(k))=2(1-2\theta)$
and $\min\left(\{\zeta(\theta(0))\}\setminus\{0\}\right)=1-2\theta$.
Hence, $e_2=1-2\theta$.

\noindent
(3)
In the above theorem, we assume that $\D$ contains the minimal geodesic.
Hence, $|\nu_{\la,a}(\D)-1|=O(e^{-\la C})$ holds as $\la\to\infty$
for a certain positive constant $C$.
This follows from the large deviation estimate for 
$\nu_{\la,a}$ (\cite{ks1},\cite{inahama},\cite{ds}).

\noindent
(4) Note that Witten Laplacian is originally defined as a differential operator acting
on differential forms on Riemannian manifolds 
without boundary (\cite{witten}).
In \cite{a2011}, we proved a vanishing theorem of Witten Laplacian acting on
$1$-forms
on based loop group
$P_{e,e}(G)$ with the measure $\nu_{1,e}$, 
where $G$ is a compact connected and simply connected Lie
group.
That is, we defined Witten Laplacian $\Box$ acting on 
$L^2(\nu_{1,e})$-$1$-forms on $P_{e,e}(G)$ and
proved that $\dim\ker \Box=0$.
The proof is not based on spectral theoretical properties of 
Witten Laplacian but
on the fact that the Betti number of $P_{e,e}(G)$ is 0.

Witten Laplacian acting on differential forms on a manifold with the boundary
also have been defined and studied in finite dimensional cases
(\cite{chang-liu}, \cite{helffer-nier}).
We may consider infinite dimensional version of them if the boundary of
$\D$ is smooth.
To explain the case where the boundary is smooth, let $G=SU(n)$. 
Then, the Malliavin covariance matrix of
the random variable $(b(t),Y(t,e,w))$ is nondegenerate.
Hence, 
$\{w\in \Omega~|~Y(1,e,w)=a, |b(1)|=M\}$ and 
$\{h\in H~|~Y(1,e,h)=a, |b(1)|=M\}$ are smooth submanifolds in
$\Omega$ and $H$ respectively.
Hence $\{w\in \Omega~|~Y(1,e,w)=a, |b(1)|<M\}=Y^{-1}(\mathscr{D}_{M})$
is a domain of $S_a=\{w\in \Omega~|~Y(1,e,w)=a\}$
with a smooth boundary.
We explain the relation $\D$ and $Y^{-1}(\D)$ in more general setting in
Theorem~\ref{D and Y^-1(D)}.

Also, note that $\varphi_{B,m,\theta}^{4m}(Y(w))$ 
is smooth function in the sense of
Malliavin for any $G$.
 If this functional is non-degenerate in the sense of Malliavin, then
$B_{M,m,\theta}$ and
$\{w\in S_a~|~\varphi_{B,m,\theta}(Y(w))<M\}$ is also another candidate of
domain for consideration.

In \cite{chang-liu}, \cite{helffer-nier},
they assume that
the Morse function $E$ on $\D$
should be a Morse function on the boundary
$\partial \D$ also.
Note that $E$ is an energy function of paths in our case and
it is not trivial to see whether this assumption generically holds or not.

Also, we refer the readers to \cite{hino} for a study on
spectral properties of 
Witten Laplacians with respect to weighted Wiener measures.
See also Remark~\ref{remark on log-Sobolev with potential}
in Section~\ref{log-Sobolev with potential function}.

\noindent
(5)
The smallness assumption on $\delta$ in the case where $\D=D_{K,\delta}$
is used in Lemma~\ref{exponential estimate0}.
\end{rem}

In the above remark, we mention the existence of infinitely many
discrete spectrum of $-\la^{-1}L_{\la,\D}$ for large $\la$ around 1.
For large discrete spectrum, we have the following result.
We give the proof in Appendix.

\begin{pro}\label{example of discrete spectrum}
We consider Example~$\ref{example su2}$.
We use the notations there in the following statement.
Assume that $\theta$ is an irrational number.
Let $k,l\in \ZZ$, $M$ be a positive integer and $p$ be a prime number satisfying
that $p>\max\left(2|k|, 2|l|, 2M|\theta+k|,M\right)$.
Then it holds that
\begin{align}
 E_0(\theta(k))+M\left(1+\frac{2|\theta+k|}{p}\right)
\notin \Lambda(\theta(l))+\mathbb{N}.\label{criteria}
\end{align}
In particular, the following statement hold.
Let $\{k_i\}_{i=1}^N\subset \ZZ$ be a set of distinct integers.
For any $M\in \NN$ and prime number $p$ satisfying
$p>\max\{\{2|k_i|,2M|k_i+\theta|\}_{i=1}^N, M\}$, it holds that
\begin{align*}
\left\{E_0(\theta(k_i))+M\left(1+\frac{2|\theta+k_i|}{p}\right)
~\Big |~1\le i\le N
\right\}\subset
 \left(\cup_{i=1}^N\Lambda(\theta(k_i))\right)\setminus
\left(\cup_{i=1}^N\Lambda(\theta(k_i))^a\right).
\end{align*}
\end{pro}

\subsection{Preliminary from rough path analysis and Malliavin calculus}
\label{preliminary}

In this section, we prepare necessary results about Brownian rough path.
First, we recall the basic notions in rough path theory.
References for rough path theory can be found
in \cite{friz-hairer}, \cite{friz-victoir}, \cite{lq}.
Let
$\bX_{s,t}=(\X_{s,t}, \XX_{s,t})_{0\le s\le t\le 1}$ be a pair of
continuous two parameter functions with values
in $\RR^d\oplus (\RR^d\otimes \RR^d)$.
Below, we may denote $\X_{s,t}$ and $\X_{0,t}$ by $X_{s,t}$ and
$X_t$ respectively.
Let $\frac{1}{3}<\alpha\le \frac{1}{2}$ and
$2\le p<3$.
We consider the following condition 
(1), (2), (3).
\begin{enumerate}
 \item[(1)] $\X_{s,t}=\X_{s,u}+\X_{u,t}$, \quad
$\XX_{s,t}=\XX_{s,u}+\XX_{u,t}+\X_{s,u}\otimes \X_{u,t}$ $(0\le s\le u\le t),$
\item[(2)] $\displaystyle{\|\bX\|_{\alpha}:=\|\X\|_{\alpha}
+\sqrt{\|\XX\|_{2\alpha}}}<\infty$,
\item[(3)] $\displaystyle{\|\bX\|_{p\hyp var}:=\|\X\|_{p\hyp var}
+\sqrt{\|\XX\|_{p/2\hyp var}}}<\infty$,
\end{enumerate}
where for $Z=(Z_{s,t})$ $(0\le s\le t\le 1)$ with
values in a finite dimensional
normed linear space $V$, 
$\|Z\|_{\beta}$ $(0<\beta\le 1)$ and
$\|Z\|_{q\hyp var}$ $(q\ge 1)$ are defined by
\begin{align*}
& \|Z\|_{\beta}=\sup_{0\le s<t\le 1}\frac{|Z_{s,t}|_V}{|t-s|^{\beta}},\\
 & \|Z\|_{q\hyp var,[s,t]}\nonumber\\
&\,=\sup\Biggl\{\left(\sum_{i=1}^N|Z_{t_{i-1},t_i}|_V^q\right)^{1/q}~\Big |~
 \text{$\{s=t_0<\cdots<t_N=t\}$ moves all finite partition of $[s,t]$}
\Biggr\},\\
&\|Z\|_{q\hyp var}=\|Z\|_{q\hyp var,[0,1]}.
\end{align*}

$(\bX_{s,t})_{0\le s\le t\le 1}$ is called
an $\alpha$-H\"older rough path (a rough path with finite $p$-variation)
on $\RR^d$
if (1) and (2) ((1) and (3)) hold respectively.
We may denote $\bX_{s,t}$ by $\overline{X}_{s,t}$.
$\mathscr{C}^{\alpha}([0,1]\to\RR^d)$ (or $\mathscr{C}^{\alpha}(\RR^d)$ simply)
and $\mathscr{V}^p([0,1]\to\RR^d)$ (or $\mathscr{V}^p(\RR^d)$ simply)
denote the
set of all $\alpha$-H\"older rough paths and all rough paths 
with finite $p$-variation
on $\RR^d$ respectively.
They are complete metric spaces 
with respect to the distance $d_{\mathscr{C}^{\alpha}}(\bX,\bY)=
\|\X-\Y\|_{\alpha}+\|\XX-\YY\|_{2\alpha}$ and
$d_{\mathscr{V}^{p}}(\bX,\bY)=
\|\X-\Y\|_{p\hyp var}+\|\XX-\YY\|_{\frac{p}{2}\hyp var}$
respectively.
Note that we have natural inclusion 
$\mathscr{C}^{\alpha}([0,1]\to\RR^d)\subset \mathscr{V}^{1/\alpha}([0,1],\RR^d)$.

The property (1) implies $\X_{s,t}=\X_{0,t}-\X_{0,s}$ $(0\le s\le t\le 1)$.
Hence $\|\X\|_{\alpha}$ and $\|\X\|_{p\hyp var}$
denote the $\alpha$-H\"older norm and the $p$-variation norm of
the path $X_t=X_{0,t}$ 
and we use the same notation 
$\|X\|_{\alpha}$ and $\|X\|_{p\hyp var}$ to denote the
$\alpha$-H\"older norm and the $p$-variation norm
of a path $X=(X_t)_{0\le t\le 1}$ with values in 
a Euclidean space respectively.
Also we use the notation $C^{\alpha}([0,1]\to\RR^d)$ and 
$\mathcal{V}^p([0,1]\to \RR^d)$ 
to denote the set of $\alpha$-H\"older continuous
paths and continuous paths with finite $p$-variations
respectively.
Below, we use the notation $X_{s,t}=X_t-X_s$ for a path $X=(X_t)_{0\le t\le 1}$.

If $(X_t)_{0\le t\le 1}$ is a piecewise $C^1$ path, then 
$\X_{s,t}=X_t-X_s$, $\XX_{s,t}=\int_s^t\X_{s,u}\otimes d\X_{u}$ defines a
$\alpha$-H\"older rough path for any $\alpha\in (1/3,1/2]$
which is called a smooth rough path.
The set of
all $\alpha$-H\"older geometric rough paths on $\RR^d$ which
is denoted by $\mathscr{C}^{\alpha}_g(\RR^d)$ is the subset of
$\mathscr{C}^{\alpha}(\RR^d)$ consists of the elements which can be approximated 
by smooth rough paths in the topology of $d_{\alpha}$.
Similarly, the set of geometric rough paths with finite $p$-variation 
$\mathscr{V}^p_g(\RR^d)$ is defined.

For the standard basis $\{e_i\}_{i=1}^d$ of $\RR^d$, 
we write components of $\X$ and $\XX$ as follows:
\begin{align*}
 X^i_{s,t}=\left(\X_{s,t},e_i\right),\qquad
X^{i,j}_{s,t}=\left(\XX_{s,t},e_i\otimes e_j\right).
\end{align*}

\begin{rem}\label{remark on Besov norm}
The following norms $\|~\|_{B,4m,\theta/2}$
$\|~\|_{B,2m,\theta}$ are also used instead of $\|~\|_{\alpha}$,
$\|~\|_{2\alpha}$.
\begin{align*}
 \|\overline{X}^1\|_{B,4m,\theta/2}^{4m}
=
\iint_{0<s<t<1}
\frac{|\overline{X}^1_{s,t}|^{4m}}{(t-s)^{2+2m\theta}}dsdt,
\quad
\|\overline{X}^2\|_{B,2m,\theta}^{2m}
=
\iint_{0<s<t<1}
\frac{|\overline{X}^2_{s,t}|^{2m}}{(t-s)^{2+2m\theta}}dsdt,
\end{align*}
where $m\in \NN$ and $0<\theta<1$.
Actually, these norms are equivalent to each other in the following sense.
\begin{align*}
 \|\X\|_{\frac{\theta}{2}}\le C_{m,\theta}\|\X\|_{B,4m,\theta/2}, \quad\quad
\|\XX\|_{\theta}\le C_{m,\theta}
\left(\|\XX\|_{B,4m,\theta}+\|\X\|_{B,4m,\theta/2}^2\right)
\end{align*}
holds.
See \cite{a2011}.
Conversely, by an elementary calculation, we see that for 
$\alpha>\frac{\theta}{2}+\frac{1}{4m}$,
 \begin{align*}
\|\X\|_{B,4m,\theta/2}\le C_{m,\theta}\|\X\|_{\alpha},\quad
\|\XX\|_{B,4m,\theta}\le C_{m,\theta}\|\XX\|_{2\alpha}
\end{align*}
holds.
\end{rem}
We now recall the notion of Brownian rough path
$\bw$ which is a geometric lift of Brownian path
$w_t$.
Let $W$ be the set of all continuous mapping from $[0,1]$ to
$\mathfrak{g}$ starting at $0$.
By using the orthonormal basis $\{\ep_i\}\subset \g\subset M(n,\CC)$,
we can identify $W$ with $C([0,1]\to \RR^d ; w_0=0)$ 
which is a set of continuous paths starting at
$0$ as follows:
$w=(w(t))_{0\le t\le 1}(=\sum_{i=1}^d w^i(t)\ep_i
=(w^1(t),\ldots,w^d(t)))\in W$.
Let $\mu$ be the standard Brownian motion measure on 
$C([0,1],\RR^d)$ (and hence on $W$).
$H=H^1([0,1]\to\RR^d~|~h_0=0)$ is the Cameron-Martin space of 
$(W,\mu)$.
Let $\mu_{\la}$ denote the image measure of $\mu$ by
the mapping $w\mapsto \la^{-1/2}w$.
We may write $\mu_{\la,W}$.
We need to lift $w$ to a geometric rough path for $\mu_{\la}$
almost all $w$ for all $\la>0$.

Let $N$ be a positive integer.
We consider dyadic polygonal approximations of $w$:
\begin{align}
 w^N(t)=w(\tNkN)+2^{-N}(t-\tNkN)w_{\tNkN,\tNk}\quad
\tNkN\le t\le \tNk, \quad 1\le k\le 2^N, \label{dyadic}
\end{align}
where $\tNk=k2^{-N}$.
Let $\bw^N_{s,t}=(\wN_{s,t},\wwN_{s,t})$ be the smooth rough path
associated with $w^N$.
Also let
${w^N}^{\perp}(t)=w(t)-w^N(t)$ and define
\begin{align*}
 C(w^N,{w^N}^{\perp})_{s,t}&=\int_s^tw^N_{s,u}\otimes d{w^N}^{\perp}_u,\qquad
C({w^N}^{\perp},w^N)_{s,t}=\int_s^t{w^N}^{\perp}_{s,u}\otimes dw^N_{u}.
\end{align*}
Also we may use the notation
$
 C(x,y)_{s,t}=\int_s^tx_{s,u}dy_u,
$
where $x$ and $y$ are real valued continuous functions
and one of them are
bounded variation.
To fix a version of the solution of (\ref{sde}), 
we introduce the following set $\Omega$.

\begin{thm}\label{lift of w}
Let $\overline{w^N}_{s,t}=(\wN_{s,t},\wwN_{s,t})$ be the smooth rough path
associated with $w^N$ and set
\begin{align*}
 \Omega=\left\{w\in W~\Big |~
\text{the following assertions $\mathrm{(i)}\sim \mathrm{(iv)}$ 
concerning $w$ holds.}\right\}
\end{align*}
\begin{itemize}
\item[{\rm (i)}] $\lim_{N\to\infty}\overline{w^N}$ converges 
in the norm of 
$\|~\cdot~\|_{\alpha}$ for all $\alpha<1/2$.
We denote the limit by $\bw_{s,t}=(\w_{s,t},\ww_{s,t})$
and write $w^i_{s,t}=(\bw_{s,t},e_i)$, 
$w^{i,j}_{s,t}=(\bw_{s,t},e_i\otimes e_j)$.

\item[{\rm (ii)}] 
$C({w^N}^{\perp},w^N)$ and $C(w^N, {w^N}^{\perp})$ 
converge to $0$ as $N\to\infty$ in
the norm $\|~\cdot~\|_{2\alpha}$ for all $\alpha<1/2$.
\item[{\rm (iii)}]  For $1\le k,l\le d$, using $w^{k,l}_{s,t}$ which is defined in
{\rm (i)}, let
$
 d^{N,k,l}_{\tNiN,\tNi}(w)=w^{k,l}_{\tNiN,\tNi}
-\frac{1}{2}w^k_{\tNiN,\tNi}w^l_{\tNiN,\tNi}
$
and set 
\begin{align}
d^{N,k,l}_t(w)=\sum_{i=1}^{2^Nt}d^{N,k,l}_{\tNiN,\tNi}(w), \quad
t\in \{\tNk\}_{k=0}^{2^N}.\label{dN on tNk}
\end{align}
Then $\lim_{N\to\infty}\|d^{N,k,l}(w)\|_{p\hyp var}=0$
for all $p>1$ and $1\le k,l\le d$, where $\|~\cdot~\|_{p\hyp var}$ denotes the discrete 
$p$-variation norm for functions defined on $\{\tNk\}$.
\end{itemize}
Then the following hold.
\begin{enumerate}
 \item[$(1)$] $H\subset\Omega$ holds.
\item[$(2)$]  $\Omega$ is invariant under the multiplication of real numbers and
the addition of the element of $H$.
\item[$(3)$] For all $\la>0$, it holds that $\mu_{\la}(\Omega)=1$
and $\Omega^{\complement}$ is a slim set with respect to $\mu_{\la}$ for all
$\la>0$.
\item[$(4)$] The mappings $\Omega\ni w\mapsto \bw
\in \mathscr{C}^{\alpha}(\RR^d)$
is $\infty$-quasi continuous for any $0<\alpha<\frac{1}{2}$,
where the topology of $\Omega$ is of $C([0,1],\RR^d)$.
\end{enumerate}
\end{thm}

\begin{proof}
 In \cite{a2011}, we proved that the assertions (1), (2) hold for 
$\Omega$ which satisfies properties (i) and (ii).
Note that $\|h\|_{1\hyp var,[s,t]}\le \|h\|_H(t-s)^{1/2}$
holds for $h\in H$ and hence by the estimate of Young integral, we have
$\|C(h,w)\|_{2\alpha}\le C\|h\|_H\|w\|_{\alpha}$.
This implies the invariance of $\Omega$ 
under the addition of the element of $H$.
Also, in the same paper, we proved (3) in the case where $\la=1$ which 
immediately implies the same results hold for any $\la$.
So we consider the property (iii).
First, we prove that if (iii) holds for $w$ then so does for
$w+h$ for any $h\in H$.
Note that $\|w\|_{\alpha}<\infty$ holds for any $w\in \Omega$ 
and $\alpha<1/2$.
For $w\in \Omega$,
we have
\begin{align*}
 d^{N,k,l}_{\tNiN,\tNi}(w+h)&=
d^{N,k,l}_{\tNiN,\tNi}(w)+C(w^k,h^l)_{\tNiN,\tNi}
+C(h^k,w^l)_{\tNiN,\tNi}+d^{N,k,l}_{\tNiN,\tNi}(h)\nonumber\\
&\quad-\frac{1}{2}w^k_{\tNiN,\tNi}h^l_{\tNiN,\tNi}
-\frac{1}{2}h^k_{\tNiN,\tNi}
w^l_{\tNiN,\tNi}-\frac{1}{2}h^k_{\tNiN,\tNi}
h^l_{\tNiN,\tNi}.
\end{align*}
We have
$
 \sum_{i=1}^{2^N}|C(w^k,h^l)_{\tNiN,\tNi}|\le
C\sup_i|w^k_{\tNiN,\tNi}|
\|h^l\|_{1\hyp var, [0,1]}\to 0\quad \text{as $N\to\infty$}.
$
Except the sum of $d^{N,k,l}_{\tNiN,\tNi}(w)$,
we can estimate the sum of the other terms similarly to this.
Consequently, it suffices to show that the complement of
$A=\{w\in W~|~\lim_{N\to\infty}\|d^{N,k,l}(w)\|_{p\hyp var}=0\,\,
\text{for all $k,l$ and $p>1$}\}$
is a slim set.
In \cite{an}, we obtained the following.
Let $\tilde{d}^{N,k,l}_t(w)$ $(t\in [0,1])$ be a piecewise linear extension of
$d^{N,k,l}_t(w)$ $(t\in \{\tNk\})$
and set
\begin{align*}
 \Phi_{N,m}(w)&=2^{N/2}
\left(
\iint_{0<s<t<1}\frac{|\tilde{d}^{N,k,l}_{s,t}|^{4m}}{|t-s|^{2+2m\theta}}dsdt
\right)^{1/(4m)},
\end{align*}
where $m$ is a positive integer and $\theta\in (0,1)$.
Then $E[\Phi_{N,m}^{8m}]\le C_{m,\theta}<\infty$ holds for 
$(m,\theta)$ satisfying
$2m(1-\theta)>1$.
This follows from the moment estimate
$E[\left(d^{N,k,l}_{s,t}\right)^{2}]\le C2^{-N} (t-s)$,
the hypercontractivity of the Ornstein-Uhlenbeck semigroup
and a similar argument to the estimate (3.17) in \cite{a2011}.
By the Garsia-Rodemich-Rumsey inequality, we have
\begin{align*}
 |\tilde{d}^{N,k,l}_{s,t}(w)|\le 2^{-N/2}C'_m\Phi_{N,m}(w)
(t-s)^{\theta/2}.
\end{align*}
Choosing $\frac{\theta}{2}=\frac{1}{2}-\frac{1}{2m}$ and 
$\ep<\frac{1}{2}$,
we obtain
\begin{align*}
 |d^{N,k,l}_{s,t}(w)|\le
(2^{-N})^{\ep}C'_m\Phi_{N,m}(w)(t-s)^{1-\frac{1}{2m}-\ep},
\quad s<t,\quad s,t\in \{\tau^N_k\}.
\end{align*}
Let
\begin{align*}
 G_m(w)&=\sum_{N=1}^{\infty}(2^{-N})^{2m\ep}\Phi_{N,m}(w)^{4m}.
\end{align*}
This belongs to Wiener chaos of at most order $8m$ and converges in 
$L^2(\mu_{\la})$ and 
hence $G_m\in \cap_{q,s>1}\DD^{s,q}(\RR)$ holds,
where $\DD^{s,q}(\RR)$ denotes a Sobolev space in Malliavin calculus.
For this and the capacity $C^{s,q}$ which we will mention later, 
we refer the readers to
Remark~$\ref{remark to definition of Omega}$ (3) and references therein.
Consequently, we arrive at
\begin{align*}
\|{d}^{N,k,l}\|_{(1-\frac{1}{2m}-\ep)^{-1}\hyp var}\le
 \|{d}^{N,k,l}\|_{1-\frac{1}{2m}-\ep}\le (2^{-N})^{\ep/2}C'_mG_m^{1/(4m)}(w).
\end{align*}
Let 
\begin{align*}
A_{N,k,l,m}=\left\{w~\big |~\|{d}^{N,k,l}\|_{(1-\frac{1}{2m}-\ep)^{-1}\hyp var}
>N^{-2}\right\}.
\end{align*}
Then by the Chebyshev type inequality of $C^{s,q}$ capacity 
(see Section 2.2 in Chapter IV in 
\cite{malliavin}), we obtain
\begin{align*}
 C^{s,q}(A_{N,k,l,m})\le 
N^2(2^{-N})^{2m \ep}C'_m M_{q,s}\|G_m\|_{\DD^{s,q}}.
\end{align*}
Combining the Borel-Cantelli type inequality 
(1.2.4. Corollary in Chapter IV in \cite{malliavin}),
we obtain 
\begin{align*}
 C^{s,q}\left(\limsup_{N\to\infty}A_{N,k,l,m}\right)=0,
\end{align*}
which implies $\limsup_{N\to\infty}A_{N,k,l,m}$ is a slim set.
Let $E$ be the slim set defined in the proof of Theorem 3.1 and Theorem 3.2 
in \cite{a2011}.
Then, by the superadditivity of $C^{s,q}$ capacity 
(1.2.1 Proposition in Chapter IV in \cite{malliavin}),
$E'=E\cup \left(\cup_{1\le k,l\le d, m, \theta\in \QQ}
\limsup_{N\to\infty}A_{N,k,l,m}\right)$
is also a slim set.
Since ${E'}^{\complement}\subset \Omega$, $\Omega^{\complement}$ is a slim set.
\end{proof}

\begin{rem}
(1) For $w\in \Omega$, we write $C(w,w)_{s,t}=\overline{w}^2_{s,t}$.
Note that, for $w\in \Omega$, $\int_s^t{w^N}^{\perp}_{s,u}\otimes d{w^N}^{\perp}_u$
is well-defined because 
$({w^N}^{\perp}_t)=(w_t-w^N_t)\in \mathcal{C}^{\alpha}_w([0,1]\to\RR^d)$.
We write $C({w^N}^{\perp},{w^N}^{\perp})_{s,t}
=\int_s^t{w^N}^{\perp}_{s,u}\otimes d{w^N}^{\perp}_u$.
Then, it holds that
\begin{align*}
 C(w,w)_{s,t}&=C({w^N}^{\perp},{w^N}^{\perp})_{s,t}
+C(w^N,{w^N}^{\perp})_{s,t}+C({w^N}^{\perp},w^N)_{s,t}+C(w^N,w^N)_{s,t}.
\end{align*}
Moreover, noting the property of the elements of 
$\Omega$ in Theorem~\ref{lift of w} (i), (ii), we see that
for all $0<\alpha<1/2$,
\begin{align*}
 \lim_{N\to\infty}\|C({w^N}^{\perp},{w^N}^{\perp})\|_{2\alpha}=0.
\end{align*}

\noindent
(2) $d^{N,k,l}_t$ appeared in the study of asymptotic error distribution of
Milstein approximation scheme driven by (fractional) Brownian motion.
For instance, see \cite{an}.
\end{rem}

\begin{dfi}
Define
\begin{align*}
 \mathscr{C}^{\Omega}(\RR^d)=\{\bw~|~w\in \Omega\}
\end{align*}
and 
for $\frac{1}{3}<\alpha<\frac{1}{2}$, we define
\begin{align*}
 \mathscr{C}^{\alpha,\Omega}(\RR^d)=\{\bw\in \mathscr{C}^{\alpha}(\RR^d)
~|~w\in\Omega\}
\end{align*}
and identify $\Omega$ with them.
\end{dfi}

\begin{rem}\label{remark to definition of Omega}
(1) $\Omega$ is a Borel measurable subset of $W$ and
$\mathscr{C}^{\Omega}(\RR^d)$
is also a Borel measurable subset of 
$\mathscr{C}^{\alpha}_g(\RR^d)$.
In the calculation below, for $w\in \Omega$, we view $\bw$ as an element
of $\mathscr{C}^{\alpha}(\RR^d)$ for some fixed $\alpha$.

\noindent
(2) 
As already explained, to treat the Brownian motion on $G$ defined by
the transition probability density function
$p(t/\la,x,y)$, we need to consider the Brownian motion 
$w=(w(t))$ on 
$\mathfrak{g}$ whose
covariance satisfies $E[w^i(t)w^j(s)]=\delta_{i,j}\la^{-1}\min(s,t)$.
The above theorem assures the Brownian motion can be lifted to
a geometric rough path for all $\mu_{\la}$ simultaneously.
Also we need to consider the derivative of the functional
$w(\in \Omega)\mapsto Y(t,e,w)$ with the direction $H$,
where $Y$ is the solution of (\ref{RDE}).
The invariance property of $\Omega$ under the addition of the elements
of $H$ in the above is important in that argument.

\noindent
(3) The notion of slim set is due to Malliaivin.
Slim set of $W$ is a negligible set for any $C^{s,q}$ capacity.
As a consequence, any measure associated with a 
positive generalized Wiener functional
does not charge in any slim set.
Let us recall an important example of positive generalized Wiener functionals.
Let $\DD^{k,p}(W,\RR^d)$ (or simply $\DD^{k,p}(\RR^d)$)
denote the Sobolev space which 
consists of the all functions
$F : W\to \RR^N$ such that
$D^lF\in L^p(\mu_{\la})$ for all $0\le l\le k$ and $p\ge 1$,
where $D^lF$ denotes the $l$-times Malliavin derivative which takes values in
$\otimes^lH$.
Also let $\DD^{\infty}(\RR^d)=\cap_{p\ge 1, k\ge 0}\DD^{k,p}(\RR^d)$.
As we will explain, $\DD^{1,2}$ is the same as the domain
$\rD(\E_{\la,W})$ in the final part of this Section.
Suppose $(DF(w)DF^{\ast}(w))^{-1}\in L^p(\mu_{\la})$ for all $p\ge 1$.
Then the composition function $\delta_a(F)$
of the delta function $\delta_a$ $(a\in \RR^N)$
and $F$ is a positive generalized Wiener functional in the sense of
Watanabe and Sugita.
More precisely, $\delta_a(F)$ belongs to a Sobolev space
$\DD^{-k,p}(\RR)$ for some $k\in \NN$ and $p>1$.
 It is known that the coupling $\langle \delta_a(F),\varphi\rangle$
where $\varphi$ is a smooth function in the sense of Malliavin coincides with
the integral $\int_W\tilde{\varphi}(w)dm_{\la,a}(w)$, where
$\tilde{\varphi}$ is an $\infty$-quasi-continuous modification of
$\varphi$ and $m_{\la,a}$ 
is a finite measure on $W$ associated with $\delta_a(F)$ and
sometimes is written as $\delta_a(F)\mu_{\la}(w)$.
This is not a 
probability measure and we denote the normalized probability measure
$p(\la^{-1},e,a)^{-1}\delta_a(F)\mu_{\la}
(=p(\la^{-1},e,a)^{-1}m_{\la,a})$ by $\mu_{\la,a}$.
Also we have the following estimate
\begin{align}
 \left|\int_W\tilde{\varphi}(w)d\mu_{\la,a}(w)\right|
&\le
p(\la^{-1},e,a)^{-1}\|\varphi\|_{\DD^{k,q}}\|\delta_a(F)\|_{\DD^{-k,p}},
\label{pgwf}
\end{align}
where $q$ is a positive number satisfying $\frac{1}{p}+\frac{1}{q}=1$.
Actually, it is proved that $\mu_{\la,a}(S_a^\complement)=0$, where
$S_a=\{w\in W~|~\tilde{F}(w)=a\}$ and $\tilde{F}$ is again 
an $\infty$-quasi-continuous modification of $F$.
These results can be naturally extended to the 
compact Riemannian manifold $M$ valued
Wiener functional $F$.
In this paper, we apply such kind of results to the case 
where $M$ is a compact Lie group $G$ and $F(w)=Y(1,e,w)$.
We refer the readers for the above subjects and results 
to \cite{malliavin}, \cite{nualart}, \cite{shigekawa1}, \cite{sugita1},
\cite{watanabe1}, \cite{airault-malliavin}.

We recall necessary results on Sobolev spaces 
over submanifold $S_a$ 
to study the operator $-L_{\la,\Pea}$
in the final part of this section.

\noindent
(4)
In \cite{a1993}, we consider submanifolds in Wiener spaces defined by
solutions of stochastic differential equations(=SDEs)
and made use of Malliavin calculus developed at that time.
However, after the works of Terry Lyons on rough path analysis
(\cite{lyons}, \cite{lq}),
it is more natural to define the submanifolds by using the solutions
to rough differential equations which give the nice versions of the solutions
of SDEs.
In this paper, we study our problem in such a framework.
\end{rem}

\begin{dfi}\label{W_0} 
We can view $W^{\ast}$ as the subspace of $H$ by the identification
$W^{\ast}\subset H^{\ast}\simeq H\subset W$.
In this paper, we call a linear subspace $W_0$ of $W$ finite codimensional subspace if
there exists a finite dimensional subspace $V\subset W^{\ast}$
such that
\[
 W_0=\{w-P_Vw~|~w\in W\},
\]
where $P_V$ denotes the orthogonal projection mapping onto $V$
defined on $W$ by
$P_Vw=\sum_{i=1}^N (w,v_i)v_i$ and
$\{v_i\}_{i=1}^N$ is an orthonormal basis of $V\subset W^{\ast}\subset H$.
Also we define $H_0=\{P_{W_0}h~|~h\in H\}$.
In the above, $(w,v_i)$ is the natural paring and coincides with the inner product
$(w,v_i)$ in $H$ if $w\in H$.
Also we write $P_{W_0}w=w-P_Vw$ $(w\in W)$ and denote $V$ by $W_0^{\perp}$.
In the following, we use the notation $\eta$ and
$\eta^{\perp}$ of the elements of $W_0$ and $W_0^{\perp}$ respectively.
Finally, we note that we allow the case $W=W_0$.
\end{dfi}

Note that $H_0$ is the Cameron-Martin space of
$W_0$ with the image measure $\mu_{\la,W_0}=(P_{W_0})_{\ast}\mu_{\la,W}$.
We now recall the statement Remark~\ref{remark to definition of Omega} (3),
``$\mu_{\la,a}(S_a^\complement)=0$''.
In this paper, we define $S_a$ by using the solution $Y(t,e,w)$ of RDE driven by
$w\in \Omega$ (c.f. (\ref{RDE}), (\ref{RDE2}))) and 
introduce a local coordinate on $S_a$ in a neighborhood of 
the point $k\in S_a^H$.
$H_0$ arises as the tangent space $T_kS_a$ of $S_a$ at $k$.
The local coordinate and the coordinate function are defined on 
$W_0=\oTkSa$.
We explain the definition of $T_kS_a$ and $\oTkSa$ in 
Definition~\ref{tangent space}.
Furthermore, we see that the measure $\mu_{\la,a}$ locally can be viewed as the
image measure of weighted Wiener measure on $W_0$.
The following proposition is trivial by the definition
and the regularity property of the element of $H$.

\begin{pro}\label{Omega0}
 Let $W_0$ be a finite codimensional subspace of $W$.
Then $P_{W_0}\Omega\subset \Omega$ and $P_{W_0}\Omega+H_0\subset P_{W_0}\Omega$.
Also it holds that $\Omega\cap W_0=P_{W_0}\Omega$.
We write $\Omega_0=\Omega\cap W_0$.
\end{pro}

We now introduce the set of pathwise smooth functions.

\begin{dfi}\label{C^k_b function}
Let $W_0$ be a finite codimensional subspace of $W$.
\begin{enumerate}
 \item[$(1)$] Let $U$ be a subset of $\Omega_0$.
For $\eta\in \Omega_0$, let $U(\eta)=\{h\in H_0~|~\eta+h\in U\}$.
We call $U$ an $H_0$-open subset if $U(\eta)$ is an open subset of $H_0$
for any $\eta\in U$.
\item[$(2)$] Let $U$ be an $H_0$-open subset of $\Omega_0$.
Let $f : U\to E$ be a measurable function with values in 
a Hilbert space $E$.
Let $m, n\in \NN\cup \{0\}$.
We say that $f$ is a $C^n_b$ function with direction $H_0$ if
$f(\eta+\cdot)\in C^n_b(U(\eta)\to E)$ and 
$\|f\|_{C^n_{b,H_0}}
:=\sum_{l=0}^n\sup\{\|D^lf(\eta+h)\|_{L_2(H_0^{\otimes l},E)}~|~
\eta\in U, h\in U(\eta)\}
<\infty$, where $L_2$ denotes the set of Hilbert-Schmidt operators.
We denote the set of all such functions by $C^{n}_{b,H_0}(U,E)$.
Let $D$ be an open subset of a compact Lie group $G$ or Euclidean spaces.
$C^{m,n}_{b,H_0}(D\times U)$ denotes the set of functions on
$D\times U$ which are of $C^m_b$ with respect to the variable of $D$
and of $C^n_{b,H_0}$ with respect to the variable of $U$.
\end{enumerate}
 \end{dfi}

\begin{exm}\label{examples of H open set}
 Let $W_0$ be a finite codimensional subspace of $W$.
For $\ep>0$ and $k\in H$, define
\begin{align*}
 B_{\ep}^{\Omega_0}(0)&=\left\{\eta\in \Omega_0~
|~\|\overline{\eta}\|_{\alpha}<\ep
\right\},\\
B_{k,\ep}^{\Omega_0}(0)&=\left\{\eta\in \Omega_0~|~
\|\overline{\eta}\|_{\alpha}<\ep,\,
\|C(k,\eta)\|_{2\alpha}<\ep,\, \|C(\eta,k)\|_{2\alpha}<\ep\right\},\\
U_{\ep}^{\Omega}(k)&=\left\{\eta\in \Omega~|~
\|\overline{\eta-k}\|_{\alpha}<\ep, 
\|C(\eta-k,k)\|_{2\alpha}<\ep, \|C(k,\eta-k)\|_{2\alpha}<\ep
\right\}.
\end{align*}
The first two sets are neighborhoods of $0$ in $\Omega_0$.
The last one is a neighborhood of $k$ in $\Omega$ and
$U_{\ep}^{\Omega}(k)=k+B_{k,\ep}^{\Omega}(0)$ holds.
\end{exm}

\begin{rem}\label{remark on controlled path}
(1) The solutions of rough differential equations(=RDEs) 
are typical examples belonging 
to $C^n_{b,H}(U_{\ep}^{\Omega}(k),\RR^d)$.
Let us recall the notion of controlled paths to explain 
RDEs and the solutions.

Let $1/3<\alpha<1/2$.
Let $V$ be a finite dimensional normed
linear space.
$V$-valued path $(Y_t)_{0\le t\le 1}$ is 
said to be cotrolled by $X\in C^{\alpha}([0,1]\to \RR^d)$ 
($X\in \mathcal{V}^p([0,1]\to\RR^d)$ respectively)
if there exists 
$(Y'_t)\in C^{\alpha}([0,1]\to \LL(\RR^d,V))$
($(Y'_t)\in \mathcal{V}^{p}([0,1]\to \LL(\RR^d,V))$ respectively)
such that
$R^Y_{s,t}=Y_{s,t}-Y'_sX_{s,t}$
satisfies $\|R^Y\|_{2\alpha}<\infty$ 
($\|R^Y\|_{\frac{p}{2}\hyp var}<\infty$ respectively).
The pair $(Y_t,Y'_t)$ is called a controlled rough path
and we denote the set of controlled rough paths associated with
$X\in C^{\alpha}([0,1]\to\RR^d)$
by
$\mathcal{C}^{\alpha}_X([0,1]\to V)$
which is a Banach space with the norm
$\|(Y,Y')\|_{\mathcal{C}^{\alpha}_X}=|Y_0|+|Y'_0|+\|Y'\|_{\alpha}
+\|R^Y\|_{2\alpha}$.
Similarly, we denote the set of controlled rough paths associated with
$X\in \mathcal{V}^p([0,1]\to\RR^d)$ and the norm by
$\mathcal{V}^{p}_X([0,1]\to\RR^d)$ and 
$\|(Y,Y')\|_{\mathcal{V}^p_X}
=|Y_0|+|Y'_0|+\|Y'\|_{p\hyp var}
+\|R^Y\|_{\frac{p}{2}\hyp var}$
respectively.
Note that for $X\in C^{\alpha}([0,1]\to\RR^d)$,
the natural inclusion $\mathcal{C}^{\alpha}_X([0,1]\to V)\subset
\mathcal{V}^{1/\alpha}_X([0,1]\to V)$ and an estimate
$\|(Y,Y')\|_{\mathcal{V}^{1/\alpha}_X}\le \|(Y,Y')\|_{\mathcal{C}^{\alpha}_X}$
hold.

Let $\bX_{s,t}=(X_{s,t},\XX_{s,t})\in \mathscr{V}^{1/\alpha}([0,1]\to \RR^d)$
and $(Y_t,Y'_t)\in 
\mathcal{V}_X^{1/\alpha}([0,1]\to \LL(\RR^d,V))$.
Note that 
$Y'_t\in \LL(\RR^d,\LL(\RR^d,V))$ and 
$Y'_t(\xi)\in \LL(\RR^d,V)$ for $\xi\in \RR^d$.
We can define rough integral 
$I_t=\int_0^tY_sd\bX_s$
as follows.
Let $\Xi_{s,t}=Y_sX_{s,t}+Y'_s\XX_{s,t}$.
Here precisely, $Y'_s\XX_{s,t}=\sum_{i=1}Y'_s(e_i)e_j X^{i,j}_{s,t}$.
Then $I_t$ is defined by 
$I_t:=\mathcal{I}(\Xi)_{0,t}
=\lim_{|\Delta|\to 0}\sum_{i=1}^M\Xi_{t_{i-1},t_i}$, where
$\Delta=\{0=t_0<\cdots<t_M=t\}$ and $|\Delta|=\max_i|t_i-t_{i-1}|$.
The existence of the limit can be checked by the Sewing lemma 
(\cite{gubinelli}, \cite{friz-hairer}).
Note that $(I_t,Y_t)\in \mathcal{V}_X^{1/\alpha}([0,1]\to V)$ holds.

We consider the following RDE driven by $\bw$ which is a lift
of $w\in \Omega$.
\begin{align}
 Y(t,x,w)=x+\int_0^tY(s,x,w)dw_s,\quad x\in M(n,\CC).
\label{RDE}
\end{align}
Actually we consider the case where $x\in G\subset SU(n)$ only.
Note that $(Y(t,x,w),Y(t,x,w))\in \mathcal{C}_w^{\alpha}([0,1],M(n,\CC)))$
and the above integral is a rough integral.
For notational simplicity, we use the notation
$Y(t,x,w)$ and $dw_s$ instead of $Y(t,x,\bw)$ and $d\bw_s$ respectively
throughout this paper.
We think this makes no confusion since
we already defined $\bw$ canonically for each $w\in \Omega$.
Later, we define rough path $\overline{w+h}$ and $\overline{w_f}$.
For these, we may not omitting the overline to denote them for the clarity.
In (\ref{RDE}),
note that $Y(t,x,w)$ takes values in $M(n,\CC)$ and
the product in the above equation should be understood as a matrix product.
That is, the equation reads
\begin{align}
 Y(t,x,w)=x+\sum_{i=1}^d\int_0^t
Y(s,x,w)\ep_i dw^i(s)\label{RDE2}
\end{align}
and the product $Y(s,x, w)\ep_i$ is the matrix product.
Note that $Y(t)$ is the solution to (\ref{RDE}) ((\ref{RDE2})) 
if and only if
$Y(0)=x$ and the following estimate holds:

There exists $C>0$ which depends only on $x$ and $w$ such that
\begin{align*}
 \left|Y(t)-Y(s)-\sum_{i}Y(s)\ep_iw^i_{s,t}
-\sum_{i,j}Y(s)\ep_i\ep_jw^{i,j}_{s,t}\right|
\le C|t-s|^{3\alpha}\qquad \text{for any $0<s\le t<1$},
\end{align*}
where $Y(s)\ep_i, Y(s)\ep_i\ep_j$ are the matrix products.
Finally, note that the solution of (\ref{RDE}) 
is a version of the solution to the Stratonovich SDE 
\begin{align}
 Y(t,x,w)=x+\int_0^tY(s,x,w)\circ dw(s), 
\qquad Y(0,x,w)=x.\label{stratonovich sde}
\end{align}

\noindent
(2)
Let $h\in H$.
Of course, the solution $Y(t,x,w+h)$ satisfies the following inequality.
\begin{align*}
&  \Bigl|Y(t,x,w+h)-Y(s,x,w+h)-
\sum_i Y(s,x,w+h)\ep_i\left(w^i_{s,t}+h^{i}_{s,t}\right)\\
&\qquad \qquad
-\sum_{i,j}Y(s,x,w+h)\ep_i\ep_j(w+h)^{i,j}_{s,t}\Bigr|\le
C|t-s|^{3\alpha}.
\end{align*}
Note that $(Y(t,x,w), Y(t,x,w))\in \mathcal{C}^{\alpha}_w([0,1],M(n,\CC))$ 
holds.
However
$(Y(t,x,w+h),Y(t,x,w+h))\in
\mathcal{C}^{\alpha}_w([0,1], M(n,\CC))$ does not hold in general
because $h$ may not belong to $C^{2\alpha}$.
That is, $R^{Y(w+h)}_{s,t}:=Y(t,x,w+h)-Y(s,x,w+h)-
Y(s,x,w+h)w_{s,t}$ may not satisfy the necessary H\"older continuity.
However, $\|h\|_{1\hyp var, [s,t]}\le \|h\|_H(t-s)^{1/2}$ holds
and hence
we see that $\|R^{Y(w+h)}\|_{1/(2\alpha)\hyp var}<\infty$
for any $1/3<\alpha<1/2$.
That is, $(Y(t,x,w+h),Y(t,x,w+h))\in 
\mathcal{V}^{1/\alpha}_w([0,1], M(n,\CC))$ holds.

For $\bX_{s,t}=(X_{s,t},\XX_{s,t})\in 
\mathscr{V}^{1/\alpha}([0,1]\to\RR^d)$ and
$(Y_t,Y'_t)\in \mathcal{V}^{1/\alpha}_X([0,1]\to \mathcal{L}(\RR^d,V))$,
we have the following estimate 
(\cite{friz-shekhar}, \cite{rtt}, \cite{friz-hairer}):
\begin{align}
 &\left|\int_s^tY_udX_u-Y_sX_{s,t}-Y'_s\XX_{s,t}\right|\nonumber\\
&\le
C\left(\|R^Y\|_{1/(2\alpha)\hyp var,[s,t]}\|X\|_{1/\alpha\hyp var, [s,t]}+
\|Y'\|_{1/\alpha\hyp var,[s,t]}\|\XX\|_{1/(2\alpha)\hyp var,[s,t]}\right).
\label{estimate of rough integral with 1/alpha variation}
\end{align}
We apply this estimate to rough integrals $I_t$ and $J_{s,t}$ below.
Let $f\in C^2_b(G,\LL(\g,\RR^m))$ and 
$g\in C^2_b(G,\LL(\RR^m\otimes \g,\RR^l))$
and set
\begin{align*}
 I_t&=\int_0^tf(Y(u,e,w+h))dw_u,\\
J_{s,t}&=\int_s^tg(Y(u,e,w+h))(I_{s,u},dw_u),
\end{align*}
where $f\in C^n_b$ means that
$f\in C^n$ and $f$ and all its derivatives are bounded continuous functions.
Let
\begin{align*}
 \Xi^I_{u,v}&=f(Y(u,e,w+h))w_{u,v}+(Df)(Y(u,e,w+h))[Y(u,e,w+h)\ep_i]
\ep_j w^{i,j}_{u,v},\\
\Xi^J_{u,v}&=g(Y(u,e,w+h))(I_{s,u},w_{u,v})
+g(Y(u,e,w+h))(f(Y(u,e,w+h)\ep_i,\ep_j))w^{i,j}_{u,v}\nn\\
&\quad
+(Dg)(Y(u,e,w+h))[Y(u,e,w+h)\ep_i]
\left(I_{s,u},\ep_j\right)w^{i,j}_{u,v}.
\end{align*}
Then by the definition of rough integrals,
we have $I_t=\mathcal{I}(\Xi^I)_{0,t}$ 
and $J_{s,t}=\mathcal{I}(\Xi^J)_{s,t}$.
Then, for any $w\in\Omega$, we have, for $0\le s\le t\le 1$,
\begin{align}
& \left|I_{s,t}-\left(f(Y(s,e,w+h))w_{s,t}+
(Df)(Y(s,e,w+h))[Y(s,e,w+h)\ep_i]\ep_j
{w}^{i,j}_{s,t}\right)\right|\nn\\
&\le C(\|\overline{w+h}\|_{\alpha})
\Bigl(
\left(\|\overline{w+h}\|_{\alpha}\|\bar{w}\|_{\alpha}^2
+\|\overline{w+h}\|_{\alpha}^2\|\bar{w}\|_{\alpha}\right)(t-s)^{3\alpha}
+\|h\|_{1\hyp var, [s,t]}\|w\|_{\alpha}(t-s)^{\alpha}
\Bigr), \label{estimate of rough integral0}\\
&\|J\|_{2\alpha}\le C(\|\overline{w}\|_{\alpha},\|h\|_H)
\|\bar{w}\|_{\alpha}^2,\label{estimate of iterated integral0}
\end{align}
where $C$ is a positive polynomial growth increasing function
on $\RR^+$.

\noindent
$(3)$
Also the integrals $I$ and $J$ in (2) can be defined using
\begin{align*}
 \tilde{\Xi}^I_{u,v}&=f(Y(u,e,w+h))w_{u,v}+(Df)(Y(u,e,w+h))[Y(u,e,w+h)\ep_i]
\ep_j
C((w+h)^i,w^j)_{u,v},\\
\tilde{\Xi}^J_{u,v}&=g(Y(u,e,w+h))(I_{s,u},w_{u,v})
+g(Y(u,e,w+h))(f(Y(u,e,w+h)\ep_i,\ep_j))w^{i,j}_{u,v}\nn\\
&\quad
+(Dg)(Y(u,e,w+h))[Y(u,e,w+h)\ep_i]
\left(I_{s,u},C((w+h)^i,w^j)_{u,v}\ep_j\right),
\end{align*}
as
$
 I_{s,t}:=\mathcal{I}(\tilde{\Xi}^I)_{s,t}, \,\,
J_{s,v}:=\mathcal{I}(\tilde{\Xi}^J)_{s,v}\,\,
(s\le v\le 1),
$
which follow from 
\[
 |C(h^i,w^j)_{u,v}|\le 
C\|h^i\|_{1\hyp var, [u,v]}\|w^j\|_{1/\alpha\hyp var, [u,v]}.
\]
By this expression, we have the following estimate.
\begin{align}
 &|I_{s,t}-\tilde{\Xi}^I_{s,t}|\le 
C(\|\overline{w+h}\|_{\alpha})
\left(\|w\|_{\alpha}+\|\overline{w}^2\|_{2\alpha}+\|C(h,w)\|_{2\alpha}\right)
(t-s)^{3\alpha},\label{estimate of rough integral}\\
&\left|J_{s,t}-g(Y(s,e,w+h)
(f(Y(s,e,w+h))\ep_i,\ep_j))w^{i,j}_{s,t}\right|\nn\\
&
\le C(\|\bar{w}\|_{\alpha}, \|\bar{h}\|_{\alpha}, \|C(h,w)\|_{2\alpha},
\|C(w,h)\|_{2\alpha})
\Bigl(\|\bar{w}\|_{\alpha}+\|C(h,w)\|_{2\alpha}+
\|C(w,h)\|_{2\alpha}\Bigr)^2(t-s)^{3\alpha}
\label{estimate of iterated integral},
\end{align}
where $C$ are again polynomial growth increasing functions.
Of course, we have $\|C(h,w)\|_{2\alpha}\le C\|h\|_H\|w\|_{\alpha}$
and this is enough for most of purposes.
However, the estimates (\ref{estimate of rough integral}) 
and (\ref{estimate of iterated integral}) are useful for the proof of
Lemma~\ref{local chart2}.
Hence, we use the estimates (\ref{estimate of rough integral}) 
and (\ref{estimate of iterated integral}) in stead of
(\ref{estimate of rough integral0}) 
and (\ref{estimate of iterated integral0})
in the calculations below.

\noindent
(4) Let $w\in \Omega$ and $h\in H$.
Then for any $u_1\in \RR$ and $u_2\in H$, $u_1w+u_2h\in \Omega$.
Therefore, the solution $Y(t,x,u_1w+u_2h)$ is defined as the functional
of the rough path
$\overline{u_1w+u_2h}$.
Clearly, the essential roughness of the driving noise of
$u_1w+u_2h$ is in $w$ only and the solution can be viewed as a functional of $w$.
To be more precise,
$Y(t,x,u_1w+u_2h)$ coincides with 
the solution to the following RDE driven by
$\bw$:
\begin{align*}
 Y_t=x+u_1\int_0^tY_sdw_s+u_2\int_0^tY_sdh_s.
\end{align*}
\end{rem}

In view of the above remark, we prepare a lemma.

\begin{lem}\label{w_f-roughpath}
 Let $f=(f(t))_{0\le t\le 1}\in H^1([0,1],\LL(\RR^d,\RR^d))$ and
set $w_f(t)=\int_0^tf(s)dw(s)$ as the Stieltjes integral for $w\in \Omega$.

\begin{enumerate}
 \item[$(1)$] $(w_f,f)\in \mathcal{V}_w^{1/\alpha}([0,1]\to\RR^d)$ holds,
where $\alpha<1/2$.
Therefore, the rough path $\overline{w_f}$ can be defined by the
rough integral against $w$.
Also for any 
$(Y_t,Y'_t)\in \mathcal{V}^{1/\alpha}_w([0,1]\to
\mathcal{L}(\RR^d,\RR^m))$,
it holds that
\begin{align}
 \int_0^tY_sd\overline{w_f}(s)=\int_0^tY_sf(s)dw_s,
\label{chain rule}
\end{align}
as rough integrals in the sense of controlled rough paths by $w$.
\item[$(2)$]Let $w_f^N$ also denote the dyadic polygonal approximation of
$w_f$ similarly to $w^N$.
Set ${w_f^N}^\perp=w_f-w_f^N$.
The following hold for all $w\in \Omega$.
\begin{itemize}
 \item[{\rm (i)}] $\lim_{N\to\infty}\|w_f^N-w_f\|_{1/\alpha\hyp var}=0$\, for all 
$\alpha<1/2$.
\item[{\rm (ii)}] 
$\lim_{N\to\infty}\|C(w_f^N,w^N_f)
-\overline{w_f}^2\|_{(2\alpha)^{-1}\hyp var}=0$\,
for all $\alpha<1/2$.
\item[{\rm (iii)}] 
$\lim_{N\to\infty}\|C(w_f^N,{w_f^N}^\perp)\|_{1/(2\alpha)\hyp var}=
\lim_{N\to\infty}\|C({w_f^N}^{\perp},{w_f^N})\|_{1/(2\alpha)\hyp var}=0
$\, for all $\alpha<1/2$.
\end{itemize}
\end{enumerate}
\end{lem}

\begin{exm}\label{example of f}
 Typical example of $f$ is as follows.
Let $k\in H$ and consider the solution $Y(t,e,k)$ and
set $f(t)=Ad\left(Y(t,e,k)\right)$.
In this case, we denote $w_f$ by $U_kw$.
That is, $(U_kw)(t)=\int_0^tAd\left(Y(s,e,k)\right)dw_s$.
This defines a unitary operator from $T_kS_a$ to $H_{0,0}$.
We use this mapping many times as the unitary operator
from $T_kS_a$ to $H_{0,0}$.
Note that $T_kS_a$ will be introduced in Definition~$\ref{tangent space}$.
\end{exm}

\begin{rem}\label{remark on wf}
 Actually, $(w_f,f)
\in \mathcal{C}^{\alpha}_w([0,1]\to\RR^d)$ holds.
To see this, note
$w_f(t)=\int_0^tf(s)dw(s)=f(t)w(t)-\int_0^tdf(u)w(u)$.
Hence
$w_f(t)-w_f(s)=f(s)w_{s,t}+\int_s^tdf_uw_{u,t}$.
Set $R^{w_f}_{s,t}=\int_s^tdf_uw_{u,t}$.
Using $\|f\|_{1\hyp var, [s,t]}\le \|f\|_H(t-s)^{1/2}$,
we have
$|R^{w_f}_{s,t}|\le \|f\|_H\|w\|_{\alpha}(t-s)^{\frac{1}{2}+\alpha}$.
However, we treat $w_f$ as a controlled path in the space of finite
$1/\alpha$-variation norm.
\end{rem}

\begin{proof}[Proof of Lemma~$\ref{w_f-roughpath}$]
\noindent
(1) 
We already showed $w_f'=f$ in Remark~$\ref{remark on wf}$.
Hence,
(\ref{chain rule}) is a well-known property of the
rough integral.

\noindent
(2) 
We have
$w_f(t)=\int_0^tf(s)dw(s)=f(t)w(t)-\int_0^tdf(u)w(u)$.
Write $R(w,f)_{t}=-\int_0^tdf(u)w(u)$.

The second level path $\overline{w_f}^2_{s,t}$ is given as follows.
\begin{align}
 \overline{w_f}^2_{s,t}&=C(f\cdot w, f\cdot w)_{s,t}+
C\left(f\cdot w, R(w,f)\right)_{s,t}
+C\left(R(w,f), f\cdot w\right)_{s,t}
+C\left(R(w,f),R(w,f)\right)_{s,t},\label{wf-representation}
\end{align}
where $(f\cdot w)(t)=f(t)w(t)$.
The latter three integrals on the R.H.S. of (\ref{wf-representation})
are Young integrals and the first one
is a rough integral.
To explain the explicit form, we use component of $f(s)=(f^i_j(s))$ and 
$w_s=(w^i_s)$.
Note that $(f\cdot w)^i_s=\sum_jf^i_j(s)w^j_s$.
For $H^1$ paths $f(s)=(f^i_k(s))$, $g(s)=(g^j_l(s))$ and two controlled paths 
$\phi=(\phi^k)$ and $\varphi=(\varphi^l)$ of $w$, we have
\begin{align}
& \left(C(f\cdot \phi, g\cdot \varphi)_{s,t}, e_i\otimes e_j\right)\nonumber\\
&=
\sum_{i,j,k,l}C(\phi^k, \varphi^l)_{0,t}(f^i_k)_{s,t}g^j_l(t)
-\sum_{i,j,k,l}\int_s^tC(\phi^k, \varphi^l)_{0,u}d_u((f^i_k)_{s,u}g^j_l(u))\nonumber\\
&\quad +\sum_{i,j,k,l}f^i_k(s)g^j_l(t)C(\phi^k,\varphi^l)_{s,t}
-\sum_{i,j,k,l}f^i_k(s)\int_s^tC(\phi^k,\varphi^l)_{s,u}d(g^j_l(u))
\nonumber\\
&\quad +\sum_{i,j,k,l}\int_s^t(f^{i}_k)_{s,u}\phi^k_u\varphi^l_udg^j_l(u)
+\sum_{i,j,k,l}f^i_k(s)\int_s^t\phi^k_{s,u}\varphi^l_udg^j_k(u)
\label{expansion of C}
\end{align}
Also we have
\begin{align}
 w_f^N(t)&=(f\cdot w)^N(t)+R(w,f)^N_t
=f^N(t)\cdot w^N(t)+\ep_N(t)+R(w,f)^N_t,\label{wfN}\\
 {w_f^N}^{\perp}(t)&=w_f(t)-w^N_f(t)\nn\\
&={f^N}^{\perp}(t)\cdot w(t)+f^N(t)\cdot {w^N}^{\perp}(t)
+{f^N}^{\perp}(t)\cdot {w^N}^{\perp}(t)+{R(w,f)^N}^{\perp}_t
-\ep_N(t).\label{wfNperp}
\end{align}
where $\ep_N(t):=(f\cdot w)^N(t)-f^N(t)\cdot w^N(t)$.
Clearly,
\begin{align*}
 \lim_{N\to\infty}\|{R(w,f)^N}^{\perp}\|_{1/(2\alpha)\hyp var}
\le\lim_{N\to\infty}\|{R(w,f)^N}^{\perp}\|_{H^1}=0.
\end{align*}
On the other hand, we have
\begin{align*}
 \ep_N(t)&=
\left(f^N_{\tNkN}-f_t^N\right)2^N(\tNk-t)w_{\tNkN}
+\left(f^N_{\tNk}-f^N_t\right)2^N(t-\tNkN)w_{\tNk}\\
&=:\ep_N^1(t)+\ep_N^2(t)
\quad \tNkN\le t\le \tNk\quad 1\le k\le 2^N,
\end{align*}
where $\tNk=k2^{-N}$.
$\ep^i_N(\tNk)=0$ $(0\le k\le 2^N, i=1,2)$ holds and
we have
$\lim_{N\to\infty}\|\ep_N^i\|_{1/(2\alpha)\hyp var}=0$ for any
$\alpha<1/2$.
To prove this estimate, we recall the following.

For $z\in C([0,1]\to \RR^m)$ and $\kappa>p-1$, let
\begin{align*}
\|z\|_{p,\kappa}:=
\left[\sum_{n=1}^{\infty}\left\{n^{\kappa}\sum_{k=1}^{2^n}
|z(\tau^n_k)-z(\tau^n_{k-1})|^p\right\}\right]^{1/p}.
\end{align*}
We have the following estimate:
there exists a positive number $C$ such that
\begin{align}
\|z\|_{p\hyp var}\le C\|z\|_{p,\kappa}\qquad 
\mbox{for all $z\in C([0,1]\to \RR^m)$},
\label{pkappanorm}
\end{align}
which can be found in Proposition~4.1.1 in \cite{lq}.
Using $\|f\|_{1/2}\le \|f\|_H$, we have
\[
 \max\{|f_{\tNkN}-f^N_t|, |f_{\tNk}-f^N_t|\}\le \|f\|_H 2^{-N/2},
\quad \tNkN\le t\le \tNk.
\]
Let $p>1$ and choose $\ep>0$ such that $p-\ep>1$.
Write $\varphi^{N,k}_t=2^N(\tau^N_k-t)w_{\tau^N_{k-1}}$ 
$(\tau^N_{k-1}\le t\le\tau^N_k)$.
Let $l\in \NN$.
Noting $|\varphi^{N,k}_t|\le \|w\|_{\infty}$,
we have
\begin{align*}
& \sum_{k=0}^{2^N-1}\sum_{i=2^lk+1}^{2^l(k+1)}
|f^N_{(i-1)2^{-(N+l)},
i2^{-(N+l)}}|^p|\varphi^{N,k}_{i2^{-(N+l)}}|^p\\
&\le \sum_{k=0}^{2^N-1}
\max_{2^lk+1\le i\le 2^l(k+1)}
|f^N_{(i-1)2^{-(N+l)},
i2^{-(N+l)}}|^{\ep}
\sum_{i=2^lk+1}^{2^l(k+1)}|f^N_{(i-1)2^{-(N+l)},
i2^{-(N+l)}}|^{p-\ep}\|w\|_{\infty}^p\\
&\le \left(\frac{1}{2^{N+l}}\right)^{\ep/2}\|f\|_{H^1}^{\ep}
\sum_{k=0}^{2^N-1}
\sum_{i=2^lk+1}^{2^l(k+1)}|f^N_{(i-1)2^{-(N+l)},
i2^{-(N+l)}}|^{p-\ep}\|w\|_{\infty}^p\\
&\le \left(\frac{1}{2^{N+l}}\right)^{\ep/2}\|f\|_{H^1}^{\ep}
\|f^N\|_{p-\ep \hyp var, [0,1]}^{p-\ep}\|w\|_{\infty}^p=:I_1(l).
\end{align*}
On the other hand,
noting
\begin{align*}
\sum_{i=2^lk+1}^{2^l(k+1)}|\varphi^{N,k}_{(i-1)2^{-(N+l)},i2^{-(N+l)}}|^p
&=\left(\frac{1}{2^l}\right)^{p-1}|w_{\tNkN}|^p,
\end{align*}
we get
\begin{align*}
& \sum_{k=0}^{2^N-1}\sum_{i=2^lk+1}^{2^l(k+1)}
|f^N_{\tau^N_k,
(i-1)2^{-(N+l)}}|^p|\varphi^{N,k}_{(i-1)2^{-(N+l)},i2^{-(N+l)}}|^p\\
&\le \sum_{k=0}^{2^N-1}
\max_{2^lk+1\le i\le 2^l(k+1)}|f^N_{\tau^N_k,(i-1)2^{-(N+l)}}|^{p-\ep}
\|f\|_H^{\ep}\left(\frac{1}{2^N}\right)^{\ep/2}\left(\frac{1}{2^l}\right)^{p-1}
|w_{\tNkN}|^p\\
&\le \left(\frac{1}{2^N}\right)^{\ep/2}\left(\frac{1}{2^l}\right)^{p-1}
\|f\|_H^{\ep}\|f^N\|_{p-\ep\hyp var, [0,1]}^{p-\ep}\|w\|_{\infty}^p=:I_2(l)
\end{align*}
By (\ref{pkappanorm}), we get
\begin{align*}
 \|\ep_N^1\|_{p\hyp var}^p&\le
C\sum_{l=1}^{\infty}(N+l)^{\kappa}(I_1(l)+I_2(l))\to 0
\qquad \text{as $N\to\infty$.}
\end{align*}
We can estimate $\ep_N^{2}$ similarly and
we get 
\begin{align}
 \lim_{N\to\infty}\|\ep_N\|_{(2\alpha)^{-1}\hyp var}=0 
\qquad \text{for $\alpha<1/2$}.\label{epsilonN}
\end{align}
We now prove (i)$\sim$(iii).
We see that (i) holds 
by using (\ref{wfNperp}), (\ref{epsilonN}) and
\begin{align}
 \|{f^N}^{\perp}\|_{(2\alpha)^{-1}\hyp var}\to 0,\quad
\|{w^N}^{\perp}\|_{\alpha^{-1}\hyp var}\to 0,\quad
\|{R(w,f)^N}^{\perp}\|_{(2\alpha)^{-1}\hyp var}
\to 0\quad \text{for $\alpha<1/2$}.\label{fNperp}
\end{align}
We next prove (ii).
By (\ref{wfN}), we have
\begin{align*}
 C(w_f^N,w_f^N)_{s,t}&=C\Big(f^N\cdot w^N,f^N\cdot w^N\Big)_{s,t}+
C\Bigl(f^N\cdot w^N, \ep_N+R(w,f)^N\Bigr)_{s,t}\nonumber\\
&\quad +
C\Bigl(\ep_N+R(w,f)^N,f^N\cdot w^N\Bigr)_{s,t}
+C\Bigl(\ep_N+R(w,f)^N,\ep_N+R(w,f)^N\Bigr).
\end{align*}
Hence, using (\ref{wf-representation}), we obtain
\begin{align*}
& C(w_f,w_f)_{s,t}-C(w_f^N,w_f^N)_{s,t}=
C(f\cdot w,f\cdot w)_{s,t}-C\Big(f^N\cdot w^N,f^N\cdot w^N\Big)_{s,t}\nonumber\\
&\quad\quad\qquad +C\Bigl(f\cdot w,R(w,f)\Bigr)_{s,t}-
C\Bigl(f^N\cdot w^N,\ep_N+R(w,f)^N\Bigr)_{s,t}\nonumber\\
&\quad\quad\qquad +C\Bigl(R(w,f), f\cdot w\Bigr)_{s,t}-
C\Bigl(\ep_N+R(w,f)^N,f^N\cdot w^N\Bigr)_{s,t}\nonumber\\
&\quad\quad\qquad +C(R(w,f),R(w,f))_{s,t}
-C\Bigl(\ep_N+R(w,f)^N,\ep_N+R(w,f)^N\Bigr)_{s,t}.
\end{align*}
Except the term $I_N(s,t):=
C(f\cdot w,f\cdot w)_{s,t}-C\Big(f^N\cdot w^N,f^N\cdot w^N\Big)_{s,t}$,
we see that the other terms converge to 0 in
$(2\alpha)^{-1}$-variation norm by the estimates 
in (\ref{epsilonN}) and (\ref{fNperp}) and the estimate of Young integrals.
We need to show $\lim_{N\to\infty}\|I_N\|_{(2\alpha)^{-1}\hyp var}=0$.
We see this by noting
$f\cdot w=f^N\cdot w^N +
{f^N}^{\perp}\cdot w^N +
f^N\cdot {w^N}^{\perp}+ {f^N}^\perp\cdot {w^N}^\perp$
and the property Theorem~\ref{lift of w} (ii) of $w$
and applying the formula (\ref{expansion of C}).
The property (iii) also follows from (\ref{expansion of C}) by 
a similar argument.
\end{proof}

For the rough path $\overline{w_f}$, we have the following result.

\begin{lem}\label{rough integral wf}
Let $w\in \Omega$.
Let $f=(f(t))\in H^1([0,1],\LL(\RR^d,\RR^d))$.
Let $(Y_t,Y_t')\in 
\mathcal{V}^{1/\alpha}_w([0,1]\to \LL(\RR^d,\RR^m))$.
Suppose $f(t)$ is invertible for all $0\le t\le 1$.
Then $\big(Y_t,Y'_t(f(t)^{-1}\cdot,\cdot)\big)\in 
\mathcal{V}^{1/\alpha}_{w_f}([0,1]\to \LL(\RR^d,\RR^m))$.
Namely, 
for $R^Y_{s,t}(w_f)=Y_{s,t}-Y'_sf^{-1}(s)w_f(s,t)$
we have
$\|R^Y(w_f)\|_{1/(2\alpha)\hyp var}<\infty$ holds.
Also, it holds that
\begin{align*}
 \int_0^tY_sf(s)dw_s=\int_0^tY_sd\overline{w_f}(s),
\end{align*}
where the integral on the L.H.S. is a rough integral of
a controlled path $Y_tf_t$ against the rough path $\bw$
and the one of the R.H.S. is a rough integral of $Y$ 
as a controlled path by $\overline{w_f}$.
\end{lem}

\begin{rem}
 In the above setting $Y'_t\in \LL(\RR^d,\LL(\RR^d,\RR^m))$.
That is, for the vector $v_1, v_2\in \RR^d$, we mean that
$Y'_t(v_1,\cdot)\in \LL(\RR^d,\RR^m)$ 
and 
$Y'_t(v_1,v_2)\in \RR^m$.
That is, $Y'_t(f(t)\cdot,\cdot)(v_1,v_2):=Y'_t(f(t)v_1,v_2)$ and
$Y'_t(\cdot,f(t)^{-1}\cdot)(v_1.v_2):=Y'_t(v_1,f(t)^{-1}v_2)$.
\end{rem}

\begin{proof}
For $(Y_t, Y'_tf(t)^{-1})$, we have
\begin{align}
 R^Y_{s,t}(w_f)&=Y_{s,t}-Y'_sf(s)^{-1}w^f_{s,t}\nonumber\\
&=Y_{s,t}-Y'_sf(s)^{-1}\Bigl(f(t)w(t)-f(s)w(s)+R(w,f)_{s,t}\Bigr)\nonumber\\
&=Y_{s,t}-Y'_sw_{s,t}-Y'_sf(s)^{-1}f_{s,t}w(t)-Y'_sf(s)^{-1}R(w,f)_{s,t}
\nonumber\\
&=R^Y(w)_{s,t}-Y'_sf(s)^{-1}f_{s,t}w(t)-Y'_sf(s)^{-1}R(w,f)_{s,t}.\label{RYwf}
\end{align}
By the assumption on $Y$ and $f\in H^1$, this implies the desired result.
We prove the last assertion.
Let $I_t=\int_0^tY_s f(s)dw_s$ and $\tilde{I}_t=\int_0^tY_sdw^f_s$.
Let
\begin{align*}
 \Xi_{s,t}&=Y_sf_sw_{s,t}
+\sum_{i,j,k}Y'_s(\ep_i,\ep_j)\int_s^tw^i_{s,u}f^j_k(s)dw^k_u,\\
\tilde{\Xi}_{s,t}&=Y_sw_f(s,t)+\sum_{i,j}Y'_s(f(s)^{-1}\ep_i,\ep_j)
\int_s^tw^i_f(s,u)dw^j_f(u).
\end{align*}
Then, using the partitions $\Delta=\{0=t_0<\cdots<t_M=t\}$,
we have
$I_t=\lim_{|\Delta|\to 0}\sum_{i=1}^M\Xi_{t_{i-1},t_i}$
and $\tilde{I}_t=\lim_{|\Delta|\to 0}\sum_{i=1}^M\tilde{\Xi}_{t_{i-1},t_i}$.
Note that
\begin{align*}
&w_f(s,t)-f(s)w_{s,t}=O(\omega_1(s,t)^{3\alpha}),\\
& \int_s^tw^i_f(s,u)dw^j_f(u)\\
&=\int_s^t\left(\int_s^uf^i_k(r)dw^k(r)\right)
f^j_l(u)dw^l(u)\\
&=f^i_k(s)\int_s^tw^k_{s,u}f^j_l(u)dw^l_u+
\int_s^t\left(\int_s^uf^i_k(s,r)dw^k_r\right)f^j_l(u)dw^l_u\\
&=f^i_k(s)\int_s^tw^k_{s,u}f^j_l(s)dw^l_u
+f^i_k(s)\int_s^tw^k_{s,u}f^j_l(s,u)dw^l_u
+\int_s^tf^i_k(s,u)w^k_{s,u}f^j_l(u)dw^l_u\\
&\qquad
-\int_s^t\left(\int_s^uw^k_{s,r}df^i_k(r)\right)f^j_l(u)dw^l_u\\
&=:I^1_{s,t}+I^2_{s,t}+I^3_{s,t}+I^4_{s,t},\\
&\quad 
I^2_{s,t}=O(\omega_2(s,t)^{3\alpha}),\qquad
I^3_{s,t}=O(\omega_3(s,t)^{3\alpha}), 
\qquad I^4_{s,t}=O(\omega_4(s,t)^{4\alpha}),\\
&\quad Y'_s(f(s)^{-1}\ep_i,\ep_j)I^1_{s,t}
=Y'_s(\ep_i,\ep_j)\int_s^tw^i_{s,u}f^j_l(s)dw^l_s,
\end{align*}
where $\omega_i(s,t)$ $(1\le i\le 4)$ are certain control functions.
Hence, we obtain $I_t=\tilde{I}_t$, which completes the proof.
\end{proof}

\begin{lem}\label{approximation of rough integral}
Let $w\in \Omega$.
Let $f=(f(t))\in H^1([0,1],\LL(\RR^d,\RR^d))$
and $(Y_t,Y_t')\in 
\mathcal{V}^{1/\alpha}_w([0,1]\to \LL(\RR^d,\RR^m))$.
Suppose $f(t)$ is invertible for all $0\le t\le 1$.
Then, we have
\begin{align}
 \lim_{N\to\infty}\int_0^tY_s d{w_f}^N(s)
&=\int_0^tY_s d\overline{w_f}(s),
\label{approximation of rough integral3}
\end{align}
where the rough integral on the R.H.S. is defined as
Lemma~$\ref{w_f-roughpath}$ $(1)$ or
as in Lemma~$\ref{rough integral wf}$.
\end{lem}

\begin{proof}
Write $I^N_t=\int_0^tY_sd{w_f}^N(s)$ and $I_t=\int_0^tY_sd\overline{w_f}(s)$.
$I^N_t$ is just a Riemann-Stieltjes integral but we may write the limit
in the following way.
Let $\Xi^N_{s,t}=Y_sw^N_f(s,t)+Y'_sf(s)^{-1}\int_s^tw_f(s,r)\otimes dw^N_f(r)$,
where
\[
 Y'_sf(s)^{-1}\int_s^tw_f(s,r)\otimes dw^N_f(r)=
\sum_{i,j}Y'_s(f(s)^{-1}(\ep_i),\ep_j)\int_s^tw^{i}_f(s,r)
d(w^N_f)^j(r).
\]
Clearly, for partitions $\Delta=\{0=t_0<\cdots<t_M=t\}$, we have
\[
 \lim_{|\Delta|\to 0}\sum_{i=1}^M\Xi^N_{t_{i-1},t_i}=\int_0^tY_sdw_f^N(s).
\]
Let $\Xi_{s,t}=Y_sw_f(s,t)+Y'_sf(s)^{-1}\overline{w_f}^2(s,t)=Y_sw_f(s,t)
+\sum_{i,j}Y'_s(f(s)^{-1}\ep_i,\ep_j)\int_s^t(w_f)^i(s,r)d(w_f)^j(r)$.
Then $\delta\Xi_{s,u,t}=\Xi_{s,t}-\Xi_{s,u}-\Xi_{u,t}$ 
$(0\le s\le u\le t\le 1)$
satisfies
\begin{align*}
 \delta\Xi_{s,u,t}&=
-R^Y(w_f)_{s,u}w_f(u,t)-(Y'\cdot f^{-1})_{s,u}\overline{w_f}^2(u,t).
\end{align*}
where
$R^Y(w_f)$ is calculated in (\ref{RYwf}).
Then, by the definition of $I_t$, we have
\begin{align*}
 I_t=\lim_{|\Delta|\to 0}\sum_{i=1}^M\Xi_{t_{i-1},t_i}.
\end{align*}
Let ${\Xi^N}^\perp_{s,t}=\Xi_{s,t}-\Xi^N_{s,t}$.
Then we have
\begin{align*}
 {\Xi^N}^\perp_{s,t}&=Y_{s}{w_f^N}^{\perp}_{s,t}
+Y_{s}'f(s)^{-1}C(w_f,{w_f^{N}}^{\perp})_{s,t},
\end{align*}
where
$Y'_sf(s)^{-1}C(w_f,(w_f^{N})^{\perp})_{s,t}
=\sum_{i,j}Y'_s(f(s)^{-1}\ep_i,\ep_j)\int_s^tw_f^i(s,r)d({w_f^N}^{\perp})^j_r$.
For $\delta{\Xi^N}^{\perp}_{s,u,t}={\Xi^N}^{\perp}_{s,t}
-{\Xi^N}^{\perp}_{s,u}-{\Xi^N}^{\perp}_{u,t}$,
we have
\begin{align*}
 \delta{\Xi^N}^{\perp}_{s,u,t}=
-R^Y(w_f)_{s,u}({w_f^N}^{\perp})_{u,t}
-(Y'\cdot f^{-1})_{s,u}C(w_f,{w_f^N}^{\perp})_{u,t}.
\end{align*}
Let us consider the following control function
\begin{align*}
 \omega_N(s,t)&=\frac{1}{2}
\left(\|R^Y(w_f)\|_{1/(2\alpha)\hyp var,[s,t]}^{1/(2\alpha)}\right)^{2/3}
\left(\|{w_f^N}^\perp\|_{1/\alpha\hyp var, [s,t]}^{1/\alpha}\right)^{1/3}
\nonumber\\
&\quad+
\frac{1}{2}
\left(\|C(w_f,{w_f^N}^{\perp})\|_{1/(2\alpha)\hyp var, [s,t]}^{1/(2\alpha)}
\right)^{2/3}
\left(\|Y'\cdot f^{-1}\|_{\alpha^{-1}\hyp var, [s,t]}^{1/\alpha}\right)^{1/3}.
\end{align*}
Then, we have
$|\delta{\Xi^N}^{\perp}_{s,u,t}|\le 2\left(\omega_N(s,t)\right)^{3\alpha}$.
By the Sewing lemma, we obtain
\begin{align}
 \left|I_{s,t}-I^N_{s,t}\right|&\le
\left|Y_{s}{w^N_f}^{\perp}_{s,t}+Y'_sf(s)^{-1}
C(w_f,{w^N_f}^{\perp})_{s,t}
\right|+C\omega_N(s,t)^{3\alpha}.
\label{I-I^N}
\end{align}
By Lemma~\ref{w_f-roughpath} (2),
$\lim_{N\to\infty}\|C(w_f,{w_f^N}^{\perp})\|_{(2\alpha)^{-1}\hyp var}=0$
and $\lim\omega_{N}(s,t)=0$,
we see that the RHS of (\ref{I-I^N})
converges to 0, which completes the proof.
\end{proof}

We return to (\ref{RDE}) and 
summarize basic properties of the solution.
The pathwise differentiability property of the solution $Y(t,e,w)$ with respect to
$w$ in (3) holds for more general case of solutions of RDE 
(see \cite{friz-hairer}).
We give a proof of the differentiability property in our case for the completeness.

\begin{lem}\label{properties of Y} Let $w\in \Omega$.
 \begin{enumerate}
\item[$(1)$] Let $\varphi\in \PeH$.
Then for all $0\le t\le 1$, we have
\begin{align}
 \varphi(t)Y(t,e,w)&=
Y\left(t,e,w+\int_0^{\cdot}Ad(Y(s,e,w)^{-1})
(\varphi^{-1}(s)\dot{\varphi}(s))ds\right)
\label{varphitimesY}
\end{align}
holds.
Note that 
$\left(w+\int_0^{\cdot}Ad(Y(s,e,w)^{-1})(\varphi^{-1}(s)\dot{\varphi}(s))ds,
I\right)\in \mathcal{V}^{1/\alpha}_w([0,1]\to \RR^d)$.
\item[$(2)$]
Let $u\in\RR$.
Let $\phi(t)$ be the solution to the following RDE,
\begin{align*}
 \phi_t=I+\int_0^t\phi(s)Ad\left(Y(s,e,w)\right)udw(s).
\end{align*}
Then $\phi(t)Y(t,e,w)=Y(t,e,(1+u)w)$ $(0\le t\le 1)$ holds.
\item[$(3)$] For $h\in H$ and $w\in \Omega$,
we consider a mapping $(u_1,u_2)(\in \RR^2)\mapsto Y(t,e,u_1w+u_2h)$.
Then this is a $C^{\infty}$ mapping and
it holds that
\begin{align}
 \frac{\partial}{\partial u_1}
Y(t,e,u_1w+u_2h)
&=
\left(\int_0^tAd(Y(s,e,u_1w+u_2h))dw_{s}\right)
Y(t,e,u_1w+u_2h)
,\label{W derivative}\\
 \frac{\partial}{\partial u_2}Y(t,e,u_1w+u_2h)
&=
\left(\int_0^tAd(Y(s,e,u_1w+u_2h))dh_{s}\right)
Y(t,e,u_1w+u_2h)
.\label{H derivative}
\end{align}
\item[$(4)$] $Y(t,e,\cdot)\in C^{\infty}_{b,H}(\Omega, M(n,\CC))$ and
$
 Y(t,x,w)=x Y(t,e,w)
$
holds for any $x\in M(n,\CC)$ and $w\in \Omega$.
\item[$(5)$] Let $h\in H$. For $r\in \RR$, set $b(t,r)=\int_0^tAd(Y(s,e,w+rh))dw_s$.
Then $r\mapsto b(t,r)$ is an $C^{\infty}$ function and
\begin{align*}
 \frac{\partial}{\partial r}
 b(t,r)&=\left[\int_0^tAd(Y(s,e,w+rh))dh_s,\int_0^tAd(Y(s,e,w+rh))dw_s\right]
\nonumber\\
&\qquad-\int_0^t\left[Ad(Y(s,e,w+rh))dh_s,\int_0^sAd(Y(u,e,w+rh))dw_u\right].
\end{align*}
Then $D_hb(t)=\int_0^tAd(Y(s,e,w))dh_s+\int_0^tD_h(Ad(Y(s,e,w)))dw_s$.
\item[$(6)$] $w(\in \Omega)\mapsto Y(t,x,w)\in G$ is 
$\infty$-quasicontinuou mapping,
where the topology of $\Omega$ is of $C([0,1],\RR^d)$.
 \end{enumerate}
\end{lem}

\begin{proof}
 (1) First, we consider the case $\varphi\in C^2$.
We have
\begin{align*}
 \varphi_tY_t-\varphi_sY_s&=
\varphi_s Y_{s,t}+\varphi_{s,t}Y_s+\varphi_{s,t}Y_{s,t}\nonumber\\
&=\varphi_s\sum_iY_s\ep_i w^i_{s,t}
+\varphi_sY_s\sum_{i,j}\ep_i\ep_jw^{i,j}_{s,t}+
\dot{\varphi}_sY_s(t-s)+O(|t-s|^{3\alpha})+O(|t-s|^2).
\end{align*}
Setting $\tY_t=\varphi_tY_t$,
the above estimate implies
\begin{align*}
 \left|\tY_t-\tY_s-\tY_s\left(\sum_i\ep_iw^i_{s,t}+\sum_{i,j}\ep_i\ep_jw^{i,j}_{s,t}
+Ad(Y_s^{-1})(\varphi_s^{-1}\dot{\varphi_s})(t-s)\right)\right|
\le C|t-s|^{3\alpha}
\end{align*}
and hence $\tY_t$ is the solution to
\begin{align*}
 \tY_t=I+\int_0^t\tY_s dw_s
+\int_0^t\tY_sAd(Y_s^{-1})(\varphi_s^{-1}\dot{\varphi_s})ds,
\end{align*}
which shows (\ref{varphitimesY}) in the case where $\varphi\in C^2$.
General cases follows from the case of $C^2$ by using the approximation of
$H^1$-path by $C^2$-paths and the continuity of
the mapping $h(\in H)\mapsto X(t,e,w+h)$.

\noindent
(2) This can be checked by calculating
the difference $\phi_tY_t-\phi_sY_s$.
The calculation is simple and
we omit the proof.

\noindent
(3)
Let $h\in H$ and
let $\varphi_{\ep}(t)$ be an $H^1$-path with values in $G$ such that
\[
 \dot{\varphi}_{\ep}(t)=\ep\varphi_{\ep}(t)Ad(Y(t,e,u_1w+u_2h))\dot{h}(t)
\quad 0\le t\le 1,\qquad \varphi_{\ep}(0)=I.
\]
Then $\varphi_{\ep}(t)Y(t,e,u_1w+u_2h)=Y(t,e,u_1w+(u_2+\ep) h)$ 
$(0\le t\le 1)$ holds by (1).
Since
\begin{align*}
 \ep^{-1}\left(\varphi_{\ep}(t)-\varphi_0(t)\right)=
 \ep^{-1}\left(\varphi_{\ep}(t)-I\right)
=\int_0^t\varphi_{\ep}(s)Ad(Y(s,e,u_1w+u_2h))\dot{h}(s),
\end{align*}
we have
\begin{align*}
 \frac{\partial}{\partial \ep}\varphi_{\ep}(t)\Bigg|_{\ep=0}
=\int_0^tAd(Y(s,e,u_1w+u_2h))dh_s,
\end{align*}
which implies
\begin{align}
&\lim_{\ep\to 0}\ep^{-1}
\left(Y(t,e,u_1w+(u_2+\ep)h)-Y(t,e,u_1w+u_2h)\right)\nonumber\\
&=
\lim_{\ep\to 0}\ep^{-1}(\varphi_{\ep}(t)-I)Y(t,e,u_1w+u_2h)\nonumber\\
&=\left(\int_0^tAd(Y(s,e,u_1w+u_2h))dh_s\right)Y(t,e,u_1w+u_2h).
\label{H-derivative}
\end{align}
We can prove (\ref{W derivative}) in a similar way
to the above by using (2).
We omit the proof.

\noindent
(4) The former statement is a consequence of (3)
and the latter is easy to check by the definition.

\noindent
(5) Write $Y^r_t=Y(t,e,w+rh)$ and set
\begin{align*}
 \Xi^r_{s,t}=Ad(Y^r_s)w_{s,t}+Ad(Y^r_s)[\ep_i,\ep_j]w^{i,j}_{s,t}.
\end{align*}
By the definition of $b(t,r)$, we have
$
 b(t,r)=\lim_{|\Delta|\to 0}\Xi^r_{t_{k-1},t_k}.
$
We have
\begin{align*}
 {\partial r}\Xi^r_{t_{k-1},t_k}=
\left[\int_0^{t_{k-1}}Ad(Y^r_u)dh_u,Ad(Y^r_{t_{k-1}})w_{t_{k-1},t_k}\right]+
\left[\int_0^{t_{k-1}}Ad(Y^r_u)dh_u,Ad(Y^r_{t_{k-1}})
[\ep_i,\ep_j]w^{i,j}_{t_{k-1},t_k}\right].
\end{align*}
Let $(Z_s)\in C^{\alpha}([0,1], \LL(\RR^d,\RR^d))$ be the following:
\begin{align*}
 Z_s\xi=\left[\int_0^sAd(Y^r_u)dh_u,Ad(Y^r_s)\xi\right].
\end{align*}
Then $Z'_s\in \LL(\RR^d, \LL(\RR^d,\RR^d))$, that is,
$Z'_s(v)\in \LL(\RR^d,\RR^d)$ is given by
\begin{align*}
 Z'_s(v)(\tilde{v})=\left[\int_0^sAd(Y^r_u)dh_u,Ad(Y^r_s)[v,\tilde{v}]\right].
\end{align*}
Therefore, by the definition of rough integral, we have
\begin{align*}
 \lim_{|\Delta|\to 0}\sum_{k=1}^M
{\partial r}\Xi^r_{t_{k-1},t_k}=
\int_0^t\left[\int_0^sAd(Y^r_u)dh_u,Ad(Y^r_s)dw_{s}\right]
\end{align*}
The convergence $\lim_{|\Delta|\to 0}\sum_{i=1}\Xi^r_{t_{k-1,t_k}}$
and $\lim_{|\Delta|\to 0}\sum_{i}
{\partial r}\Xi^r_{t_{k-1},t_k}$ are uniform in 
$r\in (-R,R)$ for any $R$.
This implies $b(t,r)$ is $C^1$ in $r\in \RR$.

\noindent
(6) This is a consequence of Theorem~\ref{lift of w} (4) and
$Y(\cdot,e,w)$ is a continuous function of
$\bw$.
\end{proof}

\begin{rem}
 In the above, we prove that the mapping 
$w(\in \Omega)\mapsto Y(t,e,w)$
belongs to $C^{\infty}_{b,H}(\Omega)$.
Actually the derivative coincides with the Malliavin derivative.
Moreover, it holds that $Y(t,e,\cdot)
\in \DD^{\infty}(W,M(n,\CC))$.
By (\ref{H derivative}),
we have
\begin{align}
 DY(1,e,w)[h]=\left(\int_0^1Ad(Y(s,e,w))dh_s\right)Y(1,e,w).
\label{C^1 property of Y}
\end{align}
This implies $DY(1,e,w)DY(1,e,w)^{\ast}=I_{T_{Y(1,e,w)}G}$ and
hence the probability measure 
$d\mu_{\la,a}(w)=p(\la^{-1},e,e)^{-1}\delta_a(Y(1,e,w))d\mu_{\la}(w)$
can be defined.
See Remark~\ref{remark to definition of Omega} (3).
\end{rem}

Denote the (stochastic) development 
of $\gamma_t$ to $T_eG$ by the left invariant connection and the right invariant
connection by $w_t$ and $b_t$ respectively.
We use the same notation for the development for the pinned process as well as 
non pinned process.
In the present case, we have
\begin{align}
 w_t&=\int_0^t\gamma_s^{-1}\circ d\gamma_s=\lim_{|\Delta|\to 0}\sum_{i=1}^N
\frac{\gamma_{t_{i-1}}^{-1}+\gamma_{t_{i}}^{-1}}{2}(\gamma_{t_i}-\gamma_{t_{i-1}}),
\label{approximation of w}
\\
b_t&=\int_0^t\circ d\gamma_s\gamma
_s^{-1}=\lim_{|\Delta|\to 0}
\sum_{i=1}^N(\gamma_{t_i}-\gamma_{t_{i-1}})
\frac{\gamma_{t_{i-1}}^{-1}+\gamma_{t_{i}}^{-1}}{2},\label{approximation of b}
\end{align}
where $\Delta=\{0=t_0<\cdots<t_N=t\}$ and the convergence is in the sense of
probability with respect to $\nu_{\la}$ and $\nu_{\la, a}$.
Of course, if (non pinned Brownian motion) $\gamma_t$ is given by
$\gamma_t=Y(t,e,w)$ via (\ref{RDE}) (or (\ref{stratonovich sde})), then
the development of $Y(t,e,w)$ by (\ref{approximation of w})
is almost surely equal to the driving process
$w_t$.
Similar result holds true for the pinned Brownian motion $\gamma_t$ but
the law of the development process $w_t$ is not the Brownian motion measure
but a measure $\mu_{\la,a}$.
We strength these results as pathwise results in the following lemma.

\begin{lem}\label{pathwise gamma}
Here, we write $Y^{t}(\tNi)=Y(\tNi\wedge t,e,w)$ for simplicity.
\begin{enumerate}
\item[$(1)$] For $w\in \Omega$, let
\begin{align}
 J^N_t=\sum_{i=1}^{2^N}\frac{Y^t(\tNiN)^{-1}+Y^t(\tNi)^{-1}}{2}
\left(Y^t(\tNi)-Y^t(\tNiN)\right), \quad t\in [0,1].
\label{stratonovich type integral}
\end{align}
Then we have $\lim_{N\to\infty}\max_{0\le t\le 1}|J^N_t-w_t|=0$.
\item[$(2)$] For $w\in\Omega$, let
\begin{align*}
 K^N_t&=\frac{1}{2}\sum_{i=1}^{2^N} 
\left(Y^t(\tNi)-Y^t(\tNiN)\right)
\left(Y^t(\tNiN)^{-1}+Y^t(\tNi)^{-1}\right), \quad t\in [0,1].
\end{align*}
Define
$\displaystyle{
b_t=\int_0^tAd(Y(s,e,w))dw_s}.
$
Here, the integral is defined as the rough integral.
Then we have
$\lim_{N\to\infty}\max_{0\le t\le 1}|K^N_t-b_t|=0$.
\end{enumerate}
\end{lem}

\begin{proof}
(1)  
We write $[t]^N_-=\max\{\tNi~|~\tNi\le t\}$.
We have
\begin{align*}
 \sup_{t\in [0,1]}\left|J^N_t-J^{N}_{[t]^N_-}\right|&\le 
C|Y_t-Y_{[t]^N_-}|\le C(w)(2^{-N})^\alpha,
\end{align*}
where $\alpha<1/2$.
Hence it suffices to show
$\lim_{N\to\infty}\max_{t\in \{\tNi\}}|J^N_t-w_t|=0$.
By an elementary calculation, we have
\begin{align*}
 \Bigl|Y_t^{-1}-Y_s^{-1}+\sum_{k}\ep_k Y_s^{-1}w^k_{s,t}
-\sum_{k,l}\ep_k\ep_l Y_s^{-1}w^{l,k}_{s,t}
\Bigr|&\le C|t-s|^{3\alpha}.
\end{align*}
Hence
\begin{align*}
 &\sum_{i}\frac{Y_{t_{i-1}}^{-1}+Y_{t_i}^{-1}}{2}
\left(Y_{t_i}-Y_{t_{i-1}}\right)=
\sum_{i}Y_{t_{i-1}}^{-1}Y_{t_{i-1},t_i}+
\frac{1}{2}\sum_iY^{-1}_{t_{i-1},t_i}Y_{t_{i-1},t_i}\\
&=\sum_iY_{t_{i-1}}^{-1}\Bigl(Y_{t_{i-1}}w_{t_{i-1},t_i}+Y_{t_{i-1}}
\sum_{k,l}\ep_k\ep_lw^{k,l}_{t_{i-1},t_i}+O(|t_i-t_{i-1}|^{3\alpha})\Bigr)
\nonumber\\
&\quad-
\frac{1}{2}
\Bigl(\sum_{k}\ep_k Y_{t_{i-1}}^{-1}w^k_{t_{i-1},t_i}
-\sum_{k,l}\ep_k\ep_l Y_{t_{i-1}}^{-1}w^{l,k}_{t_{i-1},t_i}+
O(|t_i-t_{i-1}|^{3\alpha})\Bigr)\nonumber\\
&\qquad\Bigl(\sum_{k}Y_{t_{i-1}}\ep_kw^k_{t_{i-1},t_i}+\sum_{k,l}Y_{t_{i-1}}\ep_k\ep_l
w^{k,l}_{t_{i-1},t_i}+O(|t_i-t_{i-1}|^{3\alpha}|)\Bigr)\nonumber\\
&=\sum_iw_{t_{i-1},t_i}+\sum_i\sum_{k,l}\ep_k\ep_l
\Bigl(
w^{k,l}_{t_{i-1},t_i}
-\frac{1}{2}w^{k}_{t_{i-1},t_i}w^{l}_{t_{i-1},t_i}
\Bigr)
+\sum_iO(|t_i-t_{i-1}|^{3\alpha}).
\end{align*}
By the assumption on $w$, we arrive at the conclusion of (1).

\noindent
(2) 
Similarly to $J^N$, it is sufficient to prove
$\lim_{N\to\infty}\max_{t\in \{\tNi\}}|K^N_t-b_t|=0$.
Also note that $b_t$ can be obtained as follows by the definition 
of the rough integral.
Let $\Xi_{s,t}=Ad(Y_s)w_{s,t}
+Ad(Y_s)\sum_{k,l}\ep_k\ep_l(2w^{k,l}_{s,t}-w^k_{s,t}w^l_{s,t})$.
Then
$
\displaystyle{
 b_t=\lim_{|\Delta|\to 0}\sum_{i=1}^N\Xi_{t_{i-1},t_i},
}
$
where $\Delta=\{0=t_0<\cdots t_N=t\}$ is a partition of $[0,t]$ and
$|\Delta|=\max_i|t_i-t_{i-1}|$.
On the other hand, by an elementary calculation, we have
\begin{align*}
& \sum_{i=1}^{N}(Y_{t_i}-Y_{t_{i-1}})\frac{Y_{t_{i-1}}^{-1}+Y_{t_i}^{-1}}{2}
\nonumber\\
&~=\sum_{i=1}^{N}Ad(Y_{t_{i-1}})w_{t_{i-1},t_i}+
\sum_{i=1}^NAd(Y_{t_{i-1}})\sum_{k,l}\ep_k\ep_l
\left(w^{k,l}_{t_{i-1},t_i}-\frac{1}{2}w^k_{t_{i-1},t_i}w^l_{t_{i-1}.t_i}\right)
+\sum_{i=1}^N O(|t_i-t_{i-1}|^{3\alpha})\nonumber\\
&~=\sum_{i=1}^N\Xi_{t_{i-1},t_i}-
\sum_{i=1}^NAd(Y_{t_{i-1}})\sum_{k,l}\ep_k\ep_l
\left(w^{k,l}_{t_{i-1},t_i}-\frac{1}{2}w^k_{t_{i-1},t_i}w^l_{t_{i-1}.t_i}\right)
+\sum_{i=1}^N O(|t_i-t_{i-1}|^{3\alpha}).
\end{align*}
By the assumption on $d^{N,k,l}_{t}$, if we consider dyadic partitions of
$[0,1]$, then the sum of second term converges to 0 in $p$-variation norm
applying the estimate of discrete Young integral and
the H\"older continuity property of $Y$.
This completes the proof.
\end{proof}

We now define submanifolds in $\Omega$ which is isomorphic to
$\Pea$ in a certain sense defined by $Y(t,e,w)$.

\begin{dfi}
 Let $Y$ be the solution to the RDE $(\ref{RDE})$.
For $a\in G$, we define
$S_a=\{w\in \Omega~|~Y(1,e,w)=a\}$
and write $S_a^H=S_a\cap H$.
\end{dfi}

Note that we already defined $S_a^H$ in Section~\ref{statement of main theorem}.
The above definition coincides with it.
The following result follows from that
$w(\in \Omega)\mapsto Y(1,e,w)$ is an $\infty$-quasi continuous map.

\begin{lem}
 For any $a\in G$, it holds that $\mu_{\la,a}(S_a^{\complement})=0$.
\end{lem}

The following is an immediate consequence of the above lemma
and (\ref{varphitimesY}).

\begin{pro}\label{isomorphism}
Let
\begin{align*}
 \Peomega=\left\{Y(\cdot,e,w)~|~w\in \Omega\right\},
\quad
\Peaomega=\left\{Y(\cdot,e,w)~|~w\in S_a\right\}.
\end{align*}
\begin{enumerate}
 \item[$(1)$] The mapping
$Y :w(\in\Omega)\mapsto Y(\cdot,e,w)\in \Peomega$
and $Y|_{S_a} : S_a\to \Peaomega$ $(a\in G)$
are bijective Borel measurable mappings and 
the inverse mapping $Y^{-1}$ is obtained by
the limit of Stratonovich type integral in $(\ref{stratonovich type integral})$.
Furthermore, it holds that $Y_{\ast}\mu_{\la,a}=\nu_{\la,a}$ and
$Y^{-1}_{\ast}\nu_{\la,a}=\mu_{\la,a}$.
\item[$(2)$] The sets, $\Peomega$ and $\Peaomega$ are
 invariant by the left multiplication of
$H^1$-paths. That is, for any $\gamma\in \Peomega$ and $\varphi\in \PeH$,
it holds that $\varphi\cdot\gamma\in \Peomega$.
\end{enumerate}
\end{pro}

\begin{rem}
For a fixed $\infty$-quasi continuous version $Y(t,e,w)$ of
the Stratonovich SDE's solution of (\ref{sde}),
 it is not difficult to see that there exists an almost surely defined 
 measurable bijective mapping between $\Pe$ and $W$, $\Pea$ and $S_a$, 
respectively.
Also our problem is concerned with the operator on function spaces and hence
it is sufficient to establish the isomorphism between Sobolev spaces 
on underlying
spaces which were already done in \cite{a1995}.
However, it may be interesting to 
construct an isomorphism by restricting elements of
$w\in W$ in the above way.
\end{rem}

By using the approximation (\ref{approximation of w})  and 
(\ref{approximation of b}),
we have the following.
This results can be found in \cite{g2} and \cite{a1995}.

\begin{lem}\label{derivative of b1} 
For all $0\le t\le 1$,
 $w_t, b_t\in \rD(\E_{\la})$ and
\begin{align*}
 \nabla_h w_t=\int_0^tAd(\gamma_s^{-1})dh_s,\quad 
\nabla_h b_t=h_t+\int_0^t[h_s,db_s],
\end{align*}
where 
$h\in H_{0,0}$.
\end{lem}

In the above lemma, we consider the case where
the functionals $w_t, b_t$ are defined on
$\Pea$.
Similar results hold true as functionals on $P_e(G)$ with $\nu_{\la}$.

\begin{proof}
Let $\gamma_t=Y(t,e,w)$ $(w\in \Omega)$.
$w_t$ and $b_t$ can be obtained as a pathwise limit of
$J^N_t$ and $K^N_t$.
If $\tNiN<t\le \tNi$,
then the sum is defined by
the partition of $[0,t]$ :
$0=\tau^N_0<\cdots<\tNiN<t=t$.
For simplicity, we denote this partition
by $0=t_0<\cdots<t_{t(N)}=t$.
Then, we have
\begin{align*}
 J^N_t=\sum_{i=1}^{t(N)}\frac{\gamma^{-1}_{t_{i-1}}+\gamma^{-1}_{t_i}}{2}
(\gamma_{t_i}-\gamma_{t_{i-1}})
\end{align*}
As we proved in Lemma~\ref{pathwise gamma}, this converges to $w_t$.
Actually, by checking the proof, we see that this converges in $L^2(\nu_{\la,a})$.
In the proof, we need to use
$\|Y\|_{\alpha}\in \cap_{p\ge 1}L^p(\nu_{\la,a})$ $(\alpha<1/2)$.
This is non-trivial but it can be checked to make use the fact that
the power of the Besov type norm of $Y(\cdot,e,w)$ belongs to
$\DD^{p,r}$. We refer the readers to \cite{malliavin} for this.
We calculate the derivative of $J^N_t$.
Using (\ref{derivative of gamma}), we have
\begin{align*}
 \frac{1}{2}\nabla_h\left((\gamma^{-1}_{t_{i-1}}+\gamma^{-1}_{t_i})
(\gamma_{t_i}-\gamma_{t_{i-1}})\right)&=
\frac{1}{2}\nabla_h
(\gamma_{t_{i-1}}^{-1}\gamma_{t_i}-\gamma_{t_i}^{-1}\gamma_{t_{i-1}})
\nonumber\\
&=\frac{1}{2}\left(\gamma_{t_{i-1}}^{-1}h_{t_{i-1},t_i}\gamma_{t_i}+
\gamma_{t_i}^{-1}h_{t_{i-1},t_i}\gamma_{t_i}\right)\nonumber\\
&=Ad(\gamma_{t_i}^{-1})h_{t_{i-1},t_i}
-
\frac{1}{2}
\gamma^{-1}_{t_{i-1},t_i}h_{t_{i-1},t_i}\gamma_{t_i}
-\frac{1}{2}\gamma_{t_i}^{-1}h_{t_{i-1},t_i}\gamma_{t_{i-1},t_i}.
\end{align*}
The summation of the second and third terms of the above
converges to 0 as $N\to \infty$.
Also it is easy to show $\nabla J^N_t$ converges in $L^2(\nu_{\la,a})$.
We consider the derivative of $b_t$.
We have
\begin{align*}
& \nabla_h\left((\gamma_{t_i}-\gamma_{t_{i-1}})\frac{1}{2}
(\gamma_{t_{i-1}}^{-1}+\gamma_{t_i}^{-1})\right)\nn\\
&=
\frac{1}{2}\nabla_h\left(\gamma_{t_i}\gamma_{t_{i-1}}^{-1}-\gamma_{t_{i-1}}
\gamma_{t_i}^{-1}\right)\nn\\
&=\frac{1}{2}\left(h_{t_i}\gamma_{t_i}\gamma_{t_{i-1}}^{-1}-
\gamma_{t_i}\gamma_{t_{i-1}}^{-1}h_{t_{i-1}}-h_{t_{i-1}}\gamma_{t_{i-1}}
\gamma_{t_i}^{-1}+\gamma_{t_{i-1}}\gamma_{t_i}^{-1}h_{t_i}\right)\nn\\
&=h_{t_i}-h_{t_{i-1}}+
\frac{1}{2}\left[h_{t_{i-1}},(\gamma_{t_i}-\gamma_{t_{i-1}})
(\gamma_{t_i}^{-1}+\gamma_{t_{i-1}}^{-1})\right]
+\frac{1}{2}\left(h_{t_{i-1},t_i}\gamma_{t_{i-1},t_i}\gamma_{t_{i-1}}^{-1}
-\gamma_{t_{i-1},t_i}\gamma_{t_i}^{-1}h_{t_{i-1},t_i}
\right)\nn\\
&=:I^1_i+I^2_i+I^3_i,
\end{align*}
where $[\cdot,\cdot]$ denotes the Lie bracket.
Let $I_k=\frac{1}{2}\sum_{i=1}^k(\gamma_{t_i}-\gamma_{t_{i-1}})
(\gamma_{t_i}^{-1}+\gamma_{t_{i-1}}^{-1})$.
Then 
\begin{align*}
 \sum_{i=1}^nI^2_i&=
\sum_{i=1}^n[h_{t_{i-1}},I_i-I_{i-1}]
=[h_{t_n},I_n]-\sum_{i=1}^n[h_{t_i}-h_{t_{i-1}},I_i].
\end{align*}
Considering the case $t_n=t(N)$ and using Lemma~\ref{pathwise gamma} (2),
we obtain
$\lim_{N\to\infty}\nabla_hK^N_t=h_t+\int_0^t[h_s,db_s]$.
Also, it is easy to show the convergence of
$K^N_t$ and $\nabla K^N_t$ in $L^2(\nu_{\la,a})$.
This completes the proof.
\end{proof}

As already shown in Lemma~\ref{properties of Y},
we have
\begin{align*}
 DY(1,e,w)[h]=\left(\int_0^1Ad(Y(s,e,w))dh_s\right)Y(1,e,w).
\end{align*}
Taking this formula into account,
we now introduce the notion of tangent space of the submanifold $S_a$.

\begin{dfi}[Tangent space]\label{tangent space}
(1) Let $a\in G$ and $w\in S_a$. We define the tangent space of $S_a$ at $w$ by
\begin{align*}
 T_wS_a&=\left\{h\in H~\Big | \int_0^1Ad(Y(s,e,w))\dot{h}(s)ds=0\right\}.
\end{align*}

\noindent
(2) In the case where $k\in S_a^H$, we may write $W_0=\oTkSa$,
$H_0=T_kS_a$, $H_0^{\perp}=T_kS_a^{\perp}$
for simplicity although these spaces depend on $k$.
Recall that we use the notation $\Omega_0$ to denote $\Omega\cap W_0$ and
we may denote the elements of $W_0$ and $W_0^{\perp}$ by $\eta$ and
$\eta^{\perp}$.
\end{dfi}

We can check that the above 
$W_0$ is a finite codimensional subspace of
$W$ in the sense of Definition~\ref{W_0} by setting
$V=T_kS_a^{\perp}$ in the following lemma.

\begin{lem}\label{tangent space lemma} Let $a\in G$ and choose $w\in S_a$.
 \begin{enumerate}
  \item[$(1)$] The orthogonal complement of $T_wS_a$ in $H$ is given by
\begin{align*}
 T_wS_a^{\perp}&=
\left\{Q(w)v~\Big |~
v\in \mathfrak{g}\right\},
\end{align*}
where $Q(w)v=\int_0^{\cdot}Ad(Y(s,e,w)^{-1})vds$.
\item[$(2)$] The explicit form of the projection operator 
$P(w) : H\to T_wS_a$ is given by
\begin{align}
 P(w)h=h-\int_0^{\cdot}Ad(Y(s,e,w)^{-1})
\left(\int_0^1Ad(Y(u,e,w))dh_u\right)ds.
\label{Pk}
\end{align}
\item[$(3)$] Let $k\in S_a^H$.
The direct sum decomposition $W=\oTkSa\oplus T_kS_a^{\perp}$
holds.
Explicitly, this decomposition is given as follows.
Let 
\[
 N(k)w=\int_0^{\cdot}Ad(Y(s,e,k)^{-1})\left(\int_0^1Ad(Y(u,e,k))dw_u\right)ds
\in T_kS_a^{\perp}.
\]
Write $P(k)w=w-N(k)w$.
Then $P(k)w\in \oTkSa$ and $w=P(k)w+N(k)w$ is the direct sum decomposition of $w\in W$. 
Moreover if $w\in \Omega$, then $P(k)w\in \Omega$ holds 
and hence the lift 
$\overline{P(k)w}$
can be defined.
Explicitly, we have
\begin{align*}
(N(k)w)_t&=\int_0^tAd(Y(s,e,k)^{-1})\left(\int_0^1Ad(Y(u,e,k))dw_u\right)ds,\\
 (P(k)w)_t&=w_t-\int_0^tAd(Y(s,e,k)^{-1})
\left(\int_0^1Ad(Y(u,e,k))dw_u\right)ds.
\end{align*}
\end{enumerate}
\end{lem}

\begin{rem}\label{remark on tangent space}
(1) We identify $\g$ with $T_kS_a^{\perp}$ 
by the natural isometry mapping
$v\mapsto Q(k)v$.

\noindent
(2) Since $\Omega=\cup_{a\in G}S_a$ holds, this means that
the projection operator $P(w) : H\to T_wS_{Y(1,e,w)}$ can be defined
for all $w\in \Omega$.

\noindent
(3) 
Also $P(w)w$ can be defined as the rough integral in (\ref{Pk}) for all
$w\in \Omega$.
Suppose $w\in S_a$.
We see that $P(w)w=0$ is equivalent to 
that $Y(\cdot,e,w)$ is a geodesic connecting $e$ and $a$.
The proof is as follows.
\begin{align*}
 P(w)w=0&\Longleftrightarrow w(t)=\int_0^tAd(Y(s,e,w)^{-1})
\left(\int_0^1Ad(Y(u,e,w))dw_u\right)ds\quad (0\le t\le 1)\\
&\qquad
\text{Setting $v=\int_0^1Ad(Y(u,e,w))dw_u\in\g$,}\\
&\Longleftrightarrow
w(t)=\int_0^tAd(Y(s,e,w)^{-1})vds\quad (0\le t\le 1)\\
&\Longleftrightarrow \text{$w\in H$ and}\,\,
Y(t,e,w)\dot{w}_t=vY(t,e,w)\quad (0\le t\le 1)\\
&\Longleftrightarrow
\frac{d}{dt}Y(t,e,w)=v Y(t,e,w)\quad (0\le t\le 1)\\
&\Longleftrightarrow
Y(t,e,w)=e^{t v}\quad (0\le t\le 1).
\end{align*}

\noindent
(4) We rethink the property $P(w)w=0$ when $w\in S_a^H$.
Since the mapping $h(\in S_a^H)\mapsto Y(\cdot,e,h)$ is energy preserving, that is,
$\|h\|_{H}^2=\int_0^1|\dot{Y}(t,e,h)|^2dt$ holds.
Hence if $k\in S_a^H$ is a critical point of the functional
$\frac{1}{2}\|h\|_H^2$ on $S_a^H$, then $Y(\cdot,e,k)$ is a critical point of 
the energy function $E$, that is, a geodesic and the converse also holds.
If $k$ is the critical point of $\frac{1}{2}\|h\|_H^2$ on $S_a^H$,
by the Lagrange multiplier method, we obtain
that for any $h\in H$, there exists $v\in \g$,
\[
 (h,k)_H+\left(\int_0^1Ad(Y(s,e,k))\dot{h}_s,v\right)=0.
\]
This is equivalent to $Y(t,e,k)\dot{k}_t=-vY(t,e,k)$
and hence $P(k)k=0$.

\end{rem}

We already defined the Dirichlet forms and the nonnegative generators
\[
(\E_{\la},\rD(\E_{\la}))\,\,, -L_{\la,\Pea}\,\,\,
\text{in $L^2(\Pea,\nu_{\la,a})$,}\,\,\quad
(\E_{\la,\D},\rD(\E_{\la,\D}))\,\,,-L_{\la,\D}\,\,\,
\text{in $L^2(\D,\nu_{\la,\D})$ $(\D\subset \Pea)$.}
\] 
The latter one defines the Witten Laplacian acting on functions with the Dirichlet
boundary condition on $\D$.
We now recall the counterparts
on $S_a$ and summarize necessary results in this paper.

Let $\FC(W)$ be the set of smooth cylindrical 
functions $F$ on $W$, that is, which can be written as
$F(w)=f\left((w,h_1),\ldots,(w,h_N)\right)$ where 
$h_i\in W^{\ast}$ $(1\le i\le N)$,
$N$ can be any positive integer and
$f\in C^{\infty}_b\left((\RR^d)^N\right)$.
For such an $F$, define
\begin{align*}
 (DF)(w)&=\sum_{1\le i\le N}(\partial_{i}f)(w)h_i\in H,\\
 (D_{S_a}F)(w)&=P(w)(DF)(w)\in T_wS_a.
\end{align*}
Then the following integration by parts formula holds:
For any $G\in \FC(W)\otimes H$,
\begin{align*}
 \int_{S_a}\left((D_{S_a}F)(w),P(w)G(w)\right)d\mu_{\la,a}(w)=
\int_{S_a}F(w)(D_{\mu_{\la}}^{\ast}PG)(w)d\mu_{\la,a}(w).
\end{align*}
$D_{\mu_{\la}}^{\ast}$ denotes the adjoint operator of
$D : L^2(W,d\mu_{\la})\to L^2(W\to H,d\mu_{\la})$.
Also, $\{P(w)G(w)~|~G\in \FC(W)\otimes H\}$
in $\{G\in L^2(S_a\to H,\nu_{\la,a}~|~P(w)G(w)=G(w)\,\mu_{\la,a}\, a.s.\}$
is dense.
Hence, similarly to $\Pea$, we see that
if $F_1, F_2\in \FC(W)$ is equivalent in the sense that
$
 F_1=F_2\quad \nu_{\la,a}a.s.
$
then
$
D_{S_a}F_1=D_{S_a}F_2\quad \nu_{\la,a}a.s.
$
holds.
Hence, for the equivalence class of $F\in \FC(W)|_{S_a}$,
$\DS F$ can be defined $\mu_{\la,a}$ a.s.
Let
\begin{align*}
\E_{\la,W}(F,G)&=\int_{W}(DF(w),DG(w))_{H}d\mu_{\la}(w),\quad
F, G\in \FC(W),\\
 \E_{\la,S_a}(F,G)&=\int_{S_a}(\DS F(w),\DS G(w))_{H}d\mu_{\la,a}(w),\quad
F, G\in \FC(W)|_{S_a}.
\end{align*}
These are closable and we denote the Dirichlet forms, 
closed extension $H$-derivatives, and the generators
by $\rD(\E_{\la,W}), D$, $-L_{\la,W}$ and 
$\rD(\E_{\la,S_a}), D_{S_a}, -\LS$ respectively.
In Malliavin calculus, $\rD(\E_{\la,W})$ is usually denoted by 
$\DD^{1,2}$ when $\la=1$.

The following can be found in \cite{a1993},\cite{a1995}.

\begin{lem}
$(L_{\la,\Pea},\FC(\Pea))$ and
$(L_{\la,S_a},\FC(W)|_{S_a})$ are essentially selfadjoint\\
 in
$L^2(\Pea,\nu_{\la,a})$ and
$L^2(S_a,\mu_{\la,a})$ respectively.
\end{lem}

There exists a pathwise measure preserving 
isomorphism between $(\Peaomega,\nu_{\la,a})$ and 
$(S_a,\nu_{\la,a})$.
As we remarked, $DY(\cdot,e,w) : T_{w}S_a\to T_{\gamma}\Peaomega$
is isometry. 
Hence, we can expect that there exists some isomorphism between
Sobolev spaces on $S_a$ and $\Peaomega$.
Before showing our result, we introduce 
Dirichlet forms with the Dirichlet boundary
condition on domains in $S_a$ also.

\begin{dfi}
For $\phi\in \rD(\E_{\la,S_a})$, let
$\tilde{\D}=\{w\in S_a~|~\phi(w)<R\}$.
Suppose $\mu_{\la,a}(\tilde{\D})>0$.
Then, in the same way as in Definition~\ref{Dirichlet form with DBC},
we can define the Dirichlet form $\E_{\la,S_a,\tilde{\D}}$
on $\tilde{\D}$ with the Dirichlet
boundary condition and its nonnegative generator $-L_{\la,S_a,\tilde{\D}}$.
\end{dfi}

\begin{thm}\label{D and Y^-1(D)}
Let us consider the bijective measurable mappings
$Y : \Omega\to \Pe^\Omega$ and $Y : S_a\to \Pea^\Omega$
and their inverses $Y^{-1}$ which
are defined in Proposition~$\ref{isomorphism}$.
Let $Y^{\ast}$ and $(Y^{-1})^{\ast}$
denote the pullback for functions on them.
We have the following.
\begin{enumerate}
 \item[$(1)$] Let $f$ be a Borel measurable function on $\Pea$.
Then $f\in \rD(\E_{\la})$ and $f\in \rD(L_{\la,\Pea})$ are equivalent to
$Y^{\ast}f\in \rD(\E_{\la,S_a})$ and $Y^{\ast}f\in \rD(\LS)$ respectively
and it holds that
\begin{align*}
& D_{S_a}Y^{\ast}f=\int_0^{\cdot}Ad(Y(s,e,w)^{-1})
\overbrace{(\nabla f)(Y)_s}^{\cdot}ds, \,\,
\|D_{S_a}Y^{\ast}f\|_{T_wS_a}=\|(\nabla f)(Y(w))\|_{H_{0,0}}\,
\text{$\mu_{\la,a}$-a.s.$w$}\\
& \LS Y^{\ast}f=Y^{\ast}(L_{\la,\Pea}f)\quad \text{$\mu_{\la,a}$-a.s.$w$},
\quad \E_{\la,S_a}(Y^{\ast}f,Y^{\ast}f)=\E_{\la}(\nabla f,\nabla f).
\end{align*}
Conversely, similar relations between  measurable functions $g$ on $S_a$ and 
$(Y^{-1})^{\ast}g$ on $\Pea$ hold true
with respect to $\nu_{\la,a}$
\item[$(2)$] For $\phi\in \rD(\E_{\la,S_a})$, let
$\tilde{\D}=\{w\in S_a~|~\phi(w)<R\}$.
Suppose $\mu_{\la,a}(\tilde{\D})>0$ and
$\mu_{\la,S_a,\tilde{\D}}$ denotes the normalized probability measure 
of $\mu_{\la,a}|_{\tilde{\D}}$.
Then we can define the Dirichlet form $\E_{\la,S_a,\tilde{\D}}$ 
in $L^2(\tilde{\D},\mu_{\la,S_a,\tilde{\D}})$ with the Dicichlet boundary
condition similarly to $\E_{\la,\D}$ 
in Definition~$\ref{Dirichlet form with DBC}$.
We denote the nonnegative generator of $\E_{\la,S_a,\tilde{\D}}$ by
$-L_{\la,S_a, \tilde{\D}}$.
\item[$(3)$] Let $\varphi\in \rD(\E_{\la})$ and $R>0$.
Suppose $\nu_{\la,a}(\left\{\gamma~|~\varphi(\gamma)<R\right\})>0$.
Let $\D=\{\gamma\in \Pea~|~\varphi(\gamma)<R\}$
and set $\YD=\{w\in S_a~|~Y^{\ast}(\varphi)(w)<R\}$.
Let us consider the Dirichlet form $\E_{\la,S_a,\YD}$ 
with the Dirichlet boundary condition on $Y^{-1}(\D)$
which is well-defined because $Y^{\ast}\varphi\in \rD(\E_{\la,S_a})$.
Then $f\in \rD(\E_{\la,\D})$ and $f\in \rD(L_{\la,\D})$ are equivalent to
$Y^{\ast}f\in \rD(\E_{\la,S_a,Y^{-1}(\D)})$ 
and $Y^{\ast}f\in \rD(L_{\la,S_a,\YD})$ respectively.
Furthermore, $-L_{\la,\D}$ on $L^2(\D,\nu_{\la,\D})$ and
$-L_{\la,S_a,\YD}$ on $L^2(\YD,\mu_{\la,S_a,\YD})$ are unitarily equaivalent by
the mapping
$f(\in L^2(\D,\nu_{\la,\D}))\mapsto Y^{\ast}f\in L^2(\YD,\mu_{\la,S_a,\YD})$.
\end{enumerate}
\end{thm}

\begin{proof}

(1) follows from Lemma 3.3 and Proposition 3.6 in \cite{a1995}.
In \cite{a1995}, we consider the case where $a=e$ and $\la=1$ only.
However, the argument can be applied to the case of any $a$.
The proof of (2) is similar to the case of $\E_{\la,\D}$.
We can prove (3) by establishing similar relations to (1).
This is almost trivial and we omit the proof.
\end{proof}

\begin{rem}
 For notational simplicity, we denote 
$\mu_{\la,S_a,Y^{-1}(\D)}$, $\E_{\la,S_a,\YD}$ and $L_{\la,S_a,\YD}$ by
$\mu_{\la, Y^{-1}(\D)}$, $\E_{\la,\YD}$ and $L_{\la,\YD}$ respectively.
\end{rem}

\subsection{Certain Ornstein-Uhlenbeck type operators
on Wiener spaces}\label{cons in wiener space}

We now introduce the operator $-L_{\la,T_{\xi},W_0}$ 
which we mentioned in Section~\ref{statement of main theorem}.
Below, $H_n(x)$ denotes the Hermite polynomial, that is,
the eigenfunction of the 1-dimensional Ornstein-Uhlenbeck operator
$-L=-\frac{d^2}{dx^2}+x\nabla$ in 
$L^2(\RR,(2\pi)^{-1/2}e^{-\frac{x^2}{2}}dx)$
such that $-LH_n=n H_n$ $(n\ge 0)$.

\begin{thm}\label{cons}
Let $W_0$ be a finite codimensional subspace of $W$
and $H_0$ be the Cameron-Martin subspace.
$D_0$ denote the $H$-derivative in direction $H_0$ on $W_0$.
Let $\la>0$.
Let $\mu_{\la,W_0}$ be the induced Gaussian measure on $W_0$ from $\mu_{\la,W}$ on
$W$.
 Let $T$ be a self-adjoint Hilbert-Schmidt operator on $H_0$.
Let $N\ge 0$ and $\{e_i\}_{i=1}^{\infty}$ be the 
complete orthonormal system which are eigenvectors of $T$ such that
$Te_i=\zeta_i e_i$ $(i\ge 1)$.
Suppose $1+\zeta_i<0$ $(1\le i\le N)$ and
$1+\zeta_{i}>0$ $(i\ge N+1)$.
We define
\begin{align*}
 :\left(T\eta,\eta\right):_{\mu_{\la}}=\sum_{i=1}^{\infty}
\zeta_i\left\{(\eta,e_i)^2-\frac{1}{\la}\right\},\qquad \eta\in W_0.
\end{align*}
$(\eta,e_i)$ denotes the Wiener integral and we write $\eta_i=(\eta,e_i)$.
Let $-L_{\la,T,W_0}$ be the nonnegative generator of the Dirichlet form
\begin{align*}
 \E_{\la,T,W_0}(f,f)=\int_{W_0}|D_0f(\eta)|^2_{H_0}d\mu_{\la,T,W_0}(\eta)
\end{align*}
in $L^2(W_0,d\mu_{\la,T,W_0})$.
Here $d\mu_{\la,T,W_0}(\eta)=
\exp\left(-\frac{\la}{2}:\left(T\eta,\eta\right):_{\mu_{\la,W_0}}\right)
d\mu_{\la,W_0}$
\begin{enumerate}
\item[$(1)$] $-L_{\la,T,W_0}$ has eigenvectors which constitute 
a complete orthonormal system
of $L^2(W,d\mu_{\la,T,W_0})$ such that
\begin{align*}
\mathbf{e}_{\mathbf{n},T,\la}
&=
\left|\det{}_2(I_{H_0}+(|I_{H_0}+T|-I_{H_0}))\right|^{1/4}
\prod_{j=N+1}^{\infty}H_{n_j}(\sqrt{\la(1+\zeta_j)}\eta_j)\nonumber\\
&\qquad\qquad\times \prod_{i=1}^NH_{n_i}(\sqrt{\la|1+\zeta_i|}\eta_i)
\exp\left(-\frac{\la}{2}|1+\zeta_i|(\eta_i^2-\frac{1}{\la})\right),\\
&\qquad
 \mathbf{n}:=(n_1,n_2,\ldots,)\qquad (n_i\ge 0, \sum_{i=1}^{\infty}n_i<\infty)
\end{align*}
and the eigenvalue of $\mathbf{e}_{\mathbf{n},T,\la}$ is
$
\la
\left(\sum_{i=1}^N|1+\zeta_i|+\sum_{i=1}^{\infty}n_i|1+\zeta_i|\right).
$
Note that $\mathbf{e}_{\mathbf{n},T,\la}$ depends on the order of
the pair of eigenvalue and eigenvector $\{(\zeta_i,e_i)\}_{i=1}^{\infty}$.
\item[$(2)$] For any $\mathbf{n}, \mathbf{m}$ and $\la>0$, the following identity of 
signed measures
hold.
\begin{align}
 \mathbf{e}_{\mathbf{n},T,\la}(\eta)
\mathbf{e}_{\mathbf{m},T,\la}(\eta)
d\mu_{\la,T,W_0}(\eta)
=\mathbf{e}_{\mathbf{n},|I_{H_0}+T|-I_{H_0},\la}(\eta)
\mathbf{e}_{\mathbf{m},|I_{H_0}+T|-I_{H_0},\la}(\eta)
d\mu_{\la,|I_{H_0}+T|-I_{H_0},W_0}(\eta),\label{identity of measures}
\end{align}
where the order of the pair of eigenvalue and eigenvector of $|I_{H_0}+T|-I_{H_0}$
are given by
$\{(|1+\zeta_i|-1,e_i)\}_{i=1}^{\infty}$.
In particular, the measure 
$\mathbf{e}_{\mathbf{n},T,\la}(\eta)^2d\mu_{\la,T,W_0}(\eta)$
is the Gaussian measure whose covariance operator is
$\left(\la |I_{H_0}+T|\right)^{-1}$ on $H_0$.
\item[$(3)$] Let $k$ be a nonnegative integer and $p\ge 2$.
Then, there exists a constant $C_{\mathbf{n},T,p}$ such that
\begin{align*}
\|\mathbf{e}_{\mathbf{n},T,\la}\|_{L^p(\mu_{\la,T,W_0})}\le C_{\mathbf{n},T,p},
\quad
\||D_0^k\mathbf{e}_{\mathbf{n},T,\la}|_{H.S.}\|
_{L^p(\mu_{\la,T,W_0})}\le \la^{k/2}C_{\mathbf{n},T,p},
\end{align*}
where $|\cdot|_{H.S.}$ denotes the Hilbert-Schmidt norm.
\end{enumerate}
\end{thm}

\begin{proof}
 (1) Let $K>0$ and consider a Dirichlet form
$\E(f,f)=\int_{\RR}|f'(x)|^2d\mu_{\la,K}(x)$ on
$L^2(\RR,\mu_{\la,K})$, where
$\mu_{\la,K}(dx)=\rho_K(x)dx$ 
$(\rho_K(x)=\left(\frac{\la}{2\pi}\right)^{1/2}e^{\frac{\la}{2}Kx^2})$ and
we denote the nonnegative generator of this Dirichlet form by
$-L_{\la,K}$.
Then, using the unitary transformation $\Psi_K : L^2(\RR,\mu_{\la,K})\to L^2(\RR,dx)$
defined by $\Psi_Kf=\sqrt{\rho_K}f$, we see that
$\Psi_K(-L_{\la,K})\Psi_K^{-1}f=-\Delta f+(\frac{K^2}{4}x^2+\frac{K}{2})f$.
By this, we see that 
$
 -L_{\la,K}h_n=\la K(n+1)h_n,
$
where $h_n(x)=H_n(\sqrt{\la Kx})K^{1/4}
e^{-\frac{\la K}{2}x^2}$
and eigenvectors $\{h_n\}$ constitutes the complete orthonormal system
in $L^2(\RR,\mu_{\la,K})$.
The measure $\mu_{\la,T,W_0}$ is a product measure of
the measure $\mu_{\la,K_i}$ $(1\le i\le N)$ with $K_i=|1+\zeta_i|$ 
and the remaining Gaussian measure
$\mu_{\la,T|_{{\{e_i\}_{i=1}^N}^{\perp}}, 
\overline{{\{e_i\}_{i=1}^N}^{\perp}}^{W_0}}.$
Using this, we obtain the result (1).

\noindent
(2) This can be checked by a simple calculation.

\noindent
(3) This follows from the fact that 
the measure $\mathbf{e}_{\mathbf{n},T,\la}(\eta)^2d\mu_{\la,T,W_0}(\eta)$
is the Gaussian measure.
\end{proof}

\begin{rem}
(1) In this paper, we call these kinds of $L_{\la,T,W_0}$ 
Ornstein-Uhlenbeck type operators.
Note that this is not a standard usage.

\noindent
(2)
In the above theorem, $I_{H_0}+T$ corresponds to the Hessian of 
the energy function $E$ at critical points. 
The case $N=0$ corresponds to the local minimum critical points.
The case $N\ge 1$ corresponds to the negative index case
and the measure $\mu_{\la,T,W_0}$ is an
infinite measure and the eigenfunctions
\[
 \mathbf{e}_{(n_1,\ldots,n_N,0,0,\ldots),T,\la},\qquad (n_1,\ldots,n_N)\ne 0
\]
do not belong to $L^1$ space but belong to $L^2$ space
with respect to the infinite measure.
Also, if there exists $\zeta_i$ satisfying $1+\zeta_i=0$, that is,
 the degenerate case,
the spectrum of $-L_{\la,T,W_0}$ is identical to $[0,\infty)$.
We do not consider such a situation in this paper.
\end{rem}

\subsection{Log-Sobolev inequality with a potential function}
\label{log-Sobolev with potential function}

First, we define entropy functional on a probability space
$(Y,\mathcal{F},m)$.
Let $g$ be a nonnegative integrable function on $Y$.
We define
\begin{align*}
 \Ent_m(g)=\int_Yg \log g\, dm-\|g\|_{L^1(m)}\log\|g\|_{L^1(m)}.
\end{align*}

The following lower boundedness theorem has been applied to 
generators of hyperbounded semigroups, {\it e.g.} in the 
Euclidean quantum field theory, and called 
a NGS bound (\cite{simon}) or the Federbush semi-boundedness theorem
(\cite{federbush}, \cite{g1}, \cite{g3}).

\begin{lem}\label{NGS bound}
 Let $S\subset L^2(Y,m)$ be a subset
satisfying that $S$ is invariant by the multiplication of
a real number, that is, $tf\in S$ if $f\in S$ and $t\in \RR$ hold.
Let $Q : S\to \RR$ be a measurable map satisfying
$Q(tf)=t^2Q(f)$ for all $f\in S$ and $t\in \RR$.
Let $C>0$.
Then the following $(1)$ and $(2)$ is equivalent.
\begin{enumerate}
 \item[$(1)$] 
$
\displaystyle
{\Ent_m(f^2)\le C Q(f)}
$ 
holds
for all $f\in S$.
\item[$(2)$] For all bounded measurable function $\rho$ on $Y$ and $f\in S$, 
it holds that
\[
 Q(f)+\int_Y\rho f^2dm\ge 
-\log\left(\|e^{-\rho}\|_{L^{C}(Y,m)}\right)\|f\|_{L^2(m)}^2.
\]
\end{enumerate}
\end{lem}

The above estimate can be proved by applying the Donsker-Varadhan variational formula:
For a probability measure $m$ on $Y$ and a function $f$ satisfying $\int_Yf^2dm=1$,
it holds that
\begin{align*}
&\int_{Y}f^2\log f^2 dm\\
&\qquad 
=\sup\Biggl\{
\int_Y\phi\, f^2dm
-\log\left(\int_Ye^{\phi}dm\right)
~\Bigg |~
\text{
$\phi$ is a bounded measurable function on $Y$.}
\Biggr\}
 \end{align*}

For $\E_{\la}$ on $L^2(\Pea,\nu_{\la,a})$,
we have the following inequality
which is a refined version of Gross's log-Sobolev inequality with
a potential function on $\Pea$ (\cite{a2008},\cite{g2},\cite{getzler}).

\begin{thm}\label{refined version of gross lsi}
Let $a\in G$.
Then there exist constants $C_1,C_2>0$ such that
for any sufficiently large $\la>0$ and $f\in \FC(\Pea)$,
it holds that
\begin{align}
\Ent_{\nu_{\la,a}}(f^2)
&\le
C_{\la}
\E_{\la,V_{\la,a}}(f,f), \label{grosslsi1}
\end{align}
where $C_{\la}=\frac{2}{\la}\left(1+\frac{C_1}{\la}\right)$ and
\begin{align*}
\E_{\la,V_{\la,a}}(f,f)
&=
\int_{\Pea}|(\nabla f)(\gamma)|_H^2d\nu_{\la,a}
+\int_{\Pea}
\la^2V_{\la,a}(\gamma)
f(\gamma)^2d\nu_{\la,a},\\
V_{\la,a}(\gamma)&=
\frac{1}{4}
\left\{|b(1)|^2+\frac{2}{\la}\log\left(\la^{-d/2}p(\la^{-1},e,a)\right)
\right\}+\frac{C_2}{\la}
\left\{1+|b(1)|^2+\left(\int_0^1|b(s)|ds\right)^2\right\},
\end{align*}
where $p(t,x,y)=e^{\frac{t}{2}\Delta}(x,y)$ is the heat kernel on
$G$.
\end{thm}

Note that $V_{\la,a}\in L^p(\Pea,\nu_{\la,a})$ for all $p>1$ and hence
$f^2V_{\la,a}\in L^1(\Pea,\nu_{\la,a})$ for all $f\in \FC(\Pea)$.

We derive a log-Sobolev inequality on $\D$ from the above result.

\begin{lem}\label{LSI on D}
 Let $\D$ be a domain considered in Theorem~$\ref{main theorem}$
and consider the Dirichlet form $\E_{\la,\D}$
with the Dirichlet boundary condition in
Definition~$\ref{Dirichlet form with DBC}$.
Then, there exists $C'>0$ such that for all 
$f\in \rD(\E_{\la,\D})\cap L^{\infty}(\D)$,
it holds that
\begin{align*}
 \Ent_{\nu_{\la,D}}(f^2)\le
C_{\la}\left(\E_{\la,V_{\la,a},\D}(f,f)
+e^{-\la C'}\|f\|_{L^2(\nu_{\la,\D})}^2\right),
\end{align*}
where $C_{\la}$ is the constant defined in 
Theorem~$\ref{refined version of gross lsi}$
and
\begin{align*}
 \E_{\la,V_{\la,a},\D}(f,f)&=
\int_{\D}|(\nabla f)(\gamma)|_H^2d\nu_{\la,\D}
+\int_{\D}
\la^2V_{\la,a}(\gamma)
f(\gamma)^2d\nu_{\la,\D}.
\end{align*}
\end{lem}

\begin{proof}
First note that (\ref{grosslsi1}) holds for
any $f\in \rD(\E_{\la})\cap L^{\infty}(\Pea)$
because there exists $f_n\in \FC(\Pea)$ such that
$f_n\to f$ in $L^2(\Pea,\nu_{\la,a})$ and
$\lim_{n\to\infty}\E_{\la}(f_n,f_n)=\E_{\la}(f,f)$.
Let $f\in \rD(\E_{\la,\D})\cap L^{\infty}(\D)$.
Note that $f\in \rD(\E_{\la})\cap L^{\infty}(\Pea)$.
Hence (\ref{grosslsi1}) holds for $f$.
We have
\begin{align*}
&\Ent_{\nu_{\la,D}}(f^2)=\int_{\D}f^2\log f^2d\nu_{\la,\D}-
\|f\|^2_{L^2(\nu_{\la,\D})}\log \|f\|^2_{L^2(\nu_{\la,\D})}\\
&\,=\nu_{\la,a}(\D)^{-1}\left(
\int_{\Pea}f^2\log f^2d\nu_{\la,a}\right)\\
&\qquad -
\nu_{\la,a}(\D)^{-1}\|f\|_{L^2(\Pea,\nu_{\la,a})}^2
\log\left(\nu_{\la,a}(\D)^{-1}\|f\|^2_{L^2(\Pea,\nu_{\la,a})}\right)\\
&\,=\nu_{\la,a}(\D)^{-1}\Ent_{\nu_{\la,a}}(f^2)
+\nu_{\la,a}(\D)^{-1}\log \nu_{\la,a}(\D)\|f\|_{L^2(\nu_{\la,a})}^2\\
&\le \nu_{\la,a}(\D)^{-1}
C_{\la}\left(\int_{\Pea}|\nabla f|_H^2d\nu_{\la,a}
+\int_{\Pea}
\la^2V_{\la,a}
f^2d\nu_{\la,a}\right)\\
&\qquad +\nu_{\la,a}(\D)^{-1}\log \nu_{\la,a}(\D)\|f\|_{L^2(\nu_{\la,a})}^2\\
&=C_{\la}\left(\E_{\la,V_{\la,a},\D}(f,f)+C_{\la}^{-1}
\nu_{\la,a}(\D)^{-1}\log \nu_{\la,a}(\D)\|f\|_{L^2(\nu_{\la,a})}^2\right),
\end{align*}
where we have used that 
$(\nabla f)(\gamma)=0\,\,\nu_{\la,a}\text{-} a.s.$ on $\D^{\complement}$.
This follows from the following.
We may assume $f$ is a bounded function.
By the assumption on the boundary, it suffices to show
\[
 \int_{(\D\cup \partial\D)^{\complement}}
|\nabla f(\gamma)|_H^2d\nu_{\la,a}=0.
\]
Setting $\tilde{f}(w)=(Y^{\ast}f)(w)$ and using 
Theorem~\ref{D and Y^-1(D)}, we see that this is equivalent to
for any $r>r_0$
\begin{align}
 \int_{\{\tilde{\varphi}>r\}}
|D_{S_a} \tilde{f}(w)|_H^2d\mu_{\la,a}=0.\label{nablaf=0}
\end{align}
Here $\tilde{\varphi}(w)=\varphi(Y(w))$ and $\varphi$ is the defining function
of $\D$ and $r_0$ is $\delta$ if $\varphi=\varphi_{\infty,K}$ and
$M$ if $\varphi$ is given by other functions.
By proof by contradiction,
(\ref{nablaf=0}) follows from the integration by parts formula and
the denseness of $\FC(W)\otimes H$ in $L^2(S_a\to H,\mu_{\la,a})$.
Since $\D$ contains the minimal geodesic, we see that there exists
$C'>0$ such that $\nu_{\la,a}(\D^c)\le e^{-C'\la}$, which implies
$\log \nu_{\la,a}(\D)=O(e^{-C'\la})$.
This completes the proof.
\end{proof}

Combining Lemma~\ref{LSI on D} and Lemma~\ref{NGS bound}, we obtain the following
lower bound estimate which is important in the analysis of $-L_{\la,\D}$
outside neighborhood of the geodesics $\{l_{\xi_i}\}$ in $\D$.
To state the result, we need the following estimate of
the exponential function of
$V_{\la,a}$.

\begin{lem}\label{exponential estimate0}
Let $p>1$.
 Let $\D$ be one of the domains which are given in Theorem~$\ref{main theorem}$.
We assume $\delta$ is sufficiently small if $\D=D_{K,\delta}$
according to $p$.
Then it holds that there exists $\la_p, C_p, C_p'>0$ such that
for any $\la\ge \la_p$,
$\int_{\D}e^{2p\la V_{\la,a}}d\nu_{\la,\D}\le C_p' e^{C_p\la}$
holds.
\end{lem}

\begin{proof}
Note that 
$\sup_{\la>1}\la^{-1}\left|\log\left(\la^{-d/2}p(\la^{-1},e,a)\right)\right|<\infty$
by the heat kernel estimate.
Also by using the estimate (\ref{pgwf}), 
it is not difficult to show that 
for any $p>1$
there exists $\la_p>0$ such that
\begin{align*}
\sup_{\la\ge \la_p} \int_{\Pea}\exp\left\{p
\left(1+|b(1)|^2+\left(\int_0^1|b(s)|ds\right)^2\right)\right\}
d\nu_{\la,a}(\gamma)<\infty.
\end{align*}
For this, see the estimate (\ref{estimate of W}).
Therefore, it suffices to show that
there exists $C_p, C_p'>0$ such that
$\int_{\D}e^{p\la |b(1)|^2}d\nu_{\la,\D}\le C_p' e^{C_p\la}$.
This is proved in Lemma~4.2 in \cite{a2008}.
This completes the proof.
\end{proof}

\begin{cor}\label{lower bound of an energy of f}
 Let $\D$ be a domain considered in Theorem~$\ref{main theorem}$
and consider the Dirichlet form with the Dirichlet boundary condition in
Definition~$\ref{Dirichlet form with DBC}$.
When $\D=D_{K,\delta}$, we assume $\delta$ is sufficiently small.
Let $\rho$ be a bounded measurable function on $\D$.
Then for sufficiently large $\la$,  
$\int_{\D}\exp
\left(
C_{\la}\la^2
(V_{\la,a}+\rho)\right)d\nu_{\la,\D}<\infty$ and
the following estimate holds:
For any $f\in \rD(\E_{\la,\D})$, 
it holds that
\begin{align*}
&\int_{\D}|\nabla f|^2d\nu_{\la,\D}
\\
&\ge
-C_{\la}^{-1}
\log\left\{\int_{\D}\exp
\left(
C_{\la}\la^2
(V_{\la,a}+\rho)\right)d\nu_{\la,\D}
\right\}\|f\|^2_{L^2(\D,\nu_{\la,\D})}
+\la^2\int_{\D}(\rho-e^{-C'\la})f^2d\nu_{\la,\D},
\end{align*}
where $C_{\la}$ is the constant defined 
in Theorem~$\ref{refined version of gross lsi}$.
\end{cor}

\begin{proof}
For $n\in \NN$, let $V_{\la,a,n}=(-n)\vee V_{\la,a}\wedge n$.
Let $f\in \rD(\E_{\la,\D})\cap L^{\infty}(\D)$.
Applying Lemma~\ref{NGS bound} to the case
$Q(f)=\E_{\la,V_{\la,a},\D}(f,f)+e^{-\la C'}\|f\|^2_{L^2(\nu_{\la,\D})}$, 
$S=\rD(\E_{\la.\D})\cap L^{\infty}(\D)$
and $C=C_{\la}$,
we have
\begin{align*}
& \E_{\la,V_{\la,a},\D}(f,f)+e^{-C'\la}\|f\|_{L^2(\nu_{\la,\D})}^2
-\int_{\D}\la^2(\rho+V_{\la,a,n})f^2d\nu_{\la,\D}\\
&\ge -C_{\la}^{-1}
\log\left\{
\int_{\D}\exp(C_{\la}\la^2(V_{\la,a,n}+\rho))d\nu_{\la,\D}
\right\}\|f\|_{L^2(\nu_{\la,\D})}^2.
\end{align*}
Letting $n\to\infty$ and using Lemma~\ref{exponential estimate0},
we see that the desired estimate holds for 
$f\in \rD(\E_{\la,\D})\cap L^{\infty}(\D)$.
By approximating $f\in \rD(\E_{\la,\D})$ by
$(-n)\vee f\wedge n$, we can complete the proof.
\end{proof}

\begin{rem}\label{remark on log-Sobolev with potential}
If $\exp(\frac{\la}{2}|b(1)|^2)\in L^p(\nu_{\la,a})$ for some $p>1$, then
there exists $C'_\la>0$ such that
$\Ent_{\nu_{\la,a}}(f^2)\le C_\la'\E_{\la}(f,f)$ $(f\in \rD(\E_{\la}))$.
That is, log-Sobolev inequality without potential function holds.
However, it is proved that $\int_{\Pea}e^{\frac{\la}{2}|b(1)|^2}d\nu_{\la,a}=\infty$
in the case where $G=SU(n)$ in \cite{a2008}.
At the moment, the validity of log-Sobolev inequality 
$\E_{\la}$ on $\Pea$ is still unknown.

In this remark, we explain log-Sobolev inequality with suitable potential functions
are still useful for some purpose.
 Let $F\in C^{2}_b(\RR^d)$.
Suppose that $\tilde{F}(x)=\frac{1}{2}|x|^2+F(x)$ is a Morse function on $\RR^d$,
that is, $\tilde{F}$ has finite nondegenerate critical points $c_1,\ldots,c_N$.
Let $\mu_{\la}(dx)=\left(\frac{\la}{2\pi}\right)^{d/2}e^{-\frac{\la}{2}|x|^2}dx$.
We have the following Gross's log-Sobolev inequality,
\begin{align}
 \Ent_{\mu_{\la}}(f^2)\le \frac{2}{\la}\E_{\la}(f,f),\label{gross lsi0}
\end{align}
where $\E_{\la}(f,f)=\int_{\RR^d}|\nabla f(x)|^2d\mu_{\la}(x)$
and $f\in C^1_b(\RR^d)$.
Consider a probability measure
$\mu_{\la,F}(dx)=\frac{e^{-\la F(x)}}{Z_{\la}}\mu_{\la}(dx)$ on $\RR^d$,
where $Z_{\la}$ is the normalized constant.
Let us consider a closable form given by 
$\E_{\la,F}(f,f)=\int_{\RR^d}|\nabla f(x)|^2d\mu_{\la,F}(x)$,
($f\in C^{\infty}_0(\RR^d)$).
We denote the Dirichlet form which is obtained by taking the closure
by the same notation $\E_{\la,F}$.
Then by the Holley-Stroock's perturbation result, 
we see that 
\begin{align*}
 \Ent_{\mu_{\la,F}}\le 2\la^{-1}e^{2\la\|F\|_{\infty}}\E_{\la,F}(f,f),
\quad f\in C^1_b(\RR^d),
\end{align*}
which implies the log-Sobolev inequality without potential function.
On the other hand,
we can obtain the following inequality by
substituting $f\cdot e^{-\la F/2}/\sqrt{Z_{\la}}$ into
(\ref{gross lsi0}).
\begin{align}
 \Ent_{\mu_{\la,F}}(f,f)\le \frac{2}{\la}
\left(\E_{\la,F}(f,f)+\int_{\RR^d}V_{\la,F}f^2d\mu_{\la,F}\right),
\quad f\in C^1_b(\RR^d),
\label{log-sobolev with potential}
\end{align}
where
$
V_{\la,F}(x)=
-\frac{\la^2}{4}|\nabla F(x)|^2+\frac{\la^2}{2}F(x)
+\frac{\la}{2}L_{\la}F(x)+\frac{\la}{2}\log Z_{\la}
$ and $L_{\la}F(x)=\Delta F(x)-\la (x,\nabla F(x))$.
By using Lemma~\ref{NGS bound}, for any bounded measurable function 
$\rho$ on $\RR^d$,
we have
\begin{align}
 \E_{\la,F}(f,f)&\ge
-\frac{\la}{2}\log
\left(
\int_{\RR^d}\exp\left\{-\frac{2}{\la}
\left(
\frac{\la^2}{4}|DF|^2-\frac{\la}{2}L_{\la}F-\la^2\rho-\frac{\la^2}{2}F
-\frac{\la}{2}\log Z_{\la}
\right)d\mu_{\la,F}\right\}
\right)\nn\\
&\quad +\la^2\int_{\RR^d}\rho f^2d\mu_{\la,F}.\label{finite dimension gns}
\end{align}
The integral in (\ref{finite dimension gns}) reads
\begin{align*}
& \int_{\RR^d}\exp\left\{
-\frac{2}{\la}
\left(
\frac{\la^2}{4}|DF|^2-\frac{\la}{2}L_{\la}F-\la^2\rho
-\frac{\la^2}{2}F-\frac{\la}{2}\log Z_{\la}\right)
\right\}d\mu_{\la,F}\nn\\
&=
\int_{\RR^d}\exp\left(
-\frac{\la}{2}\Bigl(
|D\tilde{F}(x)|^2-4\rho(x)\Bigr)+\Delta F(x)\right\}
\left(\frac{\la}{2\pi}\right)^{d/2}dx=:I(\la).
\end{align*}
Let $\kappa$ be a sufficiently small positive number and
set $\rho(x)=\kappa \min_{1\le i\le N}\{|x-c_i|^2,1\}$.
Then, we see that $\lim_{\la\to\infty}I(\la)$ converges by the Laplace method.
Hence, by the estimate (\ref{finite dimension gns}) and the IMS localization formula, 
we see that the 
lowlying spectrum of the generator of $\E_{\la,F}$ can be approximated by
the spectrum of the quantum harmonic oscillator which we explained in 
Section~\ref{statement of main theorem}.
Our proof of Theorem~\ref{main theorem} is an extension of this argument.
Using the inifinite dimension version of
(\ref{finite dimension gns}), I discussed the asymptotics of
the lowest eigenvalues of
Witten Laplacian acting on $p$-forms on Wiener space
(toy model of supersymmetric Hamiltonian because it does not contain renormalizations)
in the conference talk
in, ``Stochastic Analysis in Infinite Dimensional Spaces'',
RIMS November 8, 2002 and
``Geometry on Path Space (Satellite conference to the ICM 2002, Beijing)
(unpublished work).
Related arguments can be found in 
\cite{a2003-2}, \cite{a2007}, \cite{a2009}, \cite{a2012, a2015, a2016}.
See also \cite{b-digesu}.
However note that Theorem~\ref{refined version of gross lsi} can be proved by
the inheritance argument due to Gross (\cite{g2}) by using the log-Sobolev inequality
on $\Pe$ with respect to the measure $\nu_{\la}$ although $\nu_{\la,a}$ is singular
with respect to $\nu_{\la}$.
Several log-Sobolev inequalities (with potential functions) were proved on 
pinned path space on Riemannian manifolds 
(\cite{a1996}, \cite{a2000}, \cite{gong-ma}, \cite{driver3}, 
see also\cite{ae}, \cite{chl}, \cite{hsu1}).
However, it is not clear whether such inequalities can be applied to 
the present problem or not.
Finally, we note that 
log-Sobolev inequalities with potential functions on submanifolds in Euclidean spaces
have been derived from different view points 
(\cite{brendle}, \cite{ecker}).
\end{rem}

\section{A local coordinate system on submanifolds
and a change of variable formula}\label{local coordinate0}

Let $a\in G$ and
pick $k\in S_a^H$.
Recall that we denote $T_kS_a$ and $T_kS_a^{\perp}$ by $H_0$ and $H_0^{\perp}$
respectively as we explained in Definition~\ref{tangent space}.
Recall that $B_{k,\ep}^{\Omega_0}(0)
=\{\eta\in \Omega_0~|~\|\overline{\eta}\|_{\alpha}<\ep, \|C(k,\eta)\|_{2\alpha}<\ep,
\|C(\eta,k)\|_{2\alpha}<\ep\}$
 and
$U_{\ep}^{\Omega}(k)=k+B_{k,\ep}^{\Omega}(0)$
are defined in
Example~$\ref{examples of H open set}$.
Here, $\Omega_0=W_0\cap \Omega$ and $W_0=\overline{H_0}^{W}$.
Also $D_0$ denotes the $H$-derivative on $W_0$.

We will construct a neighborhood $V_{\ep}(k)$ of $k$ in $\Omega$ and 
introduce a local coordinate on $V_{\ep}(k)\cap S_a$ of $S_a$.
The positive number $\ep$ depends on 
$\|\overline{k}\|_{\alpha}$ only.
Below, $B_{\ep}(a)$ denotes the metric open ball
centered at $a$ with the radius $\ep$ in $G$.
Also, recall that $i(G)$ denotes the injectivity radius of $G$.
See (\ref{injectivity radius}).

\begin{lem}\label{local chart}
Let $k\in S_a^H$. 
\begin{enumerate}
 \item[$(1)$]  
There exist positive two numbers $\ep<\ep_0(<i(G))$ 
such that for any 
$(x,\eta)\in B_{\ep}(a)\times B_{k,\ep}^{\Omega_0}(0)$,
there exists a
unique $\varphi_x(\eta)\in H_0^{\perp}$ with
$\|\varphi_x(\eta)\|_H\le \ep_0$ such that
$\varPhi_{k,x}(\eta):=k+\eta+\varphi_x(\eta)\in S_x$. 
Positive numbers $\ep$ and $\ep_0$ depend on 
$\|\overline{k}\|_{\alpha}$.
We may denote them by $\ep(k)$ and $\ep_0(k)$ respectively.
Define $\varPhi_{k} : B_{\ep}(a)\times B_{k,\ep}^{\Omega_0}(0)\to \Omega$
by $\varPhi_{k}(x,\eta)=\varPhi_{k,x}(\eta)$.
The mapping $\varPhi_k$
is one-to-one Borel measurable mapping.
Define $V_{\ep}(k)=\varPhi_k(B_{\ep}(a)\times B_{k,\ep}^{\Omega_0}(0))$
and $V_{\ep}^{S_x}(k)=V_{\ep}(k)\cap S_x$.
It holds that
$\varphi\in C^{1,n}_{b,H_0}
(B_{\ep}(a)\times B_{k,\ep}^{\Omega_0}(0), H_0^{\perp})$,
$(\varphi(x,\eta)=\varphi_x(\eta))$
for any $n\in \NN$ holds and
the image of the derivative $D_0\varPhi_k(x,\eta)$ is included in
$T_{\varPhi_k(x,\eta)}S_a$.
\item[$(2)$] For $\varphi_{x}(\eta)$ defined in $(1)$,
there exists a unique $v=v(x,\eta)\in \g$ such that 
$\varphi_x(\eta)=Q(k)v$ and $\|\varphi_x(\eta)\|_H=|v|_{\g}$.
\item[$(3)$] Let $\ep$ be the positive number specified in $(1)$.
For $(x,\eta)\in B_{\ep}(a)\times B_{k,\ep}^{\Omega_0}(0)$,
Let $w=\varPhi_{k,x}(\eta)$. Then
$(x,\eta)=(Y(1,e,w),P(k)(w-k))$.
\item[$(4)$] Define 
$\varPsi_k : \Omega\to G\times \Omega_0$ by
$\varPsi_k(w)=(Y(1,e,w),P(k)(w-k))$.
Then $\varPsi_k(V_{\ep}(k))=B_{\ep}(a)\times B_{k,\ep}^{\Omega_0}(0)$ 
holds and the mapping 
$\varPsi_k : V_{\ep}(k)\to B_{\ep}(a)\times B_{k,\ep}^{\Omega_0}(0)$ 
and $\varPhi_k : B_{\ep}(a)\times B_{k,\ep}^{\Omega_0}(0)\to V_{\ep}(k)$ are 
bijective Borel measurable mapping and they are inverse mapping to each other.
Also define $\varPsi_{k,x}=P(k)(w-k)|_{S_x} : S_x\to \Omega_0$.
Then,
$\varPsi_{k,x}(V_{\ep}^{S_x}(k))=B_{k,\ep}^{\Omega_0}(0)$ holds.
Also
$\varPsi_{k,x} : V_{\ep}(k)^{S_x}\to B_{k,\ep}^{\Omega_0}(0)$ 
and $\varPhi_{k,x} : B_{k,\ep}^{\Omega_0}(0)\to V_{\ep}^{S_x}(k)$ are 
bijective Borel measurable mapping and they are inverse mapping to each other.
\item[$(5)$] Let $\ep$ be the positive number defined in $(1)$.
There exists $\delta>0$ which depends on 
$\|\overline{k}\|_{\alpha}$
such that $U_{\delta}^{\Omega}(k)\cap S_a\subset V_{\ep}^{S_a}(k)$.
\item[$(6)$] 
For $v(a,\eta)$ $(\eta\in B_{k,\ep}^{\Omega_0}(0))$,
there exists $C>0$ such that
\begin{align}
|D_xv(a,\eta)(R_a)_{\ast}-I_{T_e(G)}|&\le C
\left(\|\bar{\eta}\|_{\alpha}+\|C(k,\eta)\|_{2\alpha}+\|C(\eta,k)\|_{2\alpha}
\right),\label{Dxv}\\
 |v(a,\eta)|&\le 
C\left(\|\bar{\eta}\|_{\alpha}+\|C(k,\eta)\|_{2\alpha}+\|C(\eta,k)\|_{2\alpha}
\right)^2,\label{v}\\
|D_{0}v(a,\eta)|_{H_0\otimes \g}&\le 
C\left(\|\bar{\eta}\|_{\alpha}+\|C(k,\eta)\|_{2\alpha}+\|C(\eta,k)\|_{2\alpha}
\right),\label{D_0v}
\end{align}
where $C$ is a positive constant which depends on $\|\overline{k}\|_{\alpha}$
and $\ep_0$ polynomially.
\end{enumerate}
\end{lem}

\begin{proof}[Proof of Lemma~$\ref{local chart}$]
We prove (1) and (2).
First, we prove the one-to-one property.
Suppose $\varPhi_k(x,\eta)=\varPhi_k(x',\eta')$ 
$(x,x'\in B_{\ep}(a), \eta,\eta'\in B_{k,\ep}^{\Omega_0}(0))$.
Then by the definition of $\varPhi_k$, we have $x=x'$.
The identity $\eta+\varphi_x(\eta)=\eta'+\varphi_x(\eta')$ implies
$\eta-\eta'=\varphi_x(\eta')-\varphi_x(\eta)\in H_0^{\perp}\subset W^{\ast}$.
Hence $\eta-\eta'\in W^{\ast}\subset H$.
Also by the definition of $W_0$, there exists $h_n\in H_0$ such that
$\lim_{n\to\infty}\|h_n-(\eta-\eta')\|_W=0$.
Therefore, we have
\begin{align*}
 \|\eta-\eta'\|_H^2&=\left(\eta-\eta',\varphi_x(\eta')-\varphi_x(\eta)\right)_H\\
&={}_W\left(\eta-\eta',\varphi_x(\eta')-\varphi_x(\eta)\right)_{W^{\ast}}\\
&=\lim_{n\to\infty}\left(h_n,\varphi_x(\eta')-\varphi_x(\eta)\right)_H.
\end{align*}
Noting $(h_n,\varphi_x(\eta')-\varphi_x(\eta))_H=0$, we obtain
$\eta=\eta'$ which implies the injectivity of $\varPhi_k$.
We prove the unique existence $\varphi_x(\eta)$.
 By Lemma~\ref{tangent space lemma}, it suffices to show the unique existence
of $v=v(x,\eta)\in \g$ satisfying
$Y(1,e,k+\eta+Q(k)v)=x$,
where $\varphi_x(\eta)=Q(k)v$ and
$\|\varphi_x(\eta)\|_H=|v|_\g$.
Below, we do the calculation using the embedding
$\g,\, G\subset M(n,\CC)$.
For $\theta\in [0,1]$, $v\in \g$ and $\eta\in \Omega_0$,
let 
\begin{align*}
 Y^{\theta,v,\eta}(s)=
Y\Bigl(s,e, z(k,\theta,\eta,v)\Bigr),
\quad z(k,\theta,\eta,v)=k+\theta\bigl(\eta+Q(k)v\bigr).
\end{align*}

Then, by Theorem~\ref{properties of Y} (3), 
$Y$ is a $C^1$ function of $\theta$ and for any $\xi\in \RR^d$
\begin{align}
&  \partial_{\theta}Y^{\theta,v,\eta}(s)\left(Y^{\theta,v,\eta}\right)^{-1}=
\int_0^s
Ad\Bigl(Y^{\theta,v,\eta}(u))\Bigr)d\eta_u
+
\int_0^sAd\Bigl(Y^{\theta,v,\eta}(u)(Y(u,e,k)^{-1}\Bigr)vdu,\label{derivative of Y}\\
&\partial_{\theta}\left(Ad(Y^{\theta,v,\eta}(s))\xi\right)=
\left[\partial_{\theta}Y^{\theta,v,\eta}(s)\left(Y^{\theta,v,\eta}\right)^{-1},
Ad(Y^{\theta,v,\eta}(s))\xi\right]\nonumber\\
&\qquad=\left[
\int_0^s
Ad\Bigl(Y^{\theta,v,\eta}(u))\Bigr)d\eta_u
+
\int_0^sAd\Bigl(Y^{\theta,v,\eta}(u)(Y(u,e,k)^{-1}\Bigr)vdu,
Ad(Y^{\theta,v,\eta}(s))\xi
\right].\label{derivative of Ad}
\end{align}
If $|v|$, $\|\bar{\eta}\|_{\alpha}$ and $\|C(\eta,k)\|_{2\alpha}$ 
are sufficiently small,
then $d_{\alpha}(\overline{k},\overline{k+\theta(\eta+Q(k)v)})$
is small. This follows from the following calculation.
\begin{align*}
&C(k,k)_{s,t}-C\Big(k+\theta(\eta+Q(k)v),k+\theta(\eta+Q(k)v\Big)\\
&\qquad =
\theta C\Big(k,\eta+Q(k)v\Big)_{s,t}+\theta C\Big(\eta+Q(k)v,k\Big)_{s,t}
+\theta^2 C\Big(\eta+Q(k)v,\eta+Q(k)v\Big)
\end{align*}
Thus, by the locally Lipschitz continuity property of the solution of RDEs, 
for any $\ep>0$, there exists 
$\ep(\|\overline{k}\|_{\alpha})$
such that
\begin{align}
\max\left(v, \|\bar{\eta}\|_{\alpha}, \|C(\eta,k)\|_{2\alpha},
\|C(k,\eta)\|_{2\alpha}\right)
\le \ep(\|\overline{k}\|_{\alpha})
\,\,\Longrightarrow\,\,
\sup_{s,\theta}|Y^{\theta,v,\eta}(s)-Y(s,e,k)|\le \ep.
\label{local lipschitz}
\end{align}
For small positive number $r$, we consider a local coordinate 
map $\phi : B_{r}(a)\to \g$.
For example, $\phi(x)=\log(xa^{-1})$.
If $\max\left(v, \|\bar{\eta}\|_{\alpha}, \|C(\eta,k)\|_{2\alpha}\right)$
is sufficiently small, then
$\phi\left(Y^{\theta,v,\eta}(1)\right)$ can be defined.
Using this, since $Y^{1,v,\eta}(1)=Y(1,e,k+\eta+Q(k)v)$ and
$Y^{0,v,\eta}(1)=Y(1,e,k)=a$,
$Y(1,e,k+\eta+Q(k)v)=x$ is equivalent to
\begin{align}
&  \phi(x)-\phi(a)=
\int_0^1\partial_{\theta}\phi(Y^{\theta,v,\eta}(1))d\theta\nonumber\\
&=
\int_0^1(D\phi)(Y^{\theta,v,\eta}(1))R_{Y^{\theta,v,\eta}(1)}\left(\int_0^1
Ad\Bigl(Y^{\theta,v,\eta}(u)\Bigr)d\eta_u\right)d\theta\nonumber\\
&\qquad+
\left\{
\int_0^1(D\phi)(Y^{\theta,v,\eta}(1))
R_{Y^{\theta,v,\eta}(1)}
\left(\int_0^1Ad\Bigl(Y^{\theta,v,\eta}(u)(Y(u,e,k)^{-1}\Bigr)du\right)d\theta
\right\}
v.
\label{x-a}
\end{align}
Note that 
$\phi'(x):=(D\phi)(x)R_x\in \LL(\g,\g)$
is an invertible mapping for $x\in B_r(a)$.
Also note that $\phi'(x)\xi=V_{\xi}\phi(x)$, where
$V_{\xi}$ is the right invariant vector field associated with $\xi\in \g$.
Furthermore,
$x\mapsto \phi'(x)\in \LL(\g,\g)$ and
$x\mapsto \left(\phi'(x)\right)^{-1}\in \LL(\g,\g)$
are $C^{\infty}$ mapping $(x\in B_r(a))$.
Write
\begin{align*}
 A(\eta,v)&=
\int_0^1\phi'(Y^{\theta,v,\eta}(1))
\left(\int_0^1Ad\Bigl(Y^{\theta,v,\eta}(u)(Y(u,e,k)^{-1}\Bigr)du\right)d\theta,\\
B(\eta,v)&=
\int_0^1\phi'(Y^{\theta,v,\eta}(1))
\left(\int_0^1
Ad\Bigl(Y^{\theta,v,\eta}(u)\Bigr)d\eta_u\right)d\theta.
\end{align*}
Note that $A(\eta,v)\in \LL(\g,\g)$ and
$B(\eta,v)\in \g$.
By (\ref{local lipschitz}),
there exists $\ep_0>0$
which depends only on $\|\bar{k}\|_{\alpha}$
such that if 
\begin{align}
 \max\left(v, \|\bar{\eta}\|_{\alpha}, \|C(\eta,k)\|_{2\alpha}\right)
\le\ep_0,\label{smallness of eta etc}
\end{align}
is small
then $A(\eta,v)$ is close to $\phi'(a)$ and thus invertible and
$\left\|A(\eta,v)^{-1}\right\|_{\LL(\g)}\le 2\|\phi'(a)^{-1}\|_{\LL(\g)}$.
Hence under the condition 
(\ref{smallness of eta etc}),
(\ref{x-a}) is equivalent to
\begin{align}
 v&=
A(\eta,v)^{-1}\Bigl(
\phi(x)-\phi(a)-B(\eta,v)
\Bigr)
=:M_{x,\eta}(v).\label{Mxeta}
\end{align}
Again by the locally Lipschitz continuity theorem of solutions of RDEs and 
estimate of rough integrals, taking sufficiently small $\ep<\ep_0$,
we see that for $x$ and $\eta$ satisfying 
\begin{align}
 d(x,a)<\ep,\quad \|\bar{\eta}\|_{\alpha}<\ep,\quad
\|C(\eta,k)\|_{2\alpha}<\ep, \quad \|C(k,\eta)\|_{2\alpha}<\ep,\label{epsilon}
\end{align}
the following hold:
\begin{enumerate}
 \item[(i)] $M_{x,\eta}(\overline{B_{\ep_0}(0)})\subset \overline{B_{\ep_0}(0)}$,
\item[(ii)] $M_{x,\eta}$ is a contraction mapping with the Lipschitz constant 
smaller than $1/2$,
\end{enumerate}
where $B_{\ep_0}(0)$ is the ball in $\g$ with the radius $\ep_0$.
This implies that there exists a unique $v=v(x,\eta)\in \overline{B_{\ep_0}(0)}$ 
such that
$M_{x,\eta}(v)=v$ holds.
For this $v$, by setting $\varphi_x(\eta)=Q(k)v(x,\eta)$, 
$k+\eta+\varphi_x(\eta)\in S_x$ holds.
Let me be more precise about the above (i) and (ii).
First, we consider (i).
By the estimate of the solution and the rough integrals, we have
if $|v|<\ep_0$, then there exists a polynomial growth positive 
increasing function
$C$ on $\RR^{+}$ such that
\begin{align*}
 |M_{x,\eta}(v)|&\le
2\|\phi'(a)^{-1}\|
\cdot |\phi(x)-\phi(a)|
+C\left(\|\overline{k}\|_{\alpha}
+\ep_0\right)
\left(\|\bar{\eta}\|_{\alpha}+\|C(k,\eta)\|_{2\alpha}\right).
\end{align*}
Hence, if $\ep$ in (\ref{epsilon}) is sufficiently small, then
(i) holds.
For $v, v'$ satisfying $\max(|v|,|v'|)\le \ep_0$, there exists a
polynomial growth positive increasing function
$C$ on $\RR^{+}$ such that
\begin{align}
& \left|M_{x,\eta}(v)-M_{x,\eta}(v')\right|
\le C\Bigl(
\|\overline{k}\|_{\alpha}+
\ep_0\Bigr)
\left(\|\bar{\eta}\|_{\alpha}+\|C(k,\eta)\|_{2\alpha}\right)
|v-v'|.\label{contraction of M}
\end{align}
To show this,
let $f\in C^{\infty}_b(M(n,\CC),\LL(\g,\g))$
and set
\begin{align*}
  I(v)&=\int_0^1f\left(Y^{\theta,v,\eta}(s)\right)d\eta_s, \qquad v\in \g.
\end{align*}
For $v, v'\in \g$, we have
\begin{align*}
&I(v)-I(v')\nonumber\\
&=\int_0^1\int_0^1
(Df)(Y(s,e,z(k,\theta,\eta,rv+(1-r)v')))
\phantom{lllllllllllllllllllllllllllllllllllllllllll}
\nn\\
&\quad\qquad\quad\qquad
 \left[(DY)(s,e,z(k,\theta,\eta,rv+(1-r)v'))\bigl[\theta Q_k(v-v')\bigr]\right]
dr d\eta_s
\end{align*}
By the estimate of rough integral (\ref{estimate of rough integral}),
we obtain (\ref{contraction of M}).
This implies (ii) holds for small $\ep$.
We next consider the smoothness of $v(x,\eta)$.
Suppose $v=v(x,\eta)$ is a function of $(x,\eta)\in B_{\ep}(a)\times
B_{k,\ep}^{\Omega_0}(0)$ and satisfies
$v\in C^{1,1}_{b,H_0}(B_{\ep}(a)\times B_{k,\ep}^{\Omega_0}(0))$.
Then,  again, by by Theorem~\ref{properties of Y} (2),
the derivative of $Y^{\theta,v(\eta),\eta}(s)$ can be calculated as
\begin{align}
& D_0Y^{\theta,v(\eta),\eta}(s)[h]=\partial_\eta Y^{\theta,v(\eta),\eta}(s)[h]+
\partial_vY^{\theta,v(\eta),\eta}(s)[(D_0v)[h]]\nonumber\\
&\qquad =\theta\left(\int_0^s
Ad\Bigl(Y^{\theta,v,\eta}(u))\Bigr)dh_u
+
\int_0^sAd\Bigl(Y^{\theta,v,\eta}(u)(Y(u,e,k)^{-1}\Bigr)du(D_0v)[h]
\right)Y^{\theta,v(\eta),\eta}(s),\label{derivative of Y1}\\
& D_0\left(Ad(Y^{\theta,v,\eta}(s))\xi\right)[h]\nonumber\\
&=
\theta\left[\int_0^s
Ad\Bigl(Y^{\theta,v,\eta}(u))\Bigr)dh_u
+
\int_0^sAd\Bigl(Y^{\theta,v,\eta}(u)(Y(u,e,k)^{-1}\Bigr)du(D_0v)[h],
Ad(Y^{\theta,v,\eta}(s))\xi
\right]\label{derivative of Y2}
\\
& D_xY^{\theta, v(x,\eta),\eta}
=\theta
\left(\int_0^sAd(Y^{\theta,v(x,\eta),\eta}(u)Y(u,e,k)^{-1})du
D_xv(x,\eta)\right)
Y^{\theta,v(x,\eta),\eta}(s),\label{derivative of Y3}
\end{align}
\begin{align}
& D_0\left(\int_0^1Ad\left(Y^{\theta,v,\eta}(s)\right)d\eta_s\right)[h]
=\int_0^1Ad\left(Y^{\theta,v,\eta}(s)\right)dh_s
+\theta
\int_0^1D_0\left(Ad\left(Y^{\theta,v,\eta}(s)\right)\right)
[h]d\eta_s\nonumber\\
&=\int_0^1Ad\left(Y^{\theta,v,\eta}(s)\right)dh_s\nonumber\\
&\quad +\int_0^1
\left[\int_0^s
Ad\Bigl(Y^{\theta,v,\eta}(u))\Bigr)dh_u
+
\theta\int_0^sAd\Bigl(Y^{\theta,v,\eta}(u)(Y(u,e,k)^{-1}\Bigr)du(D_0v)[h],
Ad(Y^{\theta,v,\eta}(s))d\eta_s
\right]\nonumber\\
&=\int_0^1Ad\left(Y^{\theta,v,\eta}(s)\right)dh_s
+\theta\left[\int_0^1Ad\left(Y^{\theta,v,\eta}(s)\right)dh_s, 
\int_0^1Ad\left(Y^{\theta,\eta,v}(s)\right)d\eta_s\right]\nonumber\\
&\quad+\theta
\int_0^1\left[\int_0^sAd\left(Y^{\theta,\eta,v}(u)\right)d\eta_u,
Ad\left(Y^{\theta,\eta,v}(s)\right)dh_s\right]\nonumber\\
&\quad+\theta
\left[\int_0^1Ad\Bigl(Y^{\theta,v,\eta}(s)(Y(s,e,k)^{-1}\Bigr)ds(D_0v)[h],
\int_0^1Ad\left(Y^{\theta,v,\eta}(s)\right)d\eta_s\right]\nonumber\\
&\quad+\theta
\int_0^1\left[\int_0^sAd\left(Y^{\theta,v,\eta}(u)\right)d\eta_u,
Ad\Bigl(Y^{\theta,v,\eta}(s)(Y(s,e,k)^{-1}\Bigr)(D_0v)[h]\right]ds.
\label{derivative of Y4}
\end{align}
Moreover, noting $\int_0^1Ad(Y^{0,v,\eta}(s))d\eta_s=0$ and using
(\ref{derivative of Ad}),
we obtain the following expression,
\begin{align}
&\int_0^1
Ad\Bigl(Y^{\theta,v,\eta} (s)\Bigr)d\eta_s\nonumber\\
&=\int_0^1\left(\int_0^{\theta}
\left[
\int_0^s
Ad\Bigl(Y^{\tau,v,\eta}(u))\Bigr)d\eta_u
+
\int_0^sAd\Bigl(Y^{\tau,v,\eta}(u)(Y(u,e,k)^{-1}\Bigr)vdu,
Ad(Y^{\tau,v,\eta}(s))\ep_i
\right]d\tau\right) d\eta^i_s.\label{quadratic form}
\end{align}
We consider successive approximation sequence to prove the continuity and
smoothness of $v(x,\eta)$.
Let us define 
$v_n=v_n(x,\eta)$ $((x,\eta)\in B_{\ep}(a)\times B_{k,\ep}^{\Omega_0}(0)
, n\ge 0)$ inductively by $v_0=0$,
$v_{n+1}=M_{x,\eta}(v_{n})$ $(n\ge 0)$.
Note that
$v_1(\eta)$ is a $C^{\infty}_b$ function of
$Y(s,e,k+\theta\eta)$, $\int_0^1Ad(Y(s,e,k+\theta\eta))d\eta_s$
and $\phi(x)-\phi(a)$.
These functionals belong to 
$C^{1,\infty}_{b,H_0}(B_{\ep}(a)\times B_{k,\ep}^{\Omega_0}(0))$.
See Theorem~\ref{properties of Y} (3).
Hence $v_1\in C^{1,\infty}_{b,H_0}(B_{\ep}(a)\times B_{k,\ep}^{\Omega_0}(0),\RR^d)$.
Also by the contraction property of $M_{x,\eta}$,
we have $|v_{n+1}-v_n|\le 2^{-n}\ep_0$ and
$v(x,\eta)=\lim_{n\to\infty}v_n(x)$ is uniform convergence in
$B_{\ep}(a)\times B_{k,\ep}^{\Omega_0}$.
Hence $v(x,\eta)$ is a continuous function in the product topology of 
the Euclidean topology and $d_{\alpha}$.
We consider the differentiability property of $v_n(x,\eta)$ on $x,\eta$ $(n\ge 2)$.
If $v_n\in C^{1,1}_{b,H_0}(B_{\ep}(a)\times B_{k,\ep}^{\Omega_0}(0))$,
by the definition of $M_{x,\eta}(v)$,
using (\ref{derivative of Y1}),
(\ref{derivative of Y2}),
(\ref{derivative of Y3}) 
and (\ref{derivative of Y4}) and the chain rule of the derivative,
we see that
$v_{n+1}\in C^{1,1}_{b,H_0}(B_{\ep}(a)\times B_{k,\ep}^{\Omega_0}(0))$ holds
 and the derivatives $D_0v_n(x,\eta)\in \LL(H_0,\RR^d)$ satisfies
\begin{align}
 D_0v_{n+1}&=r_n D_0v_n+f_n,\label{D0vn}
\end{align}
where $r_n=r(v_n,\eta)$, $f_n=f(v_n,\eta)$ and
\begin{align*}
r\in 
C^{\infty,\infty}_{b,H_0}\Bigl(\RR^d\times B_{k,\ep}^{\Omega_0}(0),
\mathcal{L}(\RR^d,\RR^d)\Bigr),
\quad
f\in C^{\infty,\infty}_{b,H_0}
\left(\RR^d\times B_{k,\ep}^{\Omega_0}(0),\LL(H_0, \RR^d)\right).
\end{align*}
Here, $r_nD_0v_n$ stands for the composition operator.
Moreover, if $\ep_0$ is sufficiently small, then 
$|r(v,\eta)|\le 1/2$ for all $v$ and $\eta$.
By using this, we obtain
\begin{align*}
\|D_0v_{n+1}\|&\le \frac{1}{2}\|D_0v_n\|+\|f\|
\le \sum_{k=0}^n\left(\frac{1}{2}\right)^{k}\|f\|\le 2\|f\|.
\end{align*}
Thus we have
\begin{align*}
 \|D_0v_{n+1}-D_0v_n\|&\le
\|r_n(D_0v_n-D_0v_{n-1})\|+\|(r_n-r_{n-1})D_0v_{n-1}\|+\|f_n-f_{n-1}\|\\
&\le 2^{-1}\|D_0v_n-D_0v_{n-1}\|+
C|v_n-v_{n-1}|\\
&\le 2^{-1}\|D_0v_n-D_0v_{n-1}\|+C\cdot 2^{-(n-1)}\ep_0\\
&\le 2^{-n}\|D_0v_1\|+C\cdot 2^{-(n-1)}n\ep_0.
\end{align*}
This implies $\lim_{n\to\infty}D_0v_n$ converges in uniformly on
$B_{\ep}(a)\times B_{k,\ep}^{\Omega_0}(0)$.
Hence, we obtain $v$ is differentiable with respect to $\eta$ and
$D_0v(x,\eta)=r(v,\eta)D_0v(x,\eta)+f(x,\eta)$ hold.
Thus, we have
\begin{align*}
 (D_0v)(x,\eta)&=(1-r(v(x,\eta),\eta))^{-1}f(v(x,\eta),\eta).
\end{align*}
By a similar calculation, we obtain that
$v$ is differentiable with respect to
$x$ and there exist 
\begin{align*}
\tilde{r}\in 
C^{\infty,\infty}_{b,H_0}\Bigl(B_{\ep}(a)\times B_{k,\ep}^{\Omega_0}(0),
\mathcal{L}(\RR^d,\RR^d)\Bigr),
\quad
\tilde{f}\in C^{\infty,\infty}_{b,H_0}
\left(B_{\ep}(a)\times B_{k,\ep}^{\Omega_0}(0),\LL(\RR^d,\RR^d)\right),
\end{align*}
such that
\begin{align*}
(D_xv)(x,\eta)&=(1-\tilde{r}(v(x,\eta),\eta))^{-1}\tilde{f}(v(x,\eta),\eta).
\end{align*}
By these identities, we can conclude that
$v\in C^{1,\infty}_{b,H_0}(B_{\ep}(a)\times B_{k,\ep}^{\Omega_0})$.
Since $Y(1,e,k+\eta+\varphi_x(\eta))=x$ $(\eta\in B_{k,\ep}^{\Omega_0}(0))$,
$D_{0}\varPhi_{k}(x,\eta)(H_0)\subset T_{\varPhi_{k}(x,\eta)}S_x$ holds.

We prove (3).
$Y(1,e,w)=x$ is obvious.
$\eta=P(k)(w-k)$ is also obvious because
$\varphi_x(\eta)\in H_0^{\perp}$.

(4) is also obvious by the definition and (1).

We consider (5).
Let $w\in U_{\delta}^{\Omega}(k)\cap S_a$.
Let $\eta'=P(k)(w-k)$.
Note that
\begin{align*}
 \eta'=w-k-N(k)(w-k)=w-k-\int_0^{\cdot}
Ad(Y(s,e,k))\left(\int_0^1Ad(Y(u,e,k))d(w-k)_u\right).
\end{align*}
To estimate this,
we consider 
$I_t=\int_0^tf(Y(s,e,k))d(w-k)_s$, where
$f\in C^2_b(M(n,\CC),\LL(\RR^d,\RR^d))$.
For simplicity, we write $Y_s=Y(s,e,k)$.
Let $\Xi_{s,t}=f(Y_s)(w-k)_{s,t}+(Df)(Y_s)[Y_s\ep_i]\ep_jC(k^i,(w-k)^j)_{s,t}$.
Then, we have
$I_{s,t}=\lim_{|\Delta|\to 0}\sum_{i}\Xi_{t_{i-1},t_i}$,
where $\Delta=\{s=t_0<\cdots<t_N=t\}$ is a partition.
We have there exists $C>0$ such that
\begin{align*}
 \sup_{s,u,t}\frac{|(\delta \Xi)_{s,u,t}|}{|t-s|^{3\alpha}}
\le C\left(\|Y\|_{\alpha}^2\|w-k\|_{\alpha}+\|R^Y\|_{2\alpha}\|w-k\|_{\alpha}
+\|Y\|_{\alpha}\|C(k,w-k)\|_{2\alpha}\right),
\end{align*}
where $R^Y_{s,t}=Y_t-Y_s-Y_sk_{s,t}$ and
$|R^Y_{s,t}|\le C\|\bar{k}\|_{\alpha}(t-s)^{2\alpha}$ hold.
Hence
\begin{align*}
 |I_1|&\le C\left(\|w-k\|_{\alpha}+\|C(k,w-k)\|_{2\alpha}\right)\\
&\qquad+
C\left(\|Y\|_{\alpha}^2\|w-k\|_{\alpha}+\|R^Y\|_{2\alpha}\|w-k\|_{\alpha}
+\|Y\|_{\alpha}\|C(k,w-k)\|_{2\alpha}\right).
\end{align*}
This implies there exists a polynomial 
growth increasing positive function
$C$ on $\RR^+$ such that
\begin{align*}
 \|N(k)(w-k)\|_{2\alpha}\le C(\|\bar{k}\|_{\alpha})\delta.
\end{align*}
Hence, we see that 
\begin{align*}
 \|\bar{\eta'}\|_{\alpha}\le \delta +C(\|\bar{k}\|_{\alpha})\delta,\qquad
\max(\|C(\eta',k)\|_{2\alpha},\|C(k,\eta')\|_{2\alpha})
\le \delta+C(\|\bar{k}\|_{\alpha})\delta.
\end{align*}
By this, for sufficiently small $\delta$ which just depends on
$\|k\|_H$ and $\ep$, $\eta'\in B_{k,\ep}^{\Omega_0}(0)$ holds.
Taking (4) into account, this implies 
$w=k+\eta'+\varphi_a(\eta')\in V_{\ep}^{S_a}(k)$
which completes the proof of (5).
We next consider (6).
Note that
\begin{align*}
& A(0,v(a,0))=\phi'(a),\quad 
|A(\eta,v(a,\eta)-A(0,v(a,0))|=
O(\|\bar{\eta}\|_{\alpha}),\quad
B(0,v)=0,\\
&\qquad \quad |v(a,\eta)|=O(\|\bar{\eta}\|_{\alpha}), \quad
|B(\eta,v)|\le C\|\bar{\eta}\|_{\alpha}
\left(\|\bar{\eta}\|_{\alpha}+|v|\right).
\end{align*}
The first and the third identities can be checked by the definition.
The second, fourth estimates follow from (\ref{estimate of rough integral}).
The fifth estimate follows from (\ref{quadratic form})
and (\ref{estimate of iterated integral}).
We now prove (6).
(\ref{Dxv}) follows from the first and fifth identity in the above
and (\ref{derivative of Y3}) and (\ref{quadratic form}).
By combining the fourth and fifth estimates, we can get better estimate 
$|v(a,\eta)|=O(\|\bar{\eta}\|_{\alpha}^2)$ which proves (\ref{v}).
(\ref{D_0v}) follows from these inequalities and the 
explicit calculation
of $D_0v(a,\eta)$.
\end{proof}

\begin{rem}\label{local coordinate}
(1)
Suppose $Y(\cdot,e,k)$ is a geodesic,
that is, $k_t=t\xi$ for a certain $\xi\in\g$.
In this case, $P(k)k=0$ and hence $\eta=P(k)w$ holds.
In our study, we consider elements $\bw$ in a 
small neighborhood of $k$ in the topology of $\mathscr{C}^{\alpha}$.
If $a\ne e$, then $k\ne 0$ which implies $\bw$ is not small.
The smallness corresponds to the smallness of $\bar{\eta}=\overline{P(k)w}$.
Note that $\bw$ itself is not small but the projection is small.

\noindent
(2) 
The set $V_{\ep}^{S_a}(k)$ is a neighborhood of $k$ in $S_a$
and $B_{k,\ep}^{\Omega_0}(0)$
is an $H_0$-open subset in $\Omega_0$
and they are isomorphism by the mappings:
\begin{align*}
&\varPhi_{k,a} : B_{k,\ep}^{\Omega_0}(0)\ni \eta
\mapsto w=k+\eta+\varphi_a(\eta)\in V_{\ep}^{S_a}(k),\\
& \varPsi_{k,a} : V_{\ep}^{S_a}(k)\ni w\mapsto 
\eta=P(k)(w-k)\in B_{k,\ep}^{\Omega_0}(0).
\end{align*}

\noindent
(3) The property Lemma~\ref{local chart} (6) 
shows that this coordinate has
similar property to the normal coordinate system in Riemannian geometry.
This will make some calculations easy.
\end{rem}

In Lemma~\ref{local chart}, we give an local coordinate neighborhood
$V_{\ep}^{S_a}(k)$ of $k\in S_a^H$.
However, this does not imply that $S_a$ can be covered by 
$V_{\ep}^{S_a}(k)$
$(k\in S_a^H)$.
The following covering lemma is not used in this paper but
we include it for its interest.
In this proof, the property of $w\in \Omega$ in Theorem~\ref{lift of w} (ii)
is essentially important.

\begin{lem}\label{local chart2}
Let $w\in S_a$ and $w^N$ be the dyadic polygonal approximation of $w$
as in $(\ref{dyadic})$ and set $a_N=Y(1,e,w^N)$.
There exists $N(w)\in \NN$ such that 
the following statement hold for all $N\ge N(w)$.
\begin{enumerate}
 \item[$(1)$] 
There exists a unique $v_N\in \g$ such that
$\tw^N:=w^N+Q_{w^N}v_N\in S_a$.
\item[$(2)$] There exists $\ep_N(w)$ such that $\lim_{N\to\infty}\ep_N(w)=0$
and 
$\ep_N(w)<\ep(\|\overline{\tw^N}\|_{\alpha})$ 
holds, where $\ep(\|\overline{k}\|_{\alpha})$ 
is the positive number defined in
Lemma~$\ref{local chart}$.
Hence
$V_{\ep_N(w)}^{S_a}(\tw^N)$ in Lemma~$\ref{local chart}$
can be defined and furthermore $w\in V_{\ep_N(w)}^{S_a}(\tw^N)$ holds.
That is, 
$\teta^N(w)=P(\tw^N)(w-\tw^N)\in B_{\tw^N,\ep_N(w)}^{\Omega_0}$
and
$w=\tw^N+\teta^N(w)+Q_{\tw^N}v_a(\teta^N(w))$ hold.
\end{enumerate}
\end{lem}

\begin{proof}
 (1) This corresponds to the case where $x=a$ and $a=a_n$ and $\eta=0$ in 
the equation (\ref{x-a}) and (\ref{Mxeta}).
Since $\eta=0$,
to reduce the equation (\ref{x-a}) to (\ref{Mxeta}), we need only to consider
the case $a-a_n$ is sufficiently small and small $v$.
By the continuity of the solution in the topology of $\mathscr{C}^{\alpha}$,
this is possible for sufficiently large $N$.
This shows the contraction property of $M_{a,0}$ and
the unique existence of $v_N$
in $B_{\ep_0}(0)$.
Also note that there exists a $C>0$ which is independent of $N$ such that
$|v_N|\le C|a-a_N|$ holds.

\noindent
(2) 
First note that $\sup_N\|\overline{\tilde{w}^N}\|_{\alpha}<\infty$.
It suffices to show
\begin{align*}
 \|\overline{\tilde{\eta}^N}\|_{\alpha},\quad 
\|C(\tilde{\eta}^N,\tilde{w}^N)\|_{2\alpha},\quad
\|C(\tilde{w}^N, \tilde{\eta}^N)\|_{2\alpha}
\end{align*}
is sufficiently small by Lemma~\ref{local chart} (1).
We have
\begin{align*}
 \tilde{\eta}^N={w^N}^{\perp}-Q_{w^N}v_N+N(\tilde{w}^N)
({w^N}^{\perp}-Q_{w^N}v_N).
\end{align*}
It suffices to prove that
\begin{align*}
 \lim_{N\to\infty}\int_0^1Ad(Y(u,e,\tilde{w}^N))d{w^N}^{\perp}(u)=0.
\end{align*}
The proof of this is similar to the proof of 
Lemma~\ref{approximation of rough integral}.
We omit the proof.

\end{proof}

We now state our change of variable formula.
This kind of disintegration of Gaussian
measures appeared in 
Bismut's work~\cite{bismut} on short time asymptotics of heat kernels.

\begin{lem}[Change of variable formula]\label{change of variable}
Let $V_{\ep}(k)$ and $V_{\ep}^{S_x}(k)$
be the sets defined in Theorem~$\ref{local chart}$
and we assume $\ep$ there is sufficiently small.
\begin{enumerate}
 \item[$(1)$] For any smooth cylindrical functions $f$ on $W$ 
and smooth functions $g$
on $G$, we have
\begin{align}
& \int_{V_{\ep}(k)}f(w)g(Y(1,e,w))d\mu_{\la}(w)\nonumber\\
&\quad=
\int_{B_{\ep}(a)}g(x)\Biggl\{
\int_{B_{k,\ep}^{\Omega_0}(0)}f\left(\varPhi_k(x,\eta)\right)
\left(\frac{\la}{2\pi}\right)^{d/2}
\exp\left(-\frac{\la}{2}d(x,a)^2\right)
\nonumber\\
&\qquad\qquad\qquad\qquad\qquad\qquad |\det D_x\log(xa^{-1})|G(x,\eta,k)
 \exp\Bigl(-\la F(x,\eta,k)\Bigr)d\mu_{\la,W_0}(\eta)\Biggr\}
dx\label{change of variable0}
\end{align}
where $\mu_{\la}(\eta)$ denotes the Wiener measure on $W_0$,
$dx$ is the volume element on $G$.
Also $D_xv(x,\eta)$ and $D_0 v(x,\eta)$ denote the derivative
with respect to $x$ and $\eta$ respectively and
\begin{align}
G(x,\eta,k)&=\det\Bigg(I_H-P_{H^{\perp}_0}+
Q(k)\left(
(D_xv)(x,\eta)
[
(R_a)_{\ast}(D\exp)(\log(xa^{-1})Q(k)^{\ast}
]\right)\nonumber\\
&\qquad\qquad\qquad\qquad +
Q(k)(D_0 v)\left(x,\eta)\right)\Bigg)\\
 F(x,\eta,k)&=
(\eta,k)_H+(\varphi_x(\eta),k)_H+\frac{1}{2}|k|_H^2
+\frac{1}{2}|\varphi_x(\eta)|_H^2.\label{Fxetak}
\end{align}
\item[$(2)$] The image measure of 
$
 \exp\left(-\la F(a,\eta,k)\right)
G(a,\eta,k)
\left(\frac{\la}{2\pi}\right)^{d/2}
\mu_{\la}(d\eta)
$
by the mapping\\
 $\varPhi_{k,a} :B^{\Omega_0}_{k,\ep}(0)\to V_{\ep}^{S_a}(k)$
coincides with the measure $\delta_a(Y(1,e,w))\mu_{\la}(dw)$.
That is, for any $f\in \FC(W)$
\begin{align}
& \int_{V_{\ep}^{S_a}(k)}f(w)\delta_a(Y(1,e,w))\mu_{\la}(dw)\nonumber\\
& \qquad\qquad=\int_{B^{\Omega_0}_{k,\ep}(0)}
f(\varPhi_{k,a}(\eta))
 \exp\left(-\la F(a,\eta,k)\right)
G(a,\eta,k)
\left(\frac{\la}{2\pi}\right)^{d/2}
d\mu_{\la,W_0}(\eta).\label{change of variable1}
\end{align}
Here $G(a,\eta,k)$ can be calculated as
\begin{align*}
 G(a,\eta,k)&=\det\Bigg(I_H-P_{H^{\perp}_0}+
Q(k)\left(
(D_xv)(a,\eta)
[
(R_a)_{\ast}Q(k)^{\ast}
]\right)
+
Q(k)\left((D_0 v)(a,\eta)\right)\Bigg).
\end{align*}
Conversely, the image measure of $\delta_a(Y(1,e,w))\mu_{\la}(dw)$ by
the mapping $w\in (\in V_{\ep}^{S_a}(k)) \mapsto 
\varPsi_{k,a}(w)=
P(k)(w-k)\in B_{k,\ep}^{\Omega_0}(0)
$
is equal to 
$
 \exp\left(-\la F(a,\eta,k)\right)
G(a,\eta,k)
\left(\frac{\la}{2\pi}\right)^{d/2}1_{B_{k,\ep}^{\Omega_0}(0)}(\eta)
\mu_{\la}(d\eta).
$
\end{enumerate}
\end{lem}

\begin{rem}
(1) 
In (\ref{change of variable0}), two determinants appear.
The determinant for the linear operator on $H$ is well-defined because 
the image of $P_{H_0^{\perp}}$ and $Q(k)$ is finite dimensional operator,
and hence the perturbed operator is of trace class.
$\det D_x\log(xa^{-1})$ is usual determinant of the linear
mapping from $T_xG$ to $\g=T_eG$.

\noindent
(2) In the following proof, 
we use the notation 
$B_{\ep}^{H_0^{\perp}}(0)=\{h\in H_0^{\perp}~|~\|h\|_H<\ep\}$
and set $Q(k)^{\ast}h=\int_0^1Ad(Y(s,e,k))\dot{h}(s)ds$ $(h\in H)$.
Note that $Q(k)^{\ast}Q(k)v=v$ and $Q(k)^{\ast}$ is the adjoint operator
of $Q(k) :\g\to H_0^{\perp}\subset H$ 
and $P_{H_0^{\perp}}=Q(k)Q(k)^{\ast}$ holds.
\end{rem}

\begin{proof}
(1) We consider the following mapping;
\[
 \begin{array}{ccccccc}
 B_{\ep}^{H_0^{\perp}}(0)\times B_{k,\ep}^{\Omega_0}(0)&
\stackrel{Q(k)^{\ast}\times I}{\longrightarrow} & 
\log B_{\ep}(e)\times  B_{k,\ep}^{\Omega_0}(0) &
\stackrel{R_a(\exp(\cdot))\times I}\longrightarrow & 
B_{\ep}(a)\times B_{k,\ep}^{\Omega_0}(0) &
\stackrel{\varPhi_k}{\longrightarrow} &
V_{\ep}(k)
\\
 \rotatebox{90}{$\in$} & & \rotatebox{90}{$\in$} & & \rotatebox{90}{$\in$}
& & \rotatebox{90}{$\in$}
\\
 (\eta^{\perp},\eta)& \longmapsto & (Q(k)^{\ast}\eta^{\perp},\eta)=(v,\eta)
&&  \longmapsto (e^va,\eta)=(x,\eta)  & \longmapsto
& \varPhi_k(x,\eta)
\end{array}
\]
Let 
\begin{align*}
& T=(Q(k)^{\ast}\times I)^{-1}
\circ (R_a(\exp(\cdot))\times I)^{-1}\, : \,
B_{\ep}(a)\times B_{k,\ep}^{\Omega_0}(0)\to 
B_{\ep}^{H_0^{\perp}}(0)\times B_{k,\ep}^{\Omega_0}(0),\\
& S=\varPhi_k
\circ (R_a(\exp(\cdot))\times I)\circ (Q(k)^{\ast}\times I)\, :\,
B_{\ep}^{H_0^{\perp}}(0)\times B_{k,\ep}^{\Omega_0}(0)\to V_{\ep}(k).
\end{align*}
Then
\begin{align*}
 (\eta^\perp,\eta)&=T(x,\eta)
=(Q(k)(\log(xa^{-1}),\eta),\\
S(\eta^{\perp},\eta)&=k+\eta+Q(k)
v\Bigl(\exp(Q(k)^{\ast}\eta^{\perp})a,\eta\Bigr),\\
(S\circ T)(x,\eta)&=\varPhi_k(x,\eta).
\end{align*}

$S$ is a mapping from $B_{\ep}^{H_0^{\perp}}(0)\times B_{k,\ep}^{\Omega_0}(0)$
to $V_{\ep}(k)$
and 
\begin{align*}
 (DS)(\eta^{\perp},\eta)&=
I_H-P_{H^{\perp}_0}\nonumber\\
&\quad+
Q(k)\left(
(D_x v)(\exp(Q(k)^{\ast}\eta^{\perp})a,\eta)
[
(R_a)_{\ast}(D\exp)(Q(k)^{\ast}\eta^{\perp})Q(k)^{\ast}
]\right)\nonumber\\
&\quad+
Q(k)(D_0 v)\left(\exp(Q(k)^{\ast}\eta^{\perp})a,\eta)\right).
\end{align*}
Since the image of the linear operator $P_{H_0^{\perp}}$, 
$Q(k)$ is finite dimensional 
vector spaces,
the determinant $\det (DS)(\eta^{\perp},\eta)$ can be defined
and $DS(\eta^{\perp},\eta)$ is bijective linear map
if $\ep$ is sufficiently small.
Hence, we can apply the change of variable formula (\cite{kusuoka1982})
to the map $S$ and we obtain
 \begin{align}
& \int_{V_{\ep}(k)}f(w)g(Y(1,e,w))d\mu_{\la}(w)\nonumber\\
&=
\iint_{B_{\ep}^{H_0^{\perp}}(0)\times B_{k,\ep}^{\Omega_0}(0)}
f\left(S(\eta^{\perp},\eta)\right)
g(\exp(Q(k)^{\ast}\eta^{\perp})a)\nonumber\\
&\qquad    
\det((DS)(\eta^{\perp},\eta))
\exp\left(-\la F(\exp(Q(k)^{\ast}\eta^{\perp})a,\eta,k)
\right)
d\mu_{0,\la}(\eta)\left(\frac{\la}{2\pi}\right)^{d/2}
\exp\Bigl(-\frac{\la}{2}|\eta^{\perp}|_{H}^2\Bigr)d\eta^{\perp},
\label{Vk to eta etaperp}
\end{align}
Note that we do not need Carleman-Fredholm's determinant in this case.
We next use the finite dimensional change of variable formula
with respect to $\eta^{\perp}=Q(k)(\log(xa^{-1}))$.
We obtain,
\begin{align}
& \text{RHS of (\ref{Vk to eta etaperp})}\nonumber\\
&\quad=
\int_{B_{\ep}(a)}g(x)\Biggl\{
\int_{B_{k,\ep}^{\Omega_0}(0)}f\left(\varPhi_k(x,\eta)\right)
\left(\frac{\la}{2\pi}\right)^{d/2}
\exp\Bigl(-\frac{\la}{2}d(x,a)^2\Bigr)
\nonumber\\
&\qquad \det\Bigg(I_H-P_{H^{\perp}_0}+
Q(k)\left(
(D_x v)(x,\eta)
[
(R_a)_{\ast}(D\exp)(\log(xa^{-1})Q(k)^{\ast}
]\right)
+
Q(k)(D_0 v)\left(x,\eta)\right)\Bigg)\nonumber\\
&\qquad \quad |\det D_x\log(xa^{-1})|
 \exp\left(-\la F(x,\eta,k)\right)d\mu_{\la,W_0}(\eta)\Biggr\}
dx\label{change of variable2}
\end{align}
which completes the proof of (1).

\noindent
(2)
On the other hand, we have
\begin{align}
 \int_{V_{\ep}(k)}f(w)g(Y(1,e,w))d\mu_{\la}(w)&=
\int_{B_{\ep}(a)}g(x)\left(\int_{V_{\ep}^{S_x}(k)}
f(w)\delta_x(Y(1,e,w))d\mu_{\la}(w)\right)dx.\label{change of variable3}
\end{align}
By (\ref{change of variable2}) and (\ref{change of variable3}), for any 
$x\in B_{\ep}(a)$, we obtain
\begin{align*}
&\int_{V_{\ep}^{S_x}(k)}
f(w)\delta_x(Y(1,e,w))d\mu_{\la}(w)\nonumber\\
&=
\int_{B_{k,\ep}^{\Omega_0}(0)}f\left(\varPhi_{k,x}(\eta)\right)
\left(\frac{\la}{2\pi}\right)^{d/2}
\exp\Bigl(-\frac{\la}{2}d(x,a)^2\Bigr)
\nonumber\\
&\qquad \det\Bigg(I_H-P_{H^{\perp}_0}+
Q(k)\left(
(D_x v)(x,\eta)
[
(R_a)_{\ast}(D\exp)(\log(xa^{-1})Q(k)^{\ast}
]\right)
+
Q(k)(D_0 v)\left(x,\eta)\right)\Bigg)\nonumber\\
&\qquad \quad |\det D_x\log(xa^{-1})|
 \exp\left(-\la F(x,\eta,k)\right)d\mu_{\la,W_0}(\eta).
\end{align*}
Considering the case $x=a$ and noting that
$(D\exp)(0)=I_{T_eG}$ and
$\left(D_x\log (xa^{-1})\right)(R_x)_{\ast}|_{x=a}=I_{T_e(G)}$, we get
\begin{align*}
&\int_{V_{\ep}^{S_a}(k)}f(w)\delta_a(Y(1,e,w))d\mu_{\la}(w)\nonumber\\
&=
\int_{B_{k,\ep}^{\Omega_0}(0)}f\left(\varPhi_{k,a}(\eta)\right)
\det\Bigg(I_H-P_{H^{\perp}_0}+
Q(k)\left(
(D_x v)(a,\eta)
[
(R_a)_{\ast}Q(k)^{\ast}
]\right)
+
Q(k)(D_0 v)\left(a,\eta)\right)\Bigg)\nonumber\\
&\qquad 
 \exp\left(-\la F(a,\eta,k)\right)\left(\frac{\la}{2\pi}\right)^{d/2}
d\mu_{\la,W_0}(\eta)
\end{align*}
which implies the desired result.
\end{proof}

We have obtained two bijective measurable maps which are inverse to each other,
$\varPhi_{k,a} : B_{\ep}^{\Omega_0}(0)\to V_{\ep}^{S_a}(k)$ and
$\varPsi_{k,a} : V_{\ep}^{S_a}(k)\to B_{k,\ep}^{\Omega_0}(0)$ as in 
Remark~\ref{local coordinate}(2).
Moreover the image measures of these maps, 
$(\varPhi_{k,a})_{\ast}\mu_{\la,W_0}$ and
$(\varPsi_{k,a})_{\ast}\mu_{\la, a}$ are equivalent to 
$\mu_{\la, a}$ and 
$\mu_{\la,W_0}$
respectively.
Therefore the pull back of functions
$(\varPhi_{k,a}^{\ast}f)(\eta)=f(\varPhi_{k,a}(\eta))$ and
$(\varPsi_{k,a}^{\ast}g)(w)=g(\varPsi_{k,a}(w))$
can be defined for $f$ on $V_{\ep}^{S_a}(k)$
which is a $\mu_{\la,a}$-a.s.\,defined function and
$g$ on $B_{k,\ep}^{\Omega_0}(0)$ which is a
$\mu_{\la,W_0}$-a.s.\,defined function. 
We see that the following chain rule holds.

\begin{lem}
 Let $f\in \FC(W)$ and $g\in \FC(W_0)$.
Then we have,
\begin{align*}
 (D_0(\varPhi_{k,a}^{\ast}f)(\eta))\left[\cdot\right]
&=(D_{S_a}f)(\varPhi_{k,a}(\eta))[\left(I_{H_0}+D\varphi_a(\eta)\right)\cdot]
\in L^2(B_{k,\ep}^{\Omega_0}\to H_0^{\ast},\mu_{\la,W_0}),\\
 (D_{S_a}(\varPsi_{k,a}^{\ast}g))(w)\left[\,\cdot\,\right]
&=(D_0g)(\varPsi_{k,a}(w))\left[P(k)P(w)\cdot\right]
\in L^2(V_{\ep}^{S_a}(k)\to (T_{w}S_a)^{\ast}, \mu_{\la,a}).
\end{align*}
\end{lem}

\begin{proof}
Recall that
\begin{align*}
&\varPhi_{k,a} : B_{k,\ep}^{\Omega_0}\ni \eta
\mapsto w=k+\eta+\varphi_a(\eta)\in V_{\ep}^{S_a}(k),\\
& \varPsi_{k,a} : V_{\ep}^{S_a}(k)\ni w\mapsto \eta=P(k)(w-k)\in B_{k,\ep}^{\Omega_0}.
\end{align*}
For $h\in H_0$
\begin{align*}
 \left(D_0\left(f(\varPhi_{k,a}(\eta))\right),h\right)_{H_0}&=
(Df)(\varPhi_{k,a}a(\eta))\left[h+D_{0}\varphi_a(\eta)[h]\right]\\
&=(Df)(\varPhi_{k,a}(\eta))\left[P(\varPhi_a(\eta))D_{0}\varPhi_a(\eta)h\right].
\end{align*}
For $h\in H$, we have
\begin{align*}
 \left(D\left(g(P(k)(w-k))\right),h\right)_H&=(D_0g)(P(k)(w-k))[P(k)h]\\
&=(D_0g)(\varPsi_{k,a}(w))[P(k)h],
\end{align*}
which completes the proof.
\end{proof}

The following lemma gives the explicit form of the Dirichlet form on $S_a$
in the local coordinate $(\varPsi_{k,a}, V_{\ep}^{S_a}(k))$.

\begin{lem}\label{change of variable dirichlet form} 
Let $k\in S_a^H$ and $f,g \in \FC(W)$.
Set $\tilde{f}(\eta)=(\varPhi_{k,a}^{\ast}f)(\eta)$ and 
$\tilde{g}(\eta)=(\varPhi^{\ast}_{k,a}g)(\eta))$.
Then
\begin{align}
&  \int_{V_{\ep}^{S_a}(k)}\left((\DS f)(w),(\DS g)(w)\right)
d\mu_{\la,a}(w)\nonumber\\
&=
\int_{B_{k,\ep}^{\Omega_0}(0)}\left((I_{H_0}+\tilde{K}_{k}(\eta))
D_0\tilde{f}(\eta),
D_0\tilde{g}(\eta)\right)
\exp\left(-\la F(a,\eta,k)\right)G(a,\eta,k)
c_{\la,a}d\mu_{\la,W_0}(\eta),\label{change of variable4}
 \end{align}
where
$c_{\la,a}=\left(\frac{\la}{2\pi}\right)^{d/2}p(\la^{-1},e,a)^{-1}$
and
$
\tilde{K}_k(\eta)=
\left(D_0\varPhi_{k,a}(\eta)^{\ast}D_0\varPhi_{k,a}(\eta)\right)^{-1}-I_{H_0}
\in \LL(H_0,H_0).$
Explicitly, we have
\begin{align}
 \tilde{K}_k(\eta)&=
\sum_{n=1}^{\infty}\Bigl(-\left(
(D_0v_a)^{\ast}(\eta)(D_0v_a)(\eta)\right)\Bigr)^n
\label{Neumann series}
\end{align}
and we have an estimate
\begin{align}
 |\tilde{K}_k(\eta)|_{H.S.}=O(\|\bar{\eta}\|_{\alpha}^2),
\label{estimate of tildeK}
\end{align}
where $|\cdot|_{H.S.}$ denotes the Hilbert-Schmidt norm.
\end{lem}

\begin{proof}
Using (\ref{change of variable1}), we have
 \begin{align*}
&\int_{V_{\ep}^{S_a}(k)}
\left((\DS f)(w),(\DS g)(w)\right)d\mu_{\la,a}(w)
\nonumber\\
&=\int_{V_{\ep}^{S_a}(k)}
\left((\DS f)(w),(\DS g)(w)\right)\delta_a(Y(1,e,w))
p(\la^{-1},e,a)^{-1}
d\mu_{\la}(w)\nonumber\\
&=\int_{B_{k,\ep}^{\Omega_0}(0)}
\Bigl((\DS f)(\varPhi_{k,a}(\eta)),(\DS g)(\varPhi_{k,a}(\eta))\Bigr)
\exp\left(-\la F(a,\eta,k)\right)
G(a,\eta,k)c_{\la,a}d\mu_{\la,W_0}(\eta).
 \end{align*}
For any $h\in H_0$, we have
\begin{align*}
 \left(D_0(f(\varPhi_{k,a}(\eta)),h\right)=(\DS f)(\varPhi_{k,a}(\eta))
\left[D_0\varPhi_{k,a}(\eta)[h]\right]
=\Bigl((D_0\varPhi_{k,a}(\eta))^{\ast}(\DS f)(\varPhi_{k,a}(\eta)),h\Bigr).
\end{align*}
This shows 
\[
  (\DS f)(\varPhi_{k,a}a(\eta))=
\left((D_0\varPhi_{k,a}(\eta))^{\ast}\right)^{-1}(D_0\tilde{f})(\eta)
=\left((D_0\varPhi_{k,a}(\eta))^{-1}\right)^{\ast}(D_0\tilde{f})(\eta),
\]
which proves (\ref{change of variable4}).
Note that
\begin{align*}
 \left(D_0\varPhi_{k,a}(\eta)\right)^{\ast}D_0\varPhi_{k,a}(\eta)
&=I_{H_0}+D_0v_a(\eta)^{\ast}Q_k^{\ast}Q_kD_0v_a(\eta)+
D_0v_a(\eta)^{\ast}Q_k^{\ast}P_{H_0}+
P_{H_0}Q_kD_0v_a(\eta)\\
&=I_{H_0}+D_0v_a(\eta)^{\ast}D_0v_a(\eta).
\end{align*}
Here we have used that $Q_k^{\ast}Q_k=id_{\g}$ and
$Q_k^{\ast}P_{H_0}$ and $P_{H_0}Q_k$ are 0.
Hence (\ref{Neumann series}) follows from Neumann series expansion
since
$|(D_0\varphi_a)(\eta)|<1$.
Finally, (\ref{estimate of tildeK}) follows from
(\ref{D_0v}) and 
$\|C(k,\eta)\|_{2\alpha}\le C\|k\|_H\|\eta\|_{\alpha}$.
This completes the proof.
\end{proof}

Here, we summarize important notion in the analysis in the local coordinate.
Recall that $U_k : T_kS_a\to H_{0,0}(=\{h\in H~|~h(1)=0\})$ 
is the unitary operator
defined by 
$(U_kh)(t)=\int_0^tAd(Y(s,e,k))dh_s$ as in 
Example~\ref{example of f} 
and we denote $c_{\la,a}=\left(\frac{\la}{2\pi}\right)^{d/2}p(\la^{-1},e,a)^{-1}$.

\begin{dfi}
Suppose $Y(\cdot,e,k)$ is a geodesic on $G$ between $e$ and $a$.
That is, $k$ can be written $k_t=t \xi$ $(\xi\in \g)$ and $Y(1,e,k)=a$.

\noindent
(1) We consider the following measures on $W_0$.
\begin{align*}
d\mu_{\la,T,W_0}(\eta)&=\exp
\left(
-\frac{\la}{2}
:\left(U_k^{-1}T_{\xi} U_k\eta,\eta\right):_{\mu_{\la,W_0}}
\right)
d\mu_{\la,W_0}(\eta),\\
d\hat{\mu}_{\la,T,W_0}(\eta)&=
c_{\la,a}\exp\left(-\frac{\la}{2}\|k\|_H^2\right)
d\mu_{\la,T,W_0}(\eta),\\
d\tilde{\mu}_{\la,a,W_0}(\eta)&=
\exp\big(-\la F(a,\eta,k)\big)G(a,\eta, k)1_{B_{k,\ep}^{\Omega_0}(0)}(\eta)
c_{\la,a}d\mu_{\la,W_0}(\eta),
\end{align*}
where $T_{\xi}$ is the Hilbert-Schmidt operator on $H_{0,0}$
defined in Lemma~$\ref{hessian of E}$.
That is, we consider the case $T=U_k^{-1}T_{\xi}U_k$ in Theorem~$\ref{cons}$
and the quadratic functional 
$:\left(U_k^{-1}T_{\xi} U_k\eta,\eta\right):_{\mu_{\la,W_0}}$
is defined therein.

\noindent
(2)
Let $-L_{\la,W_0}, -L_{\la,T,W_0}$, $-\hat{L}_{\la,T,W_0}$ be
the generators of the following Dirichlet forms respectively,
\begin{align*}
& \E_{\la,W_0}(f,g)=\int_{W_0}(D_0f(\eta),D_0g(\eta))_{H_0}d\mu_{\la,W_0}(\eta) \quad
\text{in $L^2(W_0,d\mu_{\la})$},\\
&
\E_{\la,T, W_0}(f,g)=\int_{W_0}(D_0f(\eta),D_0g(\eta))_{H_0}d\mu_{\la,T,W_0}(\eta)
\quad \text{in $L^2(W_0,d\mu_{\la,T,W_0})$},
\\
&
\hat{\E}_{\la, T, W_0}(f,g)=\int_{W_0}(D_0f(\eta),D_0g(\eta))_{H_0}
d\hat{\mu}_{\la,T,W_0}(\eta)
\quad \text{in $L^2(W_0,d\hat{\mu}_{\la,T,W_0})$}.
\end{align*}
The above Dirichlet forms are defined as the closure of
the closable forms defined in $\FC(W_0)$.
Also, we consider a bilinear form 
\begin{align*}
\tilde{\E}_{\la,a,B_{k,\ep}^{\Omega_0}(0)}(f,g)=
\int_{B_{k,\ep}^{\Omega_0}(0)}\left((I_{H_0}+\tilde{K}_{k}(\eta))D_0f(\eta),
D_0g(\eta)\right)d\tilde{\mu}_{\la,a,W_0}(\eta), 
\quad f,g\in \DD^{\infty}(W_0,\mu_{\la,W_0}).
\end{align*}
\end{dfi}

\begin{rem}\label{remark on dirichlet form}
(1) For $f\in \FC(W)$, let $\tilde{f}(\eta)=(\varPhi^{\ast}_{k,a}f)(\eta)$.
Then
\begin{align*}
 \int_{B_{\ep,k}^{\Omega_0}(0)}\tilde{f}(\eta)d\tilde{\mu}_{\la,a,W_0}(\eta)
=\int_{V_{\ep}^{S_a}(k)}f(w)d\mu_{\la,a}(w).
\end{align*}
That is, 
$\tilde{\mu}_{\la,a,W_0}=(\varPsi_{k,a})_{\ast}\mu_{\la,a}|_{V_{\ep}^{S_a}(k)}$
and we will see that the measure 
$\hat{\mu}_{\la,T,W_0}$ approximates 
$\tilde{\mu}_{\la,a,W_0}$ very well when $\ep$ is small in
Lemma~\ref{expansion of F and G}.
Since $\hat{\mu}_{\la,T,W_0}$ is the constant multiple of $\mu_{\la,T,W_0}$,
$-\hat{L}_{\la,T,W_0}$ is essentially the same operator as
$-L_{\la,T,W_0}$.

\noindent
(2) For $f,g \in \FC(W)$, let 
$\tilde{f}(\eta)=(\varPhi_{k,a}^{\ast}f)(\eta)$ and
$\tilde{g}(\eta)=(\varPhi_{k,a}^{\ast}g)(\eta)$.
As we shown in Lemma~\ref{change of variable dirichlet form},
it holds that
\begin{align*}
\tilde{\E}_{\la,a,B_{\ep,k}^{\Omega_0}(0)}(\tilde{f},\tilde{g})
=\int_{V_{\ep}^{S_a}(k)}(D_{S_a}f(w),D_{S_a}g(w)) d\mu_{\la,a}(w)
\end{align*}
\end{rem}

We now give the more explicit form of the Wick product of
$:\left(U_k^{-1}T_{\xi} U_k\eta,\eta\right):_{\mu_{\la,W_0}}$.

\begin{lem}\label{quadratic functional}
 We have
\begin{align}
 :\left(U_k^{-1}T_{\xi} U_k\eta,\eta\right):_{\mu_{\la,W_0}}&=
\int_0^1\left([U_k\eta_s,\xi],dU_k\eta_s\right),\quad \text{$\mu_{\la,W_0}$
-a.s. $\eta$}.
\end{align}
and the integral on R.H.S. is a rough integral.
\end{lem}

\begin{proof}
 Let $f_{1,0}(t)=1$ $(0\le t\le 1)$.
For $0\le l\le 2^{n-2}-1$ and $n\ge 2$,
\begin{align*}
f_{n,l}(t)=
\begin{cases}
\sqrt{2^{n-2}} & \frac{2l}{2^{n-1}}\le
t<\frac{2l+1}{2^{n-1}}\\
-\sqrt{2^{n-2}} & \frac{2l+1}{2^{n-1}}\le t<\frac{2l+2}{2^{n-1}}\\
0 & \text{otherwise}
\end{cases}
\end{align*}
Set $g_{n,l}(t)=\int_0^tf_{n,l}(s)ds$.
Then $\{e_{n,l,i}(t)~|~1\le i\le d, n\ge 1, 0\le l\le \max(2^{n-2}-1,0)\}$
and $\{e_{n,l,i}(t)~|~1\le i\le d, n\ge 2, 0\le l\le \max(2^{n-2}-1,0)\}$
are complete orthonormal systems of $H$ and $H_{0,0}$ respectively,
where $e_{n,l,i}(t)=g_{n,l}(t)\ep_i$.
Therefore, 
$\{\tilde{e}_{n,l,i}(t)~|~1\le i\le d, n\ge 2, 0\le l\le \max(2^{n-2}-1,0)\}$
are c.o.n.s. of $H_0=T_kS_a$, where
$\tilde{e}_{n,l,i}=U_k^{-1}e_{n,l,i}$.
Let 
$P_N : H_{0,0}\to \text{linear span of}\,
\{e_{n,l,i}(t)~|~1\le i\le d, 2\le n\le N, 0\le l\le 2^{n-2}-1\}$ 
and
$\tilde{P}_N : H_{0}\to \text{linear span of}\,
\{\tilde{e}_{n,l,i}(t)~|~1\le i\le d, 2\le n\le N, 0\le l\le 2^{n-2}-1\}$ 
be the projection operators.
Note that $P_Nh$ is a dyadic polygonal approximation of $h$.
Let
\begin{align*}
 F_{N,k}(\eta)=\left(U_k^{-1}T_{\xi}U_k\tilde{P}_N\eta,
\tilde{P}_N\eta\right)_{H_0}-
\frac{1}{\la}\tr \tilde{P}_NU_k^{-1}T_{\xi}U_k\tilde{P}_N,
\end{align*}
where $\tr$ denotes the trace of the operator defined on $H_0$.
Then by the definition, we have
\begin{align*}
 :\left(U_k^{-1}T_{\xi} U_k\eta,\eta\right):_{\mu_{\la,W_0}}=
\lim_{N\to\infty}F_{N,k}(\eta), \qquad \mu_{\la,0} a.s.
\end{align*}
We calculate $F_{N,k}$.
By Lemma~6.12 in \cite{a2007} and $\sum_i([\xi,\ep_i],\ep_i)=0$, we have
\begin{align*}
\lim_{N\to\infty} \tr \tilde{P}_NU_k^{-1}T_{\xi}U_k\tilde{P}_N&=
\lim_{N\to\infty} \sum_{1\le i\le d, 1\le n\le N, 0\le l\le (2^{N-2}-1)\vee 0}
\left(T_{\xi}e_{n,l,i},e_{n,l,i}\right)_{H_{0,0}}\\
&=-\sum_{i=1}^d\int_0^1\left([\ep_i s,\xi],\ep_i\right)ds=0.
\end{align*}
Noting $U_k\tilde{P}_N\eta=P_NU_k\eta$ and applying Lemma~\ref{w_f-roughpath},
we get
\begin{align*}
\lim_{N\to\infty}\left(U_k^{-1}T_{\xi}U_k\tilde{P}_N\eta,\tilde{P}_N\eta\right)_{H_0}
&=\lim_{N\to\infty}\left(T_{\xi}P_NU_k\eta,P_NU_k\eta\right)_{H_{0,0}}\\
&=\lim_{N\to\infty}\int_0^1\left([P_NU_k\eta(s),\xi], d(P_NU_k\eta)_s\right)\\
&=\int_0^1\left([U_k\eta_s,\xi],dU_k\eta_s\right),
\end{align*}
which completes the proof.
\end{proof}

\begin{rem}
For a bounded linear operator $T$ on $H_0$,
we define
\begin{align*}
 \tilde{\tr}\, T=\lim_{N\to\infty}\tilde{P}_NT\tilde{P}_N
\end{align*}
if the R.H.S. converges, where
$\tilde{P}_N$ is the projection operator defined in the above proof.
This kind of trace appeared in \cite{a2007}.
\end{rem}

In our problem, the calculation of $\tilde{\tr}$ is reduced to the following
calculation of the sum of the infinite series.

\begin{lem}\label{trace formula}
 Let $\Gamma : \RR^d\times \RR^d\to \RR^m$ be a bilinear mapping
and let $f=(f(s)), g=(g(s))\in C([0,1],\LL(\RR^d,\RR^d))$.
Consider a continuous bilinear form
$T(h_1,h_2)=\int_0^1\Gamma(\int_0^sf(u)dh_1(u),g(s)dh_2(s))$
~$(h_1, h_2\in H_0)$.
We denote by $\{\tilde{e}_i\}_{i=1}^{\infty}$ the complete orthonormal system
of $H_0$ such that
$\{\tilde{e}_i\}_{i=1}^{a_N}$ spans the linear space of the image 
of $\tilde{P}_N$
which is defined in the proof of Lemma~$\ref{quadratic functional}$.
Note that $e_i=U_k\tilde{e}_i$ $(i\ge 1)$ is a complete orthonormal system of $H_{0,0}$.
We have
\begin{align*}
& \lim_{N\to\infty}
\sum_{i=1}^{a_N}
T(\tilde{e}_{i},\tilde{e}_{i})\nn\\
&=
\int_0^1\sum_{1\le p,q\le d}\Gamma(\ep_p,\ep_q)
\left(f(s)U_k^{-1}(s), g(s)U_k^{-1}(s)\right)
ds-\sum_{i=1}^d\int_0^1
\Gamma(\int_0^sf(u)U_k^{-1}(u)\ep_idu,g(s)U_k^{-1}(s)\ep_i)ds.
\end{align*}
\end{lem}

\begin{proof}
 The proof of this is similar to the calculation of 
Lemma 6.12 in \cite{a2007}.
We omit the proof.
\end{proof}

\begin{lem}\label{expansion of F and G}
\begin{enumerate}
\item[$(1)$] In the case where $k\in S_a^H$, it holds that
\begin{align*}
 G(a,\eta,k)&=1+G_{R}(a,\eta,k),
\end{align*}
where $|G_R(a,\eta,k)|=O(\|\bar{\eta}\|_{\alpha})$ and
$|D_{0}G_R(a,\eta,k)|=O(1)$.
 \item[$(2)$]
We consider the case where
$Y(t,e,k)=e^{t\xi}$ is a geodesic on $G$ between $e$ and $a$.
\begin{itemize}
 \item[{\rm (i)}]  We have 
\[
 F(a,\eta,k)=(\varphi_a(\eta),k)+\frac{1}{2}\|k\|_H^2
+\frac{1}{2}\|\varphi_a(\eta)\|_H^2
\]
and 
\begin{align}
 F(a,\eta,k)&=\frac{1}{2}\|k\|_H^2+
\frac{1}{2}:\left(U_k^{-1}
T_{\xi} U_k\eta,\eta\right):_{\mu_{\la,W_0}}
+F_R(a,\eta,k).\label{expansion of F}
\end{align}
where $|F_R(a,\eta,k)|=O(\|\bar{\eta}\|_{\alpha}^3)$ and 
$|D_{0}F_R(a,\eta,k)|=O(\|\bar{\eta}\|_{\alpha}^2)$.
\item[{\rm (ii)}] We have
\begin{align}
& \tilde{\mu}_{\la,a,W_0}(d\eta)\nn\\
&=
c_{\la,a}e^{-\frac{\la}{2}\|k\|_H^2}
\exp\left(-\la F_{R}(a,\eta,k)\right)
(1+G_R(a,\eta,k))1_{B_{k,\ep}^{\Omega_0}}(\eta)
\mu_{\la,T,W_0}(d\eta),\label{tildemu and hatmu near 0}
\end{align}
where $c_{\la,a}=\left(\frac{\la}{2\pi}\right)^{d/2}p(\la^{-1},e,a)^{-1}$.
\end{itemize}
\end{enumerate}
\end{lem}

\begin{proof}
(1) This is a consequence
of Lemma~\ref{local chart} (6).

\noindent
 (2) First note that $(\eta,k)=0$ which follows from 
$N(k)k=k$ and $\eta\in T_kS_a^{\perp}$.
For this, see Remark~\ref{remark on tangent space} (3).
This implies the formula of $F(a,\eta,k)$.
The proof of (\ref{expansion of F}) is as follows.
Note that $\varphi_a(\eta)=\int_0^{\cdot}Ad(e^{-s\xi})v_a(\eta)ds$
and
\begin{align}
 v_a(\eta)&=-A(\eta,v_a(\eta))^{-1}B(\eta,v_{a}(\eta))\nn\\
&=-\left(\phi'(a)+O(\|\bar{\eta}\|_{\alpha})\right)^{-1}
\left(\phi'(a)+O(\|\bar{\eta}\|_{\alpha})\right)
\left(\frac{1}{2}\int_0^1\left[U_k\eta(s),U_kd\eta(s)\right]
+O(\|\bar{\eta}\|_{\alpha}^3)\right)\nn\\
&=-\frac{1}{2}\int_0^1\left[U_k\eta(s),U_kd\eta(s)\right]
+O(\|\bar{\eta}\|_{\alpha}^3),\label{expression of v}
\end{align}
and hence,
\begin{align*}
 (\varphi_a(\eta),k)_H=\frac{1}{2}\int_0^1
\left(\left[U_k\eta_s,\xi\right],dU_k\eta_s\right)+
O(\|\bar{\eta}\|_{\alpha}^3).
\end{align*}
In the above, the term $O(\|\bar{\eta}\|_{\alpha}^3)$
can be written by three times iterated rough integral against $\eta$
and this implies the order of $O(\|\bar{\eta}\|_{\alpha}^3)$.
Also $\|\varphi_a(\eta)\|_H^2=|v_a(\eta)|^2
=O(\|\bar{\eta}\|_{\alpha}^4)$ and hence this is a negligible term.
In the above, by the explicit calculation,
we see that the term $O(\|\bar{\eta}\|_{\alpha}^k)$
can be written by $k$-times iterated rough integral against $\eta$.
This implies such terms are order of $O(\|\bar{\eta}\|_{\alpha}^k)$
and the $H$-derivative with respect to $\eta$ is of
order $O(\|\bar{\eta}\|_{\alpha}^{k-1})$.
Combining Lemma~\ref{quadratic functional}, we complete the proof.
\end{proof}

We give the explicit form of
$L_{\la,T,W_0}f$ and $\widetilde{\LS f}$,
where $\widetilde{L_{\la,S_a}f}(\eta)=
\varPhi_{k,a}^{\ast}(L_{\la,S_a}f)(\eta)$.

\begin{lem}\label{expression of L in local chart}
Suppose $Y(t,e,k)$ is a geodesic.
\begin{enumerate}
 \item[$(1)$] Let $f\in \DD^{\infty}(W_0,\mu_{\la,W_0})$.
Then
 we have
\begin{align*}
 -L_{\la,T,W_0}f(\eta)&=-L_{\la,W_0}f
+\la\left(U_k^{-1}TU_k\eta,D_0f(\eta)\right)
\end{align*}
\item[$(2)$] 
For $f\in \DD^{\infty}(W_0,\mu_{\la,W_0})$, define
\begin{align*}
 -\tilde{L}_{\la,S_a}f(\eta)&=
-L_{\la,T,W_0}f-\tr\tilde{K}_k(D_0^2f(\eta))+
\left(B_{\la, k},D_0f(\eta)\right)\nonumber\\
&\quad+\la(D_0F_R(a,\eta,k),D_0f(\eta))
-\left(D_0\log G(a,\eta,k),D_0f(\eta)\right),
\end{align*}
on $B^{\Omega_0}_{k,\ep}(0)$.
Here
\begin{align*}
B_{\la, k}
:=
\la\tilde{K}_k(\eta)\Bigl(\eta+U_k^{-1}TU_k\eta\Bigr)
-\tilde{\tr}\, D_0\tilde{K}_k
-\tilde{K}_k\Bigl(\la D_0F_R(a,\eta,k)-D_0\log G(a,\eta,k)\Bigr),
\end{align*}
and 
\begin{align}
 \tilde{K}_k(\eta)\eta
&:=\lim_{N\to \infty}
\tilde{K}_k(\eta)\tilde{P}_N\eta,
\label{tildeK rough integral}\\
\tilde{\tr}\, D_0\tilde{K}_k&:=\lim_{N\to\infty}
\tr\,D_0\tilde{K}_k[\tilde{P}_N\cdot].
\label{tildetrace}
\end{align}
For $f\in \DD^{\infty}(W,\mu_{\la,W})$ 
satisfying $f=0$ $\mu_{\la,a}$-a.s. on 
$\left(V_{\ep'}(k)^{S_a}\right)^{\complement}$ for some $\ep'<\ep$,
let 
\[
 \tilde{f}(\eta)=
\begin{cases}
 (\varPhi_{k,a}^{\ast}f)(\eta) & \eta\in B_{k,\ep}^{\Omega_0}(0),\\
0 & \text{otherwise}
\end{cases}
\]
Then $\tilde{f}\in \DD^{\infty}(W_0,\mu_{\la,W_0})$ and
it holds that
\begin{align*}
\tilde{L}_{\la,S_a}\tilde{f}(\eta)&=
\widetilde{L_{\la,S_a}f}(\eta).
\end{align*}
Moreover, we have
\begin{align}
 \left|\tilde{\tr}\, D_0\tilde{K}_k\right|_{H_0}
&=O(\|\bar{\eta}\|_{\alpha}),
\label{estimate of tildeK2}\\
\left|\tilde{K}_k(\eta)\Bigl(\eta+U_k^{-1}TU_k\eta\Bigr)\right|
_{H_0}
&=O(\|\bar{\eta}\|_{\alpha}^3).\label{estimate of tildeK3}
\end{align}
Let
\begin{align*}
 \tilde{B}_{\la,k}=B_{\la,k}+\la D_0F_R(a,\eta,k)-
D_0\log G(a,\eta,k).
\end{align*}
Then, for $\delta$ satisfying $\frac{2}{3}<\delta<1$, we have
\begin{align}
 |\tilde{B}_{\la,k}|_{H_0}=O(\la^{1-\delta}), \quad 
\text{if $\|\bar{\eta}\|_{\alpha}=O(\la^{-\frac{\delta}{2}})$}.
\label{estimate of tildeB}
\end{align}
\end{enumerate}

\end{lem}

\begin{proof}
 (1) This is a immediate consequence of the integration by parts formula.

\noindent
(2) 
Let $\{\tilde{e}_i\}_{i=1}^{\infty}$ be the complete orthonormal system
of $H_0$ which is defined in Lemma~\ref{trace formula}.
Let $g\in \FC(W)$.
By Lemma~\ref{change of variable dirichlet form} 
and Lemma~\ref{expansion of F and G} (2),
we have
\begin{align*}
& \int_{V_{\ep}^{S_a}(k)}-(\LS f)(w)g(w)d\mu_{\la,a}(w)\\
&=
\int_{V_{\ep}^{S_a}(k)}\left(D_{S_a}f,D_{S_a}g\right)d\mu_{\la,a}(w)\\
&=\int_{B_{k,\ep}^{\Omega_0}(0)}\left((I_{H_0}+\tilde{K}_k(\eta))D_0\tilde{f},
D_0\tilde{g}\right)\exp\left(-\la F_R(a,\eta,k)\right)(1+G_R(a,\eta,k))
e^{-\frac{\la}{2}\|k\|_H^2}c_{\la,a}
d\mu_{\la,T,W_0}\\
&=
\int_{B_{k,\ep}^{\Omega_0}(0)}\left(D_0\tilde{f},
D_0\tilde{g}\right)\exp\left(-\la F_R(a,\eta,k)\right)(1+G_R(a,\eta,k))
e^{-\frac{\la}{2}\|k\|_H^2}c_{\la,a}
d\mu_{\la,T,W_0}\\
&\quad +
\int_{B_{k,\ep}^{\Omega_0}(0)}\left(\tilde{K}_k(\eta))D_0\tilde{f},
D_0\tilde{g}\right)\exp\left(-\la F_R(a,\eta,k)\right)(1+G_R(a,\eta,k))
e^{-\frac{\la}{2}\|k\|_H^2}c_{\la,a}
d\mu_{\la,T,W_0}\\
&=:I_1+I_2.
\end{align*}
For $I_1$, using integration by parts formula, we have
\begin{align*}
 I_1&=\int_{B_{k,\ep}^{\Omega_0}(0)}\left(D_0\tilde{f},
\la D_0F_R(a,\eta,k)-(1+G_{R}(a,\eta,k))^{-1}D_0G_R(a,\eta,k)\right)\tilde{g}(\eta)
d\tilde{\mu}_{\la,a,W_0}(w)\\
&\quad 
-\int_{B_{k,\ep}^{\Omega_0}(0)}L_{\la,T,W_0}\tilde{f}(\eta)
\tilde{g}(\eta)d\tilde{\mu}_{\la,a,W_0}(w).
\end{align*}
Let
\begin{align*}
 I_{2,N}&=
\sum_{i=1}^{a_N}
\int_{B_{k,\ep}^{\Omega_0}(0)}
(D_0\tilde{f}(\eta),\tilde{K}_k(\eta)\tilde{e}_i)(D_0\tilde{g}(\eta),\tilde{e}_i)
\exp\left(-\la F_R(a,\eta,k)\right)(1+G_R(a,\eta,k))\\
&\qquad\qquad\qquad \times e^{-\frac{\la}{2}\|k\|_H^2}c_{\la,a}
d\mu_{\la,T,W_0}.
\end{align*}
Then $\lim_{N\to\infty}I_{2,N}=I_2$.
Using the integration by parts formula, we have
$I_{2,N}=I_{2,1,N}+I_{2,2,N}+I_{2,3,N}+I_{2,4,N}$,
where
\begin{align*}
 I_{2,1,N}&=-\sum_{i=1}^{a_N}
\int_{B_{k,\ep}^{\Omega_0}(0)}
D_{\tilde{e}_i}
\left(
(D_0\tilde{f}(\eta),\tilde{K}_k\tilde{e}_i)
\right)
\tilde{g}(\eta)
d\tilde{\mu}_{\la,a,W_0},\\
I_{2,2,N}&=-\sum_{i=1}^{a_N}
\int_{B_{k,\ep}^{\Omega_0}(0)}
\left(
(D_0\tilde{f}(\eta),\tilde{K}_k\tilde{e}_i)\right)\\
&\qquad\qquad\quad
\times
\left(
\tilde{e}_i,\la D_0F_R(a,\eta,k)-
D_0\log G(a,\eta,k)
\right)
\tilde{g}(\eta)
d\tilde{\mu}_{\la,a,W_0},\\
I_{2,3,N}&=\sum_{i=1}^{a_N}
\int_{B_{k,\ep}^{\Omega_0}(0)}
\left(
(D_0\tilde{f}(\eta),\tilde{K}_k\tilde{e}_i)\right)
\left(\tilde{e}_i,\la U_k^{-1}TU_k\eta\right)
\tilde{g}(\eta)d\tilde{\mu}_{\la,a,W_0},\\
I_{2,4,N}&=\sum_{i=1}^{a_N}\int_{B_{k,\ep}^{\Omega_0}(0)}
\la\left((D_0\tilde{f})(\eta),\tilde{K}_k\tilde{e}_i\right)(\tilde{e}_i,\eta)
\tilde{g}(\eta)d\tilde{\mu}_{\la,a,W_0}
\end{align*}
Note that
\begin{align*}
& \lim_{N\to\infty}I_{2,2,N}\\
&=
-\int_{B_{k,\ep}^{\Omega_0}(0)}
\left(
D_0\tilde{f}(\eta),\tilde{K}_k
\Bigl(
\la D_0F_R(a,\eta,k)
-D_0\log G(a,\eta,k)
\Bigr)
\right)\tilde{g}(\eta)
d\tilde{\mu}_{\la,a,W_0},\\
&\lim_{N\to\infty}I_{2,3,N}=
\int_{B_{k,\ep}^{\Omega_0}(0)}
\left(D_0\tilde{f}(\eta),\tilde{K}_kU_k^{-1}TU_k\eta\right)\tilde{g}(\eta)
d\tilde{\mu}_{\la,a,W_0}.
\end{align*}
Hence $\lim_{N\to\infty}(I_{2,1,N}+I_{2,4,N})$ also converges.
We calculate the limit.
We have
\begin{align*}
 I_{2,1,N}&=-\sum_{i=1}^{a_N}\int_{B_{k,\ep}^{\Omega_0}(0)}
(D_0^2\tilde{f})(\eta)
[\tilde{K}_k \tilde{e}_i,\tilde{e}_i]\tilde{g}(\eta)d\tilde{\mu}_{\la,a,W_0}\\
&\qquad -\int_{B_{k,\ep}^{\Omega_0}(0)}\left(D_0\tilde{f}(\eta),
\sum_{i=1}^{a_N}
D_0\tilde{K}_k[\tilde{e}_i]\tilde{e}_i\right)
\tilde{g}(\eta)d\tilde{\mu}_{\la,a,W_0}\\
&=:I_{2,1,1,N}+I_{2,1,2,N}.
\end{align*}
Since $D_0^2\tilde{f}$ and $\tilde{K}_k$ are Hilbert-Schmidt operators, we have
\begin{align*}
 \lim_{N\to\infty}I_{2,1,1,N}=
-\int_{B_{k,\ep}^{\Omega_0}(0)}
\tr\left((D_0^2\tilde{f})(\eta)\tilde{K}_k(\eta)\right)\tilde{g}(\eta)
d\tilde{\mu}_{\la,a,W_0}.
\end{align*}
Also by Lemma~\ref{trace formula}, we have
\begin{align*}
 \lim_{N\to\infty}I_{2,1,2,N}=-\int_{B_{k,\ep}^{\Omega_0}(0)}
\left(D_0\tilde{f}(\eta),\tilde{\tr} D_0\tilde{K}_k(\eta)\right)
\tilde{g}(\eta)d\tilde{\mu}_{\la,a,W_0}.
\end{align*}
Concerning $I_{2,4,N}$,
we have
\begin{align*}
 \sum_{i=1}^{a_N}\tilde{K}_k\tilde{e}_i(\tilde{e_i},\eta)&=
\sum_{i=1}^{a_N}\tilde{K}_kU_{k}^{-1}e_i(U_k^{-1}e_i,\eta)
=\tilde{K}_kU_{k}^{-1}P_NU_k\eta\to \tilde{K}_k\eta
\quad \text{in $H_0$ as $N\to\infty$},
\end{align*}
where we have used Lemma~\ref{rough integral wf}
and Lemma~\ref{approximation of rough integral}.
Hence,
\begin{align*}
 \lim_{N\to\infty}I_{2,4,N}=\int_{B_{k,\ep}^{\Omega_0}(0)}
\la((D_0\tilde{f})(\eta),\tilde{K}_k\eta)
\tilde{g}(\eta)d\tilde{\mu}_{\la,a,W_0}.
\end{align*}
Consequently, the definitions in (\ref{tildeK rough integral}) 
and (\ref{tildetrace})
are well-defined and we obtain
 \begin{align*}
\int_{V_{\ep}^{S_a}(k)}(\LS f)(w)g(w)d\mu_{\la,a}(w)
&=\int_{B_{k,\ep}^{\Omega_0}(0)}\tilde{L}_{\la,S_a}\tilde{f}(\eta)\tilde{g}(\eta)
d\tilde{\mu}_{\la,a,W_0}(\eta).
\end{align*}
By the denseness of $\FC(W)$ in $L^2(S_a,\mu_{\la,a})$, we obtain
$\widetilde{\LS f}(\eta)=\tilde{L}_{\la,S_a}\tilde{f}(\eta)$.

(\ref{estimate of tildeK2}) follows from the definition of
$\tilde{K}_k$ and Lemma~\ref{trace formula} and
(\ref{estimate of tildeK3}) follows from Lemma~\ref{approximation of rough integral}
and estimate of rough integrals.
(\ref{estimate of tildeB}) follows from the results obtained in the above.
\end{proof}

\section{Cut-off functions and approximate eigenfunctions}
\label{cut-off and eigenfunction}

\begin{lem}[Cut-off function]\label{cut-off function}
Let $0<\delta<1$.
For $w\in\Omega$, let
\begin{align*}
& \Gamma_{m,\theta,1}(w)=
\iint_{0<s<t<1}\frac{|w_{s,t}|^{4m}}{|t-s|^{2+2m\theta}}dsdt,\\
&\Gamma^{i,j}_{m,\theta,2}(w)=\iint_{0<s<t<1}
\frac{|w^{i,j}_{s,t}|^{2m}}{|t-s|^{2+2m\theta}}
dsdt,\\
& \bar{\Gamma}_{m,\theta}(w)=\left\{\Gamma_{m,\theta,1}(w)
+\sum_{1\le i\ne j\le d}\Gamma^{i,j}_{m,\theta,2}(w)\right\}^{1/(4m)},\\
&\chi_{m,\theta,\la,\delta}(w)=\chi\Bigl(\la^{2m\delta}
\bar{\Gamma}_{m,\theta}(w)^{4m}\Bigr),
\end{align*}
where $\chi$ is a $C^{\infty}$ function such that 
$\chi(u)=1$ for $u\in [-1,1]$ and $\chi(u)=0$ for $|u|\ge 2$
and $m$ is a positive integer satisfying that $m(1-\theta)>1$.
\begin{enumerate}
\item[$(1)$] For any $0<\alpha<1/2$, 
there exists $(m,\theta)$ with $m(1-\theta)>1$
and a constant $C_{\alpha,m,\theta}>0$ such that
\begin{align*}
 \|\bw\|_{\alpha}\le C_{\alpha,m,\theta}\bar{\Gamma}_{m,\theta}(w).
\end{align*}
Converse estimate holds for any $(m,\theta)$ by taking $\alpha$ to be 
sufficiently close to $1/2$.
 \item[$(2)$] $\Gamma_{m,\theta,1}, \Gamma^{i,j}_{m,\theta,2}\in 
C^{\infty}_{b,H}(B_{\ep}^{\Omega}(0))\cap \DD^{\infty}(\RR)$ for any $\ep>0$ and
it holds that
\begin{align}
 \|D\chi_{m,\theta,\la,\delta}(w)\|_H\le C\la^{\delta/2},
\qquad
\|D^2\chi_{m,\theta,\la,\delta}(w)\|_{H\otimes H}\le C\la^\delta,
\label{desirable cut-off}
\end{align}
where $C$ is a positive constant independent of $w$ and $\la$.
\item[$(3)$] We consider the function $\Gamma_{m,\theta,1}(\eta)$, 
$\Gamma^{i,j}_{m,\theta,2}(\eta)$
and $\chi_{m,\theta,\la,\delta}(\eta)$ 
in the case where $\eta\in \Omega_0\subset W_0=\overline{T_kS_a}^W$.
Then similar estimates to $(\ref{desirable cut-off})$ hold and
$
 \left|(L_{\la,W_0}\,\chi_{m,\theta,\la,\delta})(\eta)\right|
\le C \la,
$
where $C$ is a positive constant independent of $\eta$ 
and $\la$.
\item[$(4)$] For sufficiently small $\kappa>0$, it holds that
$e^{\kappa \bar{\Gamma}_{m,\theta}(\eta)^2}\in L^1(\Omega_0,\mu_{1,W_0})$.
\end{enumerate}
\end{lem}

\begin{rem}
Similar kind of cut-off function appeared in \cite{a2007}.
In it, we used the notations
\begin{align*}
 \|\overline{w}^1\|_{B,4m,\theta/2}^{4m}=\Gamma_{m,\theta,1}(w),\quad
\|\overline{w}^2\|_{B,2m,\theta}^{2m}
=\iint_{0<s<t<1}\frac{|\overline{w}^2_{s,t}|^{2m}}{(t-s)^{2+2m\theta}}dsdt.
\end{align*}
The latter function is comparable to $\sum_{i,j}\Gamma^{i,j}_{m,\theta,2}(w)$.
\end{rem}

\begin{proof}
We refer the readers to \cite{a2007} for (1) and (4).

\noindent
(2)
First note that $\Gamma_{m,\theta,1}(w)+\sum_{i,j}\Gamma^{i,j}_{m,\theta,2}(w)<\infty$ 
for all $w\in \Omega$.
This follows from that $|w_{s,t}|\le C_{\alpha,w}|t-s|^{\alpha}$ and
$|w^{i,j}_{s,t}|\le C_{\alpha,w}|t-s|^{2\alpha}$ for any
$0<\alpha<1/2$ and the assumption
$m(1-\theta)>1$.
Also it is easy to prove the first statement.
So, we prove (\ref{desirable cut-off}).
By the definition, 
 we have, for $h,k\in H$
\begin{align*}
 D_h\Gamma_{m,\theta,1}(w)&=\iint_{0<s<t<1}
\frac{4m(w_{s,t},w_{s,t})^{2m-1}(w_{s,t},h_{s,t})}{(t-s)^{2+2m\theta}}dsdt,\\
D_{k}D_{h}\Gamma_{m,\theta,1}(w)&=
\iint_{0<s<t<1}\frac{4m(4m-2)(w_{s,t},w_{s,t})^{2m-2}(w_{s,t},h_{s,t})
(w_{s,t},k_{s,t})}{(t-s)^{2+2m\theta}}dsdt\\
&\quad
+\iint_{0<s<t<1}\frac{4m(w_{s,t},w_{s,t})^{2m-1}(k_{s,t},h_{s,t})}
{(t-s)^{2+2m\theta}}dsdt.
\end{align*}
Note that we need to consider the double integrals
in the domain $\{(s,t)~|~0<\ep<s<t<1\}$ first and take the limit
$\ep\downarrow 0$ in the calculation above.
The same applies to the following.
By the above calculation, we have
\begin{align*}
 D\Gamma_{m,\theta,1}&=\iint_{0<s<t<1}
\frac{4m(w_{s,t},w_{s,t})^{2m-1}\sum_{i=1}^d\psi^i_{s,t}w^i_{s,t}}
{(t-s)^{2+2m\theta}}dsdt,\\
D^2\Gamma_{m,\theta,1}&=
\iint_{0<s<t<1}\frac{4m(4m-2)(w_{s,t},w_{s,t})^{2m-2}
\sum_{i=1}^d\psi^i_{s,t}w^i_{s,t}\odot
\sum_{i=1}^d\psi^i_{s,t}w^i_{s,t}}
{(t-s)^{2+2m\theta}}dsdt\\
&\quad+
\iint_{0<s<t<1}\frac{4m(w_{s,t},w_{s,t})^{2m-1}
\sum_{i=1}^d\psi^i_{s,t}\odot \psi^i_{s,t}}
{(t-s)^{2+2m\theta}}dsdt,
\end{align*}
where $\psi^i_{s,t}(\tau)=\int_0^\tau 1_{[s,t]}(\sigma)d\sigma \ep_i\in H$
and $\psi^i_{s,t}\odot\psi^j_{s',t'}=2^{-1}(\psi^i_{s,t}\otimes\psi^j_{s',t'}
+\psi^j_{s',t'}\otimes\psi^i_{s,t})$ is the symmetric tensor product.
Although $u\mapsto \psi^i_u(=\psi^i_{0,u})\in H$ is $1/2$-H\"older continuous map,
by the property $(\psi^i_{s,t},\psi^j_{s',t'})=0$ if $(s,t)\cap (s',t')=\emptyset$,
the integral $\int_s^tf_ud\psi^i_u$ can be defined by the
Riemann sum limit,
$\lim_{|\Delta|\to 0}\sum_{k}f_{u_{k-1}}\psi^i_{u_{k-1},u_k}$
for continuous function $f$ and
$\|\int_s^tf_ud\psi^i_u\|_H=\|f\|_{L^2([s,t])}$ holds.
Note that
\begin{align}
\|\sum_{i=1}^d\psi^i_{s,t}w^i_{s,t}\|_H&\le |w_{s,t}||t-s|^{1/2},\qquad
\left\|
\sum_{i=1}^d\psi^i_{s,t}\otimes \psi^i_{s,t}
\right\|_{H\otimes H}\le d|t-s|.\label{DwD^2w}
\end{align}
If $m(1-\theta)>1$,
\begin{align}
 \|D\Gamma_{m,\theta,1}(w)\|_H&\le 4m\Gamma_{m,\theta,1}(w)^{(4m-1)/4m}
\left(\iint_{0<s<t<1}\frac{(t-s)^{2m}}{(t-s)^{2+2m\theta}}dsdt\right)^{1/(4m)}
\nonumber\\
&\le
C_{m,\theta}\Gamma_{m,\theta,1}(w)^{(4m-1)/4m}\label{DPhi_m1}
\end{align}
and
\begin{align}
 \|D^2\Gamma_{m,\theta,1}(w)\|_{H\otimes H}&\le
4m(4m-2)d\Gamma_{m,\theta,1}^{(2m-1)/(2m)}
\left(\iint_{0<s<t<1}\frac{(t-s)^{2m}}{(t-s)^{2+2m\theta}}dsdt\right)^{1/(2m)}
\nonumber\\
&\le C_{m,\theta}\Gamma_{m,\theta,1}(w)^{(2m-1)/(2m)}.\label{D^2Phi_m2}
\end{align}
We next consider the term $\Gamma^{i,j}_{m,\theta,2}$.
\begin{align*}
& D_{h}\Gamma^{i,j}_{m,\theta,2}(w)=\iint_{0<s<t<1}
\frac{2m\left(w^{i,j}_{s,t}\right)^{2m-1}
\left(\int_s^{t}w^{i}_{s,u}dh^j_u+
\int_s^th^i_{s,u}dw^j_u\right)}{(t-s)^{2+2m\theta}}dsdt,\\
& D_{k}D_{h}\Gamma^{i,j}_{m,\theta,2}(w)=\iint_{0<s<t<1}
\frac{2m\left(w^{i,j}_{s,t}\right)^{2m-1}
\left(\int_s^{t}k^{i}_{s,u}dh^j_u+
\int_s^th^i_{s,u}dk^j_u\right)}{(t-s)^{2+2m\theta}}dsdt\\
&\quad+\iint_{0<s<t<1}
\frac{2m(2m-1)\left(w^{i,j}_{s,t}\right)^{2m-2}
\left(\int_s^{t}w^{i}_{s,u}dh^j_u+
\int_s^th^i_{s,u}dw^j_u\right)
\left(\int_s^{t}w^{i}_{s,u}dk^j_u+
\int_s^tk^i_{s,u}dw^j_u\right)
}{(t-s)^{2+2m\theta}}dsdt
\end{align*}
Hence
\begin{align*}
& D\Gamma^{i,j}_{m,\theta,2}(w)=\iint_{0<s<t<1}
\frac{2m\left(w^{i,j}_{s,t}\right)^{2m-1}
\left(\int_s^{t}w^{i}_{s,u}d\psi^j_u+
\int_s^tw^j_{u,t}d\psi^i_u\right)}{(t-s)^{2+2m\theta}}dsdt,\\
& D^2\Gamma^{i,j}_{m,\theta,2}(w)\nonumber\\
&=\iint_{0<s<t<1}
\frac{2m\left(w^{i,j}_{s,t}\right)^{2m-1}
\left(
\int_s^t\psi^i_{s,u}\odot d\psi^j_u+
\int_s^t\psi^j_{s,u}\odot d\psi^i_u
\right)}
{(t-s)^{2+2m\theta}}dsdt\nonumber\\
&\quad +
\iint_{0<s<t<1}
\frac{2m(2m-1)\left(w^{i,j}_{s,t}\right)^{2m-2}
\left(\int_s^{t}w^{i}_{s,u}d\psi^j_u+
\int_s^tw^j_{u,t}d\psi^i_u\right)^{\odot 2}}{(t-s)^{2+2m\theta}}dsdt.
\end{align*}
Note that
\begin{align}
 \left\|\int_s^{t}w^{i}_{s,u}d\psi^j_u+
\int_s^tw^j_{u,t}d\psi^i_u\right\|_{H}&\le
\left(\int_s^t(w^i_{s,u})^2du\right)^{1/2}+
\left(\int_s^t(w^j_{u,t})^2du\right)^{1/2}
\le C\|w\|_{\theta/2}|t-s|^{(\theta+1)/2}\nn\\
&\le C \Gamma_{m,\theta,1}(w)^{1/(4m)}|t-s|^{(\theta+1)/2},\label{Dw^{i,j}}
\end{align}
where $\|w\|_{\theta/2}$ is the $\theta/2$-H\"older norm of $w$
and we have used the estimate 
$\|w\|_{\theta/2}\le C\Gamma_{m,\theta,1}(w)^{1/(4m)}$.
Hence we obtain
\begin{align}
 \|D\Gamma^{i,j}_{m,\theta,2}(w)\|_H&\le
C \left(\Gamma^{i,j}_{m,\theta,2}(w)\right)^{(2m-1)/(2m)}
\Gamma_{m,\theta,1}(w)^{1/(4m)}.\label{DPhi^ij_m2}
\end{align}
Next we estimate $D^2\Gamma^{i,j}_{m,\theta,2}$.
Noting
\begin{align}
\left\|\int_s^t\psi^i_{s,u}\odot d\psi^j_u\right\|_{H\otimes H}&\le
C|t-s|,\label{psi^ipsi^j}
\end{align}
we get
\begin{align}
 \|D^2\Gamma^{i,j}_{m,\theta,2}(w)\|_{H\otimes H}&\le
C\left((\Gamma^{i,j}_{m,\theta,2}(w))^{(2m-1)/(2m)}+
(\Gamma^{i,j}_{m,\theta,2})^{(m-1)/m}\Gamma_{m,\theta,1}(w)^{1/(2m)}
\right).\label{D^2Phi^ij_m2}
\end{align}
We now prove (\ref{desirable cut-off}).
By (\ref{DPhi_m1}) and (\ref{DPhi^ij_m2}), we have
\begin{align*}
& \la^{2m\delta}\left(
\|D\Gamma_{m,\theta,1}(w)\|_H+\|D\Gamma^{i,j}_{m,\theta,2}(w)\|_{H}
\right)\nonumber\\
&\le C \la^{2m\delta}
\left(\Gamma_{m,\theta,1}(w)^{(4m-1)/4m}
+\left(\Gamma^{i,j}_{m,\theta,2}(w)\right)^{(2m-1)/(2m)}
\Gamma_{m,\theta,1}(w)^{1/(4m)}\right)\nonumber\\
&\le C \la^{2m\delta}\la^{-2m\delta+\delta/2}
\nonumber\\
&\qquad
\times \left(\left(\la^{2m\delta}\Gamma_{m,\theta,1}(w)\right)^{(4m-1)/4m}
+\left(\la^{2m\delta}\Gamma^{i,j}_{m,\theta,2}(w)\right)^{(2m-1)/(2m)}
\left(\la^{2m\delta}\Gamma_{m,\theta,1}(w)\right)^{1/(4m)}\right)\nonumber\\
&\le C\la^{\delta/2}\qquad 
\text{if $\chi'(\la^{2m\delta}\bar{\Gamma}_{m,\theta}^{4m}(w))\ne 0$.}
\end{align*}
Similarly, by (\ref{D^2Phi_m2}) and (\ref{D^2Phi^ij_m2}),
we obtain
\begin{align*}
 & \la^{2m\delta}\left(
\|D^2\Gamma_{m,\theta,1}(w)\|_{H\otimes H}
+\|D^2\Gamma^{i,j}_{m,\theta,2}(w)|_{H\otimes H}
\right)\nonumber\\
&\le C\la^{2m\delta}\left(
\Gamma_{m,\theta,1}^{(2m-1)/(2m)}+
(\Gamma^{i,j}_{m,\theta,2}(w))^{(2m-1)/(2m)}+
(\Gamma^{i,j}_{m,\theta,2})^{(m-1)/m}\Gamma_{m,\theta,1}(w)^{1/(2m)}
\right)\nonumber\\
&\le C\la^{2m\delta}
\la^{-2m\delta}\la^{\delta}
\Big\{
\left(\la^{2m\delta}\Gamma_{m,\theta,1}\right)^{(2m-1)/(2m)}+
(\la^{2m\delta}\Gamma^{i,j}_{m,\theta,2}(w))^{(2m-1)/(2m)}\nn\\
&\qquad +
(\la^{2m\delta}\Gamma^{i,j}_{m,\theta,2})^{(m-1)/m}
\left(\la^{2m\delta}\Gamma_{m,\theta,1}(w)\right)^{1/(2m)}
\Big\}\nonumber\\
&\le C\la^{\delta}\qquad 
\text{if 
$\chi'(\la^{2m\delta}\bar{\Gamma}_{m,\theta}^{4m}(w))\ne 0$},
\end{align*}
which completes the proof of (2).

\noindent
(3)
The counterpart of (\ref{desirable cut-off}) is trivial because
the direction of the derivative for $\eta$ is a subspace of $H$.
Or we may proceed as follows.
The Cameron-Martin space of $(W_0,\mu_{\la,0})$ is
$H_0=T_kS_a=P(k)H$.
Hence, $\eta^i_{s,t}=(\eta,P(k)\psi^i_{s,t})_{H_0}$ holds,
where $\psi^i_{s,t}(u)=\int_0^u1_{[s,t]}(u)du\ep_i$ and
$(\cdot,\cdot)_{H_0}$ is an Wiener integral.
Also $P(k)\psi^i_{s,t}=\psi^i_{s,t}-N(k)\psi^i_{s,t}$
and $\|N(k)\psi^i_{s,t}\|_H\le C|t-s|$.
By using this result, we can prove similar estimates to
(\ref{desirable cut-off}).
Finally, we prove the estimate of 
$(L_{\la,W_0}\chi_{m,\theta,\la,\delta})(\eta)$.
First, we note the following.
Let $f\in D(L_{\la,W_0})$ and suppose 
$f, |D_0f|_{H_0}, L_{\la,W_0}f\in \cap_{p\ge 1}L^p(W_0,\mu_{\la,W_0})$.
Then for any $n\in \mathbb{N}$, it holds that $f^n\in D(L_{\la,W_0})$ and
$L_{\la,W_0}(f^n)=nf^{n-1}L_{\la,W_0}f+
2(\sum_{k=1}^{n-1}k)|D_0f|_{H_0}^2f^{n-2}$.
We prove this by induction on $n$.
Let $n=2$.
Let $\varphi\in \FC(W_0)$.
Also let $g$ be another function such that $g\in D(L_{\la,W_0})$ and
$g,|D_0g|_H, L_{\la,W_0}g\in \cap_{p\ge 1}L^p(W_0,d\mu_{\la,W_0})$.
Then 
\begin{align*}
& \left(L_{\la,W_0}\varphi, fg\right)_{L^2(W_0,d\mu_{\la,W_0})}\\
&=
-\int_{W_0}(D_0\varphi,D_0f)gd\mu_{\la,W_0}-\int_{W_0}(D_0\varphi,D_0g)f
d\mu_{\la,W_0}\\
&=-\int_{W_0}\left(D_0(\varphi g)-\varphi D_0g,D_0f\right)d\mu_{\la,W_0}
-\int_{W_0}\left(D_0(\varphi f)-\varphi D_0f,D_0g\right)d\mu_{\la,W_0}\\
&=-\int_{W_0}\left\{\left(D_0(\varphi g),D_0f\right)+\left(D_0(\varphi f), 
D_0g\right)\right\}
d\mu_{\la,W_0}+2\int_{W_0}\varphi\left(D_0f,D_0g\right)_{H_0}d\mu_{\la,W_0}\\
&=\int_{W_0}\varphi \left(gL_{\la,W_0}f
+f L_{\la,W_0}g+2(D_0f,D_0g)_{H_0}\right)
d\mu_{\la,W_0}.
\end{align*}
Here we have used that $fg\in \rD(\E_{\la,W_0})$ and
$D_0(f g)=(D_0f) g+f D_0g$.
Noting 
$gL_{\la,W_0}f+fL_{\la,W_0}g+2(D_0f,D_0g)_{H_0}\in L^2(W_0,d\mu_{\la,W_0})$
and the essential selfadjointness of 
$L_{\la,W_0}$ with the domain $\FC(W_0))$ 
in $L^2(W_0,d\mu_{\la,W_0})$,
we obtain $fg\in \rD(L_{\la,W_0})$ and
$L_{\la,W_0}(fg)=gL_{\la,W_0}f+fL_{\la,W_0}g+2(D_0f,D_0g)_{H_0}$.
This result include the case where $n=2$ as a special case.
Next, suppose the case $n\ge 2$.
Then using the induction assumption and the result we have proved, we have
\begin{align*}
 L_{\la,W_0}(f^{n+1})&=L_{\la,W_0}(f^n\cdot f)\\
&=f L_{\la,W_0}(f^n)+f^nL_{\la,W_0}f+2nf^{n-1}|D_0f|_{H_0}^2\\
&=f \Bigl(nf^{n-1}L_{\la,W_0}f+2(\sum_{k=1}^{n-1}k)|D_0f|_{H_0}^2f^{n-2}\Bigr)+
f^nL_{\la,W_0}f+2nf^{n-1}|D_0f|_{H_0}^2\\
&=(n+1)f^nL_{\la,W_0}f+2(\sum_{k=1}^nk)f^{n-1}|D_0f|_{H_0}^2,
\end{align*}
which completes the proof of the induction.
We return to the proof of (3).
First note that $\eta^i_{s,t}$ and 
$\eta^{i,j}_{s,t}-E[\eta^{i,j}_{s,t}]$ 
belong to
the 1st Wiener chaos and the 2nd Wiener chaos respectively
in $L^2(W_0,d\mu_{\la,W_0})$.
By a simple calculation, we have
$
 f^{i,j}_{s,t}=\la E[\eta^{i,j}_{s,t}]
$
is independent of $\la$ and $|f^{i,j}_{s,t}|\le C|t-s|^2$.
Hence $L_{\la,W_0}\eta^i_{s,t}=-\la \eta^i_{s,t}$ and
$L_{\la,W_0}\eta^{i,j}_{s,t}=-2\la\eta^{i,j}_{s,t}-2f^{i,j}_{s,t}$ hold.
Also we have
\begin{align*}
 D_0\eta^i_{s,t}=P(k)\psi^i_{s,t},\quad
D_0\eta^{i,j}_{s,t}=\int_s^t\eta^i_{s,u}d(P(k)\psi^j)_u
+\int_s^t\eta^j_{u,t}d(P(k)\psi^i)_u
\end{align*}
and
\begin{align*}
 \|D_0\eta^i_{s,t}\|_{H_0}\le C|t-s|^{1/2},\quad
\|D_0\eta^{i,j}_{s,t}\|_{H_0}\le C
\Gamma_{m,\theta,1}(\eta)^{1/(4m)}|t-s|^{(\theta+1)/2}.
\end{align*}
Thus, we get
\begin{align*}
 L_{\la,W_0}(\eta^i_{s,t})^{4m}&=-4m\la (\eta^i_{s,t})^{4m}
+(\eta^{i}_{s,t})^{4m-2}O(|t-s|),\\
L_{\la,W_0}(\eta^{i,j}_{s,t})^{2m}&=-4m\la (\eta^{i,j}_{s,t})^{2m}
-4m(\eta^{i,j}_{s,t})^{2m-1} f^{i,j}_{s,t}+
\Gamma_{m,\theta,1}(\eta)^{1/2m}(\eta^{i,j}_{s,t})^{2m-2}
O(|t-s|^{\theta+1}).
\end{align*}
Noting
\begin{align*}
& L_{\la,W_0}\chi_{m,\theta,\la,\delta}(\eta)
=\chi'\Bigl(\la^{2m\delta}\Bigl(\Gamma_{m,\theta,1}(\eta)
+\sum_{i,j=1}^d\Gamma^{i,j}_{m,\theta,2}(\eta)\Bigr)\Bigr)\la^{2m\delta}
L_{\la,W_0}\Bigl(\Gamma_{m,\theta,1}(\eta)
+\sum_{i,j=1}^d\Gamma^{i,j}_{m,\theta,2}(\eta)\Bigr)\nonumber\\
&\quad+\chi''\Bigl(\la^{2m\delta}\Bigl(\Gamma_{m,\theta,1}(\eta)
+\sum_{i,j=1}^d\Gamma^{i,j}_{m,\theta,2}(\eta)\Bigr)\Bigr)\la^{4m\delta}
\Bigl\|D_{0}\Gamma_{m,\theta,1}(\eta)
+\sum_{i,j=1}^dD_{0}\Gamma^{i,j}_{m,\theta,2}(\eta)\Bigr\|_{H_0}^2.
\end{align*}
by a similar calculation to the estimate (\ref{desirable cut-off}),
we arrive at the estimate of $(L_{\la,W_0}\chi_{m,\theta,\la,\delta})(\eta)$.
\end{proof}

We now introduce approximate eigenfunctions of $-L_{\la,\Pea}$ localized in
a neighborhood of geodesic path $e^{t\xi}$ in $\Pea$ $(\xi\in \g)$
when $a$ is not a conjugate point of $e$ along any geodesics.
To this end, we use the development path $(w_t)$ of $(\gamma_t)\in \Pea$.

\begin{lem}[Approximate eigenfunctions]\label{approximate eigenfunctions}
Suppose $a$ is not a conjugate point of $e$ along any geodesics.
Let $Y(t,e,k)=e^{t\xi}$ be a geodesic between $e$ and $a$.
That is, there exists $\xi\in \g$ such that $k_t=t \xi$ $(0\le t\le 1)$.
We apply Theorem~$\ref{cons}$ to the case
$W_0=\overline{T_kS_a}^W$, $T=U_k^{-1}T_{\xi}U_k$ and 
$\{\zeta_i\}=\{\zeta_i(\xi)\}$,
where
$\{\zeta_i(\xi)\}$ are eigenvalues of $T_{\xi}$ as in 
Lemma~$\ref{hessian of E}$ and 
Remark~$\ref{remark hessian of E}$ $(1)$.
Let $\mathbf{e}_{\mathbf{n},T,\la}(\eta)$ 
be an eigenfunction defined in Theorem~$\ref{cons}$.
For $\frac{2}{3}<\delta<1$ and a positive integer $m$ and 
$\alpha<\frac{\theta}{2}<1$ satisfying
$m(1-\theta)>1$, 
using $\chi_{m,\theta,\la,\delta}$ which satisfies
Lemma~$\ref{cut-off function}$ $(1)$,
we define
$\displaystyle{
  \mathbf{\hat{e}}_{\mathbf{n},T,\la}(\eta)
=\mathbf{e}_{\mathbf{n},T,\la}(\eta)
\chi_{m,\theta,\la,\delta}(\eta)c^{-1/2}_{\la,a}e^{\frac{\la}{4}\|k\|_H^2}}
$~$(\eta\in \Omega_0)$, where
we need to take sufficiently large $\la$ such that
$\{\chi_{m,\theta,\la,\delta}\ne 0\}\subset B_{k,\ep(k)}^{\Omega_0}(0)$,
where $B_{k,\ep(k)}^{\Omega_0}(0)$ is defined in Lemma~$\ref{local chart}$.
We set
\begin{align*}
\mathbf{\tilde{e}}_{\mathbf{n},T,\la}(\gamma)&=
(Y^{-1})^{\ast}\left(\varPsi_{k,a}^{\ast}
(\mathbf{\hat{e}}_{\mathbf{n},T,\la})\right)
(\gamma), \quad \gamma\in \Peaomega.
\end{align*}
Then for any $\mathbf{n}, \mathbf{m}$ with $\mathbf{n}\ne \mathbf{m}$, 
$\mathbf{\tilde{e}}_{\mathbf{n},T,\la}\in \rD(L_{\la,\Pea})$ and
we have
\begin{align}
& \|\mathbf{\tilde{e}}_{\mathbf{n},T,\la}\|_{L^2(\Pea,\nu_{\la,a})}=
1+O(\la^{1-(3\delta/2)})+O(\la^{-\delta/2}),\label{n}\\
& \left(\mathbf{\tilde{e}}_{\mathbf{n},T,\la},
\mathbf{\tilde{e}}_{\mathbf{m},T,\la}\right)_{L^2(\Pea,\nu_{\la,a})}=
O(\la^{1-(3\delta/2)})+O(\la^{-\delta/2}),\label{nm}\\
&\|\la^{-1}L_{\la,\Pea}\mathbf{\tilde{e}}_{\mathbf{n},T,\la}
-E_{\mathbf{n}}(\xi)\mathbf{\tilde{e}}
_{\mathbf{n},T,\la}\|_{L^2(\Pea,\nu_{\la,a})}
=O(\la^{\frac{1}{2}-\delta}),\label{approximate eigenvalue}
\end{align}
where
$E_{\mathbf{n}}(\xi)=E_0(\xi)+\sum_{i=1}^{\infty}n_i|\zeta_i(\xi)|$.
\end{lem}

\begin{rem}\label{e on S_a}
(1) Since $\tilde{\mu}_{\la,a,W_0}\approx \hat{\mu}_{\la,T,W_0}=
c_{\la,a}e^{-\frac{\la}{2}\|k\|_H^2}\mu_{\la,T,W_0}$ on the neighborhood of
$0$ in $W_0$,
we need to multiply the eigenfunctions on $L^2(W_0,d\mu_{\la,T,W_0})$ by
the constant 
$c_{\la,a}^{-1/2}e^{\frac{\la}{4}\|k\|_H^2}$ for the normalization.

\noindent
(2) Write $\bce_{\mathbf{n},T,\la}(w)=
\varPsi^{\ast}_{k,a}(\mathbf{\hat{e}}_{\mathbf{n},T,\la})(w)$.
Then these functions are approximate eigenfunctions localized near $k$
in $S_a$.
Also $\mathbf{\tilde{e}}_{\mathbf{n},T,\la}\in \rD(L_{\la,\Pea})$,
(\ref{n}), (\ref{nm}) and (\ref{approximate eigenvalue})
are equivalent to 
$\mathbf{\check{e}}_{\mathbf{n},T,\la}\in \rD(\LS)$
and the following estimates respectively
by Theorem~$\ref{D and Y^-1(D)}$.
\begin{align}
& \|\bce_{\mathbf{n},T,\la}\|_{L^2(S_a,\mu_{\la,a})}=
1+O(\la^{1-(3\delta/2)})+O(\la^{-\delta/2}),\label{S_an}\\
& \left(\bce_{\mathbf{n},T,\la},
\bce_{\mathbf{m},T,\la}\right)_{L^2(S_a,\mu_{\la,a})}=
O(\la^{1-(3\delta/2)})+O(\la^{-\delta/2}),\label{S_anm}\\
&\|\la^{-1}\LS \bce_{\mathbf{n},T,\la}
-E_{\mathbf{n}}(\xi)\bce_{\mathbf{n},T,\la}\|_{L^2(S_a,\mu_{\la,a})}
=O(\la^{\frac{1}{2}-\delta}),\label{S_aapproximate eigenvalue}
\end{align}

\noindent

$(3)$ We considered domains $\D\subset \Pea$ in Theorem~\ref{main theorem}.
Suppose $l_{\xi}=Y(\cdot,e,k)\in \D$.
Then for sufficiently large $\la$,
by Remark~\ref{remark on the generator} (1), we see that
$\mathbf{\tilde{e}}_{\mathbf{n},T,\la}\in \rD(L_{\la,\D})$,
$\mathbf{\check{e}}_{\mathbf{n},T,\la}\in \rD(L_{\la,\YD})$
and
\begin{align*}
 L_{\la,\D}\,\mathbf{\tilde{e}}_{\mathbf{n},T,\la}=
\left(L_{\la,\Pea}\,\mathbf{\tilde{e}}_{\mathbf{n},T,\la}\right)|_{\D},\quad
L_{\la,\YD}\,\mathbf{\check{e}}_{\mathbf{n},T,\la}
=\left(\LS\, \mathbf{\check{e}}_{\mathbf{n},T,\la}\right)|_{\YD}
\end{align*}
 hold.
Note that we need normalize them to obtain unit vectors
in $L^2(\D,\nu_{\la,\D})$, $L^2(\YD,\mu_{\la,\YD})$.
\end{rem}

\begin{proof}
We apply Lemma~\ref{change of variable} (2) to estimate these quantities.
First note that $w\in V_{\ep}^{S_a}(k)$ is equivalent to
$w\in S_a$ and $\|\bar{\eta}\|_{\alpha}<\ep$, 
where $\eta=\varPsi_{k,a}w=P(k)(w-k)$.
Hence if $\la$ is sufficiently large, then for $w\in S_a$ satisfying that
$\chi_{m,\theta,\la,\delta}(P(k)(w-k))\ne 0$, 
$\|\bar{\eta}\|_{\alpha}=O(\la^{-\delta/2})$ 
and $\|\bar{\eta}\|_{\alpha}<\ep$ 
holds and hence
$w\in V_{\ep}(k)$. Therefore by Lemma~\ref{change of variable} (2),
and Lemma~\ref{expansion of F and G},
we have
 \begin{align*}
  \|\mathbf{\tilde{e}}_{\mathbf{n},T,\la}\|_{L^2(\Pea,\nu_{\la,a})}^2&=
\int_{S_a} 
\mathbf{e}_{\mathbf{n},T,\la}(P(k)(w-k))^2\chi_{m,\theta,\la,\delta}
(P(k)(w-k))^2c_{\la,a}^{-1}
e^{\frac{\la}{2}\|k\|_H^2}
d\mu_{\la,a}(w)
\nonumber\\
&=\int_{S_a\cap V_{\ep}(k)} 
\mathbf{e}_{\mathbf{n},T,\la}(P(k)(w-k))^2
\chi_{m,\theta,\la,\delta}(P(k)(w-k))^2c_{\la,a}^{-1}
e^{\frac{\la}{2}\|k\|_H^2}d\mu_{\la,a}(w)
\nonumber\\
&=\int_{B_{k,\ep}^{\Omega_0}(0)}
\mathbf{e}_{\mathbf{n},T,\la}(\eta)^2\chi_{m,\theta,\la,\delta}(\eta)^2
\exp\left(-\la F(a,\eta,k)\right)G(a,\eta,k)e^{\frac{\la}{2}\|k\|_H^2}
d\mu_{\la,W_0}(\eta)\nonumber\\
&=\int_{B_{k,\ep}^{\Omega_0}(0)}
\mathbf{e}_{\mathbf{n},T,\la}(\eta)^2\chi_{m,\theta,\la,\delta}(\eta)^2
\exp\left(-\la F_R(a,\eta,k)\right)
(1+G_R(a,\eta,k))
d\mu_{\la,T,W_0}(\eta)\nonumber\\
&=\int_{W_0}\mathbf{e}_{\mathbf{n},T,\la}(\eta)^2
\chi_{m,\theta,\la,\delta}(\eta)^2
\exp\left(\la O(\|\bar{\eta}\|_{\alpha}^3)\right)(1+O(\|\bar{\eta}\|_{\alpha})
d\mu_{\la,T,W_0}(\eta)\nonumber\\
&=\int_{\Omega_0}\mathbf{e}_{\mathbf{n},T,\la}(\eta)^2
\chi_{m,\theta,\la,\delta}(\eta)^2
(1+O(\la^{1-(3\delta/2)}))
(1+O(\la^{-\delta/2}))
d\mu_{\la,T,W_0}(\eta)\nonumber\\
&=1+O(\la^{1-(3\delta/2)})+O(\la^{-\delta/2})+
\int_{\Omega_0}\mathbf{e}_{\mathbf{n},T,\la}
(\eta)^2\left(\chi_{m,\theta,\la,\delta}(\eta)^2-1\right)
d\mu_{\la,T,W_0}(\eta).
 \end{align*}
Also
\begin{align*}
& \int_{\Omega_0}\mathbf{e}_{\mathbf{n},T,\la}(\eta)^2
\left(\chi_{m,\theta,\la,\delta}(\eta)^2-1\right)
d\mu_{\la,T,W_0}(\eta)\nonumber\\
&\qquad=
 \int_{\{\|\bar{\eta}\|_{\alpha}\ge C \la^{-\delta/2}\}}
\mathbf{e}_{\mathbf{n},T,\la}(\eta)^2
d\mu_{\la,T,W_0}(\eta)\nonumber\\
&\qquad=\int_{\{\|\bar{\eta}\|_{\alpha}\ge C\la^{(1-\delta)/2}\}} 
\mathbf{e}_{\mathbf{n},|I+T|-I,1}(\eta)^2
d\mu_{1,|I+T|-I,W_0}(\eta)\le C\exp(-C'\la^{1-\delta}).
\end{align*}
Here in the final inequality, 
we have used the square exponential decay estimate of the Gaussian measure
$\mu_{1,|I+T|-I,W_0}$
and $L^p$ integrability of the Hermite functions with respect to
the Gaussian measure $\mu_{1,|I+T|-I,W_0}$ for any $p\ge 1$.
The estimate (\ref{nm}) is similar to (\ref{n}).
We prove (\ref{approximate eigenvalue}).
First, note that
$\varPhi_{k,a}^{\ast}(\varPsi_{k,a}^{\ast}\mathbf{\hat{e}}_{\mathbf{n},T,\la})(\eta)
=\mathbf{\hat{e}}_{\mathbf{n},T,\la}(\eta)$.
Using this, 
we have
\begin{align*}
 -L_{\la,\Pea}\mathbf{\tilde{e}}_{\mathbf{n},T,\la}&=-(Y^{-1})^{\ast}
\left(L_{\la, S_a}\varPsi_{k,a}^{\ast}\mathbf{\hat{e}_{\mathbf{n},T,\la}}\right)
(\gamma)=-(Y^{-1})^{\ast}\left(\varPsi_{k,a}^{\ast}
\left(\tLS\mathbf{\hat{e}}_{\mathbf{n},T,\la}\right)\right)(\gamma).
\end{align*}
Using Lemma~\ref{expression of L in local chart}, we estimate the above.
First, noting 
$L_{\la,T,W_0}\mathbf{e}_{\mathbf{n},T,\la}
=-\la E_{\mathbf{n}}\mathbf{e}_{\mathbf{n},T,\la}$ 
which follows from Theorem~\ref{cons},
we have
\begin{align*}
&\la^{-1}L_{\la,T,W_0} \mathbf{\hat{e}}_{\mathbf{n},T,\la}
-E_{\mathbf{n}}\mathbf{\hat{e}}_{\mathbf{n},T,\la}\nn\\
&=
\la^{-1}c_{\la,a}^{-1/2}e^{\frac{\la}{4}\|k\|_H^2}
\Bigl\{(L_{\la,T,W_0}\chi_{m,\theta,\la,\delta})\,\mathbf{e}_{\mathbf{n},\la,T}
-2\left(D_0\mathbf{e}_{\mathbf{n},\la,T}, 
D_0\chi_{m,\theta,\la,\delta}\right)\Bigr\}=:I_{1,\la}(\eta).
\end{align*}
Hence,
\begin{align*}
& \int_{B_{k,\ep}^{\Omega_0}(0)}
I_{1,\la}(\eta)^2
d\tilde{\mu}_{\la,a,W_0}(\eta)\nn\\
&\,\le
2\la^{-2}\int_{\|\bar{\eta}\|_{\alpha}\ge C\la^{-\delta/2}}
e^{\la O(\|\bar{\eta}\|_{\alpha}^3)}(1+O(\|\bar{\eta}\|_{\alpha}))
(L_{\la,T,W_0}\chi_{m,\theta,\la,\delta})(\eta)^2\,
\mathbf{e}_{\mathbf{n},\la,T}(\eta)^2d\mu_{\la,T,W_0}(\eta)\nn\\
&\,+8\la^{-2}
\int_{\|\bar{\eta}\|_{\alpha}\ge C\la^{-\delta/2}}
e^{\la O(\|\bar{\eta}\|_{\alpha}^3)}(1+O(\|\bar{\eta}\|_{\alpha}))
|D_0\mathbf{e}_{\mathbf{n},\la,T}|^2
|D_0\chi_{m,\theta,\la,\delta}|^2
d\mu_{\la,T,W_0}(\eta)\nn\\
&=O(e^{-C\la^{1-\delta}}).
\end{align*}
Here, we have used Theorem~\ref{cons} (3) and 
the exponential decay estimate of the tail of the 
Gaussian measure $\mu_{1,|I+T|-I,W_0}$ similarly to the proof of (1)
in the final step.

By Theorem~\ref{cons} (3) and the exponential decay estimate of 
the Gaussian measures, we obtain
\begin{align*}
& \la^{-1}\|\tr(\tilde{K}_kD_0^2\mathbf{\hat{e}}_{\mathbf{n},\la,T})
\|_{L^2(\tilde{\mu}_{\la,a,W_0})}=O(\la^{-\delta/2}),\\
&\la^{-1}\|\left(B_{\la, k},
D_0\mathbf{\hat{e}}_{\mathbf{n},\la,T}(\eta)\right)
\|_{L^2(\tilde{\mu}_{\la,a,W_0})}
=O\left(\la^{-\frac{1}{2}}\left(\la^{1-\frac{3}{2}\delta}+
\la^{-\frac{\delta}{2}}\right)\right)=O\left(\la^{-(3\delta-1)/2}\right),\\
&\quad
\|(D_{0}F_R(a,\eta,k),
D_0\mathbf{\hat{e}}_{\mathbf{n},\la,T}(\eta))\|_{L^2(\tilde{\mu}_{\la,a,W_0})}
=O(\la^{\frac{1}{2}-\delta}),\phantom{llllllllllllllllllllllll}\\
&\la^{-1}
\|\left(D_{0}\log G(a,\eta,k),
D_0\mathbf{\hat{e}}_{\mathbf{n},\la,T}(\eta)
\right)\|_{L^2(\tilde{\mu}_{\la,a,W_0})}
=O(\la^{-1/2}),
\end{align*}
which completes the proof.
\end{proof}

\section{Estimate of integral of certain exponential functional}

\begin{lem}\label{expansion of b}
 Let $k\in S_a^H$ and consider the 
local coordinate defined in
Lemma~$\ref{local chart}$.
Also recall that $(U_kh)(t)=\int_0^tAd(Y(s,e,k))dh_s$.
\begin{enumerate}
\item[$(1)$] 
We have the following expansion.
\begin{align*}
& \frac{1}{2}|b(1)|^2=\frac{1}{2}|U_k(k)(1)|^2
-\frac{1}{2}\int_0^1\left([(U_k\eta)(s), 
(U_kk)(1)],d(U_k\eta)_s\right)\nonumber\\
&\qquad-\frac{1}{2}\int_0^1\left([(U_k\eta)(s), d(U_kk)_s],
[(U_k\eta)(s), (U_kk)(1)\right)
+\frac{1}{2}\left|\int_0^1[(U_k\eta)(s), d(U_kk)(s)]\right|^2\nonumber\\
&\qquad+\left(\int_0^1[v_a(\eta),sd(U_kk)(s)], (U_kk)(1)\right)\nonumber\\
&\qquad
+\frac{1}{2}\left(\int_0^1\left[\left[\int_0^s(U_k\eta)(u), 
d(U_k\eta)_u\right],
d(U_kk)(s)\right], U_k(k)(1)\right)\nonumber\\
&\qquad
+\left(\int_0^1\left[(U_k\eta)(s),d(U_kk)_s\right], (U_kk)(1)\right)
+O(\|\bar{\eta}\|_{\alpha}^3).
\end{align*}
\item[$(2)$]
Further, under the assumption that there exists $\xi\in \g$ such that
$k_t=t\xi$, it holds that
\begin{align}
 & \frac{1}{2}|b(1)|^2=\frac{1}{2}\|k\|_H^2
-\frac{1}{2}\int_0^1\left([(U_k\eta)(s), \xi],d(U_k\eta)_s\right)\nonumber\\
&\qquad\qquad-\frac{1}{2}\left(\int_0^1\left|[(U_k\eta)_s,\xi]\right|^2ds-
\left|\int_0^1[(U_k\eta)(s),\xi]ds\right|^2\right)+O(\|\bar{\eta}\|_{\alpha}^3).
\label{expansion of b 2}
\end{align}
\end{enumerate}
\end{lem}

\begin{proof}
(1) 
Below, for simplicity, we write $v_a(\eta)=v(\eta)=v$ and 
$\varphi_a(\eta)=\varphi(\eta)=\varphi$.
Using (\ref{derivative of Ad}) and $v=O(\|\bar{\eta}\|_{\alpha}^2)$,
$\varphi=O(\|\bar{\eta}\|_{\alpha}^2)$, we have
\begin{align*}
& \frac{1}{2}|b(1)|^2
=\frac{1}{2}\left|\int_0^1Ad\left(Y(s,e,w)\right)dw_s\right|^2
\nonumber\\
&\,=\frac{1}{2}
\left|\int_0^1Ad\left(Y(s,e,k+\eta+\varphi)\right)d(k+\eta+\varphi)_s\right|^2
\nonumber\\
&\,=\frac{1}{2}
\Biggl|\int_0^1
\Bigl(Ad\left(Y(s,e,k+\eta+\varphi)\right)-
Ad\left(Y(s,e,k)\right)
\Bigr)
d(k+\eta+\varphi)_s
+(U_kk)(1)+v\Biggr|^2\nonumber\\
&\,=\frac{1}{2}\Biggl|
\int_0^1\int_0^1
\Bigl[
\int_0^s
Ad\left(Y^{\tau,v,\eta}(u)\right)d\eta_u
+\int_0^sAd\left(Y^{\tau,v,\eta}(u)Y(u,e,k)^{-1}\right)vdu,\nonumber\\
&\quad\qquad\qquad\qquad Ad\left(Y^{\tau,v,\eta}(s)\right)d(k+\eta+\varphi)_s
\Bigr]d\tau+(U_kk)(1)+v
\Biggr|^2\nonumber\\
&\,=\frac{1}{2}\Biggl|
\int_0^1\int_0^1
\Bigl[
\int_0^s
Ad\left(Y^{\tau,v,\eta}(u)\right)d\eta_u
+\int_0^sAd\left(Y^{\tau,v,\eta}(u)Y(u,e,k)^{-1}\right)vdu,\nonumber\\
&\quad\qquad\qquad\qquad
Ad\left(Y^{\tau,v,\eta}(s)\right)d(k+\eta+\varphi)_s
\Bigr]d\tau\Biggr|^2\nonumber\\
&\,+\left(\int_0^1\int_0^1
\Bigl[
\int_0^sAd\left(Y^{\tau,v,\eta}(u)\right)d\eta_u,
Ad\left(Y^{\tau,v,\eta}(s)\right)d(k+\eta)_sd\tau
\Bigr], (U_kk)(1)\right)\nonumber\\
&\,+\left(\int_0^1\int_0^1
\Bigl[
\int_0^sAd\left(Y^{\tau,v,\eta}(u)Y(u,e,k)^{-1}\right)vdu,
Ad\left(Y^{\tau,v,\eta}(s)\right)d(k+\eta)_sd\tau
\Bigr], (U_kk)(1)\right)\nonumber\\
&\,+\frac{1}{2}|(U_kk)(1)|^2+
\left((U_kk)(1),v\right)+O(\|\bar{\eta}\|_{\alpha}^3)\nonumber\\
&=:I_1+I_2+I_3+I_4+I_5+O(\|\bar{\eta}\|_{\alpha}^3).
\end{align*}
We estimate each terms.
\begin{align*}
 I_1&=\frac{1}{2}\left|
\int_0^1\int_0^1\left[\int_0^sAd\left(Y^{\tau,v,\eta}(u)d\eta_u,
Ad\left(Y^{\tau,v,\eta}(s)dk_s\right)\right)\right]d\tau
\right|^2+O(\|\bar{\eta}\|_{\alpha}^3)\nonumber\\
&=\frac{1}{2}\left|
\int_0^1\int_0^1\left[\int_0^sAd\left(Y(u,e,k)d\eta_u,
Ad\left(Y(s,e,k)dk_s\right)\right)\right]d\tau\right|^2
+O(\|\bar{\eta}\|_{\alpha}^3)
\nonumber\\
&=\frac{1}{2}\left|\int_0^1
\left[
(U_k\eta)(s),d(U_kk)_s
\right]\right|^2+O(\|\bar{\eta}\|_{\alpha}^3).
\end{align*}
As for $I_2$, we have
\begin{align*}
 I_2&=\left(\int_0^1\int_0^1
\Bigl[
\int_0^sAd\left(Y^{\tau,v,\eta}(u)\right)d\eta_u,
Ad\left(Y^{\tau,v,\eta}(s)\right)dk_sd\tau
\Bigr], (U_kk)(1)\right)\\
&\qquad+
\left(\int_0^1\int_0^1
\Bigl[
\int_0^sAd\left(Y^{\tau,v,\eta}(u)\right)d\eta_u,
Ad\left(Y^{\tau,v,\eta}(s)\right)d\eta_sd\tau
\Bigr], (U_kk)(1)\right)\\
&=:I_{2,1}+I_{2,2}.
\end{align*}
We estimate $I_{2,1}$.
\begin{align*}
 I_{2,1}&=
\Biggl(\int_0^1\int_0^1
\Bigl[
\int_0^s\left(Ad\left(Y^{\tau,v,\eta}(u)\right)
-Ad\left(Y(u,e,k)\right)\right)d\eta_u
+\int_0^sAd(Y(u,e,k))d\eta_u,\\
&\quad\left(
\Big(Ad\left(Y^{\tau,v,\eta}(s)\right)-Ad(Y(s,e,k))\Big)+
Ad\left(Y(s,e,k)\right)
\right)
dk_sd\tau
\Bigr], (U_kk)(1)\Biggr)\\
&=\Biggl(\int_0^1\int_0^1
\Bigl[
\int_0^s\left(Ad\left(Y^{\tau,v,\eta}(u)\right)
-Ad\left(Y(u,e,k)\right)\right)d\eta_u,
Ad(Y(s,e,k))dk_s\Bigl]d\tau,(U_kk)(1)\Biggl)\\
&\quad+
\Biggl(\int_0^1\int_0^1
\Bigl[
\int_0^sAd\left(Y(u,e,k)\right)d\eta_u,
\left(Ad\left(Y^{\tau,v,\eta}(s)\right)-Ad(Y(s,e,k))\right)dk_s\Bigl]
d\tau,(U_kk)(1)\Biggl)\\
&\quad+\left(\int_0^1\left[\int_0^sAd(Y(u,e,k))d\eta_u, Ad(Y(s,e,k))dk_s\right],
(U_kk)(1)\right)+O(\|\bar{\eta}\|_{\alpha}^3)\\
&\quad=:I_{2,1,1}+I_{2,1,2}+I_{2,1,3}+O(\|\bar{\eta}\|_{\alpha}^3).
\end{align*}
For $I_{2,1,1}$ and $I_{2,1,2}$, using (\ref{derivative of Ad}) again,
we have
\begin{align*}
& I_{2,1,1}\\
&=\Biggl(
\int_0^1\int_0^1\Biggl[
\int_0^s\left(\int_0^\tau
\Bigl[\int_0^uAd(Y(r,e,k))d\eta_r, Ad(Y(u,e,k))d\eta_u\Bigr]d\sigma\right),\\
&\qquad\qquad Ad(Y(s,e,k))dk_s
\Biggr]d\tau,
(U_kk)(1)
\Biggr)+O(\|\bar{\eta}\|_{\alpha}^3)\\
&=\frac{1}{2}\Biggl(
\int_0^1\left[\int_0^s\left[\int_0^uAd(Y(r,e,k))d\eta_r,
Ad(Y(u,e,k))d\eta_u\right], Ad(Y(s,e,k))dk_s\right],\\
&\quad\quad (U_kk)(1)
\Biggr)+O(\|\bar{\eta}\|_{\alpha}^3),
\end{align*}
\begin{align*}
& I_{2,1,2}\\
&=
\Biggl(
\int_0^1\int_0^1\Biggl[\int_0^sAd(Y(u,e,k))d\eta_u,\\
&\, \int_0^{\tau}\left[\int_0^sAd(Y(u,e,k))d\eta_u,
Ad(Y(s,e,k))dk_s
\right]d\sigma\Biggr] d\tau, 
(U_kk)(1)
\Biggr)+O(\|\bar{\eta}\|_{\alpha}^3)\\
&=\frac{1}{2}\Biggl(
\int_0^1\left[\int_0^sAd(Y(u,e,k))d\eta_u,
\left[\int_0^sAd(Y(u,e,k))d\eta_u, Ad(Y(s,e,k))dk_s\right]\right],\\
&\qquad\qquad
(U_kk)(1)
\Biggr)+O(\|\bar{\eta}\|_{\alpha}^3).
\end{align*}
For $I_{2,2}$,
we have
\begin{align*}
 I_{2,2}&=\left(\int_0^1\left[\int_0^sAd(Y(u,e,k))d\eta_u, 
Ad(Y(s,e,k))d\eta_s\right]
, (U_kk)(1)\right)+O(\|\bar{\eta}\|_{\alpha}^3).
\end{align*}
We next consider $I_3$.
\begin{align*}
 I_3&=\left(\int_0^1\int_0^1\left[sv,Ad(Y(s,e,k))dk_s\right]d\tau, 
(U_kk)(1)\right)
+O(\|\bar{\eta}\|_{\alpha}^3)\\
&=\left(\int_0^1\left[v,sd(U_kk)_s\right],(U_kk)(1)\right)
+O(\|\bar{\eta}\|_{\alpha}^3).
\end{align*}
Finally, by the expression of $v$ 
as in (\ref{expression of v}), we have
\begin{align*}
 I_5&=-\frac{1}{2}\left(\int_0^1\left[U_k\eta_s,d(U_k\eta)_s\right],
(U_kk)(1)\right)+O(\|\bar{\eta}\|_{\alpha}^3).
\end{align*}
Consequently, we obtain the following identities, in which we mean
the difference of the R.H.S. and the L.H.S. is of order 
$O(\|\bar{\eta}\|_{\alpha}^3)$.
\begin{align*}
 I_4&=\frac{1}{2}|U_kk(1)|^2,\\
I_5&=\frac{1}{2}\int_0^1\left([(U_k\eta)(s),(U_kk)(1)],d(U_k\eta)_s\right),\\
I_{2,2}&=-\int_0^1\left([(U_k\eta)(s),(U_kk)(1)],d(U_k\eta)_s\right),\\
I_{2,1,2}&=-\frac{1}{2}\int_0^1\Bigl([(U_k\eta)(s), d(U_kk)_s],
[(U_k\eta)(s), (U_kk)(1)]\Bigr)\\
I_1&=\frac{1}{2}\left|\int_0^1[(U_k\eta)(s), d(U_kk)(s)]\right|^2,\\
I_3&=\left(\int_0^1\left[v,sd(U_kk)_s\right],(U_kk)(1)\right),\\
I_{2,1,1}&=\frac{1}{2}\left(\int_0^1\left[\left[\int_0^s(U_k\eta)(u), d(U_k\eta)_u\right],
d(U_kk)(s)\right], U_k(k)(1)\right),\\
I_{2,1,3}&=\left(\int_0^1\left[(U_k\eta)(s),d(U_kk)_s\right], (U_kk)(1)\right).
\end{align*}
This completes the proof.

\noindent
(2) When $k_t=t\xi$, $I_4=\frac{1}{2}\|k\|_H^2$ holds.
Also, noting $[\xi,\xi]=0$ for $\xi\in \g$, we see that
$I_3=I_{2,1,1}=I_{2,1,3}=0$.
Finally, noting that $I_4+I_5+I_{2,2}+I_{2,1,2}+I_1$ coincides with the main term of
the R.H.S of (\ref{expansion of b 2}), the proof is finished.
\end{proof}

\begin{lem}\label{exponential estimate}
Let $\D$ be a domain considered in Theorem~$\ref{main theorem}$.
 Let $\{l_{\xi_i}\}_{i=1}^N$ be all geodesics in $\D$.
For simplicity, we write $k_i=l_{\xi_i}$.
We take $\ep(k_i)$ which is defined in Lemma~$\ref{local chart}$ to
be smaller than $1$ and 
$V_{\ep(k_l)}^{S_a}(k_l)\cap V_{\ep(k_j)}^{S_a}(k_j)=\emptyset$
if $l\ne j$.
Let $\kappa>0$ and define a function on $\D$ by
\begin{align*}
 \rho_\kappa(\gamma)=
\begin{cases}
 \kappa \left[\bar{\Gamma}_{m,\theta}
\left(\varPsi_{k_i,a}(Y^{-1}(\gamma))\right)\right]^2 &
\text{if $Y^{-1}(\gamma)\in V_{\ep(k_i)}^{S_a}(k_i)$\,, 
$1\le i\le N$,}\\
\kappa & \text{if $Y^{-1}(\gamma)\notin 
\cup_{i=1}^{N}V_{\ep(k_i)}^{S_a}(k_i)$,}
\end{cases}
\end{align*}
Then for sufficiently small $\kappa$, 
$
 \limsup_{\la\to\infty}\int_{\D}
\exp\Bigl(C_{\la}\la^2\left(V_{\la,a}+\rho_{\kappa}\right)\Bigr)
d\nu_{\la,\D}<\infty,
$
where 
$C_{\la}=\frac{2}{\la}\left(1+\frac{C_1}{\la}\right)$ and
$V_{\la,a}$ are given 
in Theorem~$\ref{refined version of gross lsi}$.
\end{lem}

\begin{proof}
Since $a$ is not a point of the cut-locus of $e$,
the short time aymptotics formula of the heat kernel implies that
\begin{align*}
\left|\frac{2}{\la}\log\left(\la^{-d/2}p(\la^{-1},e,a)\right)
+d(e,a)^2\right|=O\left(\frac{1}{\la}\right).
\end{align*}
Hence, denoting
$W(\gamma)=C_2\left(1+|b(1,\gamma)|^2+\int_0^1|b(s,\gamma)|^2ds\right)$,
there exists $C'>0$ such that for large $la>0$,
\begin{align}
 \int_{\D}
\exp\Bigl(C_{\la}\la^2\left(V_{\la,a}+\rho_{\kappa}\right)\Bigr)
d\nu_{\la,\D}&=\int_{Y^{-1}(\D)}
\exp\Bigl\{C_{\la}\la^2\Bigl(V_{\la,a}(Y(w))+\rho_{\kappa}(Y(w))\Bigr)\Bigr\}
d\mu_{\la,a}(w)\nn\\
&\le C'\int_{Y^{-1}(\D)}\exp
\left(F_{\la}(w)\right)
d\mu_{\la,a}(w)=:C'I(\la),\label{exponential integral}
\end{align}
where
$F_{\la}(w)=\frac{\la}{2}\left(|b(1)|^2-d(e,a)^2\right)
+2\la\left(1+\frac{C_1}{\la}\right)\rho_{\kappa}(Y(w))+C'W(Y(w))$.
We estimate $I(\la)$ for large $\la$ by 
decomposing the set $\{w\in S_a~|~Y(w)\in \D\}$ into three types of sets.
Let us choose a positive number $R$ such that
$\max_{1\le i\le N} \|k_i\|_H+1<R$.
Actually, $R$ should be taken to be large enough according to
$\D$ which we will point out.
Consider the set $S_a^{H,R}=S_a^H\cap \{h\in H~|~\|h\|_H\le R\}$
and we choose the neighborhood $V_{\ep(k)}^{S_a}(k)$ of each $k\in S_a^{H,R}$
and $U_{\delta(k)}^{\Omega}(k)$ such that
$U_{\delta(k)}^{\Omega}(k)\cap S_a\subset V_{\ep(k)}(k)^{S_a}$.
Note that $U_{\delta(k)}^{\Omega}(k)$ is a neighborhood of $k$ in $\Omega$.
These choices are possible because of Lemma~\ref{local chart} (5).
Actually we need to choose $\delta(k)$ and $\ep(k)$ to be sufficiently small
according to $k$ as in the following proof.
We have the trivial covering $S_a^{H,R}\subset 
\cup_{k\in S_a^{H,R}}U_{\delta(k)}^{\Omega}(k)\cap H$.
$U_{\delta(k)}^{\Omega}(k)\cap H$ is an open set in the topology of rough path
and $S_a^{H,R}$ is a compact set in the topology of $\mathscr{C}^{\alpha}$
(precisely, we need to embed the set in $\mathscr{C}^{\alpha}$).
Hence there exists finite $\{k_i\}_{i=1}^M\subset S_a^{H,R}$ such that
$S_a^{H,R}\subset \cup_{i=1}^M
\Bigl(U_{\delta(k_i)}^{\Omega}(k_i)\cap H\Bigr)$.
We may assume that the first $N$ elements, $\{k_i\}_{i=1}^N$,  is the set of the
geodesics and $U_{\delta(k_i)}^{\Omega}(k_i)$ $(i\ge N+1)$ dose not contain
any $k_i$ $(1\le N)$.
By the property $U_{\delta(k)}^{\Omega}(k)\cap S_a\subset V_{\ep(k)}(k)^{S_a}$,
we obtain 
\begin{align*}
 S_a^{H,R}\subset \cup_{i=1}^M
\Bigl(V_{\ep(k_i)}^{S_a}(k_i)\cap H\Bigr).
\end{align*}
This implies
\begin{align*}
\inf\left\{\|h\|_H^2~\Big |~ h\in S_a^H\cap \left(\cup_{i=1}^M
\Bigl(V_{\ep(k_i)}^{S_a}(k_i)\Bigr)\right)^{\complement}
\right\}\ge R.
\end{align*}
Note that 
$S_a\cap \Bigl(V_{\ep(k_i)}^{S_a}(k_i)\Bigr)^{\complement}$
is a closed set in $\mathscr{C}^{\alpha}$.
Hence by the large deviation estimate for $\mu_{\la,a}$ 
(see Theorem~2.1 in \cite{inahama}),
for all large $\la$, we have for any $R'<R$,
\begin{align*}
 \mu_{\la,a}\left(\Bigl(\cup_{i=1}^M
V_{\ep(k_i)}^{S_a}(k_i)\Bigr)^{\complement}\right)
\le \exp\left(-\la R'/2\right).
\end{align*}
We have
\begin{align*}
 I(\la)&\le\sum_{i=1}^N\int_{V_{\ep(k_i)}^{S_a}(k_i)}
\exp\left(F_{\la}\right)d\mu_{\la,a}
+\sum_{i=N+1}^M\int_{V_{\ep(k_i)}^{S_a}(k_i)\cap 
\left(\cup_{j=1}^NV_{\ep(k_j)}^{S_a}(k_j)\right)^\complement}\,
\exp\left(F_{\la}\right)d\mu_{\la,a}(w)\nn\\
&\qquad+\int_{Y^{-1}(\D)\cap\bigl(\cup_{i=1}^M
V_{\ep(k_i)}^{S_a}(k_i)\bigr)^{\complement}
}\, \exp\left(F_{\la}\right)d\mu_{\la,a}(w),\nn\\
&=:\sum_{i=1}^MI_i(\la)+J(\la).
\end{align*}
We estimate each terms.
First, we estimate $J(\la)$.
As stated in the proof of
Lemma~\ref{exponential estimate0}, there exists $p>1$ such that
$\int_{\D}\exp\left(\frac{p\la}{2}|b(1)|^2\right)
d\nu_{\la,\D}\le Ce^{C_p\la}$.
Let $q>1$ be the positive number such that $\frac{1}{p}+\frac{1}{q}=1$.
By using the H\"older inequality, we have
\begin{align*}
 J(\la)&\le Ce^{C'_p\la+C'\kappa\la-\la R'/4q}
\left(\int_{\Pea}e^{2qC'W(\gamma)}d\nu_{\la,a}\right)^{1/(2q)}.
\end{align*}
Also
\begin{align}
& \int_{\Pea}e^{2qC'W(\gamma)}d\nu_{\la,a}\nn\\
&=p(\la^{-1},e,a)^{-1}\int_W
\exp\Biggl\{2qC'C_2\la^{-1}\Biggl(
\left|\int_0^1Ad(Y(s,e,\la^{-1/2}w))dw_s\right|^2\nn\\
&\quad+
\left|\int_0^1\left|\int_0^tAd(Y(s,e,\la^{-1/2}w))dw_s\right|^2ds\right|
\Biggr)
\Biggr\}\delta_a(Y(1,e,\la^{-1/2}w))d\mu(w)\nn\\
&\le Ce^{C''\la}.\label{estimate of W}
\end{align}
In the above, we have used the exponential functional and
$\delta_a(Y(1,e,\la^{-1/2}w))$ belong to 
$\DD^{k,r}(\RR)$ and $\DD^{-k,s}(\RR)$ 
($r,s>1$ and $\frac{1}{r}+\frac{1}{s}=1$) respectively
and the heat kernel estimate.
Also, we have used that the norm of the exponential functional is
bounded and the latter norm of $\delta_a(Y(1,e,\la^{-1/2}w))$
is of at most polynomial growth of $\la$
under the limit $\la\to\infty$.
Hence, for sufficiently large $R$, we see that
$\lim_{\la\to\infty}J(\la)=0$.

We calculate the integral on $V^{S_a}_{\ep(k_i)}(k_i)$
by using the change of variable formula 
$w=\varPhi_a(\eta)$ $(\eta\in B_{k_i,\ep}^{\Omega_0})$.
Of course, the set $\Omega_0$ depends on $k_i$, but,
we use the same notation $W_0$ for simplicity.
Also, we write $\tilde{\rho}_{\kappa}(\eta)=\rho(Y(\varPhi_{k_i,a}(\eta)))$
and $\tilde{W}(\eta)=W(Y(\varPhi_{k_i,a}(\eta)))$.
Consider the case $1\le i\le N$.
Then, using the change of the variable formula,
Lemma~\ref{expansion of F and G},
 Lemma~\ref{expansion of b} and 
$\sup_{\eta\in B_{k_i,\ep(k_i)}^{\Omega_0}(0)}\tilde{W}(\eta)<\infty$, 
there exists $C'', C'''>0$ which depend on $k_i$ such that
\begin{align*}
I_i(\la)
\le
2e^{C''}\int_{B_{k_i,\ep(k_i)}^{\Omega_0}(0)}
\exp\left(-\frac{\la}{2}G(\eta)\right)
d\mu_{\la,W_0}(d\eta)=:2C'''\tilde{I}_i(\la),
\end{align*}
where
\begin{align*}
& G(\eta)=
2\int_0^1\left(\bigl[(U_{k_i}\eta)(s), \xi_i\bigr],d(U_{k_i}\eta)_s\right)
\nonumber\\
&+\int_0^1\left|\bigl[(U_{k_i}\eta)_s,\xi_i\bigr]\right|^2ds-
\left|\int_0^1\bigl[(U_{k_i}\eta)(s),\xi_i\bigr]ds\right|^2-
C_{\alpha,m,\theta}'\left(\kappa+
\ep(k_i)\right)\bar{\Gamma}_{m,\theta}(\eta)^2,
\end{align*}
and we have used Lemma~\ref{cut-off function} (1) and
$C'_{\alpha,m,\theta}$ depends on $C_{\alpha,m,\theta}$
which is defined there.
Note that
\begin{align*}
& \int_0^1\left(\bigl[(U_{k_i}\eta)(s), \xi_i\bigr],d(U_{k_i}\eta)_s\right)=\,
:\left(U_{k_i}^{-1}T_{\xi_i}U_{k_i}\eta, \eta\right):,\\
& \int_0^1\left|\bigl[(U_{k_i}\eta)_s,\xi_i\bigr]\right|^2ds-
\left|\int_0^1\bigl[(U_{k_i}\eta)(s),\xi_i\bigr]ds\right|^2=
\|U_{k_i}^{-1}T_{\xi}U_{k_i}\eta\|_{H_0}^2,
\end{align*}
and
\begin{align}
& \inf\left\{\|h\|_{H_0}^2+2\left(U_{k_i}^{-1}T_{\xi_i}U_{k_i}h,h\right)_{H_0}
+\|U_{k_i}^{-1}T_{\xi}U_{k_i}h\|_{H_0}^2~\Big |~\|h\|_{H_0}=1\right\}\nn\\
&\qquad =\inf\sigma\Big(\left(I_{H_0}+T_{\xi_i}\right)^2\Big)>0.
\label{nondegenerate}
\end{align}
(\ref{nondegenerate}) follows from the assumption that
$a$ is not a conjugate point of $e$.
We now estimate $I_{i}(\la)$.
By using the change of variable $\eta=\la^{-1/2}\phi$, We have
\begin{align*}
 \tilde{I}_i(\la)&=
\int_{\la B_{k_i,\ep(k_i)}^{\Omega_0}(0)}
\exp\left(-\frac{1}{2}\tilde{G}(\phi)\right)
d\mu_{1, W_0}(d\phi),
\end{align*}
where 
\begin{align*}
 \tilde{G}(\phi)&=
2\int_0^1\left(\bigl[(U_{k_i}\phi)(s), \xi_i\bigr],d(U_{k_i}\phi)_s\right)
\nonumber\\
&+\int_0^1\left|\bigl[(U_{k_i}\phi)_s,\xi_i\bigr]\right|^2ds-
\left|\int_0^1\bigl[(U_{k_i}\phi)(s),\xi_i\bigr]ds\right|^2
-C_{\alpha,m,\theta}'\left(\kappa+
\ep(k_i)\right)\bar{\Gamma}_{m,\theta}(\phi)^2.
\end{align*}
For any $p>1$, if we choose sufficiently small $\kappa$ and $\ep(k_i)$, then
$\int_{W_0}e^{p C_{\alpha,m,\theta}'\left(\kappa+
\ep(k_i)\right)\bar{\Gamma}_{m,\theta}(\phi)^2}d\mu_{1, W_0}<\infty$
by Lemma~\ref{cut-off function} (4).
By this and (\ref{nondegenerate}), using the H\"older inequality,
we obtain $\limsup_{\la\to\infty}I_i(\la)<\infty$.
We next consider the case $i\ge N+1$.
In these cases, 
\begin{align}
I_i(\la)
&\le
\int_{B_{k_i,\ep(k_i)}^{\Omega_0}}
\exp
 \left(
\frac{\la}{2}\Bigl(|b(1,I(\varPhi_{k_i,a}(\eta)))|^2-d(e,a)^2\Bigr)
+2\la\kappa\left(1+\frac{C_1}{\la}\right)
+C'\tilde{W}(\eta)
\right)\nn\\
&\quad\times
\exp\left\{-\la\left((\eta,k_i)+(\varphi_a(\eta),k_i)+\frac{1}{2}\|k_i\|_H^2
+\frac{1}{2}\|\varphi_a(\eta)\|_H^2\right)\right\}\nn\\
&\qquad \times (1+O(\|\bar{\eta}\|_{\alpha}))
\left(\frac{\la}{2\pi}\right)^{d/2}p(\la^{-1},e,a)^{-1}
d\mu_{\la,W_0}(\eta)\nn\\
&\le C_{k_i}
e^{4\kappa\la}
\int_{B_{k_i,\ep(k_i)}^{\Omega_0}}\exp\left(\la G_{2}(\eta)\right)
d\mu_{\la,W_0}(\eta),\label{i is larger than N+1}
\end{align}
where the final inequality (\ref{i is larger than N+1}) holds for $\la$ with
$\frac{C_1}{\la}\le 1$ and
\begin{align*}
& G_{2}(\eta)=\frac{1}{2}\left(\left|
\int_0^1Ad\left(Y(s,e,k_i+\eta+\varphi_a(\eta))\right)
d\left(k_i+\eta+\varphi_a(\eta)\right)_s\right|^2
-\|k_i\|_H^2\right)
\nn\\
&\qquad\qquad\qquad-
\left(
(\eta,k_i)+(\varphi_a(\eta),k_i)+\frac{1}{2}\|\varphi_a(\eta)\|_H^2
\right).
\end{align*}
Since $Y(\cdot,e,k_i)$ is not a geodesics,
we have
\begin{align*}
 G_2(0)=\frac{1}{2}\left(\left|\int_0^1Ad(Y(s,e,k_i))dk_i(s)\right|^2
-\left\|\int_0^{\cdot}Ad(Y(s,e,k_i))dk_i(s)\right\|_H^2\right)=:\delta_{k_i}<0.
\end{align*}
By this and the fact that $\eta\mapsto G_2(\eta)$ is
a continuous function in the topology of $\mathscr{C}^{\alpha}$,
by choosing sufficiently small $\ep(k_i)$,
we see that
$\sup_{\{\eta\in B_{k_i,\ep(k_i)}^{\Omega_0}(0)\}}G_2(\eta)\le 
\delta_{k_i}/2$.
Thus, it holds that for $\kappa<|\delta_{k_i}|/8$,
\begin{align*}
 \limsup_{\la\to\infty}I_i(\la)\le \lim_{\la\to\infty}
C_{k_i}e^{-\la\delta_{k_i}/2}=0.
\end{align*}
Consequently, by taking sufficiently large $R>0$ and
sufficiently small $\ep(k)$ for $k\in S_a^{H,R}$
and considering
the covering above
and arguing in the above way, we can conclude that
the conclusion of the lemma holds true for sufficiently small $\kappa$.
\end{proof}

\section{Proof of the main theorem}

We prove our main theorem.
Recall that $k_i$ $(1\le i\le N)$
denotes the geodesic $l_{\xi_i}$ which is included in
$\YD$.
First, we introduce a cut-off function
$\chi_{i,\ep}$ which vanishes outside neighborhood of
$k_i$.

\begin{dfi}\label{cut-off function near k_i}
(1) Here, we further suppose that $\ep(k_i)$ which are defined in
Lemma~\ref{local chart} satisfy
$\varPhi_{k_i,a}(B_{k_i,\ep(k_i)}^{\Omega_0})\cap \YD^c=\emptyset$.
This choice is possible because $k_i$ is not included in the 
boundary of $\YD$.
For $\ep<\frac{1}{4}\min(\ep(k_i), 1)$, define
 \begin{align*}
 \chi_{i,\ep}(w)=
 \chi\left(
2
\Bigl\{\ep^{-1}C_{\alpha,m,\theta}\left(\varPsi_{k_i,a}^{\ast}
\bar{\Gamma}_{m,\theta}\right)(w)\Bigr\}^{4m}
\right),
\quad w\in \Omega
 \end{align*}
where 
$\varPsi^{\ast}_{k_i,a}$ denotes the pull back by the mapping
$\varPsi_{k_i,a} : w(\in \Omega)\mapsto 
\eta=P(k_i)(w-k_i)\in 
\Omega_0=\overline{T_{k_i}S_a}^W$ and
$\chi$, $C_{\alpha,m,\theta}$ and $\bar{\Gamma}_{m,\theta}$ 
is the function and the
number appeared in 
Lemma~$\ref{cut-off function}$.
Note that $\Omega_0$ and $\eta$ depend on $k_i$, but, 
we use the same notation
$\Omega_0$ for simplicity as in the proof of 
Lemma~\ref{exponential estimate}.

\noindent
(2) In Lemma~\ref{cut-off function}, we introduce
a function $\chi_{m,\theta,\la,\delta}(\eta)$ which is a cut-off function
near $0$.
Here, we introduce similar functions for later use.
Let $\frac{2}{3}<\delta<1$.
Let $\chi$ be the same function as in Lemma~\ref{cut-off function}
and set
\begin{align}
 \chi_{\frac{1}{3},m,\theta,\la,\delta}(\eta)&=
\chi\left(\frac{1}{3}\la^{2m\delta}
\bar{\Gamma}_{m,\theta}(\eta)^{4m}\right),\\
\chi_{\frac{1}{3},i,m,\theta,\la,\delta}(w)&=
\left(\varPsi^{\ast}_{k_i,a}\chi_{\frac{1}{3},m,\theta,\la,\delta}\right)(w).
\label{chi1/3}
\end{align}
\end{dfi}

\begin{rem}
$(1)$
Note that 
$\chi_{i,\ep}\in \DD^{\infty}(W,\RR)\subset \rD(\LS)$ and
\begin{align*}
& \{w\in S_a~|~\chi_{i,\ep}(w)\ne 0\}\subset 
\{w\in S_a~|~(\varPsi_{k_i,a}^{\ast}\bar{\Gamma}_{m,\theta})(w)
< \ep C_{\alpha,m,\theta}^{-1}\}\\
& \qquad\qquad\qquad\qquad\qquad\subset 
\{w\in S_a~|~\|\overline{\varPsi_{k_i,a}(w)}\|_{\alpha}< \ep\}
=V_{\ep}^{S_a}(k_i),\\
&\{w\in S_a ~|~
(\varPsi_{k_i,a}^{\ast}\bar{\Gamma}_{m,\theta})(w)<
2^{-\frac{1}{4m}}\ep C_{\alpha,m,\theta}^{-1}\}
\subset\{w\in S_a~|~\chi_{i,\ep}(w)=1\}
\subset V_{\ep}^{S_a}(k_i),
\end{align*}
where we have used the identification by the mapping
$\varPsi_{k_i,a} : V_{\ep(k_i)}^{S_a}\to B_{\ep(k_i),k_i}^{\Omega_0}(0)$.

\noindent
$(2)$
Considering the local coordinate $(B_{k_i,\ep}^{\Omega_0}(0),\varPhi_{k_i,a})$
we again use the notation 
$\tilde{f}(\eta)=(\varPhi_{k_i,a}^{\ast}f)(\eta)$ for functions $f$
defined on $S_a$.
Note that 
for any $\eta\in B_{k_i,\ep}^{\Omega_0}(0)$,
\begin{align*}
 \tilde{\chi}_{i,\ep}(\eta)&:=(\varPhi_{k_i,a}^{\ast}\chi_{i,\ep})(\eta)=
\chi\left(\left\{2
\ep^{-1}C_{\alpha,m,\theta}\bar{\Gamma}_{m,\theta}(\eta)\right\}^{4m}\right),\\
\tilde{\chi}_{\frac{1}{3},i,m,\theta,\la,\delta}(\eta)
&:=(\varPhi_{k_i,a}^{\ast}\chi_{\frac{1}{3},i,m,\theta,\la,\delta})(\eta)=
\chi_{\frac{1}{3},m,\theta,\la,\delta}(\eta),
\end{align*}
holds.

\noindent
$(3)$
 Let $\bce_{\mathbf{n},T_i,\la}(w)$ be the function on $S_a$ defined in
Remark~\ref{e on S_a}.
Then it holds that
\[
 \bce_{\mathbf{n},T_i,\la}(w)\chi_{i,\frac{1}{3},m,\theta,\la,\delta}(w)
=\bce_{\mathbf{n},T_i,\la}(w)\quad (w\in S_a).
\]
\end{rem}

\begin{lem}\label{lemma for f on D}
 Let $f\in \rD(L_{\la,\YD})$.
Let $\phi\in \rD(\E_{\la,\YD})\cap 
\rD(\LS)$
and further suppose that
$\phi, D_{S_a}\phi, \LS\phi\in L^{\infty}(S_a,\mu_{\la,a})$.
Then $f\phi\in \rD(L_{\la,\YD})\cap \rD(\LS)$ and it holds that
\[
L_{\la,\YD}(f\phi)=\LS (f\phi)=L_{\la,\YD}f\cdot \phi
+2\left(D_{S_a}f,D_{S_a}\phi\right)+f\LS \phi.
\]
\end{lem}

\begin{proof}
Let $\varphi\in \FC(W)$.
Note that $\varphi\phi\in \rD(\LS)$ and
\begin{align}
 L_{\la, S_a}(\varphi\phi)
=L_{\la,S_a}\varphi\cdot\phi+2(D_{S_a}\varphi,D_{S_a}\phi)
+\varphi L_{\la,S_a}\phi.\label{product formula}
\end{align}
This follows from the following argument.
Let $\psi\in \FC(W)$ and define
$I=\int_{S_a}\LS\psi \cdot \varphi\phi d\mu_{\la,a}$.
By explicit computation, we have
$\LS(\psi\varphi)=\LS\psi\cdot\varphi+2(D_{S_a}\psi,D_{S_a}\varphi)+\psi\LS\varphi$.
Hence,
\begin{align*}
 I&=\int_{S_a}\left(\LS(\psi\varphi)-2D_{S_a}\psi\cdot D_{S_a}\varphi-
\psi\LS\varphi\right)\phi d\mu_{\la,a}\\
&=\int_{S_a}\psi\left(\varphi\LS\phi-\phi\LS\varphi\right)d\mu_{\la,a}
-2\int_{S_a}\left(D_{S_a}\psi,\phi D_{S_a}\varphi\right)d\mu_{\la,a}\\
&=:I_1+I_2.
\end{align*}
To calculate $I_2$, we note the following identity.
Let $g\in \rD(\E_{\la,S_a})$ and $h\in \rD(\LS)$.
Then we have
\begin{align}
 -\int_{S_a}(D_{S_a}\psi, gD_{S_a}h)d\mu_{\la,a}=
\int_{S_a}\psi (D_{S_a}g,D_{S_a}h)d\mu_{\la,a}
+\int_{S_a}\psi g \LS h d\mu_{\la,a}.\label{ibp identity}
\end{align}
This follows from the following argument.
By the definition and the essential self-adjointness of
$(\LS,\FC(W)|_{S_{a}})$ in $L^2(S_a,\mu_{\la,a})$, there exist
$\{g_n\}, \{h_n\}\subset \FC(W)$ such that
\begin{align*}
& \lim_{n\to\infty}\left\{\|g_n-g\|_{L^2(S_a,\mu_{\la,a})}
+\|D_{S_a}g_n-D_{S_a}g\|_{L^2(S_a,\mu_{\la,a})}
\right\}=0\\
& \lim_{n\to\infty}
\left\{\|h_n-h\|_{L^2(S_a,\mu_{\la,a})}
+\|D_{S_a}h_n-D_{S_a}h\|_{L^2(S_a,\mu_{\la,a})}
+\|\LS h_n-\LS h\|_{L^2(S_a,\mu_{\la,a})}\right\}=0.
\end{align*}
Hence noting $\||D_{S_a}\psi|_H\|_{L^{\infty}}<\infty$
and taking the limit $n\to\infty$, in the identity
\[
 -\int_{S_a}(D_{S_a}\psi, g_nD_{S_a}h_n)d\mu_{\la,a}=
\int_{S_a}\psi(D_{S_a}g_n,D_{S_a}h_n)d\mu_{\la,a}
+\int_{S_a}\psi g_n \LS h_n d\mu_{\la,a}
\]
which itself follows from the integration by parts formula,
we get (\ref{ibp identity}).
Consequently,
by the essential self-adjointness of $(\LS,\FC(W)|_{S_a})$ in
$L^2(S_a,\mu_{\la,a})$, we obtain (\ref{product formula}).
By (\ref{product formula}), we have
\begin{align*}
 \left(\LS\varphi, f\phi\right)_{L^2(S_a,\mu_{\la,a})}&=
\left((\LS\varphi)\phi, f\right)_{L^2(S_a,\mu_{\la,a})}\\
&=\left(\LS(\varphi\phi)-2(D_{S_a}\varphi,D_{S_a}\phi)-
\varphi\LS\phi,f\right)_{L^2(S_a,\mu_{\la,a})}.
\end{align*}
Using Remark~\ref{remark on the generator}, we obtain
\[
  \left(\LS(\varphi\phi),f\right)_{L^2(S_a,\mu_{\la,a})}
=\left(L_{\la,\YD}(\varphi\phi),f\right)_{L^2(S_a,\mu_{\la,a})}
=\left(\varphi\phi, L_{\la,\YD}f\right)_{L^2(S_a,\mu_{\la,a})}.
\]
Also, using (\ref{ibp identity}), we have
\begin{align*}
 -2\left((D_{S_a}\varphi,D_{S_a}\phi), f\right)_{L^2(S_a,\mu_{\la,a})}
&=-2\int_{S_a}(D_{S_a}\varphi, fD_{S_a}\phi)d\mu_{\la,a}\\
&=2\int_{S_a}\varphi(D_{S_a}f,D_{S_a}\phi)d\mu_{\la,a}
+2\int_{S_a}\varphi f \LS\phi d\mu_{\la,a}.
\end{align*}

Consequently, we obtain
\begin{align*}
 \left(\LS\varphi, f\phi\right)_{L^2(S_a,\mu_{\la,a})}&=
\left(\varphi, \Bigl(L_{\la,\YD}f\cdot \phi+2
\left(D_{S_a}f,D_{S_a}\phi\right)+f\LS \phi\Bigr)\right)
_{L^2(S_a,\mu_{\la,a})}.
\end{align*}
Since $L_{\la,\YD}f\cdot \phi+2
\left(D_{S_a}f,D_{S_a}\phi\right)+f\LS \phi\in L^2(S_a,\mu_{\la,a})$,
by the essential self-adjintness of $(L_{\la,S_a},\FC(W)|_{S_a})$
in $L^2(S_a,\mu_{\la,a})$, we complete the proof.
\end{proof}

\begin{rem}
Let
 $\chi_{i,\ep}$
be the function defined in Definition~$\ref{cut-off function near k_i}$.
Then we can prove that $\chi_{i,\ep}$ satisfies the property of $\phi$ imposed
in Lemma~$\ref{lemma for f on D}$ by using Lemmas \ref{local chart}, 
\ref{expansion of F and G}, \ref{expression of L in local chart}, 
\ref{cut-off function}.
Actually, $\chi_{i,\ep}\in \DD^{\infty}(W,\RR)$ holds.
\end{rem}

Let $R, r$ be the number defined in Theorem~\ref{main theorem}.
Then there exists $\ep(r)>0$ such that $r-\ep(r)>0$ and
$N_R(r-\ep(r))=N_R(r)=N_R(r+\ep(r))$.
We choose $r_-$ and $r_+$ such that $r-\ep(r)<r_-<r<r_+<r+\ep(r)$.
 Let $F=[0,R]\cap
\left(\cup_{i=1}^{N}(\Lambda(\xi_i)^a+[-r_-,r_-])\right)$.
For $\beta\ge 0$, define
$F_{\beta,+}=F\cap [\beta,\infty)$ and
$F_{\beta,-}=F\cap (-\infty,\beta]$.
For each sets $A=F, F_{\beta,+}, F_{\beta, -}$,
we consider the projection operators on
$L^2(W_0,\hat{\mu}_{\la,T_i,W_0})$ defined by 
$P_{\la,i,A}f=1_{A}(-\la \hat{L}_{\la,T_i,W_0})f$.
Also,
we define $P_{\la,A}=1_{A}(-\la^{-1}L_{\la,\YD})$ for Borel measurable sets
$A\subset \RR$.

\begin{lem}\label{projection on F}
Let
$
 \beta\in 
[0,R]\setminus
\left(\cup_{i=1}^N(\Lambda(\xi_i)^a+(-r,r))\right).
$
Suppose 
$f_{\la,\kappa}\in L^2(\YD,\mu_{\la,\YD})$ satisfies
$P_{\la,[\beta-\kappa,\beta+\kappa]}f_{\la,\kappa}=f_{\la,\kappa}$
and $\|f_{\la,\kappa}\|_{L^2(\YD,\mu_{\la,\YD})}=1$.
Then the following hold.
\begin{enumerate}
 \item[$(1)$] Let us define $u_{\la,\kappa}$ by 
\begin{align*}
 -\la^{-1}L_{\la,\YD}f_{\la,\kappa}=\beta f_{\la,\kappa}+u_{\la,\kappa}.
\end{align*}
Then $\|u_{\la,\kappa}\|_{L^2(\YD,\nu_{\la,\YD})}\le \kappa$ holds.
\item[$(2)$] It holds that
\begin{align}
 \{\eta\in \Omega_0~|~
(\widetilde{
\chi_{\frac{1}{3},i,m,\theta,\la,\delta}})(\eta)\ne 0\}
\subset
\{\eta~|~\bar{\Gamma}_{m,\theta}(\eta)<6\la^{-\delta/2}\}
\subset \{\eta~|~
\|\bar{\eta}\|_{\alpha}< 6C_{\alpha,m,\theta}\la^{-\delta/2}\}.
\label{inclusion property}
\end{align}
Hence, for any $0<\ep'<\ep$, $f\in L^2(\YD,\mu_{\la,\YD})$,
$\widetilde{f\chi_{\frac{1}{3},i,m,\theta,\la,\delta}}(\eta)$
is $0$ on $B_{k_i,\ep'}^{\Omega_0}(0)^{\complement}$
for large $\la$.
Although this function
is defined only on $B_{k_i,\ep}^{\Omega_0}(0)$,
we can extend the function 
to a function in
$L^2(W_0,\hat{\mu}_{\la,T_i,W_0})$ by setting
$0$ on $B_{k_i,\ep}^{\Omega_0}(0)^\complement$.
\item[$(3)$] 
Let 
$\widetilde{f_{\la,\kappa}\chi_{\frac{1}{3},i,m,\theta,\la,\delta}}
\in L^2(W_0,\hat{\mu}_{\la,T_i,W_0})$
be the function defined in $(2)$.
We have
\begin{align*}
\|P_{\la,i,F}
\widetilde{(f_{\la,\kappa}\chi_{\frac{1}{3},i,m,\theta,\la,\delta})}\|^2
_{L^2(W_0,\hat{\mu}_{\la,T_i,W_0})}\le  2(r-r_-)^{-1}
C_{i,\beta,R}\, (\kappa+\la^{-\frac{1-\delta}{2}}).
\end{align*}
\end{enumerate}
\end{lem}

\begin{proof}
(1) This is an immediate consequence of the definition of $f_{\la,\kappa}$.

\noindent
(2) (\ref{inclusion property}) follows from Lemma~\ref{cut-off function}.
Another claim is trivial.

\noindent
(3)
 By Lemma~\ref{lemma for f on D}, we see that
$f_{\la,\kappa}\chi_{i,\ep}\in \rD(L_{\la,S_a})$.
Since $(L_{\la,S_a}, \FC(W)|_{S_a})$ in
$L^2(S_a,\mu_{\la,a})$ is essentially self-adjoint, for any $n\in \NN$, 
there exists $f_{\la,i,n,\kappa}\in \FC(W)$ such that
\begin{align}
 \|f_{\la,i,n,\kappa}-f_{\la,\kappa}\chi_{i,\ep}\|_{L^2(S_a,\mu_{\la,a})}
&\le \frac{1}{n},
\qquad
\left\|\la^{-1}L_{\la,S_a}f_{\la,i,n,\kappa}-
\la^{-1}L_{\la,S_a}(f_{\la,\kappa}\chi_{i,\ep})
\right\|_{L^2(S_a,\mu_{\la,a})}\le \frac{1}{n}.
\label{approximation of f near i 1}
\end{align}
To prove the desired estimate, we prepare necessary three estimates for
$f_{\la,i,n,\kappa}$.
First, let $u_{\la,i,n,\kappa}=-\la^{-1}\LS f_{\la,i,n,\kappa}-\beta
f_{\la,i,n,\kappa}$.
Define
$
\tilde{V}_{\ep}^{S_a}(k_i)=\{w\in S_a ~|~
(\varPsi_{k_i,a}^{\ast}\bar{\Gamma}_{m,\theta})(w)<
2^{-\frac{1}{4m}}\ep C_{\alpha,m,\theta}^{-1}\}.
$
Then using $\chi_{i,\ep}=1$ on $\tilde{V}_{\ep}^{S_a}(k_i)$ and 
Lemma~\ref{lemma for f on D},
we have
\begin{align*}
 -\la^{-1}\LS(f_{\la,\kappa}\chi_{i,\ep})
=-\la^{-1}L_{\la,\YD}f_{\la,\kappa}=\beta f_{\la,\kappa}+u_{\la,\kappa}
\quad \text{on $\tilde{V}_{\ep}^{S_a}(k_i)$}.
\end{align*}
Hence, together with (\ref{approximation of f near i 1}), we get
\begin{align}
 &\|u_{\la,i,n,\kappa}\|_{L^2(\tilde{V}_{\ep}^{S_a}(k_i),\mu_{\la,a})}
\nn\\
&=
\|\la^{-1}\LS f_{\la,i,n,\kappa}+\beta
f_{\la,i,n,\kappa}\|_{L^2(\tilde{V}_{\ep}^{S_a}(k_i),\mu_{\la,a})}
\nn\\
&\le \|\la^{-1}\LS (f_{\la,i,n,\kappa}-f_{\la,\kappa}\chi_{i,\ep})\|
_{L^2(\tilde{V}_{\ep}^{S_a}(k_i),\mu_{\la,a})}
+\|\la^{-1}\LS(f_{\la,\kappa}\chi_{i,\ep})+\beta f_{\la,\kappa}\|
_{L^2(\tilde{V}_{\ep}^{S_a}(k_i),\mu_{\la,a})}\nn\\
&\quad 
+\beta\cdot\|
f_{\la,i,n,\kappa}-f_{\la,\kappa}
\|_{L^2(\tilde{V}_{\ep}^{S_a}(k_i),\mu_{\la,a})}\nn\\
&\le \frac{1+\beta}{n}+C\kappa,\label{estimate of u}
\end{align}
where we have used that 
there exists $C_\la>0$ satisfying $\lim_{\la\to \infty}C_{\la}=1$
such that $C_\la \mu_{\la,\YD}=\mu_{\la,a}$
on $\YD$.
We next prove 
the following estimate of the Dirichlet norm of $f_{\la,i,n,\kappa}$.
\begin{align}
\la^{-1}\|D_{S_a}f_{\la,i,n,\kappa}\|_{L^2(S_a,\mu_{\la,a})}^2
&\le \frac{C}{n}\left(1+\la^{-1}
\|\LS(f_{\la,\kappa}\chi_{i,\ep})\|_{L^2(S_a,\mu_{\la,a})}^2\right)\nn\\
&\quad +C\left(\beta+\kappa+\frac{2C^2_m}{\ep^2\la}\right).
\label{dirichlet norm of f}
\end{align}
The proof is as follows.
We have
\begin{align*}
& \la^{-1}\|D_{S_a}f_{\la,i,n,\kappa}\|_{L^2(S_a,\mu_{\la,a})}^2\nn\\
&=
-\la^{-1}\left(L_{\la,S_a}f_{\la,i,n,\kappa},f_{\la,i,n,\kappa}\right)_{L^2(S_a,\mu_{\la,a})}
\nn\\
&=-\la^{-1}
\left(\LS(f_{\la,i,n,\kappa}-f_{\la,\kappa}\chi_{i,\ep}),
f_{\la,\kappa}\chi_{i,\ep}\right)_{L^2(S_a,\mu_{\la,a})}\nn\\
&\quad -\la^{-1}\left(\LS(f_{\la,i,n,\kappa}-f_{\la,\kappa}\chi_{i,\ep}),
f_{\la,\kappa}\chi_{i,\ep}-f_{\la,i,n,\kappa}\right)_{L^2(S_a,\mu_{\la,a})}\nn\\
&\quad
-\la^{-1}\Bigl(
L_{\la,S_a}(f_{\la,\kappa}\chi_{i,\ep}),
f_{\la,\kappa}\chi_{i,\ep}-f_{\la,i,n,\kappa}\Bigr)_{L^2(S_a,\mu_{\la,a})}\nn\\
&\quad-\la^{-1}
\left(\LS (f_{\la,\kappa}\chi_{i,\ep}),f_{\la,\kappa}\chi_{i,\ep}\right)
_{L^2(S_a,\mu_{\la,a})}\\
&=:I_1+I_2+I_3+I_4
\end{align*}
and
\begin{align*}
 |I_1|+|I_2|\le \frac{C}{n}\left(1+\frac{1}{n}\right),\quad
|I_3|\le \frac{C}{n\la}\|\LS (f_{\la,\kappa}\chi_{i,\ep})\|^2.
\end{align*}
For $I_4$, by using the integration by parts formula, we have
\begin{align*}
 |I_4|&=\frac{1}{\la}\int_{S_a}
\left|D_{S_a}f_{\la,\kappa}\cdot\chi_{i,\ep}+
f_{\la,\kappa}D_{S_a}\chi_{i,\ep}\right|_{T_wS_a}^2
d\mu_{\la,a}\\
&\le 
\frac{2}{\la}\int_{S_a}|D_{S_a}f_{\la,\kappa}|^2d\mu_{\la,a}+
\frac{2}{\la}\int_{S_a}|f_{\la,\kappa}|^2|D_{S_a}\chi_{i,\ep}|^2d\mu_{\la,a}\\
&\le
2(\beta+\kappa)+\frac{2C^2_{m}}{\ep^2\la},
\end{align*}
where we have used 
$P_{\la,[\beta-\kappa,\beta+\kappa]}f_{\la,\kappa}=f_{\la,\kappa}$,
$\|D_{S_a}\chi_{i,\ep}\|_{\infty}\le \ep^{-1}C_m$,
and $D_{S_a}f_{\la,\kappa}=0$ on $\YD^{\complement}$
which follows from that the boundary measure of
$\D$ is 0.
Using (\ref{dirichlet norm of f}), we obtain
\begin{align*}
&\|D_0\tilde{f}_{\la,i,n,\kappa}\|
_{L^2(B_{k_i,\ep}^{\Omega_0},\hat{\mu}_{\la,T_i,W_0})}\nn\\
 &\le C\left(\int_{B_{k_i,\ep}^{\Omega_0}}
\left((I_{H_0}+\tilde{K}_{k_i})D_0\tilde{f}_{\la,i,n,\kappa},
D_0\tilde{f}_{\la,i,n,\kappa}\right)_{H_0}d\tilde{\mu}_{\la,a,W_0}\right)^{1/2}
\nn\\
&\le C'\la^{\frac{1}{2}}
\left(\la^{-1}\int_{S_a}|D_{S_a}f_{\la,i,n,\kappa}(w)|^2
d\mu_{\la,a}(w)\right)^{1/2}\nn\\
&\le 
C'_{\beta}
\left(\frac{\la}{n}\right)^{1/2}\left(1+\la^{-1}
\|\LS(f_{\la,\kappa}\chi_{i,\ep})\|_{L^2(S_a,\mu_{\la,a})}^2\right)^{1/2}
+C\la^{1/2}\left(\beta+\kappa+\frac{2C_m^2}{\ep^2\la}\right)^{1/2}
\end{align*}
which implies
\begin{align}
\limsup_{n\to\infty} \|D_0\tilde{f}_{\la,i,n,\kappa}\|
_{L^2(B_{k_i,\ep}^{\Omega_0},\hat{\mu}_{\la,T_i,W_0})}
&\le
C\la^{1/2}\left(\beta+\kappa+\frac{2C_m^2}{\ep^2\la}\right)^{1/2}.
\label{D_0 dirichlet norm of f}
\end{align}
We now estimate
$\|P_{\la,i,F}\widetilde{f_{\la,\kappa}
\chi_{\frac{1}{3},i,m,\theta,\la,\delta}}\|
_{L^2(W_0,\hat{\mu}_{\la,T_i,W_0})}$.
From now on, for simplicity, we write
\begin{align*}
 \chi_{\frac{1}{3},i,m,\theta,\la,\delta}=\chi_{\frac{1}{3},i,\la,\delta},
\,\,
B_{k_i,\ep}^{\Omega_0}(0)=B_{i,\ep},\,\,
\tilde{K}_{k_i}=\tilde{K}_i,\,\,
\tilde{B}_{\la,k_i}=\tilde{B}_{\la,i},\,\,
L_{\la,T_i,W_0}=L_{\la,i},\,\,
P_{\la,i, F_{\beta,\pm}}=P_{\la,i,F_{\pm}}.
\end{align*}
First, we note a trivial estimate 
\begin{align*}
 \|P_{\la,i,F_{+}}
\widetilde{(f_{\la,\kappa}\chi_{\frac{1}{3},i,\la,\delta})}\|
_{L^2(W_0,\hat{\mu}_{\la,T_i,W_0})}\le C.
\end{align*}
This follows from
$\|\widetilde{(f_{\la,\kappa}\chi_{\frac{1}{3},i,\la,\delta})}
\|_{L^2(\hat{\mu}_{\la,T_i,W_0})}\le 1$ which itself follows from
$\|f_{\la,\kappa}\|_{L^2(\YD,\mu_{\la,\YD})}=1$.
Also, note that
\begin{align}
& \lim_{n\to\infty}
\int_{B_{i,\ep}}
\widetilde{(f_{\la,i,n,\kappa}\chi_{\frac{1}{3},i,\la,\delta})}(\eta)
\cdot (-\la^{-1}L_{\la,i})
P_{\la,i,F_{+}}
\widetilde{(f_{\la,\kappa}\chi_{\frac{1}{3},i,\la,\delta})}(\eta)
d\hat{\mu}_{\la,T_i,W_0}(\eta)
\nn\\
&\quad=\int_{B_{i,\ep}}
\widetilde{(f_{\la,\kappa}\chi_{\frac{1}{3},i,\la,\delta})}(\eta)
\cdot (-\la^{-1}L_{\la,i})
P_{\la,i,F_{+}}\widetilde{(f_{\la,\kappa}
\chi_{\frac{1}{3},i,\la,\delta})}(\eta)d\hat{\mu}_{\la,T_i,W_0}(\eta)
\nn\\
&\quad=\int_{W_0}
\widetilde{(f_{\la,\kappa}\chi_{\frac{1}{3},i,\la,\delta})}(\eta)
\cdot (-\la^{-1}L_{\la,i})
P_{\la,i,F_{+}}\widetilde{(f_{\la,\kappa}
\chi_{\frac{1}{3},i,\la,\delta})}(\eta)d\hat{\mu}_{\la,T_i,W_0}(\eta)\nn\\
&\quad \ge (\beta+r-r_-)\|P_{\la,i,F_{+}}
\widetilde{(f_{\la,\kappa}\chi_{\frac{1}{3},i,\la,\delta})}\|^2
_{L^2(W_0,\hat{\mu}_{\la,T_i,W_0})},\label{fn approximation}
\end{align}
where in the final step,
we have used 
$\inf (F_{+}\cap\sigma(-\la^{-1}L_{\la,i}))\ge \beta+r-r_{-}$
which follows from the assumption on $\beta$.
To prove the desired estimate, we are going to estimate 
$\limsup_{n\to\infty}I(n)$, where
\[
 I(n)=\int_{B_{i,\ep}}
\widetilde{(f_{\la,i,n,\kappa}\chi_{\frac{1}{3},i,\la,\delta})}(\eta)
\cdot (-\la^{-1}L_{\la,i})
P_{\la,i,F_{+}}
\widetilde{(f_{\la,\kappa}\chi_{\frac{1}{3},i,\la,\delta})}(\eta)
d\hat{\mu}_{\la,T_i,W_0}(\eta).
\]
By Lemma~\ref{expression of L in local chart}, we have
\begin{align*}
I(n)
&=-\la^{-1}\int_{B_{i,\ep}}L_{\la,i}
(\tilde{f}_{\la,i,n,\kappa}\tilde{\chi}_{\frac{1}{3},i,\la,\delta})\cdot 
P_{\la,i,F_{+}}
\widetilde{(f_{\la,\kappa}\chi_{\frac{1}{3},i,\la,\delta})}
d\hat{\mu}_{\la,T_i,W_0}\\
&=-\la^{-1}\int_{B_{i,\ep}}
\left(\tilde{\chi}_{\frac{1}{3},i,\la,\delta} 
L_{\la,i}\tilde{f}_{\la,i,n,\kappa}+\tilde{f}_{\la,i,n,\kappa}
L_{\la,i}\tilde{\chi}_{\frac{1}{3},i,\la,\delta}+
2 D_0\tilde{f}_{\la,i,n,\kappa}\cdot 
D_0\tilde{\chi}_{\frac{1}{3},i,\la,\delta}\right)
\nn\\
&\qquad\qquad\qquad \times
P_{\la,i,F_{+}}\widetilde{(f_{\la,\kappa}\chi_{\frac{1}{3},i,\la,\delta})}
d\hat{\mu}_{\la,T_i,W_0}\\
&=-\la^{-1}\int_{B_{i,\ep}}
\left(
\tLS \tilde{f}_{\la,i,n,\kappa}-
\tr\left(\tilde{K}_iD_0^2\tilde{f}_{\la,i,n,\kappa}\right)
+(\tilde{B}_{\la,i},D_0\tilde{f}_{\la,i,n,\kappa})
\right)\nn\\
&\qquad\qquad\qquad\qquad \times \tilde{\chi}_{\frac{1}{3},i,\la,\delta}
P_{\la,i,F_{+}}\widetilde{(f_{\la,\kappa}\chi_{\frac{1}{3},i,\la,\delta})}
d\hat{\mu}_{\la,T_i,W_0}\nn\\
&\quad-\la^{-1}\int_{B_{i,\ep}}
\tilde{f}_{\la,i,n,\kappa}
L_{\la,i}\tilde{\chi}_{\frac{1}{3},i,\la,\delta}P_{\la,i,F_{+}}
\widetilde{(f_{\la,\kappa}\chi_{\frac{1}{3},i,\la,\delta})}
\tilde{\chi}_{\frac{1}{3},i,\la,\delta}d\hat{\mu}_{\la,T_i,W_0}\nn
\\
&\quad-\la^{-1}
\int_{B_{i,\ep}}
2 D_0\tilde{f}_{\la,i,n,\kappa}\cdot 
D_0\tilde{\chi}_{\frac{1}{3},i,\la,\delta}
P_{\la,i,F_{+}}\widetilde{(f_{\la,\kappa}\chi_{\frac{1}{3},i,\la,\delta})}
d\hat{\mu}_{\la,T_i,W_0}\nn\\
&=:I_1+I_2+I_3.
\end{align*}
Let $I_1=:I_{1,1}+I_{1,2}+I_{1,3}$.
We have
\begin{align*}
 I_{1,1}&=
\beta\int_{B_{i,\ep}}\tilde{f}_{\la,i,n,\kappa}
\tilde{\chi}_{\frac{1}{3},i,\la,\delta}
P_{\la,i,F_{+}}
\widetilde{(f_{\la,\kappa}
\chi_{\frac{1}{3},i,\la,\delta})}d\hat{\mu}_{\la,T_i,W_0}\\
&\quad +
\int_{B_{i,\ep}}\tilde{u}_{\la,i,n,\kappa}
\tilde{\chi}_{\frac{1}{3},i,\la,\delta}
P_{\la,i,F_{+}}
\widetilde{(f_{\la,\kappa}\chi_{\frac{1}{3},i,\la,\delta})}
d\hat{\mu}_{\la,T_i,W_0}.
\end{align*}
By $\lim_{n\to\infty}\|f_{\la,i,n,\kappa}
-f_{\la,\kappa}\chi_{i,\ep}\|_{L^2(S_a,\mu_{\la})}=0$,
and the fact that $\chi_{\frac{1}{3},i,\la,\delta}\ne 0$
implies $\chi_{i,\ep}=1$ for large $\la$,
we have
\begin{align*}
 \lim_{n\to\infty}\int_{B_{i,\ep}}\tilde{f}_{\la,i,n,\kappa}
\tilde{\chi}_{\frac{1}{3},i,\la,\delta}
P_{\la,i,F_{+}}\widetilde{(f_{\la,\kappa}\chi_{\frac{1}{3},i,\la,\delta})}
d\hat{\mu}_{\la,T_i,W_0}
&=\int_{W_0}\widetilde{f_{\la,\kappa}\chi_{\frac{1}{3},i,\la,\delta}}\cdot
P_{\la,i,F_{+}}\widetilde{(f_{\la,\kappa}\chi_{\frac{1}{3},i,\la,\delta})}
d\hat{\mu}_{\la,T_i,W_0}\nn\\
&=\|P_{\la,i,F_{+}}
\widetilde{(f_{\la,\kappa}\chi_{\frac{1}{3},i,\la,\delta})}\|^2
_{L^2(W_0,\hat{\mu}_{\la,T_i,W_0})}
\end{align*}
Using (\ref{estimate of u}),
and Lemma~\ref{expansion of F and G},
we get
\begin{align*}
&\limsup_{n\to\infty} 
\left|\int_{B_{i,\ep}}\tilde{u}_{\la,i,n,\kappa}
\tilde{\chi}_{\frac{1}{3},i,\la,\delta}
P_{\la,i,F_{+}}\widetilde{(f_{\la,\kappa}\chi_{\frac{1}{3},i,\la,\delta})}
d\hat{\mu}_{\la,T_i,W_0}\right|\nn\\
&\quad \le C \limsup_{n\to\infty}
\|u_{\la, i,n,\kappa}\|
_{L^2(\tilde{V}^{S_a}_{\ep}(k_i),\mu_{\la,a})}
\|P_{\la,i,F_{+}}
\widetilde{(f_{\la,\kappa}\chi_{\frac{1}{3},i,\la,\delta})}\|
_{L^2(B_{i,\ep},\hat{\mu}_{\la,T_i,W_0})}
\le C\kappa.
\end{align*}
Hence 
\begin{align*}
  \limsup_{n\to\infty}I_{1,1}\le 
\beta \|P_{\la,i,F_{+}}
\widetilde{(f_{\la,\kappa}\chi_{\frac{1}{3},i,\la,\delta})}\|^2
_{L^2(W_0,\hat{\mu}_{\la,T_i,W_0})}+C\kappa.
\end{align*}
Using the integration by parts formula.
\begin{align*}
& I_{1,2}=\la^{-1}\int_{B_{i,\ep}}
\Bigl(
\sum_{k,l}(\tilde{K}_ie_k,e_l)
(D_0^2\tilde{f}_{\la,i,n,\kappa})[e_k,e_l]
\Bigr)
\tilde{\chi}_{\frac{1}{3},i,\la,\delta}
P_{\la,i,F_{+}}\widetilde{(f_{\la,\kappa}\chi_{\frac{1}{3},i,\la,\delta})}
d\hat{\mu}_{\la,T_i,W_0}\nn\\
&\,=\sum_{k=1}^4I_{1,2,k},
\end{align*}
where
\begin{align*}
 I_{1,2,1}&=-\la^{-1}
\int_{B_{i,\ep}}
\Bigl(\sum_{k,l}\left(D_{e_k}\tilde{K}_ie_k,e_l\right)
D_0\tilde{f}_{\la,i,n,\kappa}[e_l]\Bigr)
\tilde{\chi}_{\frac{1}{3},i,\la,\delta}
P_{\la,i,F_{+}}
\widetilde{(f_{\la,\kappa}\chi_{\frac{1}{3},i,\la,\delta})}
d\hat{\mu}_{\la,T_i,W_0}\nn\\
I_{1,2,2}&=-\la^{-1}
\int_{B_{i,\ep}}
\Bigl(\sum_{k,l}
\left(\tilde{K}_ie_k,e_l\right)
D_0\tilde{f}_{\la,i,n,\kappa}[e_l]
D_0\tilde{\chi}_{\frac{1}{3},i,\la,\delta}[e_k]
\Bigr)
P_{\la,i,F_{+}}
\widetilde{(f_{\la,\kappa}\chi_{\frac{1}{3},i,\la,\delta})}
d\hat{\mu}_{\la,T_i,W_0},\nn\\
I_{1,2,3}&=
-\la^{-1}
\int_{B_{i,\ep}}
\Bigl(\sum_{k,l}
\left(\tilde{K}_ie_k,e_l\right)
D_0\tilde{f}_{\la,i,n,\kappa}[e_l]
D_0(P_{\la,i,F_{+}}
\widetilde{(f_{\la,\kappa}\chi_{\frac{1}{3},i,\la,\delta})})[e_k]
\Bigr)\tilde{\chi}_{\frac{1}{3},i,\la,\delta}d\hat{\mu}_{\la,T_i,W_0},\nn\\
I_{1,2,4}&=
\int_{B_{i,\ep}}
\Bigl(\sum_{k,l}
\left(\tilde{K}_ie_k,e_l\right)
D_0\tilde{f}_{\la,i,n,\kappa}[e_l]
\left((\eta,e_k)_{H_0}+(T_i\eta,e_k)\right)
\Bigr)\tilde{\chi}_{\frac{1}{3},i,\la,\delta}
P_{\la,i,F_{+}}\widetilde{(f_{\la,\kappa}\chi_{\frac{1}{3},i,\la,\delta})}
d\hat{\mu}_{\la,T_i,W_0}\nn,
\end{align*}
and $\{e_k\}$ is a complete orthonormal system of $H_0$

We estimate each terms.
As for $I_{1,2,1}$, using the estimate of $\tilde{\tr}$ in 
Lemma~\ref{expression of L in local chart},
Remark~\ref{remark on dirichlet form} (2), (\ref{dirichlet norm of f})
and the Schwarz inequality, we have
\begin{align*}
 |I_{1,2,1}|&\le \la^{-1}
\int_{B_{i,\ep}}|\tr D_0\tilde{K}_i|_{H_0}|\,
|D_0\tilde{f}_{\la,i,n,\kappa}|_{H_0}\,
|P_{\la,i,F_{+}}\widetilde{(f_{\la,\kappa}\chi_{\frac{1}{3},i,\la,\delta})}|
|\tilde{\chi}_{\frac{1}{3},i,\la,\delta}| d\hat{\mu}_{\la,T_i,W_0}\nn\\
&\le C \la^{-1-\delta/2}
\|\tilde{\chi}_{\frac{1}{3},i,\la,\delta}D_0\tilde{f}_{\la,i,n,\kappa}\|
_{L^2(B_{i,\ep},\hat{\mu}_{\la,T_i,W_0})}
\|P_{\la,i,F_{+}}
\widetilde{(f_{\la,\kappa}\chi_{\frac{1}{3},i,\la,\delta})}\|
_{L^2(\hat{\mu}_{\la,T_i,W_0})}.
\end{align*}
Hence using (\ref{D_0 dirichlet norm of f}),
we obtain
\begin{align*}
 \limsup_{n\to\infty} |I_{1,2,1}|&\le C'_{\beta}\la^{-\frac{1+\delta}{2}}.
\end{align*}
For $I_{1,2,2}$, using the estimate $|\tilde{K}_i|_{\LL(H_0,H_0)}
=O(\|\bar{\eta}\|_{\alpha}^2)$ 
(Lemma~\ref{change of variable dirichlet form}), 
similarly, we have
\begin{align*}
\limsup_{n\to\infty} |I_{1,2,2}|&\le 
\limsup_{n\to\infty}\la^{-1}\int_{B_{i,\ep}}
|\tilde{K}_iD_0\tilde{\chi}_{\frac{1}{3},i,\la,\delta}|_{H_0} 
|D_0\tilde{f}_{\la,i,n,\kappa}|_{H_0}
|P_{\la,i,F_{+}}
\widetilde{(f_{\la,\kappa}\chi_{\frac{1}{3},i,\la,\delta})}|_{H_0}
d\hat{\mu}_{\la,T_i,W_0}\nn\\
&\le C_{\beta}\la^{-\frac{1+\delta}{2}}.
\end{align*}
Also,
\begin{align*}
& \limsup_{n\to\infty}
|I_{1,2,3}|\nn\\
&\le \la^{-1}\limsup_{n\to\infty}
\int_{B_{i,\ep}}|\tilde{K}_i|_{\LL(H_0,H_0)}
|D_0P_{\la,i,F_{+}}
\widetilde{(f_{\la,\kappa}\chi_{\frac{1}{3},i,\la,\delta})}|_{H_0}
|D_0\tilde{f}_{\la,i,n,\kappa}|_{H_0}
\tilde{\chi}_{\frac{1}{3},i,\la,\delta}
d\hat{\mu}_{\la,T_i,W_0}\nn\\
&\le \la^{-1-\delta}
\||D_0P_{\la,i,F_{+}}
\widetilde{(f_{\la,\kappa}\chi_{\frac{1}{3},i,\la,\delta})}|_{H_0}\|
_{L^2(W_0,\hat{\mu}_{\la,T_i,W_0})}
\limsup_{n\to\infty}\||D_0\tilde{f}_{\la,i,n,\kappa}|_{H_0}\|_
{L^2(B_{i,\ep},\hat{\mu}_{\la,T_i,W_0})}
\nn\\
&\le C_{\beta}\la^{-\delta}\left(
\|\la^{-1}\hat{L}_{\la,T_i,W_0}
P_{\la,i,F_{+}}\widetilde{(f_{\la,\kappa}\chi_{\frac{1}{3},i,\la,\delta})}
\|_{L^2(W_0,\hat{\mu}_{\la,T_i,W_0})}
\|P_{\la,i,F_{+}}\widetilde{(f_{\la,\kappa}\chi_{\frac{1}{3},i,\la,\delta})}\|
_{L^2(W_0,\hat{\mu}_{\la,T_i,W_0})}
\right)^{1/2}\nn\\
&\le C_{\beta}R^{1/2}\la^{-\delta}\|P_{\la,i,F_{+}}\widetilde{(f_{\la,\kappa}\chi_{\frac{1}{3},i,\la,\delta})}\|
_{L^2(W_0,\hat{\mu}_{\la,T_i,W_0})}\nn\\
&\le C_{\beta}R^{1/2}\la^{-\delta},
\end{align*}
where we have used that $F_+\subset [0,R]$.
For $I_{1,2,4}$, we have
\begin{align*}
& \limsup_{n\to\infty}|I_{1,2,4}|\nn\\
&\le
\la^{-1}\limsup_{n\to\infty}
\int_{B_{i,\ep}}
\left|\left(\tilde{K}_i(I_{H_0}+T_i)\eta, 
D_0\tilde{f}_{\la,i,n,\kappa}(\eta)\right)_{H_0}
\right|
|P_{\la,i,F_{+}}
\widetilde{(f_{\la,\kappa}\chi_{\frac{1}{3},i,\la,\delta})}|
|\tilde{\chi}_{\frac{1}{3},i,\la,\delta}|
d\hat{\mu}_{\la,T_i,W_0}(\eta)\nn\\
&\le C_{\beta} \la^{-\frac{1+3\delta}{2}}.
\end{align*}
We next consider $I_{1,3}$.
\begin{align*}
 \left|I_{1,3}\right|&\le
\int_{B_{i,\ep}}|\la^{-1}\tilde{B}_{\la,i}|_{H_0}
|D_0\tilde{f}_{\la,i,n,\kappa}|_{H_0}
|P_{\la,i,F_{+}}
\widetilde{(f_{\la,\kappa}\chi_{\frac{1}{3},i,\la,\delta})}|
\tilde{\chi}_{\frac{1}{3},i,\la,\delta}
d\hat{\mu}_{\la,T_i,W_0}\nn\\
&\le \la^{-\delta}\||D_0\tilde{f}_{\la,i,n,\kappa}|_{H_0}\|
_{L^2(B_{i,\ep},\hat{\mu}_{\la,T_i,W_0})}
\|P_{\la,i,F_{+}}
\widetilde{(f_{\la,\kappa}\chi_{\frac{1}{3},i,\la,\delta})}\|
_{L^2(W_0,\hat{\mu}_{\la,T_i,W_0})}.
\end{align*}
Hence, using (\ref{D_0 dirichlet norm of f}), we obtain
\begin{align*}
 \limsup_{n\to\infty}|I_{1,3}|\le C\la^{\frac{1}{2}-\delta}.
\end{align*}
We estimate $I_2$.
By the integration by parts formula, we have
\begin{align*}
 I_2
&=\la^{-1}\int_{B_{i,\ep}}
\left(D_0\tilde{f}_{\la,i,n,\kappa}, 
D_0\tilde{\chi}_{\frac{1}{3},i,\la,\delta}\right)
P_{\la,i,F_{+}}\widetilde{(f_{\la,\kappa} 
\chi_{\frac{1}{3},i,\la,\delta})}
\tilde{\chi}_{\frac{1}{3},i,\la,\delta}d\hat{\mu}_{\la,T_i,W_0}\nn\\
&\quad +
\la^{-1}\int_{B_{i,\ep}}
\left(D_0{P_{\la,i,F_{+}}
\widetilde{(f_{\la,\kappa} \chi_{\frac{1}{3},i,\la,\delta})}}, 
D_0\tilde{\chi}_{\frac{1}{3},i,\la,\delta}\right)
\tilde{f}_{\la,i,n,\kappa}\tilde{\chi}_{\frac{1}{3},i,\la,\delta}
d\hat{\mu}_{\la,T_i,W_0}\nn\\
&\quad +
\la^{-1}\int_{B_{i,\ep}}
|D_0\tilde{\chi}_{\frac{1}{3},i,\la,\delta}|_{H_0}^2
P_{\la,i,F_{+}}(\widetilde{f_{\la,\kappa} \chi_{\frac{1}{3},i,\la,\delta}})
\tilde{f}_{\la,i,n,\kappa} d\hat{\mu}_{\la,T_i,W_0}\nn\\
&=I_{2,1}+I_{2,2}+I_{2,3}.
\end{align*}
Using Lemma~\ref{cut-off function} and (\ref{D_0 dirichlet norm of f}), 
we have
\begin{align*}
& \limsup_{n\to\infty}|I_{2,1}|\nn\\
&\,\le \la^{-1}
\limsup_{n\to\infty}\||D_0\tilde{f}_{\la,i,n,\kappa}|_{H_0}\|
_{L^2(B_{i,\ep},\hat{\mu}_{\la,T_i,W_0})}\nn\\
&\qquad \times \|P_{\la,i,F_{+}}(\widetilde{f_{\la,\kappa} 
\chi_{\frac{1}{3},i,\la,\delta}})\|
_{L^2(B_{i,\ep},\hat{\mu}_{\la,T_i,W_0})}
\||D_0\tilde{\chi}_{\frac{1}{3},i,\la,\delta}|_{H_0}\|_{L^{\infty}}\nn\\
&\le C\la^{-\frac{1-\delta}{2}}.
\end{align*}
For the estimate of $I_{2,2}$, using
\begin{align*}
& \left\|\left|D_0\left\{P_{\la,i,F_{+}}(\widetilde{f_{\la,\kappa} 
\chi_{\frac{1}{3},i,\la,\delta}})\right\}\right|_{H_0}\right\|^2
_{L^2(B_{i,\ep},\hat{\mu}_{\la,T_i,W_0})}\nn\\
&\qquad\le
 \left\|\left|D_0\left\{P_{\la,i,F_{+}}(\widetilde{f_{\la,\kappa} 
\chi_{\frac{1}{3},i,\la,\delta}})\right\}\right|_{H_0}\right\|^2
_{L^2(W_0,\hat{\mu}_{\la,T_i,W_0})}\nn\\
&\qquad=\int_{W_0}
-L_{\la,i}P_{\la,i,F_{+}}(\widetilde{f_{\la,\kappa} 
\chi_{\frac{1}{3},i,\la,\delta}})\cdot
P_{\la,i,F_{+}}(\widetilde{f_{\la,\kappa} 
\chi_{\frac{1}{3},i,\la,\delta}})d\hat{\mu}_{\la,T_i,W_0}\nn\\
&\qquad  \le C R\la,
\end{align*}
we have
\begin{align*}
 |I_{2,2}|&\le \la^{-1}
\left\|\left|D_0\left\{P_{\la,i,F_{+}}(\widetilde{f_{\la,\kappa} 
\chi_{\frac{1}{3},i,\la,\delta}})\right\}\right|_{H_0}\right\|
_{L^2(B_{i,\ep},\hat{\mu}_{\la,T_i,W_0})}\nn\\
&\qquad \times 
\|\tilde{f}_{\la,i,n,\kappa}\|_{L^2(B_{i,\ep},\hat{\mu}_{\la,T_i,W_0})}\,
\||D_0\tilde{\chi}_{\frac{1}{3},i,\la,\delta}|_{H_0}\|_{L^\infty}\nn\\
& \le C R^{1/2}\la^{-1/2}\la^{\delta/2}=C R^{1/2}\la^{-\frac{1-\delta}{2}}.
\end{align*}
Also we have
\begin{align*}
 |I_{2,3}|&\le \la^{-1} \|\tilde{f}_{\la,i,n,\kappa}\|
_{L^2(B_{i,\ep},\hat{\mu}_{\la,T_i,W_0})}
\|P_{\la,i,F_{+}}(\widetilde{f_{\la,\kappa} 
\chi_{\frac{1}{3},i,\la,\delta}})\|
_{L^2(B_{i,\ep},\hat{\mu}_{\la,T_i,W_0})}
\||D_0\tilde{\chi}_{\frac{1}{3},i,\la,\delta}|_{H_0}\|_{L^{\infty}}^2\nn\\
&\le C \la^{\delta-1}.
\end{align*}
Finally, we have
\begin{align*}
 |I_3|\le C\la^{-1}
\||D_0\tilde{f}_{\la,i,n,\kappa}|_{H_0}\|
_{L^2(B_{i,\ep},\hat{\mu}_{\la,T_i,W_0})}
\|P_{\la,i,F_{+}}(\widetilde{f_{\la,\kappa} 
\chi_{\frac{1}{3},i,\la,\delta}})\|
_{L^2(B_{i,\ep},\hat{\mu}_{\la,T_i,W_0})}
\||D_0\tilde{\chi}_{\frac{1}{3},i,\la,\delta}|_{H_0}\|_{L^{\infty}},
\end{align*}
which implies
\begin{align}
 \limsup_{n\to\infty}|I_3|\le
C\la^{-\frac{1-\delta}{2}}.
\end{align}
Consequently, using (\ref{fn approximation}) and combining the above, we get
\begin{align}
&\Biggl|\int_{W_0}
\widetilde{(f_{\la,\kappa}\chi_{\frac{1}{3},i,\la,\delta})}(\eta)
\cdot (-\la^{-1}L_{\la,i})
P_{\la,i,F_{+}}\widetilde{(f_{\la,\kappa}\chi_{\frac{1}{3},i,\la,\delta})}(\eta)d\hat{\mu}
_{\la,T_i,W_0}(\eta)\nn\\
&\qquad
-\beta
\|P_{\la,i,F_{+}}\widetilde{(f_{\la,\kappa}\chi_{\frac{1}{3},i,\la,\delta})}\|^2
_{L^2(W_0,\hat{\mu}_{\la,T_i,W_0})}
\Biggr|\nn\\
&\quad \le C_{\beta,R}\,(\kappa+\la^{-\frac{1-\delta}{2}}).
\label{P_F estimate}
\end{align}
Again using (\ref{fn approximation}), we have
\begin{align*}
& (\beta+r-r_-)\|P_{\la,i,F_{+}}
\widetilde{(f_{\la,\kappa}\chi_{\frac{1}{3},i,\la,\delta})}\|^2
_{L^2(W_0,\hat{\mu}_{\la,T_i,W_0})}\\
&\qquad\le
\beta\|P_{\la,i,F_{+}}\widetilde{(f_{\la,\kappa}
\chi_{\frac{1}{3},i,\la,\delta})}\|^2
_{L^2(W_0,\hat{\mu}_{\la,T_i,W_0})}+
C_{\beta,R}\,(\kappa+\la^{-\frac{1-\delta}{2}}).
\end{align*}
and
\begin{align*}
\|P_{\la,i,F_{+}}\widetilde{(f_{\la,\kappa}
\chi_{\frac{1}{3},i,\la,\delta})}\|^2
_{L^2(W_0,\hat{\mu}_{\la,T_i,W_0})}\le (r-r_-)^{-1}
C_{\beta,R}\, (\kappa+\la^{-\frac{1-\delta}{2}}).
\end{align*}

For the estimate
$\|P_{\la,i,F_{-}}\widetilde{(f_{\la,\kappa}
\chi_{\frac{1}{3},i,\la,\delta})}\|
_{L^2(W_0,\hat{\mu}_{\la,T_i,W_0})}$,
note that
\begin{align*}
&\int_{W_0}
\widetilde{(f_{\la,\kappa}\chi_{\frac{1}{3},i,\la,\delta})}(\eta)
\cdot (-\la^{-1}L_{\la,i})
P_{\la,i,F_{-}}
\widetilde{(f_{\la,\kappa}\chi_{\frac{1}{3},i,\la,\delta})}(\eta)
d\hat{\mu}_{\la,T_i,W_0}(\eta)\nn\\
&\qquad
\le (\beta-(r-r_-))\|P_{\la,i,F_{-}}
\widetilde{(f_{\la,\kappa}\chi_{\frac{1}{3},i,\la,\delta})}\|^2
_{L^2(W_0,\hat{\mu}_{\la,T_i,W_0})}.
\end{align*}
By a similar argument to the case of the projection $P_{\la,i,F_{+}}$,
we see that the same inequality as (\ref{P_F estimate}) which is obtained 
by replacing $P_{\la,i,F_{+}}$ by $P_{\la,i,F_{-}}$ holds.
Hence we get
\begin{align*}
\|P_{\la,i,F_{-}}
\widetilde{(f_{\la,\kappa}\chi_{\frac{1}{3},i,\la,\delta})}\|^2
_{L^2(W_0,\hat{\mu}_{\la,T_i,W_0})}\le (r-r_-)^{-1}
C_{\beta,R}\, (\kappa+\la^{-\frac{1-\delta}{2}}).
\end{align*}
This completes the proof.
\end{proof}

\begin{lem}\label{projection lemma}
 Let $A$ be a self-adjoint operator on a Hilbert space $H$.
Let $\{f_i\}_{i=1}^N\in \rD(A)$, $\{\alpha_i\}_{i=1}^N\subset \RR$
and $0<\ep<\frac{1}{2}$.
Suppose for all $i,j$
\begin{align*}
 \|f_i\|_H=1,\quad |(f_i,f_j)|\le \ep,\quad 
\|Af_i-\alpha_i f_i\|_H\le \ep
\end{align*}
hold.
Let $P_B=1_B(A)$ $(B\subset \RR)$.
\begin{enumerate}
 \item[$(1)$] We have $\|P_{[\alpha_i-\sqrt{\ep},\alpha_i+\sqrt{\ep}]}f_i\|^2
\ge 1-\ep$ hold.
In particular, $\sigma(A)\cap [\alpha_i-\ep,\alpha_i+\ep]\ne \emptyset$.
\item[$(2)$] Let $\tilde{f}_i=P_{[\alpha_i-\sqrt{\ep},\alpha_i+\sqrt{\ep}]}f_i
/\|P_{[\alpha_i-\sqrt{\ep},\alpha_i+\sqrt{\ep}]}f_i\|$.
Then we have
\begin{align*}
\|A\tilde{f}_i-\alpha_i\tilde{f}_i\|\le \sqrt{\ep},\quad
\|\tilde{f}_i-f_i\|_H\le 2\sqrt{2\ep},\quad
|(\tilde{f}_i,\tilde{f}_j)|\le 17\sqrt{\ep}.
\end{align*}
\item[$(3)$] If $\ep$ is sufficiently small, 
then $\{f_i\}_{i=1}^N$ are linearly independent and
so is $\{\tilde{f}_i\}_{i=1}^N$.
\end{enumerate}
\end{lem}

\begin{proof}
(1) We have 
\begin{align*}
 \|Af_i-\alpha_if_i\|_H^2&=
\Bigl\|(A-\alpha_i)P_{[\alpha_i-\sqrt{\ep},\alpha+\sqrt{\ep}]}f_i
+(A-\alpha_i)P_{(\alpha_i+\sqrt{\ep},\infty)}f_i
+(A-\alpha_i)P_{(-\infty,\alpha_i+\sqrt{\ep})}f_i
\Bigr\|_H^2\\
&\ge \|(A-\alpha_i)P_{(\alpha_i+\sqrt{\ep},\infty)}f_i\|^2+
\|(A-\alpha_i)P_{(-\infty,\alpha_i-\sqrt{\ep})}f_i\|^2\\
&\ge \ep\left(\|P_{(\alpha_i+\sqrt{\ep},\infty)}f_i\|^2+
\|P_{(-\infty,\alpha_i-\sqrt{\ep})}f_i\|^2\right),
\end{align*}
which implies
$\|P_{[\alpha_i-\sqrt{\ep},\alpha_i+\sqrt{\ep}]}f_i\|_H^2\ge 1-\ep$.

\noindent
(2) 
The first inequality follows from
$A\tilde{f}_i-\alpha_i\tilde{f}_i
=(A-\alpha_i)P_{[\alpha_i-\sqrt{\ep},\alpha_i+\sqrt{\ep}]}\tilde{f}_i$.
We prove the remaining inequalities.
We have
\begin{align*}
 \|\tilde{f}_i-f_i\|_H&
=\left\|
\frac{P_{[\alpha_i-\sqrt{\ep},\alpha_i+\sqrt{\ep}]}f_i}
{\|P_{[\alpha_i-\sqrt{\ep},\alpha_i+\sqrt{\ep}]}f_i\|_H}
-f_i\right\|_H\\
&\le \left\|\frac{P_{[\alpha_i-\sqrt{\ep},\alpha_i+\sqrt{\ep}]}f_i-f_i}
{\|P_{[\alpha_i-\sqrt{\ep},\alpha_i+\sqrt{\ep}]}f_i\|_H}\right\|
+\left\|\frac{1-\|P_{[\alpha_i-\sqrt{\ep},\alpha_i+\sqrt{\ep}]}f_i\|_H}
{\|P_{[\alpha_i-\sqrt{\ep},\alpha_i+\sqrt{\ep}]}f_i\|_H}f_i\right\|\\
&\le 
2\frac{\|P_{[\alpha_i-\sqrt{\ep},\alpha_i+\sqrt{\ep}]^{\complement}}f_i\|_H}
{\|P_{[\alpha_i-\sqrt{\ep},\alpha_i+\sqrt{\ep}]}f_i\|_H}\\
&\le 2\frac{\sqrt{\ep}}{\sqrt{1-\ep}}\le 2\sqrt{2\ep},
\end{align*}
which implies
\begin{align*}
 |(\tilde{f}_i,\tilde{f}_j)|&\le
|(\tilde{f}_i-f_i,\tilde{f}_j-f_j)|+
|(f_i,\tilde{f}_j-f_j)|+|(\tilde{f}_i-f_i,f_j)|+
|(f_i,f_j)|\\
&\le 8\ep+4\sqrt{2\ep}+\ep\le 17\sqrt{\ep}.
\end{align*}

\noindent
(3) This is an elementary fact in linear algebra.
\end{proof}

\begin{proof}[Proof of Theorem~$\ref{main theorem}$]
For each $1\le i\le N$, counting multiplicity, let
\begin{align*}
\{e^i_{l}\}_{l=1}^{l(i)}:=\left\{\Lambda(\xi_i)
\cap [0,R]\right\}\setminus
\left(\cup_{i=1}^N(\Lambda(\xi_i)^a+(-r,r))\right).
\end{align*}
Of course
$\{e_j\}_{j=1}^L=\cup_{i=1}^N\{e^i_l\}_{l=1}^{l(i)}$ holds.
Also, by the assumption
$R\notin \cup_{i=1}^N\Lambda(\xi_i)$,
it holds that $e_j<R$ $(1\le j\le L)$.

We constructed approximate eigenfunctions
$\bce_{\mathbf{n},T_{\xi_i},\la}$
 localized in the neighborhood
of $k_i\in S_a$ whose associated eigenvalues
are $E_{\mathbf{n}}(\xi_i)$.
So for each $e^i_l$, there exists $\mathbf{n}$ such that
$e^i_l=E_{\mathbf{n}}(\xi_i)$ hold and
$\bce_{\mathbf{n},T_{\xi_i},\la}
=\varPsi^{\ast}_{k_i}(\mathbf{\hat{e}}_{\mathbf{n},T_{\xi_i},\la})
$ 
is the associated approximate eigenfunction.
For simplicity, in this proof, we denote the normalized function in
$L^2(\YD,\mu_{\la,\YD})$ by $\bce^i_{l,\la}$ corresponding to
$e^{i}_l$.
Note that this function is constant multiple of 
$\bce_{\mathbf{n},T_{\xi_i},\la}$ and the constant is of $1+O(e^{-c\la})$.
Also we denote $\mathbf{e}_{\mathbf{n},T_{\xi_i},\la}$ 
which is an eigenfunction
in $L^2(W_0,\hat{\mu}_{\la,T_i,W_0})$ by $\mathbf{e}^i_{l,\la}$.
Then we have, for large $\la$, if $l\ne m$,
\begin{align}
& \left(\bce^i_{l,\la},
\bce^i_{m,\la}\right)_{L^2(\YD,\mu_{\la,\YD})}=C^{i,l,m}_1(\la),\\
&\|\la^{-1}L_{\la,\D} \bce^i_{l,\la}
+e^i_l\bce^i_{l,\la}\|_{L^2(\YD,\mu_{\la,\YD})}=C^{i,l}_2(\la),
\label{on spectrum}
\end{align}
where we have the estimate
\begin{align*}
 C(\la):=\max_{i,l,m}(|C^{i,l,m}_1(\la)|, |C^{i,l}_2(\la)|)
=O(\la^{1-(3\delta/2)})
+O(\la^{-\delta/2})
+O(\la^{\frac{1}{2}-\delta}).
\end{align*}

Let $\mathbf{f}^i_{l,\la}$ be the normalized vector
of $P_{\la,\left[e^i_l-C(\la)^{1/2},
e^i_l+C(\la)^{1/2}\right]} 
\bce^i_{l,\la}$.
Then, if $\la$ is sufficiently large, then $\mathbf{f}^i_{l,\la}$ 
is linearly independent by Lemma~\ref{projection lemma}.
Note that if $e^i_l\ne e^{i'}_{l'}$, then 
$\mathbf{f}^{i}_l$ and $\mathbf{f}^{i'}_{l'}$ are orthogonal for large
$\la$ because $\left[e^i_l-C(\la)^{1/2},
e^i_l+C(\la)^{1/2}\right]\cap
\left[e^{i'}_{l'}-C(\la)^{1/2},
e^{i'}_{l'}+C(\la)^{1/2}\right]=\emptyset$.
We prove main theorem in two steps.

\begin{itemize}
 \item[(i)] Let $\{e_{j_k}\}_{k=1}^K$ be distinct values of $\{e_j\}_{j=1}^L$
and $N(e_{j_k})$ denotes the multiplicity.
Suppose $\la$ is large enough so that
\begin{align}
& \left[e_{j_k}-2C(\la)^{1/2},e_{j_k}+2C(\la)^{1/2}\right]\cap 
\left[e_{j_l}-2C(\la)^{1/2},e_{j_l}+2C(\la)^{1/2}\right]
=\emptyset\quad \text{if $k\ne l$}\\
& \left[e_{j_k}-2C(\la)^{1/2},e_{j_k}+2C(\la)^{1/2}\right]\cap 
\left(\cup_{i=1}^N(\Lambda(\xi_i)^a+(-r,r))\right)=\emptyset
\quad \text{for all $k$.}\label{r+}
\end{align}
Then for large $\la$, it holds that
$\dim \rIm P_{\la,\left[e_{j_k}-2C(\la)^{1/2},
e_{j_k}+2C(\la)^{1/2}\right]}=N(e_{j_k})$.
\item[(ii)] The set $[0,R]\cap \left(\cup_{k=1}^K
\left(e_{j_k}-\frac{3}{2}C(\la)^{1/2},e_{j_k}+
\frac{3}{2}C(\la)^{1/2}\right)\right)
^\complement
\cap\left(\cup_{i=1}^N(\Lambda(\xi_i)^a+(-r,r))\right)^\complement$
does not intersect the spectral set of
$-\la^{-1}L_{\la,\YD}$.
\end{itemize}
In the above (i), we can choose $\la$ for which (\ref{r+})
holds because of the assumption on $r$.
Also, we already see the existence of linearly independent
$\mathbf{f}^i_{l,\la}\ne 0$
such that 
$P_{\la,\left[e^i_l-C(\la)^{1/2},
e^i_l+C(\la)^{1/2}\right]}\mathbf{f}^i_{l,\la}=\mathbf{f}^i_{l,\la}$.
Hence, actually,
$\dim \rIm P_{\la,\left[e_{j_k}-C(\la)^{1/2},
e_{j_k}+C(\la)^{1/2}\right]}=N(e_{j_k})$ holds.
This shows $|e_j(\la)-e_j|=O(2C(\la)^{1/2})$ $(1\le j\le L)$ and
combining (ii), we see that
$-\la^{-1}L_{\la,\D}$ has only discrete spectrum
$\{e_j(\la)\}_{j=1}^L$ counting multiplicity in
$[0,R]\cap \left(\cup_{i=1}^N(\Lambda(\xi_i)^a+(-r,r))\right)^\complement$
which completes the proof of the theorem.

We now prove (i).
We already showed 
$\dim \rIm P_{\la,[e_{j_k}-2C(\la)^{1/2},
e_{j_k}+2C(\la)^{1/2}]}\ge N(e_{j_k})$.
Suppose there exists $k$ such that 
$\dim \rIm P_{\la,[e_{j_k}-2C(\la)^{1/2}
,e_{j_k}+
2C(\la)^{1/2}]}>N(e_{j_k})$.\\
Then, there exists $f\in L^2(\YD,\mu_{\la,\YD})$ with 
$\|f\|_{L^2(\YD,\mu_{\la,\YD})}$=1 such that
\[
 P_{\la,[e_{j_k}-2C(\la)^{1/2},e_{j_k}+2C(\la)^{1/2}]}f=f,\quad
(f,\mathbf{f}^i_{l,\la})_{L^2(\YD,\mu_{\la,\YD})}=0\quad \text{for all $i,l$.}
\]
This implies
\begin{align}
\la^{-1}\E_{\la,\YD}(f,f)\le e_{j_k}+2C(\la)^{1/2}.\label{upper bound of Ef}
\end{align}
By Lemma~\ref{projection lemma},
$(f,\bce^i_{l,\la})_{L^2(\YD,\mu_{\la,\YD})}=O(C(\la)^{1/4})$ and hence
\begin{align}
 (\widetilde{f\chi_{\frac{1}{3},i',m,\theta,\la,\delta}},
\be^i_{l,\la})_{L^2(W_0,\hat{\mu}_{\la,T_i,W_0})}=O(C(\la)^{1/4}).
\label{norm i'il}
\end{align}
Also, applying Lemma~\ref{projection on F} to the case $f_{\la,\kappa}=f$
and $\kappa=2C(\la)^{1/2}$, we obtain
\begin{align}
 \|P_{\la,i,F}\widetilde{(f\chi_{\frac{1}{3},i',m,\theta,\la,\delta})}\|
_{L^2(W_0,\hat{\mu}_{\la,T_{i'},W_0})}\le 2
(r-r_-)^{-1}C_{i',e_{j_k},R}\left(2C(\la)^{1/2}
+\la^{-\frac{1-\delta}{2}}\right).
\label{norm F}
\end{align}
(\ref{norm i'il}) and (\ref{norm F}) implies
\begin{align*}
 \|P_{\la,i',F\cup\{e_{j_k}\}_{k=1}^K}
\widetilde{(f\chi_{\frac{1}{3},{i'},m,\theta,\la,\delta})}\|
_{L^2(W_0,\hat{\mu}_{\la,T_{i'},W_0})}=O(C(\la)^{1/4})
+O(\la^{-\frac{1-\delta}{2}}).
\end{align*}
This shows
\begin{align*}
& \|P_{\la,i',(F\cup \{e_{j_k}\}_{k=1}^K)^\complement}
\widetilde{(f\chi_{\frac{1}{3},i',m,\theta,\la,\delta})}\|^2
_{L^2(W_0,\hat{\mu}_{\la,T_{i'},W_0})}\nn\\
&\ge
\|\widetilde{(f\chi_{\frac{1}{3},i',m,\theta,\la,\delta})}\|^2
_{L^2(W_0,\hat{\mu}_{\la,T_i',W_0})}
-\|P_{\la,i,(F\cup\{e_{j_k}\}_{k=1}^K)}
\widetilde{(f\chi_{\frac{1}{3},i',m,\theta,\la,\delta})}\|^2
_{L^2(W_0,\hat{\mu}_{\la,T_i',W_0})}
\nn\\
&\ge
\|\widetilde{(f\chi_{\frac{1}{3},i',m,\theta,\la,\delta})}\|
_{L^2(W_0,\hat{\mu}_{\la,T_i',W_0})}^2
-\left(O(C(\la)^{1/4})+O(\la^{-\frac{1-\delta}{2}})\right).
\end{align*}
On the other hand,
\begin{align*}
 \E_{\la,\YD}(f\chi_{\frac{1}{3},i',m,\theta,\la,\delta},
f\chi_{\frac{1}{3},i',m,\theta,\la,\delta})
&=\tilde{\E}_{\la,a,
B_{k_i,\ep}^{\Omega_0}(0)}
(\widetilde{f\chi_{\frac{1}{3},i',m,\theta,\la,\delta}},
\widetilde{f\chi_{\frac{1}{3},i',m,\theta,\la,\delta}}
)\nn\\
&\ge (1-O(\la^{-\delta/2}))
\int_{W_0}\|D_0\widetilde{f\chi_{\frac{1}{3},i',m,\theta,\la,\delta}}\|_{H_0}^2
d\hat{\mu}_{\la,T_i',W_0}
\end{align*}
and
\begin{align*}
&\la^{-1} 
\int_{W_0}\|D_0\widetilde{f\chi_{\frac{1}{3},i',m,\theta,\la,\delta}}\|_{H_0}^2
d\hat{\mu}_{\la,T_i',W_0}\nn\\
&\quad
\ge
\la^{-1}\int_{W_0}\|D_0
P_{\la,i',(F\cup\{e_{j_k}\}_{k=1}^K)^\complement}
\left(
\widetilde{f\chi_{\frac{1}{3},i',m,\theta,\la,\delta}}
\right)
\|_{H_0}^2
d\hat{\mu}_{\la,T_{i'},W_0}\nn\\
&\quad \ge R
\left(
\|\widetilde{(f\chi_{\frac{1}{3},i',m,\theta,\la,\delta})}\|
_{L^2(W_0,\hat{\mu}_{\la,T_i',W_0})}^2
-\left(O(C(\la)^{1/4})+O(\la^{-\frac{1-\delta}{2}})\right)
\right).
\end{align*}
where we have used there exists $\ep>0$ such that
\begin{align*}
 \inf\left\{\sigma(-\la^{-1}L_{\la,T_{i'},W_0})
\cap (F\cup \{e_{j_k}\}_{k=1}^K)^\complement\right\}\ge R.
\end{align*}
Hence
\begin{align}
&\la^{-1}
\E_{\la,\YD}(f\chi_{\frac{1}{3},i',m,\theta,\la,\delta},
f\chi_{\frac{1}{3},i',m,\theta,\la,\delta})\nn\\
&\ge 
R
(1-O(\la^{-\delta/2}))
\|f\chi_{\frac{1}{3},i',m,\theta,\la,\delta}\|^2
_{L^2(\YD,\mu_{\la,\YD})}
-R\left(O(C(\la)^{1/4})+O(\la^{-\frac{1-\delta}{2}})\right).
\label{lower bound of Ei'}
\end{align}

We now apply the IMS localization formula to finish the proof of
(i).
Let 
\begin{align*}
 \chi_{\frac{1}{3},o, m,\theta,\la,\delta}(w)=
\left(1-\sum_{i=1}^N
\chi_{\frac{1}{3},i,m,\theta,\la,\delta}(w)^2\right)^{1/2}.
\end{align*}
This function takes the value 1 outside neighborhood of
$\{l_{\xi_i}\}_{i=1}^N$ and
\begin{align*}
 \|D\chi_{\frac{1}{3},o,m,\theta,\la,\delta}(w)\|_H\le
C\la^{\delta/2}
\end{align*}
which follows Lemma~\ref{cut-off function} and
explicit computation.
By applying Corollary~\ref{lower bound of an energy of f} 
to the case $f=(f\chi_{\frac{1}{3},o, m,\theta,\la,\delta})\circ Y^{-1}$
and $\rho=\rho_{\kappa}$ which is defined in Lemma~\ref{exponential estimate},
we have
\begin{align*}
& \la^{-1}\E_{\la,\YD}(f\chi_{\frac{1}{3},o, m,\theta,\la,\delta},
f\chi_{\frac{1}{3},o, m,\theta,\la,\delta})\nn\\
&=\la^{-1}\E_{\la,\D}\left(
(f\chi_{\frac{1}{3},o, m,\theta,\la,\delta})\circ Y^{-1},
(f\chi_{\frac{1}{3},o, m,\theta,\la,\delta})\circ Y^{-1}
\right)\nn\\
&\ge -C\log\left(
\int_{\D}
\exp\Bigl(C_{\la}\la^2\left(V_{\la,a}+\rho_{\kappa}\right)\Bigr)
d\nu_{\la,\D}\right) 
\|f\chi_{\frac{1}{3},o, m,\theta,\la,\delta}\|_{L^2(\YD,\mu_{\la,\YD})}^2\nn\\
&\quad+\la
\int_{\YD}(\rho_{\kappa}(Y(w))-e^{-C'\la})
(f\chi_{\frac{1}{3},o, m,\theta,\la,\delta})^2(w)
d\mu_{\la,\D}(w)\nn\\
&\ge C\left(\kappa \la^{1-\delta}-1\right)
\|f\chi_{\frac{1}{3},o, m,\theta,\la,\delta}\|_{L^2(\YD,\mu_{\la,\YD})}^2,
\end{align*}
where we have used $\rho_{\kappa}(Y(w))\ge C\kappa\la^{-\delta}$ on
$\{\chi_{\frac{1}{3},o,m,\theta,\la,\delta}\ne 0\}$ and
Lemma~\ref{exponential estimate}.
Taking the above estimates into account,
we now apply the IMS localization formula to obtain
\begin{align*}
&\la^{-1}\E_{\la,\YD}(f,f)\nn\\
&=\la^{-1}\sum_{i=1}^N
\E_{\la,\YD}(f\chi_{\frac{1}{3},i,m,\theta,\la,\delta},
f\chi_{\frac{1}{3},i,m,\theta,\la,\delta})
-\la^{-1}\sum_{i=1}^N\int_{\YD}|D_{S_a}
\chi_{\frac{1}{3},i,m,\theta,\la,\delta}|_H^2
f(w)^2d\mu_{\la,\D}\nn\\
&\, 
+\la^{-1}\E_{\la,\YD}(f\chi_{\frac{1}{3},o,m,\theta,\la,\delta},
f\chi_{\frac{1}{3},o,m,\theta,\la,\delta})
-\la^{-1}\int_{\YD}|D_{S_a}\chi_{\frac{1}{3},i,m,\theta,\la,\delta}|_H^2
f(w)^2d\mu_{\la,\D}\nn\\
&\ge \left(\min\left[\left\{
R(1-O(\la^{-\delta/2})-N O(C(\la)^{1/4})-N O(\la^{-\frac{1-\delta}{2}}))
\right\},
C(\kappa\la^{1-\delta}-1)\right]
-CN\la^{\delta-1}
\right)\nn\\
&\quad \times\|f\|_{L^2(\YD,\mu_{\la,\YD})}^2,
\end{align*}
which contradicts with (\ref{upper bound of Ef}).

We now prove (ii).
Suppose there exists $\beta$ which belongs to the both set.
Then for any $\kappa>0$, there exists a unit vector $f_{\la,\kappa}$ such that
$P_{\la,[\beta-\kappa,\beta+\kappa]}f_{\la,\kappa}=f_{\la,\kappa}$
which implies $\la^{-1}\E_{\la,\YD}(f_{\la,\kappa},f_{\la,\kappa})\le 
(R+\kappa)$.
We already see the existence of discrete eigenvalues $e_j(\la)$
which is very close to $e_j$ and $|\beta-e_j(\la)|\ge \frac{1}{2}C(\la)^{1/2}$.
By setting $\kappa=\frac{1}{3}C(\la)^{1/2}$, we see
that $f_{\la,\kappa}$ is orthogonal with the eigenfunctions
corresponding to any $e_j(\la)$.
Arguing similarly to (\ref{norm i'il}),
we have for any $i,i'$ and $l$
\begin{align*}
 (\widetilde{f_{\la,\kappa}\chi_{\frac{1}{3},i',m,\theta,\la,\delta}},
\be^i_{l,\la})_{L^2(W_0,\hat{\mu}_{\la,T_i,W_0})}=O(C(\la)^{1/4}).
\end{align*}
On the other hand, by Lemma~\ref{projection on F},
we have
 \begin{align*}
\|P_{\la,i,F}\widetilde{(f_{\la,\kappa}
\chi_{\frac{1}{3},i,m,\theta,\la,\delta})}\|^2
_{L^2(W_0,\hat{\mu}_{\la,T_i,W_0})}\le  2(r-r_-)^{-1}
C_{i,\beta,R}\, (C(\la)^{1/2}+\la^{-\frac{1-\delta}{2}}).
\end{align*}
Note that for all $1\le i\le N$, there exists $\ep>0$ such that
\begin{align*}
 \sigma(-\hat{L}_{\la,T_i,W_0})\cap 
\left(\left(\cup_{i=1}^N(\Lambda(\xi_i)^a+(-r_-,r_-))\right)
\cup \{e_j\}_{j=1}^L\right)^\complement
\ge R+\ep.
\end{align*}
Combining the above three estimates, we can get an
lower bound estimate for\\
$\E_{\la,\YD}(f\chi_{\frac{1}{3},i',m,\theta,\la,\delta},
f\chi_{\frac{1}{3},i',m,\theta,\la,\delta})$
similarly to
(\ref{lower bound of Ei'})
replacing $R$ by $R+\ep$.
However, similarly to the proof of (i), this contradicts with
the result $\la^{-1}\E_{\la,\YD}(f_{\la,\kappa},f_{\la,\kappa})\le (R+\kappa)$.
This completes the proof.
\end{proof}

\section{Appendix}

We prove Proposition~\ref{example of discrete spectrum}.

\begin{proof}
By Lemma~\ref{accumulation point} (1), it suffices to prove (\ref{criteria}).
First, we prove that (\ref{criteria}) holds if $k\ne l$.
We consider the case $k\ge 0$ and for $M\in \NN$ set
\begin{align*}
 E&:=E_0(\theta(k))+M\left(1+\frac{2(\theta+k)}{p}\right)\\
&=4\left(k\sum_{m=2}^{2k}\frac{1}{m}+\theta\sum_{m=1}^{2k}\frac{1}{m}\right)
+M\left(1+\frac{2(\theta+k)}{p}\right).
\end{align*}
Let $l\in \ZZ_{\ge 0}$.
For $n_m, n'_m, N'_m, N_m\in \ZZ_{\ge 0}$ and
$n\in \NN$, let
\begin{align*}
 E'&:=E_0(\theta(l))+\sum_{m=1}^{2l}\left(\frac{2(\theta+l)}{m}-1\right)n_m+
\sum_{m=1}^{2l}\left(\frac{2(\theta+l)}{m}+1\right)n_m'\\
&\quad+\sum_{m\ge 2l+1}\left(1+\frac{2(\theta+l)}{m}\right)N'_m
+\sum_{m\ge 2l+1}\left(1-\frac{2(\theta+l)}{m}\right)N_m+n.
\end{align*}
Suppose $E=E'$.
Then, comparing the coefficient of $2\theta$, we have
\begin{align*}
 2\sum_{m=1}^{2l}\frac{1}{m}+\sum_{m=1}^{2l}\frac{1}{m}(n'_m+n_m)+
\sum_{m\ge 2l+1}\frac{1}{m}(N'_m-N_m)=
2\sum_{m=1}^{2k}\frac{1}{m}+\frac{M}{p}
\end{align*}
and hence
\begin{align}
 \sum_{m=1}^{2l}\frac{1}{m}(n'_m+n_m+2)+
\sum_{m\ge 2l+1}\frac{1}{m}(N'_m-N_m)=
2\sum_{m=1}^{2k}\frac{1}{m}+\frac{M}{p}.\label{coefficient of alpha}
\end{align}
Comparing the other coefficient,
\begin{align*}
 &4l\sum_{m=2}^{2l}\frac{1}{m}+
\sum_{m=1}^{2l}\left(\frac{2l}{m}-1\right)n_m
+\sum_{m=1}^{2l}\left(\frac{2l}{m}+1\right)n_m'\\
&\quad
+\sum_{m\ge 2l+1}\left(\frac{2l}{m}+1\right)N'_m
+\sum_{m\ge 2l+1}\left(-\frac{2l}{m}+1\right)N_m+n
=4k\sum_{m=2}^{2k}\frac{1}{m}+M\left(1+\frac{2k}{p}\right)
\end{align*}
and so
\begin{align*}
& 4l\sum_{m=2}^{2l}\frac{1}{m}+2l\sum_{m=1}^{2l}\frac{1}{m}(n_m+n'_m)
+2l\sum_{m\ge 2l+1}\frac{1}{m}(N'_m-N_m)
+\sum_{m=1}^{2l}(n'_m-n_m)+\sum_{m\ge 2l-1}(N'_m+N_m)+n\\
&\qquad =4k\sum_{m=2}^{2k}\frac{1}{m}+M\left(1+\frac{2k}{p}\right).
\end{align*}
Hence
\begin{align}
& 2l\left\{
\sum_{m=1}^{2l}\frac{1}{m}(n_m+n'_m+2)
+\sum_{m\ge 2l+1}\frac{1}{m}(N'_m-N_m)
\right\}
+\sum_{m=1}^{2l}(n'_m-n_m)+\sum_{m\ge 2l-1}(N'_m+N_m)+n-4l\nonumber\\
&\qquad =2k\left(2\sum_{m=2}^{2k}\frac{1}{m}+\frac{M}{p}\right)+M.
\label{coefficient of the other}
\end{align}
Calculating 
$(\ref{coefficient of alpha})\times 2k-(\ref{coefficient of the other})$,
\begin{align}
& 2(k-l)\left\{
\sum_{m=1}^{2l}\frac{1}{m}(n_m+n'_m+2)
+\sum_{m\ge 2l+1}\frac{1}{m}(N'_m-N_m)
\right\}\nonumber\\
&\quad -\sum_{m=1}^{2l}(n'_m-n_m)
-\sum_{m\ge 2l-1}(N'_m+N_m)-n+4l\nonumber\\
&\qquad =4k-M.\label{pm-qm}
\end{align}
Consequently,
if $k\ne l$, by dividing the both sides of (\ref{pm-qm}) by
$2(k-l)$, we obtain
\begin{align*}
& \sum_{m\ge 2l+1}\frac{1}{m}(N'_m-N_m)\nonumber\\
&\quad=
-\sum_{m=1}^{2l}\frac{1}{m}(n_m+n'_m+2)+
\frac{1}{2(k-l)}\left\{
\sum_{m=1}^{2l}(n'_m-n_m)+\sum_{m\ge 2l-1}(N'_m+N_m)+n-4l
+4k-M
\right\}.
\end{align*}
Since $\sum_{m\ge 2l+1}\frac{1}{m}(N'_m-N_m)$ is a rational number, 
it can be written 
as a reduced fraction $\frac{A}{B}$ where $A$ and $B$ are integers.
By the assumption on $2k, 2l$, we can conclude that 
$B$ does not have $p$ as a divisor.
But, this contradicts with the expression (\ref{coefficient of alpha}).
Next we consider the case $l<0$.
We set
\begin{align*}
 E'&:=E_0(\theta(l))+\sum_{m=1}^{2|l|-1}
\left(\frac{2(|l|-\theta)}{m}-1\right)n_m+
\sum_{m=1}^{2|l|-1}\left(\frac{2(|l|-\theta)}{m}+1\right)n_m'\\
&\quad+\sum_{m\ge 2|l|}\left(1+\frac{2(|l|-\theta)}{m}\right)N'_m
+\sum_{m\ge 2|l|}\left(1-\frac{2(|l|-\theta)}{m}\right)N_m+n.
\end{align*}
Suppose $E'=E$ holds.
Then comparing the coefficient of $2\theta$,
\begin{align}
 -\sum_{m=1}^{2|l|-1}\frac{1}{m}(n_m+n_m'+2)
-\sum_{m\ge 2|l|}\frac{1}{m}(N'_m-N_m)
=2\sum_{m=1}^{2k}\frac{1}{m}+\frac{M}{p}.\label{coefficient of alpha2}
\end{align}
As to the other coefficient,
\begin{align*}
& 2\sum_{m=1}^{2|l|-1}\left(\frac{2|l|}{m}-1\right)+
 \sum_{m=1}^{2|l|-1}\left(\frac{2|l|}{m}-1\right)n_m
 +\sum_{m=1}^{2|l|-1}\left(\frac{2|l|}{m}+1\right)n_m'\nonumber\\
&\quad +\sum_{m\ge 2|l|}\left(1+\frac{2|l|}{m}\right)N'_m+
\sum_{m\ge 2|l|}\left(1-\frac{2|l|}{m}\right)N_m+n\nonumber\\
&\qquad=4k\sum_{m=2}^{2k}\frac{1}{m}+M\left(1+\frac{2k}{p}\right)
\end{align*}
and so
\begin{align}
 &2|l|\left\{\sum_{m=1}^{2|l|-1}\frac{1}{m}(n_m+n_m'+2)
+\sum_{m\ge 2|l|}\frac{1}{m}(N'_m-N_m)\right\}\nonumber\\
&\quad +\sum_{m=1}^{2|l|-1}(n'_m-n_m)+\sum_{m\ge 2|l|}(N'_m+N_m)-2(2|l|-1)+n\nonumber\\
&=4k\sum_{m=1}^{2k}\frac{1}{m}-4k+M\left(1+\frac{2k}{p}\right).
\label{coefficient of the other2}
\end{align}
Calculating 
$(\ref{coefficient of the other2})-(\ref{coefficient of alpha2})\times 2k$,
we obtain
\begin{align*}
& \sum_{m\ge 2|l|}\frac{1}{m}(N'_m-N_m)\nonumber\\
&=-\sum_{m=1}^{2|l|-1}\frac{1}{m}(n_m+n_m'+2)\nonumber\\
&\qquad -
\frac{1}{2(k+|l|)}\left(
\sum_{m=1}^{2|l|-1}(n'_m-n_m)+\sum_{m\ge 2|l|}(N'_m+N_m)-2(2|l|-1)+n
-M+4k\right).
\end{align*}
Again, by the proof of contradiction, we can conclude that
$E=E'$ dose not hold.
So far, we consider the case where $k\ge 0$.
The above proof works for the case $k\le -1$.
Consequently, we need only to consider the case $k=l$.

Assume there exists $k\in \ZZ$ and
$n_m, n'_m, N^+_m, N^-_m\in \ZZ_{\ge 0}$
such that
\begin{align}
& E_0(\theta(k))+M\left(1+\frac{2|\theta+k|}{p}\right)\nonumber\\
&=E_0(\theta(k))+\sum_{m=1}^{(|2k|\vee |2k+1|)-1}n_m
\left(\frac{2|\theta+k|}{m}-1\right)
+\sum_{m=1}^{(|2k|\vee |2k+1|)-1}n_{m}'
\left(\frac{2|\theta+k|}{m}+1\right)\nonumber\\
&\quad
+\sum_{|2k|\vee |2k+1|\le m\le K'}N'_m(1+\frac{2|\theta+k|}{m})
+\sum_{|2k|\vee |2k+1|\le m\le K}N_m(1-\frac{2|\theta+k|}{m})+n.
\label{identity of eigenvalue2}
\end{align}
Since $\theta$ is an irrational number and $p$ is a prime number, we get
\begin{align}
\sum_{j=1}^{J'}\frac{N'_{m'_{j}}}{m'_{j}}-\sum_{j=1}^{J}\frac{N_{m_j}}{m_j}
=
\frac{M}{p},\label{prime number}
\end{align}
where 
$\{m'_{j}\}_{j}$ and $\{m_j\}_j$ are all numbers 
of $\{|2k|\vee |2k+1|,\ldots,K'\}$
and $\{|2k|\vee |2k+1|,\ldots, K\}$ such that
the denominators of the reduced fractions of
$\frac{N'_{m'_{j}}}{m'_{j}}$ and $\frac{N_{m_j}}{m_j}$
have $p$ as a divisor
respectively.
This follows from the following.
Let $a$ be positive integer,
$\frac{A_i}{B_i}$ $(A_i, B_i\in \ZZ, 1\le i\le 3)$ are reduced fractions 
and suppose $B_1$ and $A_2$ do not have $p$ as a divisor.
Further suppose it holds that
\[
 \left(\frac{A_1}{B_1}+\frac{A_2}{p^aB_2}\right)\theta+\frac{A_3}{B_3}=0.
\]
Then $A_1=A_2=A_3=0$ holds.
We go back to the equation (\ref{prime number}).
This implies $\sum_{j=1}^{J'}N'_{m'_j}\ge M$.
Hence
\begin{align*}
 \text{The RHS of (\ref{identity of eigenvalue2})}
\ge E_0(\theta(k))+M+n.
\end{align*}
This implies $\frac{2M|\theta+k|}{p}\ge n$.
But this contradicts with the assumption $p>2M|\theta+k|$
and hence (\ref{identity of eigenvalue2}) does not hold.
\end{proof}

\end{document}